\documentclass[12pt]{memoir} 
\usepackage{amsmath,amsthm,graphicx}

\usepackage{amssymb}

\usepackage
[backref,hyperindex,colorlinks,citecolor=blue,pdftex,pdfpagelabels,plainpages=false]{hyperref}

\setlrmargins{*}{*}{1}
\checkandfixthelayout

\counterwithout{section}{chapter}
\counterwithout{figure}{chapter}
\firmlists

\usepackage[all]{hypcap} 

\newcommand{\mathboldit}[1]{\boldsymbol{#1}}
\newcommand{\myTimes}{
\usepackage{mathptmx}
\DeclareSymbolFont{letters}{OML}{txmi}{m}{it}
\DeclareMathAlphabet{\mathbit}{OT1}{ptm}{bx}{it}
\renewcommand{\mathboldit}[1]{\mathbit{##1}}
}

\myTimes

\newcommand{\add}[1] {{\color{blue}{#1}}}

 \newcommand{\df}[1]{\emph{#1}}

\renewcommand{\le}{\leqslant}
\renewcommand{\ge}{\geqslant}

\newcommand{\txt}[1]{\text{\rmfamily\mdseries\upshape{#1}}}

\newenvironment{bullets}{\begin{enumerate}[\textbullet]}{\end{enumerate}}

\newcommand{\customqed}[1]{{\renewcommand{\qedsymbol}{#1}\qed}}
\newcommand{\varqed}{\customqed{\hbox{$\lrcorner$}}}

\newcommand{\intheoremstyle}{
 \newtheorem{lemma}{Lemma}[section]

\newtheorem{theorem}[lemma]{Theorem}
 \newtheorem{proposition}[lemma]{Proposition}
 
 \newtheorem{corollary}[lemma]{Corollary}

 \newtheorem{question}{Question}
}

\newcommand{\indefinitionstyle}{
\newtheorem{Definition}[lemma]{Definition}
\newenvironment{definition}{\begin{Definition}}{\varqed\end{Definition}}

 \newtheorem{Notation}[lemma]{Notation}
  \newenvironment{notation}{%
   \begin{Notation}}{\varqed\end{Notation}}

 \newtheorem{Condition}{Condition}[section]

}

\newcommand{\standardTheorems}{
\theoremstyle{plain}
\intheoremstyle
\theoremstyle{definition}
\indefinitionstyle
}

\newcommand{\restofTheorems}{
 \theoremstyle{remark}
 \newtheorem{Remark}[lemma]{Remark}
 \newtheorem{Remarks}[lemma]{Remarks}

 \newenvironment{remark}{%
   \begin{Remark}}{\varqed\end{Remark}}
 \newenvironment{remarks}{%
   \begin{Remarks}}{\varqed\end{Remarks}}

 \newtheorem{Example}[lemma]{Example}
 \newtheorem{Examples}[lemma]{Examples}

 \newenvironment{example}{%
   \begin{Example}}{\varqed\end{Example}}
 \newenvironment{examples}{%
   \begin{Examples}}{\varqed\end{Examples}}
} 

\standardTheorems
\restofTheorems

\newenvironment{alphenum} 
 {\begin{enumerate}[{\upshape (a)}]}{\end{enumerate}}

\newcommand{\opint}[2]{(#1,#2)}
\newcommand{\clint}[2]{[#1,#2]}
\newcommand{\lint}[2]{[#1,#2)}

\newcommand{\pair}[2]{(#1,#2)}

\newcommand{\lea}{\stackrel{{}_+}{<}}
\newcommand{\gea}{\stackrel{{}_+}{>}}
\newcommand{\eqa}{\stackrel{{}_+}{=}}

\newcommand{\lem}{\stackrel{{}_*}{<}}
\newcommand{\gem}{\stackrel{{}_*}{>}}
\newcommand{\eqm}{\stackrel{{}_*}{=}}

\newcommand\evof[1]{\mathopen[\,#1\,\mathclose]}
\newcommand\Paren[1]{{\left( #1\right)}}
\newcommand\setOf[2]{\mathopen\{\,#1 : #2\,\mathclose\}}
\newcommand\set[1]{\mathopen\{#1\mathclose\}}
\newcommand\tup[1]{( #1)}

\newcommand{\cei}[1]{{\lceil #1\rceil}}
\newcommand{\flo}[1]{{\lfloor #1\rfloor}}

\newcommand{\bB}{\mathbf{B}}
\newcommand{\bM}{\mathbf{M}}
\newcommand{\bX}{\mathbf{X}}
\newcommand{\bY}{\mathbf{Y}}
\newcommand{\bZ}{\mathbf{Z}}

\newcommand{\cA}{\mathcal{A}}
\newcommand{\cB}{\mathcal{B}}
\newcommand{\cC}{\mathcal{C}}
\newcommand{\cE}{\mathcal{E}}
\newcommand{\cM}{\mathcal{M}}
\newcommand{\cS}{\mathcal{S}}

\newcommand{\bbB}{\mathbb{B}}
\newcommand{\bbN}{\mathbb{N}}

\newcommand{\bbR}{\mathbb{R}}

\newcommand{\prefix}{\sqsubseteq}
\newcommand{\postfix}{\sqsupseteq}

\renewcommand{\b}{\mathbf{b}}

\newcommand{\KP}{\mathit{Kp}}
\newcommand{\Km}{\mathit{Km}}
\newcommand{\KM}{\mathit{KM}}
\renewcommand{\d}{\mathbf{d}}
\newcommand{\m}{\mathbf{m}}
\newcommand{\M}{M}
\renewcommand{\P}{P} 

\newcommand{\Q}{Q} 
\renewcommand{\r}{r} 
\newcommand{\R}{R} 
\newcommand{\s}{s}
\renewcommand{\t}{\mathbf{t}}

\newcommand{\Y}{Y}

\newcommand{\len}[1]{|#1|}

\newcommand{\Rands}{\mathrm{Randoms}}
\newcommand{\aver}{\txt{aver}}
\newcommand{\prob}{\txt{prob}}

\begin{document}

\title{Algorithmic tests and randomness with respect to a class
      of measures}

\author{Laurent Bienvenu\thanks{%
     LIAFA, CNRS \& Universit\'e Paris Diderot, Paris 7, Case 7014, 75205
     Paris Cedex 13, France,
     e-mail: \texttt{Laurent dot Bienvenu at liafa dot jussieu dot fr}},
Peter G\'acs\thanks{%
     Department of Computer Science, Boston University,
     111 Cummington st., Room 138, Boston, MA 02215,
     e-mail: \texttt{gacs at bu dot edu}},
Mathieu Hoyrup\thanks{%
     LORIA -- B248, 615, rue du Jardin Botanique, BP 239, 54506
     Vand\oe uvre-l\`es-Nancy, France,
     e-mail: \texttt{Mathieu dot Hoyrup at loria dot fr}},\\
Cristobal Rojas\thanks{%
     Department of Mathematics, University of Toronto,
     Bahen Centre, 40 St. George St., Toronto, Ontario, Canada, M5S 2E4,
     e-mail: \texttt{crojas at math dot utoronto dot ca}},
Alexander Shen\thanks{%
     LIF, Universit\`e Aix -- Marseille, CNRS,
     39, rue Joliot-Curie, 13453 Marseille cedex 13, France,
     on leave from IITP RAS, Bolshoy Karetny, 19, Moscow.
     Supported by NAFIT ANR-08-EMER-008-01, RFBR 0901-00709-a
     grants. e-mail: \texttt{sasha dot shen at gmail dot com}.}
}

\maketitle
\thispagestyle{empty}

\begin{abstract}
This paper offers some new results on randomness with respect to classes of
measures, along with a didactical exposition of their context based on
results that appeared elsewhere.

We start with the reformulation of the Martin-L\"of definition
of randomness (with respect to computable measures) in terms of
randomness deficiency functions.
A formula that expresses the randomness deficiency in
terms of prefix complexity is given (in two forms).
Some approaches that go in another direction (from deficiency to
complexity) are considered.

The notion of Bernoulli randomness (independent coin tosses
for an asymmetric coin with some probability $p$ of head) is defined.
It is shown that a sequence is Bernoulli if it is random with respect to
\emph{some} Bernoulli measure $B_p$.
A notion of ``uniform test'' for Bernoulli sequences is
introduced which allows a quantitative strengthening of this result.
Uniform tests are then generalized to arbitrary measures.

Bernoulli measures $B_{p}$ have the important property that
$p$ can be recovered from each random sequence of $B_{p}$.
The paper studies some important consequences of this
orthogonality property (as well as most other questions mentioned above)
also in the more general setting of constructive metric spaces.
\end{abstract}

\section{Introduction}

This paper, though intended to be rather self-contained,
can be seen as a continuation of~\cite{GacsUnif05} (which itself
built on earlier work of Levin) and~\cite{HoyrupRojasRandomness09}.

Our enterprise is to develop the theory of
randomness beyond the framework where the underlying
probability distribution is the uniform distribution or a computable
distribution.
A randomness test $\t(\omega,\P)$ of object $\omega$ with respect to measure $\P$
is defined to be a function of both the measure $\P$ and the point $\omega$.

In some later parts of the paper, we will also go beyond the case where the
underlying space is the set of finite or infinite sequences: rather,
we take a constructive metric space with its algebra of Borel sets.

We will apply the above notion of test to define, following ideas
of~\cite{LevinRand73}, for a class $\cC$
of measures having some compactness property,
a ``class test'' $\t_{\cC}(\omega)$.
This is a test to decide whether object $\omega$ is random with
respect to any one measure $\P$ in the class $\cC$.
We will show that in case of the class of Bernoulli measures over binary
sequences, this notion is equivalent to the class tests introduced by
Martin-L\"of in~\cite{MLof66art}.

In case there is an effective sense in which the
elements of the class are mutually orthogonal, we obtain an especially simple
separation of the randomness test $\t(\omega,\P)$ into two parts: the class test
and an arbitrarily simple test for ``typicality'' with respect to the measure
$\P$.
In some natural special cases, the typicality test corresponds to a
convergence property of relative frequencies, allowing to apply the theory to
any general effectively compact class of ergodic stationary processes.

There are some properties of
randomness tests $\t(\omega,\P)$ that depend on the measure $\P$, which our tests
do not necessarily possess, for example a kind of monotonicity in $\P$.
It is therefore notable that in case of the orthogonal classes,
randomness is equivalent to an ``blind'' notion of
randomness, that only considers randomness tests that do not depend on the
measure $\P$.

Here is an outline of the paper.
We start with the reformulation of the Martin-L\"of definition
of randomness (with respect to computable measures) in terms of
tests.
A randomness test provides a quantitative measure of
non-randomness, called ``randomness deficiency''; it is finite
for random sequences and infinite for non-random ones.
There are two versions of these tests (``average-bounded'' and
``probability-bounded'' ones); a relation between them is established.

A formula that expresses the (average-bounded) randomness deficiency in
terms of prefix complexity is given (in two forms).
It implies the Levin-Schnorr criterion of randomness (with prefix complexity,
as in the special case first announced in Chaitin's paper~\cite{Chaitin75}).
Some approaches that go in another direction (from deficiency to
complexity) are considered.

The notion of Bernoulli sequence (looking like the outcome of
independent coin tosses for an asymmetric coin) is defined.
It is shown that the set of Bernoulli sequences is the union
(over all $p\in\clint{0}{1}$) of the sets of sequences that are random
with respect to $B_p$, the Bernoulli measure with probability
$p$; here we assume that $p$ is given as an oracle).
A notion of ``uniform test'' for Bernoulli sequences is
introduced.
Then the statement above is proved in the following
quantitative form: the Bernoulli deficiency is the infimum of
$B_{p}$ deficiencies over all $p\in\clint{0}{1}$.

The notion of general uniform test (not restricted to the class
of Bernoulli measures) is introduced.
It is shown that it generalizes Martin-L\"ofs
earlier definition of randomness (which was given only for computable
measures).

Bernoulli measures $B_{p}$ have the important property that
$p$ can be recovered from each random sequence of $B_{p}$.
The paper studies some important consequences of this
orthogonality property (as well as most other questions mentioned above)
also in the more general setting of constructive metric spaces.

The following notation is useful, since inequalities hold frequently only within
an additive or multiplicative constant.
\begin{notation}
We will write $f(x)\lem g(x)$ for inequality
between positive functions within a multiplicative constant,
that is for the relation $f(x)=O(g(x))$: precisely, if
there is a constant $c$ with $f(x)\le c g(x)$ for all $x$.
The relation $f\eqm g$ means $f\lem g$ and $f\gem g$.
Similarly, $f\lea g$ and $f\eqa g$
means inequality within an additive constant.

Let $\Lambda$ denote the empty string.

Logarithms are taken, as a default, to base 2.
We use $\len{x}$ to denote the length of a string $x$.
For finite string, $x$ and finite or infinite string $y$ let
$x\prefix y$ denote that $x$ is a prefix of $y$.
If $x$ is a finite or infinite sequence then its elements are written
as $x(1), x(2),\dots$, and its prefix of size $n$ will be
denoted by $x(1:n)$.

Let $\overline\bbR_{+}=\clint{0}{\infty}$ be the set of nonnegative reals,
with the special value $\infty$ added.
The binary alphabet $\{0,1\}$ will also be denoted by $\bbB$.
\end{notation}

\section{Randomness on sequences, for computable measures}

\subsection{Lower semicomputable functions on sequences}

In the first sections, we will study randomness over infinite binary sequences.

\begin{definition}[Binary Cantor space, Baire space]\label{def:binary-Cantor}
  We will denote by $\Omega$ the set of infinite binary sequences,
and call it also the \df{binary Cantor space}.
For a finite string $x$ let $x\Omega$ be the set of all infinite sequences that have
finite prefix $x$.
These sets will be called \df{basic open sets}, the set of all basic open set is
called the \df{basis} of $\Omega$ (as a topological space).
A subset of $\Omega$ is \df{open} if it is the union of a set of basis elements.

The set of infinite sequences of natural numbers will be called the \df{Baire
  space}.
Basic open sets and open sets can be defined for it analogously.
\end{definition}

A notion somewhat weaker than computability will play crucial role.

\begin{definition}\label{def:lower-semicomp.seqs}
An open set $G\subseteq\Omega$ is called \df{effectively open}, or
\df{lower semicomputable open}, or \df{c.e.~open}, or \df{r.e.~open} if it is
the union of a computable sequence $x_{i}\Omega$ of basic elements.
A set is \df{upper semicomputable closed}, or \df{effectively closed} if its
complement is effectively open.

A set $\Gamma$ is called \df{effectively $G_{\delta}$} if there is a sequence of
sets $U_{k}$, $k=1,2,\dots$   effectively open uniformly in $k$
such that $\Gamma=\bigcap_{k}U_{k}$.

A function $t:\Omega\to\clint{0}{\infty}$ is \df{lower semicomputable} if
\begin{alphenum}
\item\label{i:lower-semicomp.seqs.lower-semicont}
 For any rational $r$ the set
        $
\{\omega : r < t(\omega)\}
        $
is open in $\Omega$, that is is a union of intervals $x\Omega$.

\item Moreover, this set is effectively open uniformly in $r$,
that is there exists an algorithm that gets $r$ as input and
generates strings $x_0,x_1,\ldots$ such that the union of
interval $x_{i}\Omega$ is equal to $\{\omega: r<t(\omega)\}$.
\end{alphenum}
\end{definition}

This definition is a constructive version of the classical
notion of lower semicontinuous function as in
requirement~(\ref{i:lower-semicomp.seqs.lower-semicont}).
The same class of lower semicomputable functions has other
(equivalent) definitions; here is one of them.

\begin{definition}\label{def:basic-func.seqs}
A function $u$ defined on $\Omega$ and having rational values is
called \df{basic} if the value $u(\omega)$ is determined by
some finite prefix of $\omega$.
\end{definition}

If this prefix has length $N$,
the function can be presented as a table with $2^N$ rows; each
row contains $N$ bits (the values of the first $N$ bits of
$\omega$) and a rational number (the value of the function).
Such a function is a finite object.

The proof of the following proposition is a simple exercise:

\begin{proposition}\label{propo:lower-semi-limit.seqs}
The (pointwise) limits of monotonic sequences of basic functions
are exactly the lower semicomputable functions on $\Omega$.
\end{proposition}

Since the difference of two basic functions is a basic function,
we can reformulate this criterion as follows: lower
semicomputable functions are (pointwise) sums of computable
series made of non-negative basic functions.

One more way to define a lower semicomputable function goes as
follows.

\begin{definition}[Generating]\label{def:generate-lower}
Let $T$ be a lower semicomputable function on the set
$\{0,1\}^{*}$ of finite sequences of zeros and ones with
non-negative (finite or infinite) values.
This means that the
set of pairs $\pair{x}{r}$ such that $r<T(x)$ is enumerable.
Then function $t$ defined as
\begin{equation*}
t(\omega)=\sup_{x\prefix\omega}T(x)
\end{equation*}
is a lower semicomputable function on~$\Omega$: we will say that
function $T(\cdot)$ \df{generates} function $t(\cdot)$ if it is also monotone:
$T(x)\le T(y)$ if $x\prefix y$.
\end{definition}

The monotonicity requirement can always be satisfied by taking
$T'(x)=\max_{z\prefix x}T(z)$.

\begin{proposition}\label{propo:generate-lower}
Any lower semicomputable function $t$ on $\Omega$ is generated by an appropriate
function $T$ on $\{0,1\}^{*}$ this way.
\end{proposition}

We may also assume that $T$ is a computable function with rational values.
Indeed, since only the supremum
of $T$ on all the prefixes is important, instead of
increasing $T(x)$ for some $x$ we may increase $T(y)$ for all
$y\postfix x$ of large length; this delay allows $T$ to be computable.

For a given lower semicomputable function
$t$ on $\Omega$ there exists a maximal monotonic function $T$ on finite strings
that generates $t$ (in the sense just described).
This maximal $T$ can be defined as follows:
\begin{equation}\label{eq:T-as-inf}
T(x)=\inf_{\omega\postfix x} t(\omega).
\end{equation}
Let us now exploit the finiteness of the binary alphabet $\{0,1\}$, which
implies that the space $\Omega$ is a compact topological space.

\begin{proposition}\label{propo:inf-lower-semicomp.seqs}
The function $T$ defined by~\eqref{eq:T-as-inf} is lower semicomputable.
In the definition, we can replace $\inf$ by $\min$.
\end{proposition}
\begin{proof}
Indeed, $r<\inf_{\omega\postfix x} t(\omega)$ if and
only if there exists some rational $r'>r$ with $r'<t(\omega)$ for
all $\omega\postfix x$.
The last condition can be
reformulated: the open set of all sequences $\omega$ such that
$t(\omega)>r'$ is a superset of $x\Omega$.
This open set is a
union of an enumerable family of intervals; if these intervals
cover $x\Omega$, compactness implies that
this is revealed at some finite stage, so the
condition is enumerable (and the existential quantifier over $r'$
keeps it enumerable).

Since the function $t(\omega)$ is lower semicontinuous, it
actually reaches its infimum on the compact set
$x\Omega$, so $\inf$ can be replaced with $\min$.
\end{proof}

\subsection{Randomness tests}

We assume that the reader is familiar with the basic concepts of measure theory
and integration, at least in the space $\Omega$ of infinite binary sequences.
A measure $\P$ on $\Omega$ is determined by the values
 \begin{align*}
  \P(x)=\P(x\Omega)
 \end{align*}
which we will denote by the same letter $P$, without danger of confusion.
Moreover, any function $\P:\{0,1\}^{*}\to\clint{0}{1}$ with the properties
 \begin{align}\label{eq:measure.Omega}
  \P(\Lambda)=1,\quad\P(x)=\P(x0)+\P(x1)
 \end{align}
uniquely defines a measure (this is a particular case of Caratheodory's theorem).

\begin{definition}[Computable measure]\label{def:computable-measure-Omega}
A real number is called \df{computable} if there is an algorithm that for all
rational $\varepsilon$ returns a rational approximation of $x$ with error not greater
than $\varepsilon$.
Computable numbers can also be determined as limits of sequences
$x_{1},x_{2},\dots$ for which $|x_{n}-x_{n+k}|\le 2^{-n}$.
An infinite sequence $s_{1},s_{2},\dots$ of real numbers is a
\df{strong Cauchy} sequence if for all $m<n$ we have $|s_{m}-s_{n}|\le 2^{-m}$.

A function determined on words (or other constructive objects) is
\df{computable} if its values are computable uniformly from the input,
that is there is an algorithm that from each input and $\varepsilon>0$ returns an
$\varepsilon$-approximation of the function value on this input.

Measure $\P$ over $\Omega$ is said to be \df{computable}
if the function $\P:\{0,1\}^{*}\to\clint{0}{1}$ is computable.
\end{definition}

\begin{definition}[Randomness test, computable
  measure]\label{def:test-computable-measure}
Let $\P$ be a computable probability distribution (measure) on~$\Omega$.
A lower semicomputable function $t$ on $\Omega$
with non-negative (possible infinite) values is an
(\df{average-bounded}) \df{randomness test} with respect to $\P$ (or
\df{$\P$-test})
if the expected value of $t$ with respect to $\P$ is at most~$1$, that is
\begin{equation*}
\int t(\omega)\,d\P\le 1.
\end{equation*}
A sequence $\omega$ \df{passes} a test $t$ if
$t(\omega)<\infty$.
A sequence is called \df{random} with
respect to $\P$ it is passes all $\P$-randomness tests (as defined
above).
\end{definition}

The intuition: when $t(\omega)$ is large, this means that test
$t$ finds a lot of ``regularities'' in $\omega$.
Constructing a
test, we are allowed to declare whatever we want as a
``regularity''; however, we should not find too many of them on
average: if we declare too many sequences to be ``regular'', the
average becomes too big.

This definition turns out to be equivalent to randomness as defined by
Martin-L\"of (see below).
But let us start with the universality theorem:

\begin{theorem}\label{thm:universality.cptable}
For any computable measure $\P$ there exists a
\df{universal} (maximal) $\P$-test $u$: this means that for any other
$\P$-test $t$ there exists a constant $c$ such that
\begin{equation*}
t(\omega)\le c\cdot u(\omega)
\end{equation*}
for every $\omega\in\Omega$.
\end{theorem}
In particular, $u(\omega)$ is finite if and only if $t(\omega)$
is finite for every $\P$-test $t$, so the sequences that pass
test $u$ are exactly the random sequences.
\begin{proof}
Let us enumerate the algorithms that generate
all lower semicomputable functions.
Such an algorithm produces a
monotone sequence of basic functions.
Before letting through the
next basic function of this sequence, let us check that its
$\P$-expectation is less than $2$.
If the algorithm considered
indeed defines a $\P$-test, this expectation does not exceed~$1$,
so by computing the values of $\P$ with sufficient precision we
are able to guarantee that the expectation is less than~$2$.
If this checking procedure does not terminate (or gives a negative
result), we just do not let the basic function through.

In this way we enumerate all tests as well as some lower
semicomputable functions that are not exactly tests but are at
most twice bigger than tests.
It remains to sum up all these
functions with positive coefficients whose sum does not exceed
$1/2$ (say, $1/2^{i+2}$).
\end{proof}

Recall the definition of randomness according to Martin-L\"of.

\begin{definition}
Let $\P$ be a computable distribution over $\Omega$.
A sequence of open sets $U_{1},U_{2},\dots$ is called a \df{Martin-L\"of test}
for $\P$
if the sets $U_{i}$ are effectively open in a uniform way (that is
$U_{i}=\bigcup_{j}x_{ij}\Omega$ where the double sequence $x_{ij}$ of strings
is computable),
moreover $\P(U_{k})\le 2^{-k}$ for all $k$.

A set $N$ is called a \df{constructive (effective) null set} for the measure $\P$
if there is a Martin-L\"of test $U_{1},U_{2},\dots$ with the property
$N=\bigcap_{k}U_{k}$.
Note that effective null sets are constructive $G_{\delta}$ sets.

A sequence $\omega\in\Omega$ is said to \df{pass} the test $U_{1},U_{2},\dots$ if it is
not in $N$.
It is \df{Martin-L\"of-random} with respect to measure
$\P$ if it is not contained in any constructive null set for $\P$.
\end{definition}

The following theorem is not new, see for example~\cite{LiViBook97}.

\begin{theorem}
A sequence $\omega$ passes all average-bounded $\P$-tests
(=passes the universal $\P$-test) if and only if it is
Martin-L\"of random with respect to $\P$.
\end{theorem}
\begin{proof}
If $t$ is a test, then the set of all
$\omega$ such that $t(\omega)>N$ is an effectively open set that
can be found effectively given $N$.
This set has $\P$-measure at
most $1/N$ (by Chebyshev's inequality), so the sets of sequences $\omega$ that
do not pass
$t$ (that is $t(\omega)$ is infinite) is an effectively $\P$-null set.

On the other hand, let us show that for every effectively null set $Z$ there
exists an average-bounded test that is infinite at all its elements.
Indeed, for every effectively open set $U$ the function $1_{U}$ that is
equal to $1$ inside $U$ and to $0$ outside $U$ is lower
semicomputable.
Then we can get a test $\sum_{i} 1_{U_i}$.
The average of this test
does not exceed $\sum_{i} 2^{-i}=1$, while the sum is infinite for all elements
of $\bigcap_{i} U_i$.
\end{proof}

When talking about randomness for a computable measure,
we will write \df{randomness} from now on, understanding
Martin-L\"of randomness, since no other kind will be considered.

Sometimes it is useful to switch to the logarithmic scale.

\begin{definition}\label{def:deficiency}
 For every computable measure $\P$, we will fix a universal $\P$-test and denote
it by $\t_{\P}(\omega)$.
Let $\d_{\P}(\omega)$ be the logarithm of the universal test~$\t_{\P}(\omega)$:
\begin{equation*}
\t_{\P}(\omega)=2^{\d_{\P}(\omega)}.
\end{equation*}
With other kinds of test also, it will be our convention to use $\t$ (boldface)
for the universal test, and $\d$ (boldface) for its logarithm.
\end{definition}

In a sense, the function $\d_{\P}$ measures the randomness
deficiency in bits.

The logarithm, along with the requirement $\int \t_{\P}(\omega)\,d\P\le 1$,
implies that $\d_{\P}(\omega)$ may have some negative values, and even values
$-\infty$.
By just choosing a different universal test we
can always make $\d_{\P}(\omega)$ bounded below by, say, $-1$, and also
integer-valued.
On the other hand, if we want to make it nonnegative, we will have to lose the
property $\int 2^{\d_{\P}(\omega)}d\P\le 1$, though we may still have
$\int 2^{\d_{P}(\omega)}d\P\le 2$.
It will still have the following property:

\begin{proposition}
The function $\d_{P}(\cdot)$ is lower semicomputable and is the largest (up
to an additive constant) among all lower semicomputable
functions such that the $\P$-expectation of $2^{\d_{\P}(\cdot)}$ is
finite.
\end{proposition}

As we have shown, for any fixed computable measure $\P$ the
value $\d_{\P}(\omega)$ (and $\t_{\P}(\omega)$) is finite if
and only if the sequence $\omega$ is Martin-L\"of random with respect to~$\P$.

\begin{remarks}\label{rem:ML-tests}
  \begin{enumerate}[1.]
\item\label{i:prob-bounded} Each Martin-L\"of's test ($U_{1},U_{2},\dots$) is
more directly related to
a lower semicomputable function $F(\omega)=\sup_{\omega\in U_{i}} i$.
This function has the property $P\evof{F(\omega)\ge k}\le 2^{-k}$.
Such functions will be called \df{probability-bounded} tests, and were used
in~\cite{ZvLe70}.
We will return to such functions later (subsection~\ref{subsec:average-probability}).

  \item
We have defined $\d_\P(\omega)$ separately for each computable measure $\P$
(up to a constant).
We will later give a more general definition of
randomness deficiency $\d(\omega,\P)$ as a function of two variables $\P$ and
$\omega$
that coincides with $\d_\P(\omega)$ for every computable $\P$ up to a constant
depending on~$\P$.
  \end{enumerate}
\end{remarks}

\subsection{Average-bounded and probability-bounded deficiencies}
\label{subsec:average-probability}
Let us refer for example to~\cite{LiViBook97,ShenUppsala00} for the definition
of and basic properties of plain and prefix (Kolmogorov) complexity.
We will define prefix complexity in Definition~\ref{def:prefix} below, though.
We will not use complexities explicitly in the present section, just refer to some of
their properties by analogy.

The definition of a test given above resembles the definition of
prefix complexity; we can give another one which is closer
to plain complexity.
For that we use a weaker requirement: we
require that the $\P$-measure of the set of all sequences~$\omega$ such
that $t(\omega)>N$ does not exceed $1/N$.
(This property is true if
the integral does not exceed $1$, due to Chebyshev's inequality.)

In logarithmic scale this requirement can be restated as
follows: the $\P$-measure of the set of all sequences whose
deficiency is greater than $n$ does not exceed $2^{-n}$.
If we restrict tests to integer values, we arrive at the classical Martin-L\"of
tests: see also Remark~\ref{rem:ML-tests}, part~\ref{i:prob-bounded}.

While constructing an universal test in this sense, it is
convenient to use the logarithmic scale and consider only
integer values of $n$.
As before, we enumerate all tests and
``almost-tests'' $d_i$ (where the measure is bounded by twice
bigger bound) and then take the weighted maximum in the
following way:
\begin{equation*}
\d(\omega)=\max_i [d_i(\omega)-i] - c.
\end{equation*}
Then $\d$ is less than $d_i$ only by $i+c$,
and the set of all $\omega$ such that $\d(\omega)>k$ is the union
of sets where $d_i(\omega)>k+i+c$.
Their measures are bounded by
$O(2^{-k-i-c})$ and for a suitable $c$ the sum of measures is at
most $2^{-k}$, as required.

In this way we get two measures of non-randomness that can be
called ``average-bounded deficiency'' $\d^{\aver}$ (the first one,
related to  the tests called ``integral tests'' in~\cite{LiViBook97}) and ``probability
bound deficiency'' $\d^{\prob}$ (the second one).
It is easy to see that they
define the same set of nonrandom sequences (=sequences that have
infinite deficiency).
Moreover, the finite values of these two
functions are also rather close to each other:

\begin{proposition}
  \begin{equation*}
 \d^{\aver}(\omega)\lea \d^{\prob}(\omega)
\lea \d^{\aver} (\omega)+ 2 \log \d^{\aver}(\omega).
  \end{equation*}
\end{proposition}
\begin{proof}
Any average-bounded test is also a probability-bounded test, therefore
$\d^{\aver}(\omega) \lea \d^{\prob}(\omega)$.

For the other direction, let $d$ be a probability-bounded test (in the
logarithmic scale).
Let us show that $d-2\log d$ is an average-bounded test.
Indeed, the probability of the event ``$d(\omega)$ is between $i-1$ and
$i$'' does not exceed $1/2^{i-1}$, the integral of $2^{d-2\log d}$
over this set is bounded by $2^{-i+1}2^{i-2\log i}=2/i^2$ and
therefore the integral over the entire space converges.

It remains to note that the inequality $a\lea b+2\log b$
follows from $b\gea a-2\log a$.
Indeed, we have $b\ge a/2$ (for large enough $a$), hence $\log a\le\log b+1$,
and then $a\lea b+2\log a\lea b+2\log b$.
\end{proof}

In the general case the question of the connection between boundedness in
average and boundedness in probability is addressed in the
paper~\cite{ShaferShenVereshVovk09}.
It is shown there (and this is not difficult) that if
$u:\clint{1}{\infty}\to\clint{0}{\infty}$ is a monotonic continuous function
with $\int_{1}^{\infty}u(t)/t^{2}\,d t\le 1$, then $u(t(\omega))$ is an
average-bounded test for every probability-bounded test $t$, and that this
condition cannot be improved.
(Our estimate is obtained by choosing $u(t)=t/\log^{2} t$.)

\begin{remark}
This statement resembles the relation between
prefix and plain description complexity.
However, now the difference is
bounded by the logarithm of the \emph{deficiency} (that is bounded independently
of length for the sequences that are close to random), not
of the \emph{complexity} (as usual), which would be normally growing with the
length.
\end{remark}

\begin{question}
It would be interesting to understand whether the two tests differ only by a
shift of scale or in some more substantial way.
For the confirmation of such a more substantial difference could serve two
families of sequences $\omega_{i}$ and $\omega'_{i}$ for which
 \begin{align*}
   \d^{\aver}(\omega'_{i})-\d^{\aver}(\omega_{i}) \to \infty
 \end{align*}
for $i\to\infty$, while
 \begin{align*}
   \d^{\prob}(\omega'_{i})-\d^{\prob}(\omega_{i}) \to -\infty.
 \end{align*}
The authors do not know whether such a family exists.
\end{question}

\subsection{A formula for average-bounded deficiency}

Let us recall some concepts connected with the prefix description complexity.
For reference, consult for example~\cite{LiViBook97,ShenUppsala00}.

\begin{definition}\label{def:prefix}
A set of strings is called \df{prefix-free} if no element of it is a prefix of
another element.
A computable partial function $T:\{0,1\}^{*}\to\{0,1\}^{*}$
is called a \df{self-delimiting interpreter} if its domain of definition
is a prefix-free set.
We define the complexity $\KP_{T}(x)$ of a string $x$ with respect to $T$
as the length of a shortest string $p$ with $T(p)=x$.
It is known that there is an \df{optimal} (self-deliminiting) interpreter: that is
a (self-delimiting) interpreter $U$ with the property that for every
self-delimiting interpreter $T$ there is a constant $c$ such that for every
string $x$ we have $\KP_{U}(x)\le \KP_{T}(x)+c$.
We fix an optimal self-delimiting interpreter $U$ and denote
$\KP(x)=\KP_{U}(x)$.

We also denote $\m(x)=2^{-\KP(x)}$, and call it sometimes
\df{discrete a priori probability}.
\end{definition}

The ``a priori'' name comes from some interpretations of a property that
distinguishes the function $\m(x)$ among certain ``weight distributions'' called
semimeasures.

\begin{definition}\label{def:semimeasure}
  A function $f:\{0,1\}^{*}\to\lint{0}{\infty}$ is called a
\df{discrete semimeasure} if $\sum_{x}f(x)\le 1$.
\end{definition}

Lower semicomputable semimeasures arise as the output distribution of a
randomized algorithm using a source of random numbers, and outputting some word
(provided the algorithm halts; with some probability, it may not halt).

It is easy to check that $\m(x)$ is a lower semicomputable discrete
semimeasure.

Recall the following fact.

\begin{proposition}[Coding Theorem]\label{propo:coding}
Among lower semicomputable discrete semimeasures, the function
$\m(x)$ is maximal within a multiplicative constant: that is for every lower
semicomputable discrete semimeasure $f(x)$ there is a constant $c$
with $c\cdot\m(x)\ge f(x)$ for all $x$.
\end{proposition}

The universal average-bounded randomness test $\t_{\P}1$ (the largest lower
semicomputable function with bounded expectation) can be
expressed in terms of a priori probability (and therefore prefix
complexity):

\begin{proposition}\label{propo:sum-characteriz}
Let $\P$ be a computable measure and let $\t_{\P}$ be
the universal average-bounded randomness test with respect to~$\P$.
Then
\begin{equation*}
\t_{\P}(\omega)\eqm\sum_{x\prefix\omega}\frac{\m(x)}{\P(x)}.
\end{equation*}
(If $\P(x)=0$, then the ratio $\m(x)/\P(x)$ is considered to be infinite.)
\end{proposition}
\begin{proof}
A lower semicomputable function on sequences is a limit of
an increasing sequence of basic functions.

Withouth loss of generality we may assume that each increase is made
on some cylinder $x\Omega$.
In other terms, we increase the ``weight'' $w(x)$
of $x$ and let our basic function on $\omega$ be the sum of the
weights of all prefixes of $\omega$.
The weights increase gradually: at any moment, only finitely many weights differ
from zero.
In terms of weights, the average-boundedness condition translates into
 \begin{align*}
   \sum_{x}\P(x)w(x)\le 1,
 \end{align*}
so after multiplying the weights by $\P(x)$, this condition corresponds exactly
to the semimeasure requirement.
Let us note that due to the computability of $\P$, the lower semicomputability is
conserved in both directions (multiplying or dividing by $\P(x)$).
More formally, the function
 \begin{align*}
   \sum_{x\prefix\omega}\frac{\m(x)}{\P(x)}
 \end{align*}
is a lower semicomputable average-bounded test: its integral is exactly
$\sum_{x}\m(x)$.
On the other hand, every lower semicomputable test can be presented in terms of
an increase of weights, and the limits of these weights, multiplied by $\P(x)$,
form a lower semicomputable semimeasure.
(Note that the latter transformation is not unique: we can redistribute the
weights among a string and its continuations without altering the sum over the
infinite sequences.)
\end{proof}
Note that we used that both $\P$ (in the second part of the
proof) and $1/\P$ (in the first part) are lower semicomputable.

In Proposition~\ref{propo:sum-characteriz},
we can replace the sum with a least upper bound.
This way, the following theorem connects three quantities,
$\t_{\P}$, the supremum and the sum, all of which are equal within a
multiplicative constant.

\begin{theorem}\label{thm:randomness-complexity}
We have $
     \t_{\P}(\omega) \eqm\sup_{x\prefix\omega}\frac{\m(x)}{\P(x)}
                                 \eqm\sum_{x\prefix\omega}\frac{\m(x)}{\P(x)}
  $,
or in logarithmic notation
\begin{align}
\label{eq:randomness-complexity}
  \d_{\P}(\omega) &\eqa\sup_{x\prefix\omega}\Paren{-\log\P(x)-\KP(x)}.
  \end{align}
\end{theorem}
\begin{proof}
The supremum is now smaller, so only the
second part of the proof of Proposition~\ref{propo:sum-characteriz}
should be reconsidered.

The lower semicomputable function $\cei{\d_{\P}(\omega)}$
can be obtained as the supremum of a sequence of
integer-valued basic functions of the form $k_{i}g_{x_{i}}(\omega)$, where
$g_{x}(\omega)=1_{x\Omega}(\omega)=1$ if $x\prefix\omega$ and 0 otherwise.
We can also require that if $i\ne j$, $x_{i}\prefix x_{j}$ then
$k_{i}\ne k_{j}$: indeed, suppose $k_{i}=k_{j}$.
If $i<j$ then we can delete the $j$th element, and if $i>j$, then we can replace
$2^{k_{i}}g_{x_{i}}$ with the sequence of all functions $2^{k_{i}}g_{z}$ where
$z$ has the same length as $x_{j}$ but differs from it.
We have
 \begin{align*}
  2\t_{\P}(\omega)&\ge 2^{\cei{\d_{\P}(\omega)}}
=\sup_{i}2^{k_{i}}g_{x_{i}}(\omega)=\sup_{x_{i}\prefix\omega}2^{k_{i}}
\ge 2^{-1}\sum_{i:x_{i}\prefix\omega}2^{k_{i}}=2^{-1}\sum_{i}2^{k_{i}}g_{x_{i}}(\omega).
 \end{align*}
The last inequality holds since according to our assumption, all the values
$k_{i}$ belonging to prefixes $x_{i}$ of the same sequence $\omega$ are
different, and the sum of different powers of 2 is at most twice larger than its
largest element.
Integrating by $\P$, we obtain $4\ge\sum_{i}2^{k_{i}}\P(x_{i})$, hence
$2^{k_{i}}\P(x_{i})\lem\m(x_{i})$ by the maximality of $\m(x)$, so
$2^{k_{i}}\lem\frac{\m(x_{i})}{\P(x_{i})}$.
We found
 \begin{align*}
 \t_{\P}(\omega)\lem\sup_{i:x_{i}\prefix\omega}\frac{\m(x_{i})}{\P(x_{i})}
\le\sup_{x\prefix\omega}\frac{\m(x)}{\P(x)}.
 \end{align*}
\end{proof}

Here is a reformulation:
\begin{equation*}
  \d_{\P}(\omega) \eqa\sup_{n}\Paren{-\log\P(\omega(1:n))-\KP(\omega(1:n))}.
\end{equation*}
This reformulation can be generalized:

\begin{theorem}\label{thm:seq-of-prefixes}
Let $n_{1}<n_{2}<\dotsm$ be an arbitrary computable sequence of natural
numbers.
Then
\begin{equation*}
  \d_{\P}(\omega) \eqa\sup_{k}\Paren{-\log\P(\omega(1:n_{k}))-\KP(\omega(1:n_{k}))}.
\end{equation*}
The constant in the $\eqa$ depends on the sequence $n_{k}$.
\end{theorem}
\begin{proof}
  Every step of the proof of Theorem~\ref{thm:randomness-complexity} generalizes
to this case straightforwardly.
\end{proof}
This theorem has interesting implications of the case when instead of a sequence
$\omega$ we consider an infinite two-dimensional array of bits.
Then for the randomness deficiency, it is sufficient to compare complexity and
probability of squares starting at the origin.

\subsubsection*{Historical digression}

The above formula for randomness deficiency is a quantitative refinement of the
following criterion.

\begin{theorem}[Criterion of
randomness in terms of prefix complexity]\label{thm:rand-prefix}
A sequence $\omega$
is random with respect to a computable measure $\P$ if and
only if the difference $-\log\P(x) - \KP(x)$ is bounded above for its
prefixes.
\end{theorem}
(Indeed, the last theorem says that the maximum value of this
difference over all prefixes is exactly the average-bounded randomness
deficiency.)
This characterization of randomness was announced first, without proof,
in~\cite{Chaitin75}, with the proof attributed to Schnorr.
The first proof, for the case of a computable measure, appeared
in~\cite{GacsExact80}.

The historically
first clean characterizations of randomness in terms of complexity,
by Levin and Schnorr independently in~\cite{LevinRand73} and~\cite{Schnorr73}
have a similar form, but use complexity and a priori probability coming from
a different kind of interpreter called ``monotonic''.
(In the cited work, Schnorr uses a slightly different form of complexity, but later,
he also adopted the version introduced by Levin.)

\begin{definition}[Monotonic complexity]\label{def:monotonic}
Let us call to strings \df{compatible} if one is the prefix of the other.
An enumerable subset $A\subseteq \{0,1\}^{*}\times\{0,1\}^{*}$
is called a \df{monotonic interpreter} if for every $p,p',q,q'$, if
$\pair{p}{q}\in A$ and $\pair{p'}{q'}\in A$ and $p$ is compatible with $p'$ then
$q$ is compatible with $q'$.
For an arbitrary finite or infinite $p\in\{0,1\}^{*}\cup\Omega$,
we define
 \begin{align*}
 A(p)=\sup\setOf{x}{\exists p'\prefix p\;\pair{p'}{x}\in A}.
 \end{align*}
The monotonicity property implies that this limit, also in $\{0,1\}^{*}\cup\Omega$,
is well defined.

We define the (monotonic)
complexity $\Km_{A}(x)$ of a string $x$ with respect to $A$
as the length of a shortest string $p$ with $A(p)\postfix x$.
It is known that there is an \df{optimal} monotonic interpreter, where
optimality has the same sense as above, for prefix complexity.
We fix an optimal monotonic interpreter $V$ and denote
$\Km(x)=\Km_{V}(x)$.
\end{definition}

\begin{remark}[Oracle computation]\label{rem:oracles}
A monotonic interpreter is a slightly generalized version of what can be
accomplished by a Turing machine with a one-way read-only input tape containing
the finite or infinite string $p$.
The machine also has a working tape and a one-way output tape.
In the process of work, on this tape appears a finite or infinite sequence
$T(p)$.
The work may stop, if the machine halts or passes beyond the limit of the input
word; it may continue forever otherwise.
It is easy to check that the map $p\mapsto T(p)$ is a monotonic interpreter
(though not all monotonic interpreters correspond to such machines, resulting in a
somewhat narrower class of mappings).

These machines can be viewed as the definition of what we will later call \df{oracle
computation}: namely, a computation that uses $p$ as an oracle.

In our applications, such a machine would have the form
$T(p,\omega)$ where the machine works on both infinite strings $p$ and $\omega$ as
input, but considers $p$ the oracle and $\omega$ the string it is testing for
randomness.

The class of mappings is narrower indeed.
Let $S$ be an undecidable recursively enumerable set of integers.
Set $T(0^{n}1)=0$ for all $n\in S$, and $T(0^{n}10)=0$ for all $n$.
Now after reading $0^{n}1$, the machine $T$ has to decide whether to
output a 0 before reading the next bit, which is deciding the undecidable set
$S$.
It is unknown to us whether this class of mappings yields also a different
monotonic complexity.
\end{remark}

A monotonic interpreter will also give rise to something like a distribution
over the set of finite and infinite strings.

\begin{definition}\label{def:contin-semim}
Let us feed a monotonic interpreter $A$ a sequence of independent random bits
and consider the output distribution on the finite and infinite sequences.
Denote $\M_{A}(x)$ the probability that the output sequence begins with $x$.
Denote $\KM_{V}(x)=-\log \M_{V}(x)$.

Recall that $\Lambda$ denotes the empty string.
A function $\mu:\{0,1\}^{*}\to\clint{0}{1}$ is called a
\df{continous semimeasure} over the Cantor space
$\Omega$ if $\mu(\Lambda)\le 1$ and
$\mu(x)\ge\mu(x0)+\mu(x1)$ for all $x\in\{0,1\}^{*}$.
\end{definition}

It is easy to check that $\M_{V}(x)$ is a lower semicomputable continuous
semimeasure.
The proposition below is similar in form to the Coding Theorem
(Proposition~\ref{propo:coding}) above, only weaker, since it does not connect
to the complexity $\Km(x)$ defined in terms of shortest programs.
(It cannot, as shown in~\cite{GacsRel83}.)

\begin{proposition} (see~\cite{ZvLe70})
  \begin{alphenum}
\item Every lower semicomputable continuous semimeasure is the output distribution of some
  monotonic interpreter.
\item Among lower semicomputable continuous semimeasures, there is one
that is maximal within a multiplicative constant.
  \end{alphenum}
\end{proposition}

\begin{definition}[Continuous a priori
  probability]\label{def:continuous-apriori}
Let us fix a maximal lower semicomputable continuous semimeasure and denote it
$\M(x)$.
We call $\M(x)$ sometimes the \df{continuous a priori probability}, or
\df{apriori probability on a tree}.
\end{definition}

Now, the characterization by Levin (and a similar one by Schnorr) is the
following.
Its proof, technically not difficult, can be found
in~\cite{GacsLnotesAIT88,LiViBook97,ShenUppsala00}.

\begin{proposition}\label{propo:rand-monot}
Let $\P$ be a computable measure over $\Omega$.
Then the following properties of an infinite sequence $\omega$ are equivalent.
\begin{enumerate}[\upshape (i)]
\item $\omega$ is random with respect to $\P$.
\item $\lim\sup_{x\prefix\omega}-\log\P(x)-\Km(x)<\infty$.
\item $\lim\inf_{x\prefix\omega}-\log\P(x)-\Km(x)<\infty$.
\item $\lim\sup_{x\prefix\omega}-\log\P(x)-\KM(x)<\infty$.
\item $\lim\inf_{x\prefix\omega}-\log\P(x)-\KM(x)<\infty$.
\end{enumerate}
\end{proposition}

Theorem~\ref{thm:rand-prefix} proved above adds to this a next equivalent
characterization, namely that $-\log\P(x)-\KP(x)$ is bounded above.
It is different in nature from the one in Proposition~\ref{propo:rand-monot}:
indeed, the expressions $-\log\P(x)-\Km(x)$ and $-\log\P(x)-\KM(x)$ are
\emph{always bounded from below} by a constant depending only on the measure
$\P$ (and not on $x$ or $\omega$), while $-\log\P(x)-\KP(x)$ is not.

Moreover, in the latter we cannot replace $\lim\sup$ with $\lim\inf$, as the
following example shows.
Note that we can add to every word $x$ some bits to achieve $\KP(y)\ge\len{y}$
(where $\len{y}$ is the length of word $y$).
Indeed, if this was not so, then for the continutations of the word we would
have $\m(y)\ge 2^{-\len{y}}$, and the sum $\sum_{y}\m(y)$ would be infinite.
Let us build a sequence, adding alternatingly long stretches of zeros to make the
complexity substantially less than the length, then bits that again bring the
complexity up to the length (as shown, this is always possible).
Such a sequence will not be random with respect to the uniform measure (since
the $\lim\sup$ of the difference is infinite), but has infinitely many prefixes
for which the complexity is not less than the length, making the $\lim\inf$
finite.

The following statement is interesting since no direct proof of it is known: the
proof goes through Theorem~\ref{thm:seq-of-prefixes}, and noting that since the
permutation of terms of the sequence does not change the coin-tossing
distribution, it does not change the notion of randomness.
More general theorems of this type, under the name of
\emph{randomness conservation}, can be found
in~\cite{LevinUnif76,LevinRandCons84,GacsUnif05}.

\begin{corollary}
Consider the uniform distribution (coin-tossing) $\P$ over binary sequences.
The maximal difference between $\len{x}$ and $\KP(x)$ for prefixes $x$ of a random
sequence is invariant (up to a constant) under any computable
permutation of the sequence terms.
(The constant depends on the permutation, but not on the sequence.)
\end{corollary}

Here is another corollary, a reformulation of Proposition~\ref{propo:sum-characteriz}:

\begin{corollary}[Miller-Yu ``ample excess'' lemma]\label{coroll:ample-excess}
A sequence $\omega$ is random with respect to a
computable measure $\P$ if and only if
\begin{equation*}
\sum_{x\prefix\omega} 2^{-\log\P(x) - \KP(x)}<\infty.
\end{equation*}
\end{corollary}

This corollary also implies the fact mentioned above already:

\begin{corollary}\label{coroll:extension}
Every finite sequence $x$ has an extension $y$ with $\KP(y)>\len{y}$.
\end{corollary}
\begin{proof}
Take $\omega$ random, then $x\omega$ is random, and therefore by the
Miller-Yu lemma $x\omega$
has arbitrarily long prefixes whose complexity is larger than the length.
\end{proof}

\subsection{Game interpretation}

The formula for the average-bounded deficiency can be interpreted in terms of
the following game.
Alice and Bob make their moves having no information about
the opponent's move.
Alice chooses an infinite binary sequence $\omega$,
Bob chooses a finite string $x$.
If $x$ turns out to be a prefix of $\omega$,
then Alice pays Bob $2^n$ where $n$ is the length of $x$.
(This version of the game corresponds to the uniform Bernoulli measure,
in the general case Alice pays $1/\P(x)$.)
Recall the game-theoretic notions of \df{pure} strategy, as a deterministic
choice by a player, and \df{mixed} strategy, as a probability distribution over
deterministic choices.

Bob has a trivial strategy (choosing the empty string) that guarantees him $1$
whatever Alice does.
Also Alice has a mixed strategy (the uniform distribution, or, in general case,
$\P$) that guarantees her the average loss $1$ whatever Bob does.
Bob can devise a strategy that will benefit him
in case (for whatever reason) Alice brings a nonrandom sequence.

A randomized algorithm that has no input and produces a string (or nothing)
can be considered a mixed strategy for Bob (if the algorithm does not
produce anything, Bob gets no money).
For any such algorithm $D$ the expected payment (if Alice produces
$\omega$ according to distribution $\P$) does not exceed $1$.
Therefore, the set of sequences $\omega$ where the expected payment (averaged
over Bob's random bits) is infinite, is a null set.
Observe the following:
\begin{enumerate}[\upshape (i)]
\item For every probabilistic strategy of Bob, his expected gain
(as a function of Alice's sequence) is an average-bounded test.
(From here already follows that this expected value will be finite, if Alice's
sequence is random in the sense of Martin-L\"of.)

\item If $m(x)$ is the probability of $x$ as Bob's move with algorithm $D$, his
expected gain  against $\omega$ is equal to
 \begin{align*}
 \sum_{x\prefix\omega} m(x)/\P(x).
 \end{align*}

\item Therefore if we take the algorithm outputting the discrete apriori
  probability $\m(x)$, then Bob's expected gain will be a universal test (by the
  proved formula for the universal test).

\end{enumerate}
Using the apriori probability as a mixed strategy enables Bob to
punish Alice with an infinite penalty for any non-randomness in her sequence.

One can consider more general strategies for Bob: he can give for a pure
strategy, not only a string $x$, but some basic function
$f$ on $\Omega$ with non-negative values.
Then his gain for the sequence $\omega$ brought by Alice is set to
$f(\omega)/\int f(\omega)\,d\P$.
(The denominator makes the expected return equal to $1$.)
To the move $x$ corresponds the basic function that assigns $2^{\len{x}}$ to
extensions of $x$ and zero elsewhere.
This extension does not change anything, since this move is a
mixed strategy and we allow Bob to mix his strategies anyway.
(After producing
$f$, Bob can make one more randomized step and choose some of the intervals on
which $f$ is constant, with an appropriate probability.)
In this way we get another formula for the universal test:
\begin{equation*}
\t_{\P}(\omega)\eqm\sum_{f} \frac{\m(f) f(\omega)}{\int f(\omega)\, d\P},
\end{equation*}
where the sum is taken over all basic functions $f$.
This formula might be useful in more general situations (not Cantor space)
where we do not work with intervals and consider some class of basic functions
instead.

On concluding this part let us point to a similar
game-theoretical interpretation of probability theory developed in the
book~\cite{ShaferVovkGame01} of Shafer and Vovk.
There, the randomness of an object is not its
property but, roughly speaking, a kind of guarantee with which it is being sold.

\section{From tests to complexities}

Formula~\eqref{eq:randomness-complexity}
expresses the randomness deficiency (the logarithm of the
universal test) of an infinite sequence in terms of complexities of its finite
prefixes.
A natural question arises: can we go in the other direction?
Is it
possible to express the complexity of a finite string $x$, or some kind of
``randomness deficiency'' of $x$, in terms of the deficiencies of $x$'s infinite
extensions?
Proposition~\ref{propo:generate-lower} and the discussion following it already
brought us from infinite sequences to finite ones.
This can also be done for the universal test:

\begin{definition}\label{def:bar-u}
Fix some computable measure $\P$, and
let $t$ be any (average-bounded) test for $\P$.
For any finite string $x$ let $\bar t(x)$ be the minimal deficiency
of all infinite extensions of $x$:
\begin{equation*}
  \bar t(x) = \inf_{\omega\postfix x} t(\omega).
\end{equation*}
\end{definition}

By Proposition~\ref{propo:inf-lower-semicomp.seqs}, $\bar t$
is a lower semicomputable function
defined on finite strings, and the function $t$ can be
reconstructed back from $\bar t$; so if $\t_{\P}$ is our fixed universal test then
$\bar\t_{\P}$
can be considered as a version of randomness deficiency for finite strings.

The intuitive meaning is clear: a finite sequence $z$ looks non-random if
\emph{all} infinite sequences that have prefix $z$ look non-random.

\begin{question}
Kolmogorov~\cite{KolmPPI69} had a somewhat
similar suggestion: for a given sequence $z$ we may consider the
minimal deficiency (with respect to the uniform distribution,
defined as a difference between length and complexity) of all
its \emph{finite} extensions.
Are there any formal connections?
\end{question}

Let us spell out what we found, in more general terms.

\begin{definition}[Extended test for a computable
  measure]\label{def:extended-test.computable}
A lower semicomputable, monotonic (with respect to the prefix relation)
function $T:\{0,1\}^{*}\to\clint{0}{\infty}$
is called an \df{extended test} for computable measure $P$ if for all $N$ the
average over words of length $N$ is bounded by 1:
 \begin{align*}
 \sum_{x:|x|=N}\P(x)T(x) \le 1.
\end{align*}
\end{definition}
Monotonicity guarantees that the sum over words of a given length can be
replaced by the sum over an arbitrary finite (or even infinite) prefix-free set $S$:
 \begin{align}\label{eq:extended-test.computable}
 \sum_{x\in S}\P(x)T(x) \le 1.
\end{align}
(Indeed, extend the words of $S$ to some common greater length.)

 \begin{proposition}\label{propo:extended-test.computable}
  Every extended test generates (in the sense of
Definition~\ref{def:generate-lower}) some averge-bounded test on the infinite
strings.
Conversely, every average-bounded test on the infinite sequences is generated by
some extended test.
\end{proposition}
\begin{proof}
The first part follows immediately from the definition (and the theorem of monotone
convergence under the integral sign).
In the opposite direction, we can set for example $T(x)=\bar t(x)$, or refer to
Proposition~\ref{propo:lower-semi-limit.seqs} if we do not want to rely on
compactness.
\end{proof}

The existence of a universal extended test is proved by the usual methods:

\begin{proposition}\label{propo:extended-univ.computable}
Among the extended tests $T(x)$ for a computable measure $\P(x)$
there is a maximal one, up to a multiplicative constant.
 \end{proposition}

 \begin{definition}\label{def:extended-test}
Let us fix some dominating extended test and call it the \df{universal} extended test.
 \end{definition}

 \begin{proposition}\label{propo:extended-to-bar}
The universal extended test
coincides with $\bar\t_{\P}(x)$ to within a bounded factor.
 \end{proposition}
 \begin{proof}
Since $\bar\t_{\P}$ is an extended test, it is not greater than the universal
test (to within a bounded factor).
On the other hand, the universal extended test generates a test on the infinite
sequences, it just remains to compare it with the maximal one.
 \end{proof}
If our space is not compact (say, it is the set of infinite sequences of
integers), then $\bar\t_{\P}(x)$ is not defined, but there is still a universal
extended test, which we will denote by $\t_{\P}(x)$.

Warning: not all extended tests generating $\t_{\P}(\omega)$ are maximal.
(For example, one can make the test equal to zero on all short words,
transferring its values to its extensions.)

The advantage of the function $\t_{\P}(x)$ is that it is defined on
finite strings, the condition~\eqref{eq:extended-test.computable}
(for finite sets $S$) imposed
on it is also more elementary than the integral condition,
but clearly implies that it generates a test.

The method just shown is not the only way
to move to tests on prefixes from tests on infinite sequences:

\begin{definition}\label{def:hat-u}
Assume that the computable measure $\P$ is positive on all intervals:
$\P(x)>0$ for all $x$.
Let $\hat\t_{\P}(x)$ be the conditional expected value of $\t_{\P}(\omega)$ if a
random variable $\omega\in\Omega$ has distribution~$\P$ and
the condition is $\omega\postfix x$.
In other terms: let $\hat\t_{\P}(x)$ be the average of $\t_{\P}$ on the interval
$x\Omega$, that is let $\hat \t_{\P}(x)=U(x)/\P(x)$ where
 \begin{align*}
  U(x)=\int_{x\Omega}\t_{\P}(\omega)\,dP(\omega).
 \end{align*}
  \end{definition}
The function $U$ is a lower semicomputable semimeasure.
(It is even a measure, but the measure
is not guaranteed to be computable and the measure of the entire space
$\Omega$ is not necessarily~$1$.
In other words, we get
a measure on $\Omega$ that has density $\t_{\P}$ with respect to $\P$.)
This implies that the function $\hat\t_{\P}(x)$ is a martingale,
according to the following definition.

\begin{definition}\label{def:martingale}
  A function $g:\{0,1\}^{*}\to\bbR$ is called a \df{martingale} with respect to
the probability measure $\P$ if
 \begin{align*}
   \P(x)g(x) = \P(x0)g(x0)+\P(x1)g(x1).
 \end{align*}
It is a \df{supermartingale} if at least the inequality $\ge$ holds here.
\end{definition}

Note that, as a martingale, the function $\hat\t_{\P}(x)$ is \emph{not}
monotonically increasing with respect to the prefix relation.

\begin{theorem}\label{thm:long-chain}
  \begin{equation}\label{eq:long-chain}
\frac{\m(x)}{\P(x)}\lem \t_{\P}(x) \lem \hat\t_{\P}(x) \lem
\frac{\M(x)}{\P(x)},
  \end{equation}
where $\m$ is the a priori probability on strings as isolated
objects (whose logarithm is minus prefix complexity) and $\M$ is
the a continuous priori probability as introduced in
Definition~\ref{def:continuous-apriori}.
\end{theorem}
\begin{proof}
In fact, the first inequality can be made stronger: we can replace
$\m(x)/\P(x)$ by $\sum_{t\prefix x}\m(t)/\P(t)$.
Indeed, this sum is a part of the expression for $\t_{\P}(\omega)$
for every $\omega$ that starts with $x$.

The second inequality uses Proposition~\ref{propo:extended-to-bar} and
relates the minimal and average values of a random variable.
The third inequality just compares
the lower semicomputable semimeasure $U(x)$ and the maximal
semimeasure $\M(x)$.
\end{proof}

Note that while $\hat\t_{\P}(x)$ is a martingale, $\frac{\M(x)}{\P(x)}$
is a supermartingale: it is actually maximal
within multiplicative constant, among the lower semicomputable
supermartingales for $\P$.

\begin{remarks}
  \begin{enumerate}[\upshape 1.]
\item We may insert
\begin{equation}\label{eq:inserted}
    \lem \max_{t\prefix x} \frac{\m(t)}{\P(t)}
   \lem \sum_{t\prefix x} \frac{\m(t)}{\P(t)}\lem
\end{equation}
between the first and the second terms of~\eqref{eq:long-chain}.
 \item
Using the logarithmic scale, we get
        \begin{equation*}
  -\log\P(x)-\KP(x) \lea \log\t_{\P}(x) \lea \log \hat\t_{\P}(x)
 \lea -\log\P(x)- \KM(x).
        \end{equation*}
\item The Measure $U$ depends on $\P$ (recall that $U$
is a maximal measure that has density with respect to $\P$), so for
different $\P$'s,  for example with different supports, like the Bernoulli measures
with different parameters, we get different measures.
But this dependence is bounded by the inequality above: it
shows that the possible variations do not exceed the difference
between $\KP(x)$ and $\KM(x)$.

\item The rightmost inequality cannot be replaced by an equality.
For example, let $\P$ be the uniform (coin-tossing) measure.
Then the value of  $U(x)$ tends to $0$ when $x$ is an
increasing prefix of a computable sequence (we integrate over
decreasing intervals whose intersection is a singleton that has zero
uniform measure).
On the other hand, the value $\M(x)$ is bounded
by a positive constant for all these~$x$.

\item We used compactness (the finiteness of the alphabet $\{0,1\}$) in proving
Proposition~\ref{propo:inf-lower-semicomp.seqs}.
But we could have used
Proposition~\ref{propo:generate-lower} and the discussion following it for a
starting point, obtaining analogous results for the Baire space of infinite
sequences of natural numbers.
  \end{enumerate}
\end{remarks}

All quantities listed in Theorem~\ref{thm:long-chain}
can be used to characterize randomness: a
sequence $\omega$ is random if the values of the quantity in question are
bounded for its prefixes.
Indeed, the Levin-Schnorr theorem guarantees that for a
random sequence the right-hand side is bounded, and for a non-random
one the left-hand side is unbounded.
The monotonicity of
the second term guarantees that all expressions except the first
one tend to infinity.
As we already mentioned above, one cannot say this about the first quantity.

\begin{question}
Some quantities used in the theorem ($\t_{\P}(x)$ and two added ones
in~\eqref{eq:inserted})
are monotonic (with respect to the prefix partial order of $x$) by definition.
We have seen that $\hat\t_{\P}(x)$, as a martingale, is not monotonic.
What can be said about $\frac{\M(x)}{\P(x)}$?
\end{question}

All these quantities are ``almost monotonic'' since they do not
differ much from the monotonic ones.

\section{Bernoulli sequences}

One can try to define randomness not only with respect to some
fixed measure but also with respect to some family of measures.
Intuitively a sequence is random if we can believe that it is
obtained by a random process that respects \emph{one of} these
measures.
As we show later, this definition can be given for
any \emph{effectively compact} class of measures.
But to make
it more intuitive, we start with a specific example: \emph{Bernoulli
measures}.

\subsection{Tests for Bernoulli sequences}

The Bernoulli measure $B_p$ arises from independent
tossing of a non-symmetric coin, where the
probability of success $p$ is some real number in $\clint{0}{1}$
(the same for all trials).
Note that we do not require $p$ to be computable.

\begin{definition}[Average-bounded Bernoulli test]\label{def:Bernoulli-test}
A lower semicomputable function $t$ on
infinite binary sequences is a \df{Bernoulli test} if its
integral with respect to any $B_p$ does not exceed $1$.
\end{definition}

\begin{proposition}[Universal Bernoulli test]
There exists a universal
(maximal up to a constant factor) Bernoulli test.
\end{proposition}
\begin{proof}
A lower semicomputable function is the
monotonic limit of basic functions.
If the integral of a given
basic function with respect to every $B_{p}$ is less or equal than $1$ for all $p$,
this fact can be established effectively (indeed, the integral
is a polynomial in $p$ with rational coefficients).
This allows us to eliminate all functions unfit to be tests, and to list all
Bernoulli tests.
Adding these up with appropriate coefficients, we obtain a universal one.
\end{proof}

\begin{definition}\label{def:univ.Bernoulli}
We fix a universal Bernoulli test and denote it $\t_{\cB}(\omega)$.
Its logarithm will be called \df{Bernoulli deficiency} $\d_{\cB}(\omega)$.
A sequence is called a \df{Bernoulli sequence} if its Bernoulli deficiency is finite.
\end{definition}

Again, we may modify the definition to within an additive
constant, to make it nonnegative and integer.

The informal motivation is the following: $\omega$ is a Bernoulli sequence if the
claim that it is obtained by independent coin tossing (coin
symmetry is not required) looks plausible.
And this statement is
not plausible if one can formulate some property that is true
for $\omega$ but defines an ``effectively Bernoulli null set''
(we did not formally introduce this notion, but could,
analogously to effective null sets).


Analogously to the case of computable measures, we can extend the class test
to finite sequences:

  \begin{definition}[Extended Bernoulli test]\label{def:extended-test.Bernoulli}
A lower semicomputable monotonic
function $T:\{0,1\}^{*}\to\clint{0}{\infty}$ is called an
\df{extended Bernoulli tests} if for all natural numbers $N$
and for all $p\in\clint{0}{1}$ the inequality
$\sum_{x: |x|=N}B_{p}(x)T(x) \le 1$ holds.
\end{definition}

As for computable measures, there is a connection between tests for finite and
tests for infinite sequences:

\begin{proposition}\label{propo:extended-univ.Bernoulli}
Every extended Bernoulli test generates a Bernoulli test over $\Omega$.
On the other hand, every Bernoulli test over $\Omega$ is generated by some
extended Bernoulli test.
\end{proposition}

There is a dominating universal extended Bernoulli test: it generats a universal
Bernoulli test on $\Omega$.
As earlier, we wil use the same notation $\t_{\cB}$ for the maximal tests on the
finite and on the infinite sequences.
Of course, it generates a universal Bernoulli test.

\subsection{Other characterizations of the Bernoulli property}

Just as for the randomness with respect to computable measures,
several equivalent definitions exist.
One may consider probability-bounded tests  (the probability of the event
$t(\omega)>N$ on any of the measures $B_{p}$ must be not greater than $1/N$).
One may call a test, following Martin-L\"of's definition for the computable
measures, any computable sequence of effectively open sets $U_{i}$ with
$B_{p}(U_{i})\le 2^{-i}$ for all $i$ and all $p\in\clint{0}{1}$.
All these variant definitions are equivalent (and this is proved just as for
randomness with respect to a computable measures).

 \begin{notation} Let $\bbB(n,k)$ denote the set of binary strings of length $n$
with $k$ ones (and $n-k$ zeroes).
 \end{notation}

Martin-L\"of defined a Bernoulli test as a family of sets of words
$U_{1}\supseteq U_{2}\supseteq U_{3}\supseteq\dotsm$;
each of these sets is hereditary upward, that is for every word contains all of
its extensions.
The following restriction is made on these sets: consider arbitrary integer
$n\ge 0$ and $k$ from $0$ to $n$; it is required that for all $i$ the share of
words in $\bbB(n,k)$ belonging to $U_{i}$ is not greater than $2^{-i}$.

For convenience of comparison let us replace the sets $U_{i}$ with an
integer-valued lower semicomputable function $d$ for which
$U_{i}=\setOf{x}{d(i)\ge i}$.
The hereditary property of the sets $U_{i}$ implies the monotonicity of this
function $d$ with respect to the prefix relation.
Besides this, it is required that the event $d\ge i$ within each set $\bbB(n,k)$
is not greater than $2^{-i}$.
Clearly, these requirements correspond to probability-bounded extended tests (in
the logarithmic scale), only in place of the class $B_{p}$ on words of length
$n$ another set of measures is considered, those concentrated on words of a
given length with a given number of ones.
The measures in the class $B_{p}$ take equal values on words of equal lengths
with equal number of ones, and are therefore representable by a mixture of
uniform measures on $\bbB(n,k)$ with some coefficients.
Replacing $B_{p}$ with these measures, the condition becomes stronger.

Let us show that nonetheless, the set of Bernoulli sequences does not change
from such a replacement; moreover, the universal test (as a function on infinite
sequences) does not change (as usual, to within a bounded factor).
We will show this for the average-bounded variant of tests (changing
Martin-L\"of's definition accordingly); this does not change the class of
Bernoulli sequences.
The reasoning is analogous for the probability-bounded tests.

 \begin{definition}\label{def:combinat-Bernoulli}
A \df{combinatorial Bernoulli test} is a function
$f:\{0,1\}^{*}\to\clint{0}{\infty}$ with the following constraints:
 \begin{alphenum}
  \item It is lower semicomputable.
  \item It is monotonic with respect to the prefix relation.
  \item For all integer $n,k$ with $0\le k\le n$ the average of the function $f$
    on the set $\bbB(n,k)$ remains below 1:
 \begin{align}\label{eq:combinat-Bernoulli-sum}
 |\bbB(n,k)|^{-1}\sum_{x\in\bbB(n,k)} f(x) &\le 1.
 \end{align}
 \end{alphenum}
 \end{definition}
The last condition says that not only is the average of $f(x)$ bounded by 1
over the set $\{0,1\}^{n}$, as in extended tests for the unbiased
coin-tossing measure, but its average is bounded by 1 separately in each set
$\bbB(n,k)$ whose union is $\{0,1\}^{n}$.

Having such a test for words of bounded length, it can be continued by monotonicity:

 \begin{proposition}\label{propo:Bernoulli-extend}
If a combinatorial Bernoulli test $f(x)$ is given on strings $x$
of length less than $n$, then extending it to longer strings using
monotonicity we get a function that is still a combinatorial Bernoulli test.
 \end{proposition}
 \begin{proof}
We extend $f$ to words of length $n$, setting $f(x0)=f(x1)=f(x)$ for words $x$ of
length $n-1$.
The set $\bbB(n,k)$ consists of two parts: words ending on zero and words ending
on one.
The first ones are in a one-to-one correspondence with $\bbB(n-1,k)$, the second ones
with $\bbB(n-1,k-1)$.
The function conserves the values in this correspondence, therefore the average
in both parts is not greater than 1.
Hence, the average over the whole $\bbB(n,k)$ is not greater than 1.
 \end{proof}

The following is obtained by standard methods:

 \begin{proposition}[Universal combinatorial Bernoulli test]
Among combinatorial Bernoulli tests, there is one that is maximal to within a
bounded factor.
 \end{proposition}

\begin{definition}
Let us fix a universal combinatorial Bernoulli test $\b(x)$ and extend it to
infinite sequences $\omega$ by
 \begin{align*}
   \b(\omega) = \sup_{x\prefix\omega} \b(x).
 \end{align*}
We will call the function obtained this way a universal combinatorial test on
$\Omega$ and will denote it also by $\b$.
\end{definition}
(By monotonicity, the least upper bound in this definition can be replaced with
a limit.)
Let us show that the this test coincides (to within a bounded factor) with the
Bernoulli tests introduced earlier in Definition~\ref{def:univ.Bernoulli}.

 \begin{theorem}\label{thm:combinat-Bernoulli}
$\b(\omega) \eqm \t_{\cB}(\omega)$.
 \end{theorem}
 \begin{proof}
We have already seen that $\b(x)$ is an extended Bernoulli test (from the bounds
on the average on each part $\bbB(n,k)$ follows the bound on the expected value
by the measure $B_{p}$, since this measure is constant on each part).
Consequently $\b(\omega)\lem\t_{\cB}(\omega)$.

The converse is not true: an extended Bernoulli test may not be a combinatorial
test.
But it is possible to construct a combinatorial test that takes the same values
(to within a bounded factor) on the infinite sequences, and only this is
asserted in the theorem.

Here is the idea.
Consider an extended Bernoulli test $t$ on words of length $n$ and transfer it
to words of much greater length $N$ (applying the old test to its beginnings of
length $n$).
We obtain a certain function $t'$.
We have to show that $t'$ is close to some combinatorial test (that is only exceeds it by a
constant factor).
For this, $t'$ must be averaged over the set $\bbB(N,K)$ for an arbitrary $K$
between $0$ and $N$.
In other words, we must average $t$ by the probability distribution on the
$n$-bit prefixes of sequences of length $N$ containing $K$ ones.
With $N\gg n$ this distribution will be close to the Bernoulli one with distribution
$p=K/N$.

In terms of elementary probability theory, we have an urn with $N$ balls, $K$ of
which is black, and take out from it $n$ balls.
We must compare the probability distribution with the Bernoulli one that would
have been obtained at sampling with replacement.
Let us show that
\begin{quote}
  \emph{for $N=n^{2}$ the distribution without replacement does not exceed the one with
  replacement more than $O(1)$ times}.
\end{quote}
(The inequality does not hold in the other direction: for $K=1$ without replacement
we cannot obtain a word with two ones, and with replacement we can.
But we only need the inequality in the given direction.)

Indeed, in sampling without replacement the probability that a ball of a given
color will be drawn is equal to the quotient
 \begin{align*}
   \frac{\text{the number of remaining balls of this color}}{\text{the number of
       all remaining balls}}.
 \end{align*}
The number of balls of this color is not more than in the case with replacement,
on the other hand the denominator is at least $N-n$.
Therefore the probability of any combination during sampling with replacement is
at most the probability of the same combination with replacement, multiplied by
$N/(N-n)$ to the power $n$.
For $N=n^{2}$ the multiplier $(1+O(1/n))^{n}=O(1)$ is obtained.

This way, taking the extended Bernoulli test $t$ and then defining $t'(x)$ on a
word $x$ of length $N$ as $t$ on the prefix of $x$ of length $\flo{\sqrt{N}}$,
the obtained function $t'$ will be a combinatorial test to within a
bounded factor.
(Note that its monotonicity follows from that of $t$.)
 \end{proof}

\subsection{Criterion for Bernoulli sequences}

It is natural to compare the notion of Bernoulli sequence (those sequences for
which the Bernoulli test is finite) with the notion of a sequence random with
respect to the measure $B_{p}$.
But Martin-L\"of definition of randomness assumes that
the measure is computable.
Therefore it cannot be applied
directly to $B_p$ if $p$ is non-computable.
But this definition
can be relativized, and if (the binary expansion of) $p$ is
given as an oracle (see Remark~\ref{rem:oracles}),
then the measure $B_p$ becomes computable
and randomness is well defined.
The following theorem supports an intuitive idea of Bernoulli
sequence as a sequence that is random with respect to some
Bernoulli measure:

\begin{theorem}\label{thm:Bernoulli-oracle}
A sequence $\omega$ is a Bernoulli sequence if
and only if it is random with respect to some
measure $B_{p}$, with oracle~$p\in\clint{0}{1}$.
\end{theorem}
By ``with oracle $p$'', we understand the possiblity to obtain from each $i$ the
$i$th bit in the binary expansion of the real number $p$
(which is essentially unique, except in those cases when $p$ is binary-rational,
and in these cases both expansions are computable, and the oracle is trivial).

Before proving the theorem (even in a stronger quantitative form),
we introduce a new notion,
of a test depending explicitly on the parameter $p$ of the Bernoulli measure
$B_{p}$,  which later will be extended to arbitrary (not just Bernoulli)
measures.
The required result will be obtained as the combination of the following claims:
\begin{alphenum}
\item Among the ``uniform'' randomness tests, there exists a maximal test
  $\t(\omega,p)$.
\item The function $\omega\mapsto\inf_{p}\t(\omega,p)$ coincides (as usual, to
  within a bounded factor) with the universal Bernoulli test.
\item For a fixed $p$, the function $\omega\mapsto\t(\omega,p)$ coincides (to
  the same precision) with the maximal randomness test for the ($p$-computable)
  measure $B_{p}$, relativized to $p$.
\end{alphenum}
These three assertions imply Theorem~\ref{thm:Bernoulli-oracle} easily: sequence
$\omega$ is Bernoulli, if the Bernoulli test is finite; the latter is equal to the
greatest lower bound of $\t(\omega,p)$, hence its finiteness means
$\t(\omega,p)<\infty$ for some $p$, which is equivalent to the relativized
randomness with respect to the measure $B_{p}$.

We need some technical preparation.
The randomness tests (as functions of two variables) will also be lower
semicomputable, but the definition of this concept needs to be extended,
since an additional real parameter is involved.
(In what follows we will also consider a more general situation, in which the
second argument is a measure.)

\begin{definition}\label{def:lower-semicomp.product}
In the space $\Omega\times\clint{0}{1}$,
let us call \df{basic rectangles}
all sets of the form $x\Omega\times\opint{u}{v}$, where $u<v$ are rational
numbers.
(A technical point: we allow $u,v$ to be outside $\clint{0}{1}$, but in this
case the rectangle we mean is $x\Omega\times(\clint{0}{1}\cap\opint{u}{v})$.)

A function $f:\Omega\times\clint{0}{1}\to\clint{-\infty}{\infty}$
is called \df{lower semicomputable} if there is an algorithm that, given a
rational $r$ on its input, enumerates a sequence of basic rectangles whose union
is the set of all pairs $\pair{\omega}{p}$ with $f(\omega,p)>r$.

The notion of \df{upper semicomputability} is defined analogously, and is
equivalent to the lower semicomputablity of $(-f)$.

A function with finite real values is called \df{computable} if it is both upper
and lower semicomputable.
\end{definition}

This definition, as earlier, requires that the preimage of $\opint{-\infty}{r}$
be an effective open set uniformly in $r$, only now we consider effectively open
sets in $\Omega\times\clint{0}{1}$, defined in a natural way.

Since the intersection of effective open sets is effective open, the
following---more intuitive---formulation is obtained for computability:

\begin{proposition}
  A real function $f:\Omega\times\clint{0}{1}\to\bbR$ is computable if
and only if for every rational interval $\opint{u}{v}$ its preimage
is the union of a sequence of basic rectangles that are effectively enumerated,
uniformly in $u$ and $v$.
\end{proposition}
The intuitive meaning of this characterization will become clearer after
observing that to ``give approximations to $\alpha$ with any given precision''
is equivalent to ``enumerate all intervals containing $\alpha$''.
Therefore for a computable function $f$ we can find approximations to
$f(\omega,p)$, if we are given appropriate approximations to $\omega$ and $p$.

We can reformulate the definition of (non-negative) lower semicomputable
function, introducing the notion of basic functions.
It is important for us that the basic functions are continuous, therefore the
dependence on the real argument will be piecewise linear, without jumps.

\begin{definition}[Basic functions, Bernoulli case]\label{def:basic-functions.Bernoulli}
We define an enumerated list of \df{basic} functions $\cE=\{e_{1},e_{2},\dots\}$
over the set $\Omega\times\clint{0}{1}$ as follows.
For $x\in\{0,1\}^{*}$, positive integer $k$ and
rational numbers $u,v$ with $u+2^{-k}<v-2^{-k}$ define the function
$g_{x,u,v,k}(\omega,p)$ as follows.
If $x\not\prefix\omega$, then it is 0.
Otherwise, its value does not depend on $\omega$ and
depends piecewise linearly on $p$:
it is 0 if $p\not\in\opint{u}{v}$ and 1 if $u+2^{-k}\le p\le v-2^{-k}$, and
varies linearly in between.
Now $\cE$ is the smallest set of functions containing all
$g_{x,u,v,k}$, and closed under maxima, minima and rational
linear combination.
\end{definition}

Now lower semicomputable functions admit the following equivalent characterization:

\begin{proposition}\label{propo:lower-semicomp-as-limit}
A function $f:\Omega\times\clint{0}{1}\to\clint{0}{\infty}$ is lower
semicomputable if and only if it is the pointwise limit
of an increasing computable sequence of basic functions.
(It follows that basic functions are computable.)
\end{proposition}
\begin{proof}
  This would be completely clear if for basic functions we also allowed the
indicator functions of basic rectangles and the maxima of such functions.
But we want the basic functions to be continuous (this will be important in what
follows).
One must note therefore that for $k\to\infty$ the function $g_{x,u,v,k}$
converges to the indicator function of a rectangle.
\end{proof}

The continuity of the basic functions guarantees the following important property:

\begin{proposition}\label{propo:integral-computable.Bernoulli}
  Let $f:\Omega\times \clint{0}{1}\to\bbR$ be a basic function.
The integral $\int f(\omega,p)\,B_{p}(d\omega)$ is a computable function of
the parameter $p$, uniformly in the code of the basic function $f$.
\end{proposition}
(Computability is understood in the above described sense; we remark that every
computable function is continuous.
An analogous statement holds for an arbitrary computable function $f$, not only
for basic functions, but we do not need this.)

The following fact, proved in~\cite{HoyrupRojasRandomness09},
will be used in the present paper a number of times,
also in generalizations, but with essentially the same proof.

 \begin{proposition}[Trimming]\label{propo:trim}
Let $\varphi:\Omega\times\clint{0}{1}\to\clint{0}{\infty}$
be a lower semicomputable function.
There is a lower semicomputable function $\varphi'(\omega,p)$ not exceeding
$\varphi(\omega,p)$ with the property that for all $p$:
 \begin{alphenum}
  \item $\int \varphi'(\omega,p)B_{p}(d\omega)\le 2$;
  \item If $\int\varphi(\omega,p)B_{p}(d\omega)\le 1$ then
    $\varphi'(\omega,p)=\varphi(\omega,p)$ for all $\omega$.
 \end{alphenum}
 \end{proposition}
 \begin{proof}
By Proposition~\ref{propo:lower-semicomp-as-limit}, we can represent
$\varphi(\omega,p)$  as a sum of a series of basic functions
$\varphi(\omega,p)=\sum_{n}h_{n}(\omega,p)$.
The integral $\int\sum_{i\le n} h_{i}(\omega,p)B_{p}(d\omega)$ is computable
by Proposition~\ref{propo:integral-computable.Bernoulli},
as a function of $p$ (uniformly in $n$), therefore the set $S_{n}$ of all $p$
where this integral is less than $2$
is effectively open, uniformly in $n$.

Define now $h'_{n}(\omega,p)$ as $h_{n}(\omega,p)$ for all $p\in S_{n}$,
and 0 otherwise.
The function $h'_{n}(\omega,p)$ is lower semicomputable, and the integral
$\int\sum_{i\le n} h'_{i}(\omega,p)B_{p}(d\omega)$ will be less than $2$ for all $p$.
Defining $\varphi'=\sum_{n}h'_{n}$ we obtain a lower semicomputable function,
and the theorem on the integral of monotonic limits gives that
$\int \varphi'(\omega,p)B_{p}(d\omega)$ is less than $2$ for all $p$.

It remains to note that if for some $p$ the integral
$\int\varphi(\omega,p)B_{p}(d\omega)$ does not exceed 1, then this $p$ enters all
sets $S_{n}$, and the change from $h_{n}$ to $h'_{n}$ as well as the change from
$\varphi$ to $\varphi'$ does not change it.
 \end{proof}

Now we are ready to introduce tests depending explicitly on $p$:

\begin{definition}\label{def:uniform.Bernoulli}
A \df{uniform test for Bernoulli measures} is a function $t$
of two arguments $\omega\in\Omega$ and $p\in\clint{0}{1}$; informally,
$t(\omega,p)$ measures the amount of nonrandomness (``regularity'') in
the sequence $\omega$ with respect to distribution $B_p$.
We require the following:
\begin{alphenum}
\item\label{i:uniform.Bernoulli.joint-lower}
$t(\omega,p)$ is lower semicomputable jointly as a function of the pair
  $\pair{\omega}{p}$.

 \item For every $p\in\clint{0}{1}$ the expected
value of $t(\omega,p)$ (that is $\int t(\omega,p)B_p(d\omega)$) does
not exceed~$1$.
\end{alphenum}
\end{definition}

It remains to prove the three assertions promised earlier:

\begin{lemma}\label{lem:univ-unif.Bernoulli}
There exists a universal uniform test $\t(\omega,p)$, that is a test that
multiplicatively dominates all uniform tests for Bernoulli measures.
\end{lemma}

\begin{lemma}\label{lem:Bernoulli-from-unif}
For the universal uniform test $\t$ of lemma~\ref{lem:univ-unif.Bernoulli},
the function $\t'(\omega)=\inf_p \t(\omega,p)$
coincides (to within a bounded factor in both
directions) with the universal Bernoulli test of
Definition~\ref{def:univ.Bernoulli}.
\end{lemma}

This lemma implies that $\omega$ is a Bernoulli sequence iff
$\t'(\omega)$ is finite, that is $\t(\omega,p)$ is finite for some
$p\in\clint{0}{1}$.

\begin{lemma}\label{lem:Bernoulli-oracle}
For a fixed~$p$ the function $\t_p(\omega)=\t(\omega,p)$
coincides (to within a bounded factor)
with the universal randomness test with respect to
$B_p$ relativized with oracle~$p$.
\end{lemma}

\begin{proof}[Proof of Lemma~\protect\ref{lem:univ-unif.Bernoulli}]
Generate all lower semicomputable functions;
using Proposition~\ref{propo:trim}, they
can be then trimmed to guarantee that all expectations do not
exceed, say, $2$, and all uniform tests should get through
unchanged.
Sum up all the trimmed functions with coefficients whose sum is less than $1/2$.
\end{proof}

\begin{proof}[Proof of Lemma~\protect\ref{lem:Bernoulli-from-unif}]
Let us show that $\t'(\omega)$ is a universal Bernoulli test.
The integral of this
function with respect to $B_p$ does not exceed $1$ since this
function does not exceed $\t(\omega,p)$ for that $p$.
The statement that this function is lower semicomputable
(as a function of $\omega$)
is analogous to Proposition~\ref{propo:inf-lower-semicomp.seqs}, and the
proof is also analogous, relying on compactness.
Both are special cases of the general theorem given in
Proposition~\ref{propo:lsc-min}.

Therefore the function $\inf_p \t(\omega,p)$ is a Bernoulli test.
The universality (maximality) follows obviously, since
any Bernoulli test can be considered a uniform Bernoulli
test of two variables that does not depend on variable $p$.
\end{proof}


\begin{proof}[Proof of Lemma~\protect\ref{lem:Bernoulli-oracle}]
Consider first the case when $p$ is a computable real number.
Then the function $\t_p\colon \omega\mapsto \t(\omega,p)$
(where $\t$ is a uniform randomness test for Bernoulli measures) is lower
semicomputable (we can enumerate all intervals that contain
$p$ and combine then with an algorithm for $\t$; in this way we
represent $\t_p$ as the least upper bound of the computable
sequence of basic functions).

A similar argument works for an arbitary $p$ and shows that $\t_{p}$ is lower
semicomputable with a $p$-oracle.
Thus, $\t_{p}$ does not exceed the universal relativized test with
respect to~$B_p$.

The reverse implication is a bit more difficult.
Assume that
$t$ is a lower semicomputable (with oracle $p$) randomness test
with respect to $B_p$.
We need to find a uniform Bernoulli test $t'$ that majorizes it
(for a given $p$).
This $t'$ must be lower
semicomputable, now (a subtle but important point)
using $p$ as an argument of the function $t'$, not as an
oracle.
In other words, one has to extend a function defined initially only for a single
$p$, to all values of $p$, while also guaranteeing the bound on the integral.

As a warmup consider the case of computable $p$.
Then no oracle is needed, and $t$ is lower semicomputable.
Adding dummy variable $p$ we get a
lower semicomputable function of two arguments.
But this
function may not be a uniform test since its expectation with
respect to $B_q$ may be arbitrary if $q\ne p$.
However, Proposition~\ref{propo:trim} helps
transform it into a $t'$ (which will now really depend on
$q$) with $\int  t'(\omega,q)B_{q}(d\omega) \le 2$ for all $q$ and
$t'(\cdot,p)=t(\cdot,p)$.
Dividing $t'$ into half provides a uniform test.

Now consider the case of noncomputable $p$.
In this case $p$ is irrational, so the bits of its
binary expansion can be obtained from any sequence of decreasing rational
intervals that converge to $p$.
Therefore an oracle machine that
enumerates approximations for $t$ from below (having $p$ as an
oracle) can be transformed into a machine that enumerates from
below some function $\tilde t(\omega,q)$, that coincides with
$t(\omega)$ if $q=p$.
The function $\tilde t$ may not be a uniform
Bernoulli test (its expectations for $q\ne p$ can be arbitrary);
but it again can be trimmed with the help of Proposition~\ref{propo:trim}.
\end{proof}


\section{Arbitrary measures over binary sequences}

In this section, we generalize the theory to arbitrary measures, not only
Bernoulli ones, but still stay in the space $\Omega$ of binary sequences.

\begin{notation}
The set of all probability measures over the space
$\Omega$ is denoted by $\cM(\Omega)$.
(Recall that the measure of the whole space $\Omega$ is equal to 1.)
\end{notation}

\subsection{Uniform randomness tests}\label{subsec:uniform-tests}

\begin{definition}[Uniform tests]\label{def:uniform-test.bin-Cantor}
 A \df{uniform} test is a lower semicomputable function $t(\omega,\P)$ of
two arguments ($\omega$ is a sequence, $\P$ is a measure on $\Omega$)
with
\begin{equation*}
\int t(\omega,\P)\, \P(d\omega) \le 1
\end{equation*}
for every measure $\P$.
\end{definition}

However, we have to define carefully the
notion of a lower semicomputability in this case.
The set $\cM(\Omega)$ of all measures
is a closed subset of the infinite (countable) product
\begin{equation}\label{eq:inf-product}
 \Xi=\clint{0}{1}\times \clint{0}{1}\times \clint{0}{1}\times\dotsm
\end{equation}
(the measure is defined by the values $\P(x)$ for all strings $x$;
these values should satisfy the equations~\eqref{eq:measure.Omega}, so we get a
closed subset).

Let us introduce basic open sets and computability notions
for the set $\Omega\times\cM(\Omega)$.

\begin{definition}
  An (open) \df{interval} (basic open set) in the space of measures is
given by a finite set of conditions of type $u<\P(y)<v$ where $y$ is some binary
string and $u,v$ are some rational numbers; the basic open set consists of the
measures $\P$ that satisfy these conditions.
A \df{basic open set} in $\Omega\times\cM(\Omega)$ has the form $x\Omega\times
\beta$, (product of intervals in $\Omega$ and $\cM(\Omega)$)
where $\beta$ is a basic open set of measures.
Now lower and upper semicomputability and computability are defined in terms of
these basic open sets just as they were defined for $\Omega\times\clint{0}{1}$
in Definition~\ref{def:lower-semicomp.product}.
\end{definition}

In much of what follows, we will exploit the fact that, due to the finiteness of
the alphabet $\{0,1\}$, the space $\Omega$ of infinite binary sequences is
compact, and also the set of measures $\cM(\Omega)$ is compact.
Recall that a set $C$ is compact if every cover of $C$ by open sets contains a
finite subcover.
We need, however, an effective version of compactness:

\begin{definition}[Effective compactness]\label{def:effectively-compact}
  A compact subset $C$ of $\cM(\Omega)$ is called
\df{effectively compact} if the set
 \begin{align*}
   \setOf{S}{S \text{ is a finite set of basic open sets and } \bigcup_{E\in S}E\supseteq C}
 \end{align*}
is enumerable.
\end{definition}

The set $\cM(\Omega)$ itself is, as it is easy to see, compact and effectively
compact.
It is compact, as said above, as a closed set in the product of compact spaces,
and the effectivity follows from the fact that we can check whether some given
basic sets cover the whole space (we are dealing with linear equations and
inequalities in a finite number of variables, where everything is
algorithmically decidable).
From here, it also follows:

\begin{proposition}\label{propo:closed-to-compact}
Every effectively closed subset of $\cM(\Omega)$ is effectively compact.
\end{proposition}
\begin{proof}
Let an effectively closed subset $C$ of $\cM(\Omega)$.
be the complement of the union of a list $B_{1},B_{2},\dots$
of basic open sets.
Then a finite set $S$ of basic open sets covers $C$
if and only if together with a finite set of the $B_{i}$, it covers the whole space.
And this property is decidable.
\end{proof}

Effective compactness implies effective closedness.
This follows from the following two properties of our space and our basic open
sets:
\begin{alphenum}
  \item\label{i:separation} For every closed set $F$ and every point $x$ outside
    $F$ there are two disjoint open sets containing $F$ and $x$.
  \item\label{i:disjointness}
For every pair of basic open sets, it is uniformly decidable whether they are disjoint.
\end{alphenum}
Let $F$ be an effectively compact set.
We call a basic open set $B$ \df{manifestly disjoint} of $F$, if there is a finite
set of basic open sets $S$ disjoint of $B$ covering $F$.
Due to the effective compactness of $F$ and property~\eqref{i:disjointness},
the set of all basic open sets
manifestly disjoint of $F$ is enumerable.
Property~\eqref{i:separation} implies that it covers the complement of $F$.

In view of later generalization to cases where the space
itself may not be compact, we will refer to
some effectively closed sets of $\cM(\Omega)$ as effectively compact.

Now we introduce a dense set of computable functions called
basic functions on the set $\Omega\times\cM(\Omega)$, similarly
to Definition~\ref{def:basic-functions.Bernoulli}.
Their specific form is not too important.

\begin{definition}[Basic functions for binary sequences and arbitrary measures]
The set of \df{basic} functions over the set $\Omega\times\cM(\Omega)$ is
defined analogously to Definition~\ref{def:basic-functions.Bernoulli},
starting from the functions
 \begin{align*}
 g_{x,y,u,v,k}:\Omega\times\cM(\Omega)\to\clint{0}{1}
 \end{align*}
with $x,y\in\{0,1\}^{*}$ defined as follows.
If $x\not\prefix\omega$, then $g_{x,y,u,v,k}(\omega,\P)=0$.
Otherwise, its value does not depend on $\omega$ and
depends piecewise linearly on $\P(y)$ in a way that
it is 0 if $\P(y)\not\in\opint{u}{v}$ and
1 if $u+2^{-k}\le\P(y)\le v-2^{-k}$.
\end{definition}

The analogue of Proposition~\ref{propo:lower-semicomp-as-limit} holds again:
a lower semicomputable function is the monotonic limit of a computable sequence
of basic functions (which themselves are computable).

The analogue of Proposition~\ref{propo:integral-computable.Bernoulli} holds also: the
integral $\int f(\omega,\P)\P(d\omega)$
of a basic function is computable as a function of the measure $\P$, uniformly
in the number of the basic function.

Finally, the analogue of Proposition~\ref{propo:trim} holds again:

\begin{theorem}[Trimming]\label{thm:trim.Cantor}
Let $\varphi(\omega,\P)$ be a lower semicomputable function.
Then there exists a lower semicomputable function $\varphi'(\omega,\P)$ such that
for all $\P$:
\begin{alphenum}
 \item $\int \varphi'(\omega,\P)\P(d\omega)\le 2$,
 \item if $\int \varphi(\omega,\P)\P(d\omega)\le 1$ then
$\varphi'(\omega,\P)=t(\omega,\P)$ for all $\omega$.
\end{alphenum}
\end{theorem}

The proof is completely analogous to the proof we gave for
Proposition~\ref{propo:trim}.


This allows the construction of a universal test as a function of a sequence and
an arbitrary measure over $\Omega$:


\begin{theorem}
There exists a maximal (maximal to within a bounded factor)
uniform randomness test.
\end{theorem}
\begin{proof}
We use the same approach as before: we
trim a lower semicomputable function in such a way that it
becomes a test (or almost a test) and remains untouched if it
were a test in the first place.
\end{proof}

\begin{definition}
Let us fix a universal uniform randomness test $\t(\omega,\P)$.

We call a sequence $\omega$ \df{uniformly random}
with respect to a (not necessarily computable) measure $\P$
if $\t(\omega,\P)<\infty$.
\end{definition}

Let us show that for computable measures, the new definition coincides with the
old one.

\begin{proposition}\label{propo:unif-comput.Cantor}
Let $\P$ be a computable measure, let $\t_{\P}(\omega)$ be a
universal (average-bounded) randomness test for $\P$ as,
and $\t(\omega,\P)$ the universal uniform test defined above.
Then there are constants $c_{1},c_{2}>0$ such that
$c_{1}\t_{\P}(\omega)\le \t(\omega,\P)\le c_{2}\t_{\P}(\omega)$.
\end{proposition}
The constants $c_{1},c_{2}$ here depend on the choice of measure $\P$ and of the
choice of the test $\t_{\P}$ for this measure (this choice was done in an
arbitrary way for each computable measure).

This proposition shows, that in the case of the computable measures, uniform
randomness coincides with randomness in the sense of Martin-L\"of.
\begin{proof}
Let us show $\t(\omega,\P)\le c_{2}\t_{\P}(\omega)$ first.
The function $\omega\mapsto \t(\omega,\P)$ is lower
semicomputable since we can effectively enumerate all intervals
in the space of measures that contain $\P$; therefore it is dominated by
$\t_{\P}(\omega)$.

To prove $\t(\omega,\P)\ge c_{1}\t_{\P}(\omega)$,
consider the lower semicomputable function
 \begin{align*}
 t(\omega,\Q)=\t_{\P}(\omega).
 \end{align*}
The function $\pair{\omega}{\Q}\mapsto t(\omega)$ is
not guaranteed to be a uniform randomness
test, since its integral can be greater than~$1$ if $\Q\ne\P$.
However, it can be trimmed without changing it at $\P$, and then it still
remains (almost) a test.
\end{proof}

We are also interested in tests defined just for one, not necessarily
computable, measure $\P$:

 \begin{definition}\label{def:P-test.Cantor}
We will call a function $f:\Omega\to\clint{0}{\infty}$ \df{lower semicomputable}
relatively to measure $\P$ if it is obtained from a lower semicomputable
function on the set $\Omega\times\cM(\Omega)$ after fixing the second argument
at $\P$.

For a measure $\P\in\cM(\Omega)$, a $\P$-\df{test of randomness}
is a function $f:\Omega\to\clint{0}{\infty}$ lower semicomputable from $\P$ with
the property $\int f(\omega)\,d P \le 1$.
\end{definition}

It seems as if a $\P$-test may capture some nonrandomnesses that uniform tests
cannot---however, this is not so, since trimming (see
Theorem~\ref{thm:trim.Cantor}) generalizes:

\begin{theorem}\label{propo:uniformize.Cantor}
Let $\P_{0}$ be some measure along with
some $\P_{0}$-test $t_{\P_{0}}(\omega)$.
There is a uniform test $t'(\cdot,\cdot)$ with $t_{\P_{0}}(\omega)\le 2
t'(\omega,\P_{0})$.
On the other hand, the restriction of any uniform test to the measure $\P$ is a
$\P$-test.
\end{theorem}
The notion of extended text can be generalized to uniform tests:

\begin{definition}[Extended uniform test]\label{def:extended-test.uniform}
A lower semicomputable function
$T:\{0,1\}^{*}\times\cM(\Omega)\to\clint{0}{1}$ monotonic with respect to
the prefix relation is called an
\df{extended uniform test} if for all $n$
and all distributions $P$ we have $\sum_{x: |x|=n}T(x,\P)\P(x)\le 1$.
\end{definition}
As earlier, due to monotonicity, we could sum not only over words of a given
length, but over an arbitrary prefix-free set.

The following follows from
the analogue of Proposition~\ref{propo:lower-semicomp-as-limit}
(representing a nonnegative
lower semicomputable function as a sum of nonnegative basic
functions):

\begin{proposition}\label{propo:extended-test.uniform}
Every uniform test $t(\omega,\P)$ can be generated by an extended uniform test
in the sense of $t(\omega,\P)=\sup_{x\prefix\omega}T(x,\P)$.
Conversely, every extended uniform test $T$ generates a uniform test $t$.
\end{proposition}

Among the uniform extended tests, it is also possible to select a maximal one
(using an analogous trimming method and summing the results).
We fix an extended uniform test and denote it $\t(x,\P)$ (where
$x\in\{0,1\}^{*}$, and $\P$ is a measure over $\Omega$).
It generates a maximal uniform test $\t(\omega,\P)$ (to within a bounded
factor).

\begin{remark}\label{rem:seq-of-natural}
Much of the theory worked out at the beginning of this paper for 0-1 sequences
holds also for sequences whose elements are arbitrary natural numbers.
The extended tests of Definition~\ref{def:extended-test.uniform}
generalize, and the existence of a uniform universal extended test is proven in the same
way.
But it becomes important to define extended tests directly, and not via tests
for infinite sequences, since compactness may not hold.
\end{remark}

Proposition~\ref{propo:unif-comput.Cantor}
allows us to generalize a result about Bernoulli measures:

\begin{theorem}\label{thm:oracle2uniform}
Let $\P$ be a measure computable with some oracle $A$.
Assume also that $A$ can be effectively reconstructed as the values of the
measure are provided with more and more precision.
Then a sequence $\omega$ is
uniformly random with respect to $\P$ if and only if it is random with respect
to $\P$ with oracle $A$.
\end{theorem}

(Since the oracle $A$ makes $\P$ computable, the notion of Martin-L\"of randomness
is well defined.)

\begin{proof}
 Assume that $\t(\omega,\P)=\infty$ for the universal uniform test $\t$.
Note that
 $\t(\cdot,\P)$ is an $A$-lower semicomputable function and is a $\P$-test, so
 $\omega$ is nonrandom with respect to $\P$ with oracle $A$.

On the other hand, let $t(\omega,A)$ be some $A$-lower semicomputable $\P$-test with
$t(\omega,A)=\infty$.
That $A$ can be reconstructed from $\P$ means that there is a computable
mapping $f$ from measures to binary sequences (oracles) defined at least over $\P$,
with $A=f(\P)$.
But then $\pair{\omega}{\P}\mapsto t(\omega,f(\P))$ is a $\P$-test.
The uniformization theorem~\ref{propo:uniformize.Cantor} converts it
into a uniform test that is infinite on $\pair{\omega}{\P}$.
\end{proof}

Let us note that not all measures $\P$ satisfy the condition of the theorem (it
means that the mass problem of ``show approximations to the values of $\P$'' is
equivalent to the decision problem of some set; on the degrees of such mass
problems, see~\cite{MillerDegreesCont04}).
Later, in Theorem~\ref{thm:some-oracle}, we show a characterization of uniform
randomness for arbitrary measures (in terms of Martin-L\"of randomness with
oracle).

Another application of the trimming technique:
let us show that the notion of uniform randomness test is indeed a
generalization of the notion of an uniform Bernoulli test we introduced earlier
in Definition~\ref{def:Bernoulli-test}.

\begin{theorem}
Let $\t(\omega,\P)$ be the universal uniform test and let $\t(\omega,p)$ be
the universal uniform Bernouli test defined in Lemma~\ref{lem:univ-unif.Bernoulli}.
Then $\t(\omega,B_{p})\eqm \t(\omega,p)$.
\end{theorem}
(Here $B_p$ is the Bernoulli measure with parameter $p$.)
\begin{proof}
For the inequality $\lem$ note that
the function $\pair{\omega}{p}\mapsto \t(\omega,B_{p})$
is an uniform Bernoulli test, since the mapping $p\mapsto B_p$
is computable mapping (in a natural sense).

For the other direction, there exists a computable function on measures
that maps $B_p$ to $p$ (just take the probability of the one-bit string).
Combining
this function with $\t(\omega,p)$, we get a lower semicomputable function
$f(\omega,\P)$ on general measures $\P$ with
$f(\omega,B_p)=\t(\omega,p)$.
The function $f(\omega,p)$ is not a uniform test yet,
but again the trimming technique given by Theorem~\ref{propo:trim}
yields the desired result.
\end{proof}

\subsection{Apriori probability with an oracle, and uniform tests}
\label{subsec:uniform-exact}

For a computable measure, we had an expression for the universal test via
apriori probability in Proposition~\ref{propo:sum-characteriz}.
An analogous expression exists also for the universal uniform test:

\begin{theorem}\label{thm:uniform-test-sum}
\begin{align*}
   \t(\omega,\P) \eqm \sum_{x\prefix\omega}\frac{\m(x\mid\P)}{\P(x)}.
 \end{align*}
\end{theorem}
To be honest, we still owe the reader the definition of the concept of apriori
probability with respect to a measure, that is the quantity $\m(x\mid\P)$.
We do this right away, before returning to the proof.

\begin{definition}\label{def:uniform-semimeasure}
A nonnegative function $t(x,\P)$ whose arguments are the binary word $x$ and the measure
$\P$ will be called a \df{uniform lower semicomputable semimeasure}, if it is
lower semicomputable and $\sum_{x} t(x,\P)\le 1$ for all measures $\P$ over
$\Omega$.
\end{definition}

\begin{proposition}\label{propo:apriori-universal-uniform}
Among the uniform lower semicomputable semimeasures, there is a largest one to
within a multiplicative constant.
\end{proposition}

This is proved by the same method as the existence of a universal test (and even
simpler, since the constraints on the values of the test do not depend on the
measure).

\begin{definition}
We will fix one such largest semimeasure, and call it the \df{apriori probability
with respect to} $\P$.
We will denote it by $\m(x\mid\P)$.
\end{definition}
(The vertical bar in place of a comma emphasizes the similarity
to the conditional apriori probability normally considered.)

\begin{proof}[Proof of Proposition~\protect\ref{thm:uniform-test-sum}]
We need to check two things.
First we need to convince ourselves that the right-hand side of the formula
defines a uniform test.
Every member of the sum can be considered to be a function of two arguments, equal
to 0 outside the cone of extensions of $x$, and equal to $\m(x\mid\P)/\P(x)$ inside
the cone.
For every $x$, the functions $\m(x\mid\P)$ and $1/\P(x)$ are lower semicomputable
(uniformly in $x$), and the sum gives a lower semicomputable
function.
The integral of this function by any measure $\P$ is equal to the sum of the
integrals of the members, that is $\sum_{x}\m(x\mid\P)$, and therefore does not
exceed 1.

There is a special case, when $\P(x)=0$ for some $x$.
In this case the corresponding member of the sum becomes infinite for any
$\omega$ extending $x$.
But since the measure of this cone is zero, the integral by this measure is by
definition zero, and therefore the additive term, if is not equal to
$\m(x\mid\P)$, is simply smaller.
This way, the right-hand side of the formula is a uniform test, and therefore
does not exceed the universal uniform test: we proved the inequality $\gem$.

The second part of the proof is not so simple: observing the increase of the
values of the uniform test, we must distribute this increase among the different
members of the sum of the right-hand side, while preserving lower
semicomputability.
The difficulty is that if, say, the lower semicomputable function was 1 on some
effectively open set $A$, and outside it was zero, and then this set was changed
to a larger set $B$, then the difference (the characteristic function of
$B\setminus A$) will not in general be lower semicomputable since in the set of
measures (as also on a segment) the difference of two intervals will not be an
open The.

This problem is solved by moving to continuous functions.
Let us be given an arbitrary uniform test $t(\omega,\P)$.
Since it is lower semicomputable, it can be represented as the limit of a
nondecreasing sequence of basic functions, or---passing to differences---in the
form of a sum of a series of nonnegative basic functions:
$t(\omega,\P)=\sum_{i}t_{i}(\omega,\P)$.

Being basic, the function $t_{i}$ of $\omega$ depends only on a finite prefix of the sequence
$\omega$; denote the length of this prefix by $n_{i}$.
For every word $x$ of length $n_{i}$ we get some lower semicomputable function
$t_{i,x}(\P)$, where $t_{i}(\omega,\P)=t_{i,x}(\P)$ if $\omega$ begins by $x$.
Now define $m_{i}(x,\P)=t_{i,x}(\P)\cdot\P(x)$, if $x$ has length $n_{i}$ (for
the other lengths, zero).
The function $m_{i}$ is lower semicomputable (as the product of two lower
semicomputable functions) uniformly in $i$, and therefore the sum
$m(x,\P)=\sum_{i}m_{i}(x,\P)$ will be lower semicomputable.

Let us show that $m$ is a semimeasure, that is $\sum_{x}m(x,\P)\le 1$ for all
$\P$.
Indeed, in $\sum_{i}m_{i}(x,\P)$ the nonzero terms correspond to words of length
$n_{i}$, and this sum is equal to $\sum_{x}t_{i,x}(\P)\cdot\P(x)$, that is exactly the
integral $\int t_{i}(\omega,\P)\,\P(d\omega)$, and the sum of these integrals
does not exceed 1 by our condition.

Moreover, if for all prefixes $x$ of the sequence $\omega$ the measure $\P(x)$
is not equal to zero, then
 \begin{align*}
   \sum_{x\prefix\omega}\frac{m_{i}(x,\P)}{\P(x)}=\frac{t_{i,x_{i}}(\P)\cdot\P(x_{i})}{\P(x_{i})}
 = t_{i}(\omega,\P)
 \end{align*}
(here $x_{i}$ is the prefix of length $n_{i}$ of $\omega$), hence after summing
over $i$
 \begin{align*}
   \sum_{x\prefix\omega}\frac{m(x,\P)}{\P(x)}=t(\omega,\P),
 \end{align*}
and it just remains to apply the maximality of the apriori probability to obtain
the $\lem$-inequality for the case that all prefixes of $\omega$ have nonzero
$\P$-measure.
On the other hand, if one of these has zero $\P$-measure, then the right-hand
side is infinite, and so here the inequality is also satisfied.
\end{proof}

\begin{question}
  For the universal randomness test with respect to a computable measure,
in this formula one could replace the sum with a maximum.
Is this possible for uniform tests?
(The reasoning applied there encounters difficulties in the uniform case.)
Can one define apriori probabilithy on the tree in a reasonable way, and prove a
uniform variant of the Levin-Schnorr theorem?
\end{question}

\subsection{Effectively compact classes of measures}

We have considered Bernoulli tests, that is lower
semicomputable functions $t(\omega)$ that are tests with respect
to all Bernoulli measures.
In this definition, in place of Bernoulli measures, an arbitrary effectively
compact class can be taken:

\begin{definition}\label{def:test-effectively-compact.Cantor}
Let $\cC$ be an effectively compact class of measures over $\Omega$.
We say that lower semicomputable
function $t$ on $\Omega$ is a \df{$\cC$-test} if
$\int t(\omega) \,d\P\le 1$ for every $\P\in\cC$.
\end{definition}

\begin{theorem}\label{thm:class-test}
Let $\cC$ be an effectively compact class of measures.
  \begin{alphenum}
 \item There exists a universal $\cC$-test $\t_{\cC}(\cdot)$.
 \item $\t_{\cC}(\omega)=\inf_{\P\in\cC}\t(\omega,\P)$.
  \end{alphenum}
\end{theorem}
\begin{proof}
Both of these statements are proved analogously to
Lemmas~\ref{lem:univ-unif.Bernoulli} and~\ref{lem:Bernoulli-from-unif}.
\end{proof}

\begin{remark}
Since $\cC$ is compact and the function $\t(\omega,\P)$
is lower semicomputable, the $\inf$-operation can be replaced by the $\min$-operation.
\end{remark}

\begin{question}
Can we give criteria for randomness with respect to natural closed
classes of measures (for example in terms of complexity)?
How can we describe
Bernoulli sequences in terms of complexities of their initial segments?
It is known that the main term of the randomness deficiency is
 \begin{align*}
   \log\binom{n}{k}-\KP(x\mid n,k).
 \end{align*}
The lecture notes~\cite{GacsLnotesAIT88} contains a characterization
Bernoulli sequences, but it is rather messy.

What about Markov measures?
Shift-invariant measures?
\end{question}

\subsection{Sparse sequences}

There are several situations closely related to some intuitive
understanding of randomness, but not fitting directly into the framework of
the question of a randomness of a given outcome $\omega$ to a given model
(measure $\P$).
Our example is here a natural notion of sparsity,
introduced in~\cite{BienvenuRomashShenSparse08},
but another example, online
tests, will be considered in Section~\ref{sec:too-strong}.

It is natural to call ``$p$-sparse'' a sequence $\omega$, when its 1's come from
some $p$-random sequence $\omega'$, but we allow some of its 0's to also come
from the 1's of $\omega'$.
For example, the 1's of $\omega'$ may be a sequence of miracles, and
$\omega$ is the sequence of those miracles that have been reported.
The tacit hypothesis is, of course, that all reported miracles actually happened.

\begin{definition}[Sparse sequences]
Let us introduce a coordinate-wise order between
infinite binary sequences (or binary sequences of the same length):
we say $\omega\le \omega'$ if this is true
coordinate-wise, that is $\omega(i)\le \omega'(i)$ for all $i$: in other words,
$\omega'$ is obtained from $\omega$ replacing some $0$'s with $1$s.

Let $B_p$ be a Bernoulli measure with some computable $p$.
We say that a binary
sequence $\omega$ is $p$-\df{sparse} if $\omega\le\omega'$
for some $B_p$-random sequence $\omega'$.
(In terms of sets, $p$-sparse sets are subsets of $p$-random sets).
\end{definition}

We will show that in the definition of sparsity, the existential quantifier can
be eliminated, giving a criterion in terms of monotonic tests.

\begin{definition}
A real function $f$ on $\Omega$ will be called \df{monotonic} if
$\omega'\ge\omega$ implies $f(\omega')\ge f(\omega)$.

A monotonic lower semicomputable function $t:\Omega\to\clint{0}{\infty}$ is a
$p$-\df{sparsity test} if $\int t(\omega)\, d B_{p}\le 1$.
A  $p$-sparsity test is \df{universal} if it multiplicatively dominates all
other sparsity tests for $p$.
\end{definition}

The monotonicity of tests guarantees, informally speaking, that only the
presence of some $1$s is counted as regularity, not their absence.
(Note that earlier we spoke of an entirely different kind of monotonicity, while
defining extended tests: there we compared the values of a function on a finite
word and its extension.)

\begin{proposition}
  Consider the universal test $\t(\omega,\P)$.
The expression
 \begin{align*}
   \r_{p}(\omega) =\min_{\omega'\ge\omega}\t(\omega',B_{p})
 \end{align*}
defines a universal $p$-sparsity test.
\end{proposition}
\begin{proof}
  Each $p$-sparsity test is by definition a test with respect to the measure
$B_{p}$.
Using its monotonicity and comparing it with the universal test we obtain that
no sparsity test exceeds $r_{p}$ (to within a bounded factor).

In the other direction it must be shown that the minimum in the expression for
$r_{p}$ is achieved, and that this function is a $p$-sparsity test.
The lower semicomputability is proved usin that the property $\omega\le\omega'$
gives an effectively closed set of the effectively compact space
$\Omega\times\Omega$.
The monotonicity and the integral inequality follow immediately from the definition.
\end{proof}

From this follows the following characterization in terms of tests:

\begin{theorem}
A sequence $\omega$ is $p$-sparse (is obtained from a $p$-random by replacing
some $1$s with zeros) if and only if the universal sparsity test
$\r_{p}(\omega)$ is finite.
\end{theorem}

Sparsity is equivalent to randomness with respect to a certain class of measures.
To define this class, we introduce the notion of coupling of measures.

\begin{definition}
For measures $\P,\Q$ we say $\P\preceq \Q$, or that
$\P$ can be \df{coupled below} $\Q$ if there exists a
probability distribution $R$ on pairs of sequences $\pair{\omega}{\omega'}$ such
that
\begin{alphenum}
  \item The first projection (marginal distribution) is $\P$ and the second one
    is $\Q$.
  \item Measure $R$ is entirely concentrated on pairs $\pair{\omega}{\omega'}$
    with $\omega\le\omega'$ (the probability of this event by the measure $R$ is $1$.
\end{alphenum}
\end{definition}

The following characterization of coupling is well known: it has many proofs,
but all seem to go back to~\cite{Strassen65} (Theorem  11, p. 436).
A proof can be found in~\cite{BienvenuRomashShenSparse08}.

\begin{proposition}\label{propo:coupling-characteriz}
The property $\P\preceq \Q$ is equivalent to the following:
for all monotonic basic functions $f$ the following inequality holds:
 \begin{align*}
\int f(\omega)\,d\P\le \int f(\omega)\,d \Q.
 \end{align*}
\end{proposition}

In this characterization, we could have said ``all monotonic integrable
functions'' as well.

\begin{definition}
  Let $\cS_{p}$ be the set of measures that can be coupled below $B_{p}$.
\end{definition}

\begin{proposition}
The set $\cS_{p}$ of measures is effectively closed (and thus effectively compact).
\end{proposition}
\begin{proof}
For each function $f$ in Proposition~\ref{propo:coupling-characteriz},
the condition defines an effectively closed set, and their intersection will
also be effectively closed.
\end{proof}

\begin{theorem}\label{thm:mon-class-test}
  The universal $p$-test $\r_{p}(\omega)$ is a universal class test for class $\cS_{p}$.
\end{theorem}
Thus, a sequence is $p$-sparse if and only if it is random with respect to some
measure that can be coupled below $B_{p}$.

The following lemma will be key to the proof.

\begin{lemma}[Monotonization]\label{lem:monotonize}
Let $t:\Omega\to\bbR$ be a basic function with
$\int t(\omega)\,d \Q\le 1$ for all $\Q\in\cS_{p}$.
Define the monotonic function $\hat t(\omega)=\max_{\omega'\le\omega}t(\omega')$
(the maximum is achieved since $t(\omega)$ depends only on finitely many
positions of $\omega$).
Then $\int\hat t(\omega)\,d B_{p}\le 1$.
\end{lemma}
\begin{proof}
Let function $t$ depend only on the first $n$ coordinates.
For each $x\in\{0,1\}^{n}$ fix $x'\le x$ for which $t(x')$ reaches the maximum
(among all such $x'$).
Besides the distribution $B_{p}$ consider a distribution $\Q$ in which the
Bernoulli measure of $x$ is tranferred to $x'$ (the measures of several $x$ may
be transferred to the same $x'$ and then be added).
We described the behavior of $\Q$ on the first $n$ bits; the following bits are
chosen independently, and the probability of 1 in each position is equal to $p$.
Note also that for the expected values of the functions $t$ and $\hat t$ only
the first $n$ bits count.

By the construction, $\Q\preceq B_{p}$ (essentially, we described a measure on
pairs), therefore $\int t(\omega)\,d\Q\le 1$.
But this integral is equal to $\int\hat t(\omega)\, d B_{p}$.
\end{proof}

Let us return to the theorem.

\begin{proof}[Proof of Theorem~\protect\ref{thm:mon-class-test}]
Every $p$-sparsity test $t$ is a class test for $\cS_{p}$.
Indeed, its integral by a measure in the class $\cS_{p}$ does not exceed its
integral by the measure $B_{p}$, by the monotonicity of the test and the
possibility of coupling.

On the other hand, let us show that for every test $t$ for the class $\cS_{p}$,
there is a a $p$-sparsity test that is not smaller.
Indeed, the test $t$ is the limit of an increasing sequence $t_{n}$ of basic
functions.
Applying to them the monotonization lemma~\ref{lem:monotonize}, we obtain a
sequence of basic functions $\hat t_{n}$ that are everywhere greater or equal to
$t_{n}$ and have integrals bounded by 1 with respect to the measure $B_{p}$.
Their limit is the needed $p$-sparsity test.
\end{proof}

\subsection{Different kinds of randomness}

There are several ways to define randomness with respect to an arbitrary (not
necessarily computable) measure.
We have already defined uniform randomness.
Here are some other ways.

\paragraph{Oracles}
We can use the Martin-L\"of definition (or its average-bounded
version) with oracles.
We would call a sequence $\omega$ random with respect
to $\P$, if there exists an oracle $A$ that makes $\P$ computable such that
$\omega$ is ML-random with respect to $\P$ with oracle $A$.
(We say ``there exists an oracle that makes $\P$ computable'' but not
``for all oracles that make $\P$ computable'': indeed, some powerful oracle
can always make
$\omega$ computable and therefore non-random, unless $\omega$ is an atom of
$\P$.)
As Adam Day and Joseph Miller have shown~\cite{DayMiller10}, this
definition turns out to be equivalent to uniform randomness.
The proof of this equivalence needs some preparation.

First let us look into why is it not possible to take for oracle the measure
itself (as was done for the Bernoulli measures, where for oracle we chose a
binary expansion of the number $p$).
Well, the choice of such a representation is not unique
($0.01111\dots=0.10000\dots$).
When all we have is a single number $p$ then this is not important, as the
non-uniqueness arises only for rational $p$, and in this case both
representations are computable.
But for measures this is not so: a measure is represented by a countable number
of reals (say, the probabilities of individual words, or the conditional
probabilities), and the arbitrariness in the choice of representation is not
reduced to a finite number of variants.

\begin{definition}\label{def:r-test}
Fix some representation of measures by infinite binary sequences, that is a
computable (and therefore continuous)
mapping $\pi\mapsto\R_{\pi}$ of $\Omega$ onto the space of measures.
For example, we may split the binary sequence $\pi$ into countably many parts
and use these parts as binary representations
of the probability that the sequence continues with $1$ after a certain prefix.

Define the notion of an \df{r-test} (representation-test, test of randomness
relative to a given representation of the measure) as a lower
semicomputable function $t(\omega,\pi)$ with
$\int t(\omega,\pi)\R_{\pi}(d\omega)\le 1$ for all $\pi$.
\end{definition}

This notion of r-test depends on the
representation method chosen; there are no intuitive reasons to choose one
specific representation and declare it to be ``natural'', but any representation
is good for the argument below and we assume some representation fixed.
The following statements can be proven just as similar statements before:
  \begin{alphenum}
  \item\label{i:some-oracle.trim}
 Every lower semicomputable function $t(\omega,\pi)$ can be trimmed to make it
not greater than twice an r-test
(not changing it for those $\pi$ where it already was a r-test).
 \item\label{i:some-oracle.univ}
 There exists an universal (maximal to within a bounded factor) r-test $\t(\omega,\pi)$.
  \end{alphenum}

For a fixed $\pi$, the function $\t(\cdot,\pi)$ is universal among the
$\pi$-computable average-bounded tests with respect to the measure $\R_{\pi}$.
Indeed, it is such a test; on the other hand, any such test can be lower
semicomputed by the oracle machine.
This machine is applicable to any oracle (though may not give a test), giving
a lower semicomputable function $t'(\omega,\pi)$ that is equal to the starting
test for the given $\pi$.
It remains to apply property~\eqref{i:some-oracle.trim}.

As a consequence of this simple reasoning we obtain that the quantity
$\t(\omega,\pi)$ is finite if and only if the sequence $\omega$ is random
relative to the oracle $\pi$, with respect to measure $\R_{\pi}$.

\begin{theorem}[Day-Miller]\label{thm:some-oracle}
A sequence $\omega$ is uniformly random with respect to measure $\P$ if and only
if there is an oracle computing $\P$ that makes $\omega$ random
(in the original Martin-L\"of sense).
More precisely,
\begin{equation}
  \label{eq:some-oracle.inf}
    \t(\omega,\P)\eqm\inf_{\R_{\pi}=\P}\t(\omega,\pi).
\end{equation}
\end{theorem}
\begin{proof}
Let us prove the equality shown in the theorem.
Note that if $t$ is a uniform test, then $t(\omega,\R_{\pi})$ as a function of
$\omega$ and $\pi$ is an r-test, and is therefore dominated by the universal r-test.

The other direction is somewhat more difficult.
We have to show that the function on the right-hand side is lower semicomputable
as a function of the sequence $\omega$ and the measure $\P$.
(The integral condition is obtained easily afterwards, as the measure $\P$ has
at least one representation $\pi$.)
This can be proved using the effective compactness of the set of those pairs
$\pair{\P}{\pi}$ with $\P=\R_{\pi}$.
In the general form (for constructive metric spaces) this statement forms the
content of Lemma~\ref{lem:push-forward}.

It remains to explain the connection between the given equality and randomness
relative to an oracle.
If $\t(\omega,\P)$ is finite, then by the proved equality a $\pi$ exists with
$\R_{\pi}=\P$ and finite $\t(\omega,\pi)$.
As we have seen, this in turn means that $\omega$ is random with respect to the
measure $\P$, with an oracle $\pi$ that makes $\P$ computable.
Conversely, if $\t(\omega,\P)$ is infinite, and some oracle $A$ makes $\P$
computable, then the function $\t(\cdot,\P)$ becomes $A$-lower semicomputable,
and its integral by measure $\P$ does not exceed 1, hence the sequence $\omega$
will not be random relative to oracle $A$ and with respect to measure $\P$.
\end{proof}

\paragraph{Blind (oracle-free) tests}
We can define the notion of an effectively null set as before, even if
the measure is not computable.
The maximal effectively null set may not exist.
For example, if
measure $\P$ may be concentrated on some non-computable sequence $\pi$, then all
intervals not containing $\pi$ will be effective null sets, and their union (the
complement of the singleton $\{\pi\}$) will not be, otherwise $\pi$ would be
computable.

However, we
still can define random sequence as a sequence that does not belong to
\emph{any} effectively null set.
Kjos-Hanssen suggested the name
``Hippocratic randomness'' for this definition (referring to a certain legend
about the doctor Hippocrates), but we prefer the more neutral
name ``blind randomness''.

\begin{definition}[Blind tests]\label{def:blind-test.Cantor}
A lower semicomputable function $t(\omega)$ with integral bounded by $1$
will be called a \df{blind, or oracle-free, test} for measure $\P$.
A sequence $\omega$ is \df{blindly random}
iff $t(\omega)<\infty$ for all blind tests.
\end{definition}

As seen, there may not exist a maximal blind test.

This oracle-free notion of randomness can be characterized in the terms
introduced earlier:

\begin{theorem}\label{thm:characterize-obliv}
Sequence $\omega$ is blindly random with respect to
measure $\P$ if and only if $\omega$ is random with respect to any effectively
compact class of measures that contains~$\P$.
\end{theorem}
\begin{proof}
  Assume first that $\omega$ is not random with respect to some effectively
compact class of measures that contains $\P$.
Then the universal
test with respect to this class is a blind test that shows that
$\omega$ is not blindly random with respect to $\P$.

Now assume that there exists some blind test $t$ for measure $\P$
with $t(\omega)=\infty$.
Then just consider the class $\cC$ of measures
$\Q$ with $\int t(\omega)\, d\Q\le 1$.
This class is effectively closed, (and thus effectively compact).
Indeed, $t$ be the supremum of the computable sequence of basic functions
$t_{n}$.
The class of measures $\Q$ with $\int t_{n}(\omega)\,d Q>1$ is effectively open,
uniformly in $n$, and $\cC$ is the complement of the union of these sets.
\end{proof}

It is easy to see from the definition (or from the last theorem)
that uniform randomness implies blind randomness (either
directly or using the last theorem).
The reverse statement is not true:

\begin{theorem}\label{thm:nonrandom-blind}
There exists a sequence $\omega$ and measure $\P$ such that $\omega$ is
blindly random with respect to $\P$ but not uniformly random.
\end{theorem}
\begin{proof}
Indeed, oblivous randomness does not change if we change the
measure slightly (up to $O(1)$-factor).
On the other hand, the changed measure
may have much more oracle power that makes a sequence non-random.
For example,
we may start with uniform Bernoulli random measure $B_{1/2}$ (coin tosses with
probabilities $1/2, 1/2, 1/2,\ldots$ and fix some random sequence
$\omega=\omega(1)\omega(2)\dots$.
Then consider a (slightly) different measure $B'$ with probabilities
$1/2+\omega(1)\varepsilon_{1},
1/2+\omega(2)\varepsilon_2,\dots$ where $\varepsilon_1,\varepsilon_2,\dots$ are
so small and converge to zero so fast that they do not change the measure more
than by $O(1)$-factor while being all positive.
Then $B'$ encodes $\omega$, which makes it easy to construct a uniform test $t$
with $t(\omega, B')=\infty$.
\end{proof}

However, there are some special cases (including Bernoulli measures) where
uniform and blind randomness are equivalent.
In order to formulate the sufficient conditions for such a coincidence, let us
start with some definitions.

\begin{definition}[Effective orthogonality]\label{def:effectively-orth}
  For a probability measure $\P$, let $\Rands(\P)$ denote the set of sequences
uniformly random with respect to $\P$.
A class of measures is called \df{effectively orthogonal} if
$\Rands(\P)\cap\Rands(\Q)=\emptyset$ for any two different measures in it.
\end{definition}

\begin{theorem}\label{thm:orthogonal-blind.bin-Cantor}
  Let $\cC$ be an effectively compact, effectively orthogonal class of measures.
Then
for every measure $\P$ in $\cC$ the uniform randomness with respect to $\P$ is
equivalent to blind randomness with respect to $\P$.
\end{theorem}

The statement looks strange: we claim something about randomness with respect
to measure $\P$, but the
condition of the claim is that $\P$ can be included into a class of measures with some
properties.
(It would be natural to have a more direct requirement for $\P$ instead.)
The theorem implies that the measures of Theorem~\ref{thm:nonrandom-blind}
do not belong to any such class.

\begin{proof}
We have noted already that in one direction the statement is obviously true.
Let us prove the converse.
Assume that sequence $\omega$ is blindly random with respect to measure $\P$.
By Theorem~\ref{thm:characterize-obliv}, it is random
with respect to the class $\cC$.
So, $\omega$ is uniformly random with respect to some measure $\P'$ from the
class $\cC$.
It remains to show $\P=\P'$.

Imagine that this is not the case.
Then we can construct an effectively compact class
of measures $\cC'$ that contains $\P$ but not $\P'$.
Indeed, since $\P$ and $\P'$ are different, they assign
different measures to some finite string,
and this fact can be used, in form of a closed condition separating $\P$ from
$\P'$, to construct $\cC'$.
Consider now the effectively compact class $\cC\cap\cC'$.
It contains $\P$, and therefore $\omega$ will be random with respect to this class.
Hence the class contains some measure $\P''$ with respect to which $\omega$ is
uniformly random.
But $\P'\ne\P''$ (one measure is in $\cC'$, the other one is not),
so we get a contradiction with the assumption with the effective orthogonality
of the class $\cC$.
\end{proof}

\begin{remark}
The proved theorem is applicable in particular to the class of Bernoulli
measures.
It is tempting to think that there is a simpler proof, at least for this case:
if $\omega$ is random with respect to $p$ we can compute $p$ from $\omega$ as
the limit of relative frequency, and no additional oracle is needed.
This is not so: though $p$ is \emph{determined} by $\omega$, it
does not even depend continuously on $\omega$.
Indeed, no initial segment of the sequence guarantees that its limiting
frequency is in some given interval.
However, we can apply an analogous reasoning to those sequences
$\omega$ with the randomness deficiency bounded by some constant.
(See~\cite{HoyrupRojasCiE2009} which introduces the notion of \df{layerwise
  computability}.)
In particular, it can be shown that if $\omega$ is
random with respect to the measure
$B_{p}$ then $p$ is computable with oracle $\omega$.
\end{remark}

\section{Neutral measure}

The following theorem, first published in~\cite{LevinUnif76} and then again
in~\cite{GacsUnif05}, points to a curious property of uniform randomness which
distinguishes it from randomness using an oracle.

\begin{definition}
A measure is called \df{neutral} if every sequence is uniformly
random with respect to it.
\end{definition}

\begin{theorem}
There exists a neutral measure; moreover, there is a measure $N$ with
$\t(\omega,N)\le 1$ for all sequences $\omega$.
\end{theorem}

Note that a neutral measure cannot be computable.
Indeed, for a computable measure there exists a computable sequence that is not
an atom (adding bits sequentially, we choose the next bit in such a way that its
conditional probability is at most $2/3$).
Such a sequence cannot be random with respect to $N$.
For the same reason a neutral measure cannot be equivalent to an oracle
(for a neutral measure $N$ one cannot find an oracle $A$ that make it computable
and at the same time can be uniformly reconstructed from every approximation
of~$N$).
Indeed, in this case uniform randomness (as we have shown) is
equivalent to randomness with respect to $N$ with oracle $A$, and the same
argument works.

A neutral measure cannot be lower or upper semicomputable either, but this
statement does not seem interesting, since here a semicomputable measure
over $\Omega$ is also computable.
Some more meaningful (and less trivial) versions of this fact are proved
in~\cite{GacsUnif05}.

\begin{proof}
Consider the universal test $\t(\omega,\P)$.
We claim that
there exists a measure $N$ with $\t(\omega,N)\le 1$ for every $\omega$.
In other terms, for every $\omega$ we have a condition on $N$ saying that
$\t(\omega,N)\le 1$ and we need to prove that these conditions (there is
continuum of them) have non-empty intersection.
Each of these condition is a
closed set in a compact space (recall that $\t$ is lower semicontinuous), so it
is enough to show that finite intersections are non-empty.

So let $\omega_1,\ldots,\omega_k$ be $k$ sequences.
We want to prove that there
exists a measure $N$ such that $\t(\omega_i,N)\le 1$ for every $i$.
This measure
will be a convex combination of measures concentrated on
$\omega_1,\ldots,\omega_k$.
So we need to prove that $k$ closed subsets of a
$k$-vertex simplex (corresponding to $k$ inequalities) have a common point.
It is a direct consequence of the following classical topology result formulated
in Lemma~\ref{lem:Sperner} below (which is used in the standard proof of
Brouwer's fixpoint theorem).

To show that the lemma gives us what we want, consider any point of some face.
For example, let $X$ be a measure that is a mixture of, say, $\omega_1$,
$\omega_5$ and $\omega_7$.
We need to show that $X$ belongs to $A_1\cup A_5\cup A_7$: in our terms,
that one of the numbers $\t(\omega_1,X)$, $\t(\omega_5,X)$ and $\t(\omega_7,X)$
does not exceed $1$.
It is easy since we know $\int t(\omega,X)\,d X(\omega)\le 1$ (by the
definition of the test), and this integral is a convex combination of the above
three numbers with some coefficients (the weights of $\omega_1$, $\omega_5$ and
$\omega_7$ in $X$.
\end{proof}
\begin{lemma}\label{lem:Sperner}
Let a simplex with vertices $1,\dots,n$ be covered
by closed sets $A_1,\ldots,A_{k}$ in such a way
that vertex $i$ belongs to $A_i$ (for every $i$), edge $i$-$j$ is covered
by $A_i\cup A_j$, and so on (formally, face $(i_1,\dots,i_s)$ of the simplex is
a subset of $A_{i_1}\cup\ldots\cup A_{i_s}$; in particular, the union
$A_1\cup\dots\cup A_{k}$ is the entire simplex).
Then the intersection $A_1\cap\dots\cap A_{k}$ is not empty.
\end{lemma}
For completeness, let us reproduce the standard proof of this lemma.
\begin{proof}
Consider a disjoint division $T$ of the simplex
into smaller $n$-dimensional simplices (in such a way that every vertex in the
division is a vertex of every simplex containing it).
Let $S$ be the set of vertices of $T$.
A \df{Sperner-labeling} is a covering of $S$ by sets $A_{1},\dots,A_{k}$
such that the points of $S$ belonging to each
lower-dimensional simplex formed by some vertices $i_1<i_2<\dots <i_r\le k$
are covered by $A_{i_1}\cup\dots\cup A_{i_{r}}$.
(A point gets label $i$ if it belongs to $A_{i}$.)
Sperner's famous combinatorial lemma (see for example the Wikipedia)
implies that in any Sperner labeling,
there is a simplex whose vertices are labeled with all $k$ colors.

To apply the Sperner's lemma, note that our closed sets $A_{i}$
satisfy the rules of Sperner coloring.
Sperner's lemma guarantees the existence of a simplex that has
all possible labels on its vertices.
In this way we can get arbitrarily small
simplices with this property; compactness then shows that all $A_i$ have a
common point.
\end{proof}

\section{Randomness in a metric space}

Most of the theory presented above for infinite binary sequences generalizes to
infinite sequences of natural numbers.
Much of it generalizes even further, to an arbitrary metric space.
In what follows below we not only generalize; some of the results are
new also for the binary sequence case.

\subsection{Constructive metric spaces}

We rely on the definition of a constructive metric space,
and the space of measures on it, as defined in~\cite{GacsUnif05}
and~\cite{HoyrupRojasRandomness09} (the lecture notes~\cite{GacsLnotesAIT88} are
also recommended).

 \begin{definition}
A \df{constructive metric space} is a tuple $\bX = (X, d, D, \alpha)$ where $(X,d)$
is a complete separable metric space, with a countable dense subset $D$
and an enumeration $\alpha$ of $D$.
It is assumed that the real function $d(\alpha(v),\alpha(w))$ is computable.
Open balls with center in $D$
and rational radius are called \df{ideal balls}, or \df{basic open sets}, or
\df{basic balls}.
The (countable) set of basic balls will also be called the \df{canonical basis}
in the topology of the metric space.

An infinite sequence $s_{1},s_{2},\dots$ with $s_{i}\in D$ is called a
\df{strong Cauchy} sequence if for all $m<n$ we have $d(s_{m},s_{n})\le 2^{-m}$.
Since the space is complete, such a sequence always has a (unique) limit, which
we will say is \df{represented} by the sequence.
 \end{definition}

We will generally use the notational convention of this definition: if there is
a constructive metric space with an underlying set $X$ then the we will use
$\bX$ (boldface) to denote the whole structure $(X,d,D,\alpha)$.
But frequently, we just use $X$ when the structure is automatically understood.

 \begin{examples}\label{example:metric}
 \begin{enumerate}[\upshape 1.]

 \item\label{i:example.metric.discrete}
A set $X=\{s_{1},s_{2},\dots\}$ can be turned into a constructive
\df{discrete} metric space by making the distance between any two different
points equal to 1.
The set $D$ consists of all points $\alpha(i)=s_{i}$.

 \item\label{i:example.metric.one-point-compactif}
The set $\overline\bbN=\bbN\cup\{\infty\}$ can be turned into a constructive
metric space by making the distance between any two different
points with the distance
function $d(x,y)=|\frac{1}{x}-\frac{1}{y}|$, where of course,  $\frac{1}{\infty}=0$.
The set $D$ consists of all points of $\bbN$.
This metric space is called the \df{one-point compactification}, in a
topological sense, of the \df{discrete metric space} $\bbN$ of
Example~\ref{i:example.metric.discrete}.

  \item The real line $\bbR$ with the distance $d(x,y) = |x - y|$ is a constructive
metric space, and so is $\bbR_{+}=\lint{0}{\infty}$.
We can add the element $\infty$ to get $\overline\bbR_{+}=\clint{0}{\infty}$.
This is not a metric space now, but is still equipped with a natural
constructive topology (see Remark~\ref{rem:constr-top} below).
It could be equipped with a new metric in a way that would not change
this constructive topology.

\item\label{i:example.metric.L1}
If $\bX,\bY$ are constructive metric spaces, then we can define
a constructive metric space $\bZ=\bX\times\bY$ with
one of its natural metrics, for example the sum of distances in both
coordinates.
In case when $\bX=\bY=\bbR$, this is called the $L_{1}$ metric.
Let $D_{\bZ}$ be the product $=D_{\bX}\times D_{\bY}$.

  \item\label{i:example.metric.Cantor}
Let $X$ be a finite or countable (enumerated) alphabet, with a
fixed numbering, and let $X^{\bbN}$ be the set of infinite sequences
$x = \tup{x(1), x(2), \dots}$
with distance function $d^{\bbN}(x,y) = 2^{-n}$ where $n$ is the first $i$ with
$x(i)\ne y(i)$.
This space generalizes the binary
Cantor space of Definition~\ref{def:binary-Cantor}, to the case mentioned in
Remark~\ref{rem:seq-of-natural}.
The balls in it are cylinder sets: for a given finite sequence $z$, we take all
continuations of $z$.
 \end{enumerate}
 \end{examples}

 \begin{remark}
Each point $x$ of a constructive metric space $\bX$ can be viewed as an
``approximation mass problem'': the set of total functions that for any given
rational $\varepsilon>0$ produce a $\varepsilon$-approximation to $x$ by a point
of the canonical dense set $D$.
This is a mass problem in the sense of~\cite{MedvedevMass55}.
One can also note that this
mass problem is Medvedev equivalent to the enumeration problem: enumerate all
basic balls that contain $x$.
 \end{remark}

 \begin{remark}\label{rem:constr-top}
A constructive metric space is special case of a more general concept,
which is often useful: a constructive topological space.

A \df{constructive topological space} $\bX = (X, \tau, \nu)$
is a topological space over a set $X$ with a basis $\tau$ effectively
enumerated (not necessarily without repetitions) as a list
$\tau =\set{\nu(1),\nu(2),\dots}$.

For every nonempty subset $Z$ of the space $X$, we can equip $Z$ with
a constructive topology: we intersect all basic sets with $Z$, without changing
their numbering.
On the other hand, not every subset of a constructive metric space
naturally has the structure of a constructive metric space (the
everywhere dense set $D$ is not inherited).

But instead of introducing constructive topological spaces formally, we prefer
not to burden the present paper with more abstractions, and will speak about
some concepts like effective open sets and continuous functions, as defined on
an arbitrary subset $Z$ of the constructive metric space $X$.
 \end{remark}

 \begin{definition}
   An open subset of a constructive metric space is \df{lower semicomputable
open} (or r.e.~open, or c.e.~open), or \df{effectively open}
if it is the union of an enumerable set of elements of the canonical basis.
It is \df{upper semicomputable closed}, or \df{effectively closed}
if its complement is effectively open.
Given any set $A\subseteq X$, a set $U$ is \df{effectively open on $A$} if
there is an effective open set $V$ such that $U\cap A=V\cap A$.
 \end{definition}

Note that in the last definition, $U$ is not necessarily part of $A$, but only
its intersection with $A$ matters.

Computable functions can be defined in terms of effectively open sets.

 \begin{definition}[Computable function]\add{\label{def:comp-func}}
   Let $X,\Y$ be constructive metric spaces and $f:X\to \Y$ a function.
Then $f$ is \df{continuous} if for each element $U$ of the canonical
basis of $\Y$ the set $f^{-1}(U)$ is an open set.
It is \df{computable} if $f^{-1}(U)$ is also an effectively open set,
uniformly in $U$.
A partial function $f:X\to \Y$ defined at least on a set $A$ is
\df{computable} if for each element $U$ of the canonical basis of $\Y$ the set
$f^{-1}(U)$ is effectively open on $A$, uniformly in $U$.

An element $x\in X$ is called \df{computable} if the
function $f:\{0\}\to X$ with $f(0)=x$ is computable.

When $f(x)$ is defined only in a single
point $x_{0}$ then we say that the element $y_{0}=f(x_{0})$ is
$x_{0}$-\df{computable}.
When $f:X\times Z\to \Y$, defined on $X\times\set{z_{0}}$, is computable,
then we say that the
function $g:X\to \Y$ defined by $g(x)=f(y, z_{0})$ is $z_{0}$-\df{computable}, or
computable \df{from} $z_{0}$.
 \end{definition}

There are several alternative characterizations of a computable element.

\begin{proposition}
The following statements are equivalent for an element $x$ of a constructive
metric space $\bX=(X,d,D,\alpha)$.
\begin{enumerate}[\upshape (i)]
\item $x$ is computable.
\item the set of basic balls containing $x$ is enumerable.
\item There is a computable sequence $z_{1},z_{2},\dots$ of elements of $D$ with
  $d(x,z_{n})\le 2^{-n}$.
\end{enumerate}
\end{proposition}

The following proposition connects computability with a more intuitive concept
based on representation by strong Cauchy sequences.

\begin{proposition}
   Let $\bX,\bY$ be constructive metric spaces and $f:X\to \Y$ a function.
Then $f$ is computable if and only if there is a computable transformation that
turns each strong Cauchy sequence $s_{1},s_{2},\dots$ with $s_{i}\in D_{\bX}$
converging to a point $x\in X$ into a strong Cauchy sequence $t_{1},t_{2},\dots$ with
$t_{i}\in D_{\bY}$ converging to $f(x)$.

If $f$ is a partial function with domain $Z$
then $f$ is computable if and only if there is a computable transformation that
turns each strong Cauchy sequence $s_{1},s_{2},\dots$ with $s_{i}\in D_{\bX}$
converging to some point $x\in Z$ into a strong Cauchy sequence $t_{1},t_{2},\dots$ with
$t_{i}\in D_{\bY}$ converging to $f(x)$.
\end{proposition}
We omit the---not difficult---proof of this statement.

\begin{remark}
Though $x_{0}$-computability means computability from a strong Cauchy sequence
$s_{1},s_{2},\dots$ converging to $x_{0}$, it should not be considered the same
as computability using a machine that treats this sequence as an ``oracle''.
In case of $x_{0}$-computability, the resulting output
must be independent of the strong Cauchy sequence $s_{1},s_{2},\dots$
representing $x_{0}$.
\end{remark}

The following definition of lower semicomputability is also a straightforward
generalization of the special case in Definition~\ref{def:lower-semicomp.seqs}.

 \begin{definition}[Lower semicomputability]
Let $\bX=(X,d,D,\alpha)$ be a constructive metric space.
A function $f:X\to\clint{-\infty}{\infty}$ is \df{lower semicontinuous}
if the sets $\setOf{x}{f(x)>r}$ are open, for every rational number $r$
(from here it follows that they are open for all $r$, not only rational).

It is \df{lower semicomputable} if these sets are effectively open, uniformly
in the rational number $r$.
It is \df{upper semicomputable} if $-f$ is lower semicomputable.

A partial function $f:X\to \Y$ defined at least on a set $A$ is
\df{lower semicomputable} on $A$ if the sets $\setOf{x}{f(x)>r}$ are effectively
open in $A$, uniformly for every rational number $r$.
  \end{definition}

It is easy to check that a real function over a constructive metric space is
computable if and only if it is lower and upper semicomputable.
As before, one can define semicomputability equivalently with the help of basic
functions.

Let us introduce an everywhere dense set of simple functions.

 \begin{definition}[Hat functions, basic functions]\label{example:hat-funcs}
We define an enumerated list of \df{basic} functions $\cE=\{e_{1},e_{2},\dots\}$
in the constructive metric space $\bX=(X,d,D,\alpha)$ as follows.
For each point $u\in D$ and positive
rational numbers $r,\varepsilon$ let us define the \df{hat function}
$g_{u,r,\varepsilon}$: its value in point $x$ is determined by  the distance of
$x$ to $u$ and is equal to $1$, if this distance is at most $r$, equal to zero,
if the distance is not less than $r+\varepsilon$, and varies linearly as the
distance runs through the segment $\clint{r}{r+\varepsilon}$: see
Figure~\ref{fig:hat-function}.
\begin{figure}
  \centering
  \includegraphics[scale=0.5]{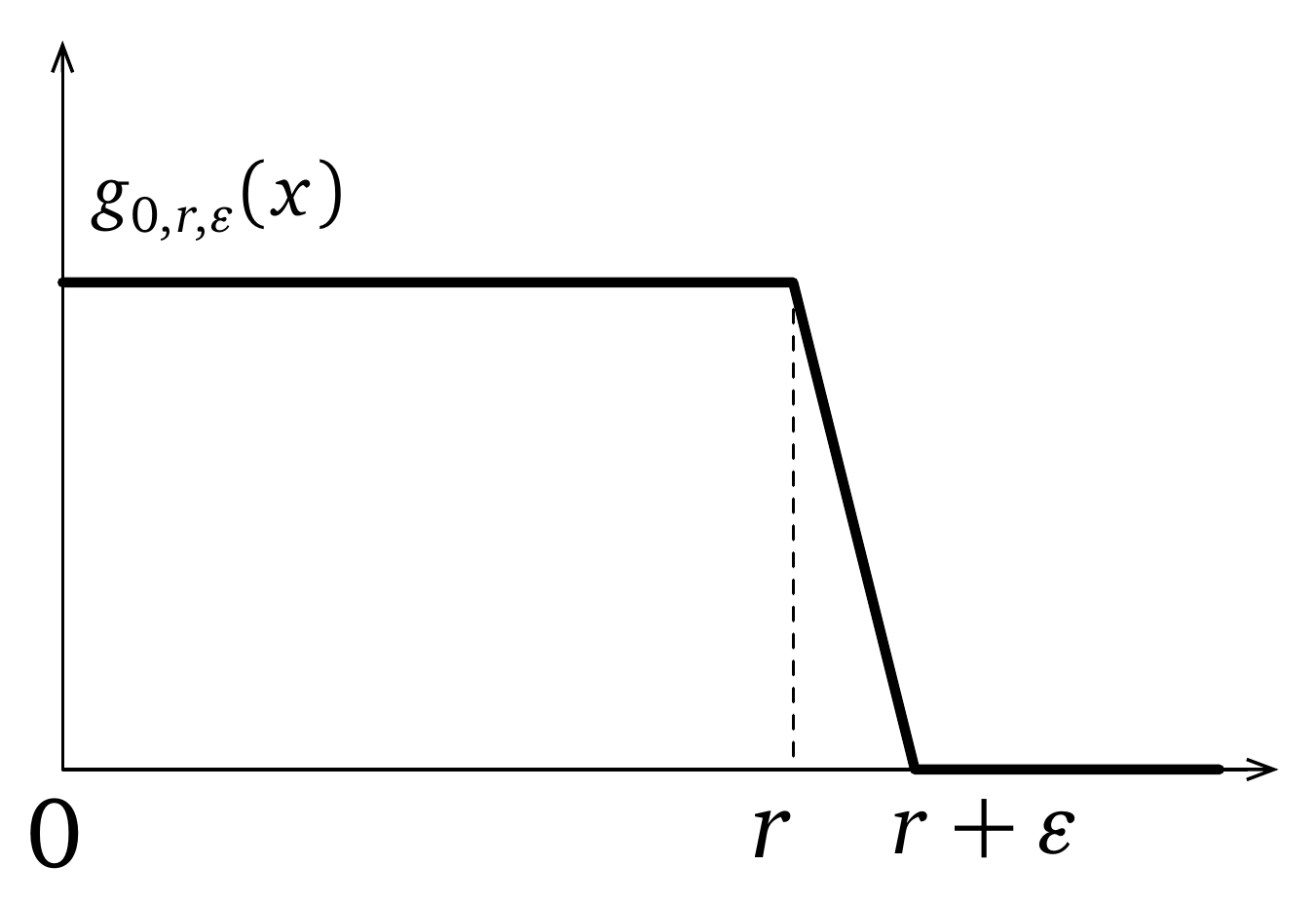}
  \caption{A hat function}
  \label{fig:hat-function}
\end{figure}
Let $\cE$ be the smallest set of functions containing all hat functions that is
closed under $\min,\max$ and rational linear combinations.
 \end{definition}

\begin{proposition}\label{propo:lower-semi}
A function $f:X\to\clint{0}{\infty}$ defined on a constructive metric space
is lower semicomputable if and only if it is the limit of a computable
increasing sequence of basic functions.
\end{proposition}

Note that the above characterization holds also for lower
semicontinous functions, if we just omit the requirement that the sequence
$g_{n}$ be computable.

\begin{definition}
We can introduce the notion of lower semicomputability \df{from} $z_{0}$,
or \df{$z_0$-lower semicomputability},
similarly to the $z_{0}$-computability of Definition~\ref{def:comp-func}, as
lower semicomputability of a function defined on the set $X\times \{z_{0}\}$.
\end{definition}

Sometimes two metrics on a space are equivalent from the point of
view of computability questions.
Let us formalize this notion.


\begin{definition}[Uniform continuity, equivalence]\label{def:uniform-contin}
  Let $X,Y$ be two metric spaces, and $f:X\to Y$ a function.
We say that $f$ is \df{uniformly continuous} if for each $\varepsilon>0$
there is a $\delta>0$ such that
$d_{X}(x,y)\le\delta$ implies
$d_{Y}(f(x),f(y))\le\varepsilon$.

If $\bX,\bY$ are constructive metric spaces and function $f$ is computable,
we will call it \df{effectively uniformly continuous} if $\delta$ can be
computed from $\epsilon$ effectively.

Two metrics $d_{1},d_{2}$ over the same space are \df{(effectively) equivalent} if
the identity map is (effectively) continuous in both directions.
\end{definition}

For example, the Euclidean metric and the $L_{1}$ metric introduced in
Example~\ref{example:metric}.\ref{i:example.metric.L1} are equivalent in
the space $\bbR^{2}$.

Effective compactness was introduced in
Definition~\ref{def:effectively-compact}: this generalizes immediately to
arbitrary metric spaces.
A weaker notion, local compactness, also has an effective version.

\begin{definition}[Effective compactness and local
  compactness]\label{def:effectively-compact-metric}
  A compact subset $C$ of a constructive metric space $\bX=(X,d,D,\alpha)$ is called
\df{effectively compact} if the set
 \begin{align*}
   \setOf{S}{S \text{ is a finite set of basic open sets and } \bigcup_{E\in S}E\supseteq C}
 \end{align*}
is enumerable.

A subset $C$ of a metric space is called \df{locally compact} if it is covered by
the union of a set of balls $B$ such that $\overline B\cap C$ is compact.
Here $\overline B$ is the closure of $B$.
It is \df{effectively locally compact} if it is covered by the union of an enumerated
sequence of basic balls $B_{k}$ such that $\overline B_{k}\cap C$ is effectively
compact, uniformly in $k$.
\end{definition}

 \begin{examples}
   \begin{enumerate}[\upshape 1.]
  \item The countable discrete space of
Example~\ref{example:metric}.\ref{i:example.metric.discrete} is effectively
compact if it is finite, and effectively locally compact otherwise.

   \item The segment $\clint{0}{1}$ is effectively compact.
The line $\bbR$ is effectively locally compact.

   \item If the alphabet $X$ is finite then the space $X^{\bbN}$ of infinite
     sequences is effectively compact.
Otherwise it is not even locally compact.

  \item
Let $\alpha\in\clint{0}{1}$ be a lower semicomputable
real number that is not computable.
(It is known that there are such numbers, for example
 $\sum_{x\in \bbN} 2^{-\KP(x)}$.)
The lower semicomputability of $\alpha$ allows to enumerate the rationals
less than $\alpha$ and allows for the segment $\clint{0}{\alpha}$ to inherit
the constructive metric (and topology) from the real line.
This space is compact, but not effectively so.

   \end{enumerate}
 \end{examples}

The following is a useful characerization of effective compactness.

\begin{proposition}\label{propo:epsilon-net}
  \begin{alphenum}
  \item
A compact subset $C$ of
a constructive metric space $\bX=(X,d,D,\alpha)$ is effectively compact if and only if
from each (rational) $\varepsilon$ one can compute a finite set of $\varepsilon$-balls
covering $C$.
 \item
For an effectively compact subset $C$ of a constructive metric space, in
every enumerable set of basic open sets covering
$C$ one can effectively find a finite covering.
  \end{alphenum}
 \end{proposition}
 \begin{proof}
Assume that for all $\varepsilon$ we can show a finite covering
$S_{\varepsilon}$ of the set $C$ by balls of radius $\varepsilon$.
Along with such a covering, we can enumerate all coverings with
\emph{guaranteedly} large balls (this means that for all balls $B(x,\varepsilon)$
from the covering $S_{\varepsilon}$ there is a ball $B(y,\sigma)$ from the new
covering with $\sigma>\varepsilon+d(x,y)$).
The compactness of $C$ guarantees that while $S_{\varepsilon}$ runs through all
$\varepsilon$-coverings of $C$, this way all coverings of $C$ will be enumerated.
(Indeed, if there is some covering $S'$ not falling into the enumeration, then
for all $\varepsilon$ there is a ball of the covering $S_{\varepsilon}$ not
guaranteedly contained in any ball of $S'$.
Applying compactness and taking a limit point of the centers of these
non-contained balls, we obtain contradiction.)

The remaining statements are proved quite easily.
 \end{proof}

The following
statement generalizes Proposition~\ref{propo:closed-to-compact}, with the
same proof.

  \begin{proposition}
Every effectively closed subset $E$ of an effectively compact set $C$
is also effectively compact.
  \end{proposition}

As earlier, the converse also holds: every effectively compact subset of a
constructive metric space is effectively closed.
Indeed, we can consider all possible coverings of this set by basic balls, and
also outside balls that manifestly (by the relation of the distances of their centerse
and their radiuses) are disjoint from the balls of the covering.
The union of all these outside balls provide the complement of our effectively
compact set.

It is known that a continuous function maps compact sets into compact ones.
This statement also has a constructive counterpart, also provable by a standard argument:

 \begin{proposition}\label{propo:image-of-eff-compact}
Let $C$ be an effectively compact subset of a constructive metric space
$X$, and $f$ a computable function from $X$
into another constructive metric space $Y$.
Then $f(C)$ is effectively compact.
 \end{proposition}

The statement that a lower semicontinuous function on a compact set reaches its
minimum has also a computable analog (we provide a parametrized variant):

 \begin{proposition}[Parametrized minimum]\label{propo:lsc-min}
Let $\bY,\bZ$ be constructive metric spaces, let $f:Y\times Z\to\clint{0}{\infty}$ be a
lower semicomputable function, and $C$ an effectively closed subset of $Y\times Z$.
If it is also effectively compact, then the function
 \begin{align*}
  g(y)=\inf_{z:(y,z)\in C} f(y,z)
 \end{align*}
is lower semicomputable from below (and the $\inf$ can be
replaced with $\min$ due to compactness).

Instead of effective compactness of $C$, it is sufficient to require that
its projection $C_{Y}=\setOf{y}{\exists z\,(y,z)\in C}$ is effectively closed and
covered by an enumerated sequence of basic balls $B_{k}$ such that
$\overline B_{k}\times Z\cap C$ is effectively compact, uniformly in $k$.
 \end{proposition}
The weaker condition formulated at the end holds for example if $Y$ is
effectively locally compact and $Z$ is effectively compact.
 \begin{proof}
For start, we reproduce the classical proof of lower semicontinuity.
One needs to check that the set $\setOf{y}{r<g(y)}$ is open for all $r$.
This set can be represented in the form of a union, noting that the condition
$r<g(y)$ is equivalent to the condition
 \begin{align*}
   (\exists r'>r)\forall z\,[(y,z)\in C\Rightarrow f(y,z)>r'],
 \end{align*}
and it is sufficient to check the openness of the set
 \begin{align*}
   U=\setOf{y}{\forall z\,[(y,z)\in C\Rightarrow f(y,z)>r']}.
 \end{align*}
Now, $U=(Y\setminus C_{Y})\cup\bigcup_{k}(B_{k}\cap U)$.
Since $Y\setminus C_{Y}$ is assumed to be open,
it is sufficient to show that each $B_{k}\cap U$ is open.
Let $F_{k}=\overline B_{k}\times Z$, then by the assumptions, $F_{k}\cap C$ is compact.
The condition $f(y,z)>r'$ by the assumption defines a certain open set $V$ of
pairs, hence $F_{k}\cap C\setminus V$ is closed, and as a subset of a compact set, compact.
It follows that its projection
$\setOf{y\in \overline B_{k}}{\exists z\,(y,z)\in F_{k}\cap C\setminus V}$, as a
continuous image of a compact set, is also compact, and so closed.
Its complement in $B_{k}$, which is $B_{k}\cap U$, is then open.

Now this argument must be translated to an effective language.
First of all note that it is sufficient to consider rational $r$ and $r'$.
Then the set $V$ is effectively open, the set $F_{k}\cap C\setminus V$ is effectively
closed, and as a subset of an effectively compact set, also effectively compact.
Its projection, as a computable image of an effectively compact set, is also
effectively compact, and as such, effectively closed.
The complement of the projection is then effectively open.
 \end{proof}

The following lemma is an application:

\begin{lemma}\label{lem:push-forward}
Let $X,Z,Z'$ be metric spaces, where $X$ is locally compact and $Z$ is compact.
Let $f\colon Z\to Z'$ be continuous and surjective,
and $t\colon X\times Z\to \clint{0}{\infty}$
a lower semicontinuous function.
Then the function $t_{f}\colon X\times Z'\to\clint{0}{\infty}$ defined by the formula
\begin{equation*}
  t_{f}(x,z')=\inf_{z:f(z)=z'} t(x,z)
\end{equation*}
is lower semicontinuous.

If $X,Z,Z'$ are constructive metric spaces, $X$ is effectively locally compact, $Z$ is
effectively compact
and $f$ is computable, further $t$ is lower semicomputable, then $t_{f}$ is lower
semicomputable.
\end{lemma}
\begin{proof}
We will prove just the effective version.
We will apply Proposition~\ref{propo:lsc-min} with $Y=X\times Z'$, and
$C=X\times\setOf{(f(z),z)}{z\in Z}$.
Then $t_{f}(x,z')=\inf_{(x,z',z)\in C} t(x,z)$.
The set $Y$ is effectively locally compact, as the product of an effectively
locally compact set and an effectively compact set.
The projection of the set $C$ onto $Y$ is the whole set $Y$, and hence it is
closed.
Hence the proposition is applicable, according to the remark following it.
\end{proof}

\subsection{Measures over a constructive metric space}

On a metric space, the \df{Borel sets} are the smallest $\sigma$-algebra containing
the open sets.
We can define measures on Borel sets.
These measures have the following \df{regularity} property:

\begin{proposition}[Regularity]\label{propo:regular}
Let $\P$ be a measure over a complete separable metric space.
Then every measureable set $A$ can be approximated by large open sets:
$\P(A)=\inf_{G\supseteq A}\P(G)$, where $G$ is open.
\end{proposition}

It is possible to introduce a metric over measures:

 \begin{definition}[Prokhorov distance]\label{def:prokh}
For a set $A$ and point $x$ let us define the distance of $x$ from $A$ as
$d(x, A) = \inf_{y \in A} d(x,y)$.
The $\varepsilon$-neighborhood of a set $A$ is defined as
$A^{\varepsilon} = \setOf{x}{d(x, A) < \varepsilon}$.

The \df{Prokhorov distance} $\rho(\P, \Q)$ of two measures is
the greatest lower bound of all those $\varepsilon$ for which, for all Borel
sets $A$ we have $\P(A) \le \Q(A^{\varepsilon}) + \varepsilon$
and $\Q(A) \le \P(A^{\varepsilon}) + \varepsilon$.
 \end{definition}
It is known that $\rho(\P,\Q)$ is indeed a metric, and it turns
the set of probability measures over metric space $X$
into a metric space.
There is a number of other metrics for measures that are equivalent, in the
sense of Definition~\ref{def:uniform-contin}.

 \begin{definition}[Space of measures]
For a constructive metric space, $\bX$, let
$\bM=\cM(\bX)$ define the metric space of the set of probability measures over $\bX$,
with the metric $\rho(\P, \Q)$.
The dense set $D_{\bM}$ is the set of
those probability measures that are concentrated on finitely many points of
$D_{\bX}$ and assign rational values to them.
Let $\alpha_{\bM}$ be a natural enumeration of $D_{\bM}$, this turns $\bM$ into
a constructive metric space, too.

A probability measure is called \df{computable} when it is a computable element
of the space $\bM$.
 \end{definition}

Computability of measures is a particularly simple property
for the Cantor space of binary sequences
in Definition~\ref{def:computable-measure-Omega} (which
is easily shown to be equivalent to the definition given here);
it is just as simple for the Baire space of sequences over a countable alphabet.

The analogue of Proposition~\ref{propo:integral-computable.Bernoulli} holds again: the
integral $\int f(\omega,\P)\P(d\omega)$
of a basic function is computable as a function of the measure $\P$, uniformly
in the code of the basic function.
Here is a closely related result:

\begin{proposition}\label{propo:integral-computable}
If $f$ is a bounded, effectively uniformly
continuous function then its integral by the measure $\P$ is an effectively
uniformly continuous function of $\P$.
\end{proposition}
\begin{proof}
It can be assumed without loss of generality that $f$ is nonnegative (add a
constant).
Let measures $\P$ and $\P'$ be close.
Then $\P'(A)\le\P(A_{\varepsilon})+\varepsilon$, where $A_{\varepsilon}$ denotes the
$\varepsilon$-neighborhood of $A$.
Then
 \begin{align*}
   \int f\, d\P'\le \int f_{\varepsilon}\, d\P+\varepsilon,
 \end{align*}
where $f_{\varepsilon}(x)$ is the least upper bound of $f$ on the
$\varepsilon$-neighborhood of $x$.
(The integral of a nonnegative function $g$ is defined by the measures of the
sets $G_{t}=\setOf{x}{g(x)\ge t}$; by Fubini's theorem on the change of the
order of integration, this measure must be
integrated by $t$ as a function of $t$.
Now, if $f(x)\ge t$ then $f_{\varepsilon}(x)\ge t$ in the
$\varepsilon$-neighborhood of point $x$.)
It remains to apply the effective uniform continuity of $f$ to find out the
precision by which the measure must be given in order to obtain a given
precision in the integral.
  \end{proof}

On the other hand, the measure of $\P(B)$ of a basic ball $B$ is not necessarily
computable, only lower semicomputable.
It is shown in~\cite{HoyrupRojasRandomness09} that this property also
characterizes the computability of measures:
$\P$ is computable if and only if $\P(B)$ is lower semicomputable, uniformly in
the basic ball $B$.

It is known that if a complete separable
metric space is compact then so is the set of measures with the described
metric.
The following constructive version is proved by standard means:

\begin{proposition}
If a constructive metric space
$\bX$ is effectively compact then its space of probability measures
$\cM(\bX)$ is also effectively compact.
\end{proposition}

For the binary Cantor space, this was proved in
Proposition~\ref{propo:closed-to-compact}.
There, the topology of the space of measures was simply derived from
the topology of the space $\clint{0}{1}\times\clint{0}{1}\times\dotsm$.
It can be seen that the Prokhorov metric leads to the same topology.

\begin{example}\label{example:discrete-measures}
Another interesting simple metric space is the infinite discrete space, say on
the set of natural numbers $\bbN$.
This is not a compact space, and the set of measures, namely the set of all
functions $\P(x)\ge 0$ with $\sum_{x\in\bbN}\P(x)=1$, is not compact either.

On the other hand, the set of semimeasures (see Definition~\ref{def:semimeasure})
is compact.
Indeed, recall that the space
$\clint{0}{1}\times\clint{0}{1}\times\dotsm$, of functions
$\P:\bbN\to\clint{0}{1}$ is compact.
Hence also for each $n$ the subset $F_{n}$ of this set of consisting of
functions $\P$ obeying the restriction
$\P(0)+\P(1)+\dots+\P(n)\le 1$ is compact, as the product of a
compact finite-dimensional set
$\setOf{(\P(0),\P(1),\dots,\P(n))\in\clint{0}{1}^{n}}{\P(0)+\dots+\P(n)\le1}$
and the compact infinite product set $\setOf{(\P(n+1),\P(n+1),\dots)}{
  0\le\P(x)\le 1 \text{ for } x>n}$.
The intersection of all sets $F_{n}$ is then also compact, and is equal to
the set of semimeasures.

Equivalently, we can consider the one-point
compactification $\overline\bbN$ of $\bbN$ given in
Example~\ref{example:metric}.\ref{i:example.metric.one-point-compactif}.
Measures $\P$ on this space can be identified with semimeasures over $\bbN$:
we simply set $\P(\infty)=1-\sum_{n<\infty}P(n)$.
\end{example}

\subsection{Randomness in a metric space}\label{subsec:randomness-metric}

In the Cantor space $\Omega$ of infinite binary sequences we defined
\begin{bullets}
  \item randomness with respect to computable measures (in the sense of
    Martin-L\"of); see Definition~\ref{def:test-computable-measure};
  \item uniform randomness with respect to arbitrary measures (when the test is
    a function of the sequence and the measure),
    Definition~\ref{def:uniform-test.bin-Cantor};
 \item Randomness with respect to an effectively compact class of measures,
   Definition~\ref{def:test-effectively-compact.Cantor};
 \item Blind (oracle-free) randomness in Definition~\ref{def:blind-test.Cantor};
\end{bullets}
All these notions carry over with minor changes to an arbitrary constructive
metric space.
In the present section we discuss these generalizations and their properties,
and then consider in more detail randomness with respect to an orthogonal class
of measures.

For computable measures, a test is defined as a lower semicomputable function on
a constructive metric space, whose integral is bounded by $1$.
Among such tests, there is a maximal one to within a multiplicative constant.
As earlier, this is proved with the help of trimming: we list all lower
semicomputable functions, forcing them into tests or almost tests, and then add
them up with coefficients from a converging series.

This is done as before, by considering lower semicomputable functions as
monotonic limits of basic ones.
It is used that the integral of a basic function by a measure is
computable as a function of the measure: see
Propositions~\ref{propo:integral-computable} and the discussion preceding it.

The uniform tests introduced in Definition~\ref{def:uniform-test.bin-Cantor}
generalize immediately to the case of constructive metric spaces.
Such a test is a lower semicomputable function of two arguments $t(x,\P)$, where
$x$ is a point of our metric space, and $\P$ is a measure over this space.
The integral condition has the same form as earlier:
$\int t(x,\P)\,\P(d x)\le 1$.

As earlier, there exists a universal test, and this can be proved by the
technique of trimming:

\begin{theorem}[Trimming in metric spaces]\label{thm:trim}
Let $u(x,\P)$ be a lower semicomputable function whose first argument is a point
of a constructive metric space, and the second one is  measure over this space.
Then there exists a uniform tests $t(x,\P)$ satisfying $u(x,\Q)\le 2 t(x,\Q)$ for all
$\Q$ such that the function $u_{\Q}:x\mapsto u(x,\Q)$ is a test by the measure
$\Q$, that is $\int u(x,\Q)\,\Q(d x)\le 1$.
\end{theorem}

The proof repeats the reasoning of the proof of Theorem~\ref{thm:trim.Cantor},
while using the fact that for a basic function $b(x,\P)$ on the product space
the integral $\int b(x,\P)\,\P(d x)$ is a computable (continuous) function of
$\P$ (which is proved analogously to our above argument on the computability of
the integral).

We will denote the universal uniform test again by $\t(x,\P)$.
Strictly speaking, it depends also on the constructive metric space on which it
is defined, but in general it is evident, which space is being considered,
therefore it is not shown in the notation.

Definition~\ref{def:P-test.Cantor} and Proposition~\ref{propo:uniformize.Cantor}
extend without difficulty.

 \begin{definition}[Tests for arbitrary measures]\label{def:P-test}
Let $\bX = (X, d, D, \alpha)$ be a constructive metric space.
For a measure $\P\in\cM(\bX)$, a $\P$-\df{test of randomness}
is a function $f:X\to\clint{0}{\infty}$ lower semicomputable from $\P$ with
the property $\int f(x)\,d P \le 1$.
\end{definition}

It seems as if a $\P$-test may capture some nonrandomnesses that uniform tests
cannot---however, this is not so, since trimming (see
Theorem~\ref{thm:trim.Cantor}) generalizes:

\begin{theorem}[Uniformization]\label{propo:uniformize}
Let $\P$ be some measure over a constructive metric space $X$, along with
some $\P$-test $t_{\P}(x)$.
There is a uniform test $t'(\cdot,\cdot)$ with $t_{\P}(x)\le 2 t'(x,\P)$.
\end{theorem}

Theorem~\ref{thm:some-oracle} generalizes to the case of constructive metric spaces.
Let us mention one of the facts that generalize to uniform tests.

\begin{proposition}[Kurtz tests, uniformly]\label{propo:Kurtz-uniform}
Let $S$ be an effectively open subset of the space $X\times\cM(X)$.
If the set $S_{\P}=\setOf{x}{\pair{x}{\P}\in S}$, has $\P$-measure $1$ for some
measure $\P$,
then the set $S(\P)$ contains all uniformly $\P$-random points.
\end{proposition}
\begin{proof}
The indicator function $1_{S}(x,\P)$ of the set $S$, that is equal to unity on $S$
and to zero outside, is lower semicomputable.
According to Proposition~\ref{propo:lower-semi},
it can be written as the limit of a computable increasing sequence
of basic functions $0\le g_{n}(x,\P)\le 1$.
The sequence $G_{n}:\P\mapsto\int g_{n}(x,\P)\,d\P$ is an
increasing sequence of functions computable uniformly in $n$.
The motonone convergence theorem implies
$G_{n}(\P)\to 1$ for all $\P\in\cC$.
Let us define for each measure $\P$ the numbers $n_{K}(\P)$ as the minimal
values of $n$ for which $G_{n}(\P)>1-2^{k}$.
These numbers are upper semicomputable as functions of $\P$ (in a natural sense;
for measures $\P$ with $\P(S_{\P})<1$, some of these $n_{k}(\P)$ are infinite).
Correspondingly, the functions $1-g_{n_{k}(\P)}(x,\P)$, as functions of $x$ and
$\P$ (define such a function to be zero for infinite $n_{k}(\P)$,
independently of $x$) are lower semicomputable, uniformly in $k$.
Then $t(x,\P) = \sum_{k>0}(1-g_{n_{k}(\P)}(x,\P))$ is a uniform test, since
at a given $\P$, if its $k$th addend is zero if $n_{k}(\P)$ is infinite, and
is not greater than $2^{-k}$ for finite $n_{k}(\P)$.

The conditions of the theorem talk about a measures $\P$ with $\P(S_{\P})=1$.
Then all numbers $n_{k}(\P)$ are finite.
Consider an $x$ outside $S_{\P}$: then $g_{n_{k}(\P)}(x,\P)=0$ by definition.
Therefore all addends of the test sum are equal to unity, thus $x$ is
is not $\P$-random point.
Consequently, $S_{\P}$ includes all uniformly $\P$-random points.
\end{proof}

\subsection{Apriori probability, with an oracle}

In Section~\ref{subsec:uniform-exact} we defined apriori probability with a
condition whose role was played by a measure over the Cantor space $\Omega$.
Now, having introduced the notion of a constructive metric space, we can note
that this definition extends naturally to an arbitrary such space $\bX$: we
consider nonnegative lower semicomputable functions
$m:\bbN\times X\to\clint{0}{\infty}$ for which $\sum_{i}m(i,x)\le 1$, for all
$x\in X$.

Among such functions, there is a maximal one to within a multiplicative
constant.
This is proved by the method of trimming: the lower semicomputable function
$m(i,x)$ can be obtained as a sum of a series of basic functions each of which
differs form zero only for one $i$; these basic functions must be multiplied by
correcting coefficients that depend on the sum over all $i$.
(In each stage, this sum has only finitely many members.)

We will call the maximal function of this kind \df{apriori probability with
condition $x$}, and denote it $\m(i\mid x)$.
We consider the first argument a natural number, but this is not essential: it
is possible to consider words (or any other discrete constructive objects).
As a special case we obtain the definition of apriori probability conditioned on
a measure (Section~\ref{subsec:uniform-exact}), and also the standard notions of
apriori probability with an oracle (which corresponds to $\bX=\Omega$, the Cantor
space of infinite sequences), and the conditional apriori probability
(corresponding to $\bX=\bbN$).

In analogy with Martin-L\"of's theorem, the apriori probability with a condition
is expressible, in an arbitrary \emph{effectively compact} constructive
metric space $\bX$ by apriori probability with an oracle.

\begin{proposition}\label{propo:day-miller-apriori}
Let $F:\Omega\to X$ be a computable map whose image is the whole space $X$.
Then
 \begin{align*}
   \m(i\mid x) \eqm \min_{\pi: F(\pi)=x}\m(i\mid\pi).
 \end{align*}
\end{proposition}
\begin{proof}
  We reason as in the proof of Theorem~\ref{thm:some-oracle}.
The function $\pair{i}{\pi}\mapsto \m(i\mid F(\pi))$ is lower semicomputable on
$\bbN\times\Omega$, hence the $\lem$-inequality.

In order to obtain the reverse inequality, we use Lemma~\ref{lem:push-forward}
and note that the function on the right-hand side is correctly defined (the
minimum is achieved) and is lower semicomputable.
\end{proof}

Note that $\m(i\mid\pi)$, the apriori probability with an oracle on the right-hand side
of Proposition~\ref{propo:day-miller-apriori},
is expressible by prefix complexity with an oracle.
For the case of prefix complexity with condition in metric spaces it is not
clear, how to define prefix complexity with such a condition (one can speak of
functions whose graph is enumerable with respect to $x$, but it is not clear how
to build a universal one).
But one can define formally $\KP(i\mid x)$ as $\max_{\pi:F(\pi)=x}\KP(i\mid
\pi)$, and then $\KP(i\mid x)\eqa -\log\m(i\mid x)$,
but it is questionable whether this can be considered a satisfactory definition
of prefix complexity (say, the usual arguments using the self-delimiting property of
programs are not applicable at such a definition).
It is more honest to simply speak of the logarithm of apriori probability.
Many results still stay true: for example the formula
$\KP(i,j\mid x)\lea \KP(i\mid x) + \KP(j\mid x)$ can be proved,
without introducing self-delimiting programs, just reasoning about
probabilities.

\begin{remark}
Analogously, it is possible to supply points in constructive metric spaces as
conditions in some of our other definitions.
For example, we can consider uniform tests over the Cantor space $\Omega$ of
infinite binary sequences, with condition in an arbitrary constructive
metric space $X$: these will be lower semicomputable functions $t(\omega,\P,x)$
with $\int t(\omega,\P,x)\,\P(d\omega)\le 1$ for all $\P,x$.
It is also possible to fix a computable measure $\P$, say the uniform one, and
define tests with respect to this measure with conditions in $X$.
\end{remark}

\section{Classes of orthogonal measures}

The definition of a class test for an effectively compact class
of measures, as well as Theorem~\ref{thm:class-test} about the expression of a
class test, generalizes, with the same proof.

The set of Bernoulli measures has an important property shared by many
classes considered in practice: namely that a random sequence determines the
measure to which it belongs.
A consequence of this was spelled out in
Theorem~\ref{thm:orthogonal-blind.bin-Cantor}.
This section explores the topic in a more general setting.

There are some examples naturally generalizing the Bernoulli case:
finite or infinite ergodic Markov chains, and ergodic stationary processes.
Below, we will dwell a little more on the latter, since it brings up a rich complex of new
questions.

We will consider orthogonal classes in the general setting of metric spaces:
from now on, our measureable space is the one obtained from
a constructive metric space $\bX = (X,d,D,\alpha)$.
The following classical concept is analogous to effective orthogonality, introduced in
Definition~\ref{def:effectively-orth}.

\begin{definition}[Orthogonal measures]\label{def:orthogonal}
  Let $\P,\Q$ be two measures over a measureable space $(X,\cA)$, that is a
space $X$ with a $\sigma$-algebra $\cA$ of measureable sets on it.
We say that they are \df{orthogonal} if the space can be partitioned into
measureable sets $U,V$ with the property $\P(V)=\Q(U)=0$.

Let $\cC$ be a class of measures.
We say that $\cC$ is \df{orthogonal} if there is a measureable function
$\varphi:X\to\cC$ with the property $\P(\varphi^{-1}(\P))=1$.
\end{definition}

Note that the space $\cM(X)$, as a metric space, also allows the definition of
Borel sets, and it is in this sense that we can talk about $f$ being
measureable.

\begin{examples}\label{example:orthogonal}
  \begin{enumerate}[\upshape 1.]
  \item
In an orthogonal class, any two (different) measures $\P$ and $\Q$ are
orthogonal.
Indeed, the sets $\{\P\}$ and $\{\Q\}$ are Borel (since closed), hence their
preimages are measureable (and obviously disjoint).
The converse statement is false:
A class $\cC$ of mutually orthogonal probability
measures is not necessarily orthogonal,
even if the class is effectively compact.
For example, let $\lambda$ be the uniform distribution over the interval
$\clint{0}{1}$, and let for each $x\in\clint{0}{1}$
the probability measure $\delta_{x}$ be concentrated on $x$.
Then the class $\{\lambda\}\cup\setOf{\delta_{x}}{x\in\clint{0}{1}}$ is effectively
compact, and its elements are mutually orthogonal.
But the whole class is not orthogonal: the orthogonality condition requires
$\phi(x)=\delta_{x}$, but then $\phi^{-1}(\lambda)$ will be empty.

\item\label{i:example.orthogonal.rand}
Let $\P,\Q$ be two probability measures.
Of course, if $\Rands(\P)$ and $\Rands(\Q)$ are disjoint, then $\P$ and
$\Q$ are orthogonal.
The converse is not always true: for example it fails if
$\lambda,\delta_{x}$ are as above, where $x$ is random with
respect to $\lambda$.
  \end{enumerate}
\end{examples}

The following definition introduces the important
example of stationary ergodic processes.

\begin{definition}
The Cantor space $\Omega$ of infinite binary sequences is equipped with an
operation $T: \omega(1)\omega(2)\omega(3)\dots\mapsto
\omega(2)\omega(3)\omega(4)\dots$ called the \df{shift}.
A probability distribution $\P$ over $\Omega$ is \df{stationary} if for every
Borel subset $A$ of $\Omega$ we have $\P(A)=\P(T^{-1}(A))$.
It is easy to see that this property is equivalent to requiring
 \begin{align*}
 \P(x)=\P(0x)+\P(1x)
 \end{align*}
for every binary string $x$.

A Borel set $A\subseteq\Omega$ is called \df{invariant} with respect to the shift
operation if $T(A)\subseteq A$.
For example the set of all sequences in which the relative frequency converges
to $1/2$ is an invariant set.
A stationary distribution is called \df{ergodic} if every invariant Borel set
has measure 0 or 1.
\end{definition}

Here is a new example of a stationary process (all Bernoulli measures and
stationary Markov chains are also examples).

\begin{example}
Let $Z_{1},Z_{2},\dots$ be a sequence of independent, identically distributed
random variables taking values $0,1$ with probabilities $0.9$ an $0.1$
respectively.
Let $X_{0},X_{1},X_{2},\dots$ be defined as follows:
$X_{0}$ takes values $0,1,2$ with equal probabilities, and independently of all
$Z_{i}$, further $X_{n}=X_{0}+\sum_{i=1}^{n}Z_{i}\bmod 3$.
Finally, let $Y_{n}=0$ if $X_{n}=0$ and $1$ otherwise.
The process $Y_{0},Y_{1},\dots$ is clearly stationary, and can also be proved to
be ergodic.
As a function of the Markov chain $X_{0},X_{1},\dots$, it is also called a
\df{hidden Markov chain}.
\end{example}

The following theorem is a consequence of Birkhoff's pointwise ergodic theorem.
For each binary string $x$ let
 \begin{align*}
  g_{x}(\omega)=1_{x\Omega}(\omega)
 \end{align*}
be the indicator function of the set $x\Omega$: it is 1 if and only if $x$ is a prefix
of $\omega$.

\begin{proposition}\label{propo:Birkhoff}
 Let $\P$ be a stationary process over the Cantor space
$\Omega$.
\begin{alphenum}
\item\label{i:Birkoff.converge} With probability $1$, the average
 \begin{align}\label{eq:Axn}
   A_{x,n}(\omega)=\frac{1}{n}(g_{x}(\omega)+g_{x}(T\omega)+\dots+g_{x}(T^{n-1}\omega))
 \end{align}
converges.
 \item\label{i:Birkhoff.ergodic}
If the process is ergodic then the sequence converges to $\P(x)$.
\end{alphenum}
\end{proposition}
(For non-ergodic processes, the limit may depend on $\omega$.)
Birkhoff's theorem is more general, talking about more general spaces and
measure-preserving transformations $T$, arbitrary integrable functions
in place of $g_{x}$, and convergence to the expected value in the ergodic case.
But the proposition captures its essence (and can also be used in the derivation
of the more general versions).

Part~\eqref{i:Birkhoff.ergodic} of Proposition~\ref{propo:Birkhoff} implies that
the class $\cC$ of ergodic measures is an orthogonal class.
Indeed, let us call a sequence $\omega$ ``stable'' if for all strings $x$, the
averages $A_{x,n}(\omega)$ of~\eqref{eq:Axn} converge.
It is easy to see that in this case, the numbers $\P(x)$ determine
some probability measure $\Q_{\omega}$.
Now, let $\varphi:\Omega\to\cC$ be a function that assigns to each stable sequence
$\omega$ the measure $\Q_{\omega}$ provided $\Q_{\omega}$ is ergodic.
If the sequence is not stable or $\Q_{\omega}$ is not ergodic, then let $\varphi(\omega)$
be some arbitrary fixed ergodic measure.
It can be shown that $\varphi$ is a measureable function:
here, we use the fact that the set of stable sequences is a Borel set.
By part~\eqref{i:Birkhoff.ergodic} of Proposition~\ref{propo:Birkhoff}, the
relation $\P(\varphi^{-1}(\P))=1$ holds for all ergodic measures.

Note that the class of all ergodic measures is not closed, but we did not rely on the
closedness of this class in the definition.

Example~\ref{example:orthogonal}.\ref{i:example.orthogonal.rand} shows that two
measures can be orthogonal and still have common random sequences.
But, for computable measures, as we will show right away, this is not possible.

We called a class of
measures $\P$ effectively orthogonal in Definition~\ref{def:effectively-orth},
if all sets of random sequences
$\Rands(\P)$ for measures $\P$ in the class are disjoint from each other.

\begin{theorem}
  Two computable probability
measures on a constructive metric space
are orthogonal if and only they are effectively orthogonal.
\end{theorem}
Speaking of the effective orthogonality of two measures, we mean that they have
no common (uniform) random sequences.
In the effective case, pairwise orthogonality within the class and the
orthogonality of the whole class are equivalent by definition.
\begin{proof}
  We only need to prove one direction.
Assume that $\P,\Q$ are orthogonal, that is there is a
measureable set $A$ with $\P(A)=1$, $\Q(A)=0$.
By Proposition~\ref{propo:regular}, these measures are regular, so
there is a sequence $G_{n}\supseteq A$ of open sets with $\Q(G_{n})< 2^{-n}$.
Then for every $n$ there is also a finite union $H_{n}$ of basic balls with
$\P(H_{n})>1-2^{-n}$ and $\Q(H_{n})<2^{-n}$; moreover, there is a
computable sequence $H_{n}$ with this property.
Let $U_{m}=\bigcup_{n>m}H_{n}$.
By Proposition~\ref{propo:Kurtz-uniform}, $\bigcap_{m} U_{m}$
contains all random points of $\P$.
On this other hand, the sets $U_{m}$ form a
Martin-L\"of test for measure $\Q$, so the intersection contains no random
points of $\Q$.
\end{proof}

We have shown above that ergodic measures form an orthogonal class.
Careful analysis shows that this is also true effectively.

\begin{theorem}\label{thm:effective-Birkhoff}
The set of ergodic measures over the Cantor set $\Omega$
forms an effectively orthogonal class.
\end{theorem}
\begin{proof}
The paper~\cite{VyuginErgodic98} (more precisely, an analysis of it that will
create uniform tests) shows that
\begin{alphenum}
\item\label{i:effective-Birkhoff.stable}
Sequences uniformly random with respect to some stationary measure
are stable (in the sense that the above indicated limit of averages exists for them).
 \item\label{i:effective-Birkhoff.ergodic}
Uniformly random sequences with respect to an ergodic measure
are ``typical'' in the sense that these averages converge to
$\P(x)$.
\end{alphenum}
To show~\eqref{i:effective-Birkhoff.stable}, the paper introduces the function
 \begin{align*}
 \sigma(\omega,\alpha,\beta)
 \end{align*}
for rationals $0<\alpha<\beta$, which is the maximum number of times that
$A_{x,n}(\omega)$ crosses from below $\alpha$ to above $\beta$.
This function is lower semicomputable, uniformly in the rationals $\alpha,\beta$.
Then it shows
 \begin{align*}
(1+\alpha^{-1})(\beta-\alpha)\int\sigma(\omega,\alpha,\beta)\,d\P\le 1,
 \end{align*}
that is that $(1+\alpha^{-1})(\beta-\alpha)\sigma(\omega,\alpha,\beta)$ is an
average-bounded test,
implying that for Martin-L\"of-random sequences, the average $A_{x,n}(\omega)$
crosses from below $\alpha$ to above $\beta$ only a finite number of times.
Now one can combine all these tests, for all strings $x$ and all rational
$0<\alpha<\beta$, into a single test.
This test is uniform in $\P$: we did not rely on the computability of $\P$.

To express~\eqref{i:effective-Birkhoff.ergodic}, in view of
part~\eqref{i:effective-Birkhoff.stable},
it is sufficient, for each $x$, to prove
 \begin{align}\label{eq:liminf-limsup}
\lim\inf_{n} A_{x,n}(\omega)\le\P(x)\le\lim\sup_{n} A_{x,n}(\omega)
 \end{align}
for random $\omega$.
Take for example the statement for the lim inf.
It is sufficient to show for each $k,m$ that
$\inf_{n\ge m} A_{x,n}(\omega)\le\P(x)+2^{-k}$ for a random $\omega$.
The set
 \begin{align*}
S_{x,k,m} &=\setOf{\pair{\omega}{\P}}{\exists{n\ge m}\;A_{x,i}(\omega)<\P(x)+2^{-k}}
 \end{align*}
is effectively open, and the Birkhoff theorem implies $\P(S_{x,k,m}(\P))=1$ for
for all ergodic measures $\P$, for the set
$S_{x,k,m}(\P)=\setOf{x}{\pair{x}{\P}\in S_{x,k,m}}$.
Proposition~\ref{propo:Kurtz-uniform} implies that then for each $\P$, the set
$S_{x,k,m}(\P)$ contains all $\P$-random points.

Another approach is a proof that just shows~\eqref{i:effective-Birkhoff.ergodic}
for computable ergodic measures (in a relativizable way), without an explicit test,
as done in~\cite{BienvDayHoyrMezhShen10}.
Then a reference to Theorem~\ref{thm:oracle2uniform} allows us to conclude the
same about uniformly random sequences.
\end{proof}

It is convenient to treat orthogonality of a class in terms of separator
functions.
For this, note that by a measureable real
function we mean a Borel-measureable real function, that is a function with the
property that the inverse images of Borel sets are Borel sets.

\begin{definition}[Separator function]
Let $\cC$ be a class of measures over the metric space $X$.
A measureable function $\s:X\times\cM(X)\to\clint{0}{\infty}$,
is called a \df{separator function} for the class $\cC$
if for all measures $\P$ we have $\int \s(x,\P)\,d\P\le 1$,
further for $\P,\Q\in\cC$, $\P\ne\Q$ implies that only one of the values
$\s(x,\P)$, $\s(x,\Q)$ is finite.

In case we have a constructive metric space $\bX$,
a separator function $\s(x,\P)$ is called a \df{separator test} if it is lower
semicomputable in $\pair{x}{\P}$.
\end{definition}

We could have required the integral to be bounded only for
measures on the class, since trimming allows the extension
of the boundedness property to all measures, just as in the remark after
Definition~\ref{def:P-test}.

The following observation connects orthogonality with separator functions and
also shows that in case of effective orthogonality, each measure can be
effectively reconstructed from any of its random elements.

\begin{theorem}\label{thm:eff-orthog}
Let $\cC$ be a class of measures.
  \begin{alphenum}
 \item\label{i:thm.orthog.sep-fun}
If class $\cC$ is Borel and orthogonal then there is a separator function for
it.
  \item\label{i:thm.orthog.sep-test}
Class $\cC$ is effectively orthogonal if and only if there is a separator test for it.
 \end{alphenum}
\end{theorem}
The converse of part~\eqref{i:thm.orthog.sep-fun} might not hold: this needs
further investigation.
\begin{proof}
Let us prove~\eqref{i:thm.orthog.sep-fun}.
If $\varphi(x)$ is a measureable function assigning measure
$\P\in\cC$ to each element $x\in X$ as required in the definition of
orthogonality, then by a general theorem of topological measure theory
(see~\cite{KuratowskiTop}), its graph is measureable.
This allows the following definition: for $\P\not\in\cC$
set $\s(x,\P)=1$, further for $\P\in\cC$, set $\s(x,\P)=1$ if $\varphi(x)=\P$,
and $\s(x,\P)=\infty$ otherwise.

Let us prove now~\eqref{i:thm.orthog.sep-test}.
If $\cC$ is effectively orthogonal
then the uniform test $\t(x,\P)$ is a separator test for the class $\cC$.
Suppose now that there is a separator test $s$
for the class $\cC$, and let $\P,\Q\in\cC$, $\P\neq\Q$, $x\in\Rands(\P)$.
Since $\s$ is a randomness test, $\s(x,\P)<\infty$, which
implies $\s(x,\Q)=\infty$, hence $x\not\in\Rands(\Q)$.
\end{proof}

The following result is less expected: it shows that if the class of
measures is effectively compact then the existence of a lower semicontinuous
separator
function implies the existence of a lower semicomputable one (that is a
separator test).

\begin{theorem}
  If for an effectively compact class of measures there is a lower
semicontinuous separator function $\s(x,\P)$, then
this class is effectively orthogonal.
\end{theorem}
\begin{proof}
Let $\cC$ be an effectively compact class of measures on a constructive metric
space.
We need to show that under the conditions of the theorem, for any two distinct
measures $\P_{1},\P_{2}$ in $\cC$, the sets of random sequences are disjoint:
 \[
\Rands(\P_{1})\cap\Rands(\P_{2})=\emptyset.
 \]
Take two disjoint closed basic balls $B_{1}$ and $B_{2}$
in the constructive metric space $\bM$ of measures, containing the measures $\P_{1},\P_{2}$.
The classes $\cC_{i}=\cC\cap B_{i}$, $i=1,2$ of measures are disjoint effectively
compact classes of measures, containing $\P_{1}$ and $\P_{2}$.
Consider the functions
 \begin{align*}
 t_{i}(x)=\inf_{\P\in\cC_{i}}\s(x,\P).
 \end{align*}
For all $x$ at least one of the values $t_{1}(x)$, $t_{2}(x)$ is infinite.
By (a version of) Proposition~\ref{propo:lsc-min}, the functions $t_{i}(x)$ are lower
semicontinuous, and hence $\cC_{1}$- and $\cC_{2}$-tests respectively.

Now we follow some of the reasoning of the proof of
Proposition~\ref{propo:Kurtz-uniform}.
For integer $k>1$, consider the open set $S_{k}=\setOf{x}{t_{1}(x)>2^{k}}$.
Since $t_{1}$ is a $\cC_{1}$-test, then $\P(S_{k})<2^{-k}$ for all $\P\in\cC_{1}$.
On the other hand, since for all $x$ one of the two values $t_{1}(x)$,
$t_{2}(x)$ is infinite, $\P(S_{k})=1$ for all $\P\in\cC_{2}$.
The indicator function $1_{S_{k}}(x)$ of the set $S_{k}$ is lower semicontinuous,
therefore it can be written as the limit of an
increasing sequence (now not necessarily computable!)
of basic functions $g_{k,n}(x)$.
We conclude as in the proof of Proposition~\ref{propo:Kurtz-uniform}, that for each
$\P$ there is an
$n=n_{k}(\P)$ with $\int g_{k,n}(x)\,d\P>1-2^{-k}$ for all $P\in\cC_{2}$.
The effective compactness
of $\cC$ implies then that there is an $n$ independent of $\P$ with
the same property.
In summary, for each $k>0$ a basic function $h_{k}$ is found with
 \begin{align*}
     \int h_{k}\,d\P  &< 2^{-k} \text{ for all }\P\in\cC_{1},
\\  \int h_{k}\,d\P &> 1-2^{-k} \text{ for all }\P\in\cC_{2}.
 \end{align*}
Such a basic function $h_{k}$ can be found effectively from $k$, by complete enumeration.
Now we can construct a lower semicomputble function
 \begin{align*}
t'_{1}(x)=\sum_{k} h_{k}(x).
 \end{align*}
It is a test for the class $\cC_{1}$, while
$t'_{2}(x)=\sum_{k}(1-h_{k}(x))$ is a test for all $\P\in\cC_{2}$ for the same
reasons.
These tests must be finite for elements random for $\P_{1}$ and $\P_{2}$, and this
cannot happen simultaneously for both tests.
\end{proof}

The meaning of separator tests introduced above introduced notion of
can be clarified as follows.
Due to effective orthogonality of $\cC$,
the universal uniform test $\t(\omega,\P)$ allows to separate the sequences into
random ones according to different measures of the class $\cC$: looking at a
sequence $\omega$, random with respect to some measure of this class (=random
with respect to the class), we are looking for a $\P\in\cC$ for which
$\t(\omega,\P)$ is finite.
This measure is unique in the class $\cC$ (by the definition of effective
orthogonality).

This separation property, however, can be satisfied also by a non-universal
test, and we called such tests separator tests.
The non-universal test is less demanding about the idea of randomness,
giving it, so to say, a ``first approximation'': it might accept a
sequence as random that will be rejected by a more serious test.
(The converse is impossible, since the universal test is maximal.)
What matters is only that this preliminary crude triage separates
the measures of the class $\cC$, that is that no sequence should appear
``random'' even ``in first approximation'', with respect to two measures at the
same time.

For brevity, just for the purposes of the present paper, we will call
``typicality'' this ``randomness in first approximation'':

\begin{definition}
Given a separator test $\s(x,\P)$ we call an element $x$
\df{typical} for $\P\in\cC$ (with respect to the test $s$) if $\s(x,\P)<\infty$.
 \end{definition}
A typical element determines uniquely the measure $\P$ for which
it is typical.

For an example, consider the class of $\cB$ of Bernoulli measures.
For a test in ``first approximation'', we may recall von Mises, who
called the first property of a random sequence (``Kollektiv'' in his words) the
stability of its relative frequencies.
The stability of relative frequencies (strong law of large numbers in today's
terminology) means $S_{n}(\omega)/n\to p$.
Here $S_{n}(\omega)$ is the number of ones in the initial segment of length $n$ of
the sequence $\omega$, and $p$ is the parameter of the Bernoulli measure $B_{p}$.

There are several requirements close to this in this spirit:
\begin{enumerate}[(1)]
\item\label{i:typical.converge-fast} $S_{n}(\omega)/n\to p$ with a certain convergence speed.
\item\label{i:typical.converge} $S_{n}(\omega)/n\to p$.
\item\label{i:typical.Vyugin}
For the case when $\cC$ is the class of all ergodic stationary measures over the
Cantor space $\Omega$,
convert the proof of Theorem~\ref{thm:effective-Birkhoff} into a test,
implying $A_{x,n}(\omega)\to \P(x)$ for all $x$.
\end{enumerate}
Among these requirements, the one that seems most natural to a
mathematician, namely~\eqref{i:typical.converge}, is not expressible
in a semicomputable way.
Requirement~\eqref{i:typical.converge-fast} has many possible formulations,
depending on the convergence speed: we will show an example below.

Requirement~\eqref{i:typical.Vyugin} is significantly more complicated to
understand, but is still much simpler than a universal test.
It \emph{does not} imply a computable convergence speed directly; indeed, as
Vyugin showed in~\cite{VyuginErgodic98}, a computable convergence speed does not
exist for the case of computable non-ergodic measures.
But later works, starting with~\cite{AvigadGerhardyTowsner2010}, have shown that
the the convergence for ergodic measures has a speed computable from
$\P$.

Here is an example of a test expressing requirement~\eqref{i:typical.converge-fast}.
(For simplicity, we obtain the
convergence of relative frequencies not on all segments, only on lengths that
are powers of two.
With more care, one could obtain similar bounds on all initial segments.)
By Chebyshev's inequality
 \begin{align*}
   B_{p}(\setOf{x\in\{0,1\}^{n}}{|\sum_{i}x(i)-n p|\ge\lambda n^{1/2}(p(1-p))^{1/2}}) \le \lambda^{-2}.
 \end{align*}
Since $p(1-p)\le 1/4$, this implies
 \begin{align*}
   B_{p}(\setOf{x\in\{0,1\}^{n}}{|\sum_{i}x(i)-n p>\lambda n^{1/2}/2}) < \lambda^{-2}.
 \end{align*}
Setting $\lambda=n^{0.1}$ and ignoring the factor $1/2$ gives
 \begin{align*}
   B_{p}(\setOf{x\in\{0,1\}^{n}}{|\sum_{i}x(i)-n p|> n^{0.6}})< n^{-0.2}.
 \end{align*}
Setting $n=2^{k}$:
 \begin{align}\label{eq:2^k-Cheb}
   B_{p}(\setOf{x\in\{0,1\}^{2^{k}}}{|\sum_{i}x(i)-2^{k}p|> 2^{0.6 k}})< 2^{-0.2k}.
 \end{align}
Now, for a sequence $\omega$ in $\bB^{\bbN}$, and for $p\in\clint{0}{1}$ let
 \begin{align*}
   g(\omega, B_{p}) = \sup\setOf{k}{|\sum_{i=1}^{2^{k}}\omega(k)-2^{k}p|> 2^{0.6k}}.
 \end{align*}
Then
 \begin{align*}
   \int g(\omega,B_{p})\,B_{p}(d\omega) \le \sum_{k} k\cdot 2^{-0.2 k} = c<\infty.
 \end{align*}
Dividing by $c$, we obtain a test.
This is a separator test, since $g(\omega,B_{p})<\infty$ implies that $2^{-k}S_{2^{k}}(\omega)$
converges to $p$, and this cannot happen for two different $p$.


Theorem~\ref{thm:orthogonal-blind.bin-Cantor} generalizes, with essentially
the same proof (using basic balls instead of initial sequences): it says that in
an effectively compact, effectively orthogonal class of measures, blind
randomness is the same as uniform Martin-L\"of randomness.
This raises the question whether every ergodic measure
belongs to some effectively compact class.
The answer is negative:

\begin{theorem}\label{thm:ergodic-non-eff}
Consider stationary measures over $\Omega$ (with the shift transformation).
Among these, there are some ergodic measures that do not belong to any
effectively compact class of ergodic measures.
\end{theorem}

Before proving the theorem, let us prove some preparatory statements.

\begin{proposition}\label{propo:ergodic-nonergodic-dense}
  Both the ergodic measures and the nonergodic measures are dense in the set of
stationary  measures $\cM(\Omega)$ over $\Omega$.
\end{proposition}
\begin{proof}
First we will show how to approximate an arbitrary stationary measure $\P$ by
ergodic measures.
Without loss of generality assume that all probabilities $\P(x)$ for finite
strings $x$ are positive.
(If not, then we can mix in a little of the uniform measure.)
For a fixed $n$, consider the values $\P(x)$ on strings $x$ of length at most
$n$.
There is a process that reproduces these probabilities and that
is isomorphic to an ergodic Markov process on $\{0,1\}^{n-1}$.
In this process, for an arbitrary $x\in\{0,1\}^{n-2}$, $b,b'\in\{0,1\}$
the transition probability from $bx$ to $xb'$ is $\P(bxb')/\P(bx)$.
Since both transition probabilities are positive, this Markov process is
ergodic.

Now we show how to approximate an arbitrary ergodic measure by nonergodic
measures.
Let $\P$ be ergodic.
Let us fix some $n>0$ and $\varepsilon>0$.
By the pointwise ergodic theorem, there is a sequence in which the limiting
frequencies of all words converge to the measure (almost all sequences---with respect
to this measure---are such).
Taking a long piece of this sequence and repeating it leads to a periodic
sequence in which the frequencies of words of length not exceeding $n$ differ
from the measure $\P$ by at most $\varepsilon$ (for any given $n$ and
$\varepsilon>0$).
(The repetition forms new words on the boundaries, but at a large length, this
effect is negligible.)
Consider now the measure concentrated on the shifts of this sequence, assigning
the same weight to each of them (their number is equal to the minimum period).
This measure is not ergodic, but is close to $\P$.
\end{proof}

\begin{proposition}\label{propo:ergodic-G-delta}
  The set of ergodic measures is a $G_{\delta}$ set in the metric space of all
stationary measures over $\Omega$.
\end{proposition}
\begin{proof}
We can restrict attention to the (closed) set of stationary measures.
Let $\P$ be a stationary probability measure over $\Omega$.
Consider the function $A_{x,n}$ over $\Omega$, defining $A_{x,n}(\omega)$ to be
equal to the average number of occurrences of the word $x$ in the $n$ first possible
positions of $\omega$.
By the ergodic theorem, the sequence of functions $A_{x,1},A_{x,2},\dots$ converges in the
$L_{1}$ sense.
Moreover, the stationary measure
$\P$ is ergodic if and only if the limit of this convergence is the constant function
with value $\P(x)$.

Since the limit exists for all stationary measures, it is sufficient to check
that the constant $\P(x)$ is a limit point.
For each $x,N$ and each rational $\varepsilon$ the
set $S_{x,N,\varepsilon}$ of those $\P$ for which there is an $n\ge N$ with
 \begin{align*}
 \int|A_{x,n}(\omega)-\P(x)|\P(d\omega)<\varepsilon
 \end{align*}
is open, and set of ergodic stationary measures is the
intersection of these sets for all $x,N,\varepsilon$.
\end{proof}

\begin{proof}[Proof of Theorem~\protect\ref{thm:ergodic-non-eff}]
The union of all effectively compact classes of ergodic measures is
$F_{\sigma}$.
Suppose that it is equal to the set of all ergodic measures.
Then the set of nonergodic measures is a $G_{\delta}$ set which is also dense by
Proposition~\ref{propo:ergodic-nonergodic-dense}.

As shown in Propositions~\ref{propo:ergodic-nonergodic-dense}
and~\ref{propo:ergodic-G-delta}, the set of ergodic measures
is a dense $G_{\delta}$ set.
But by the Baire category theorem, two dense $G_{\delta}$ sets cannot have an
empty intersection.
This contradiction proves the theorem.
\end{proof}

The following question still remains open:

\begin{question}
  Is there an ergodic measure over $\Omega$ for which uniform and blind
randomness are different?
\end{question}

Returning to arbitrary effectively compact, effectively orthogonal classes, we
can connect the universal tests with class tests of Theorem~\ref{thm:class-test}
and separator tests.

\begin{theorem}\label{thm:test-sep}
Let $\cC$ be an effectively compact, effectively orthogonal class of measures,
let $\t(x,\P)$ be the universal uniform test and
let $\t_{\cC}(x)$ be a universal class test for $\cC$.
Assume that $\s(x,\P)$ is a separator test for $\cC$.
Then we have the representation
 \begin{align*}
  \t(x,\P) \eqm \max(\t_{\cC}(x), \s(x,\P))
 \end{align*}
for all $\P\in\cC$, $x\in X$.
\end{theorem}
 \begin{proof}
Let is note first that $\t_{\cC}(x)$ and $\s(x,\P)$ do not exceed the universal
uniform test $\t(x,\P)$.
Indeed $\s(x,\P)$ is a uniform test by definition.
Also by definition, the universal class test $\t_{\cC}(x)$ is a uniform test.

On the other hand, let us show that if
$\t_{\cC}(x)$ and $\s(x,\P)$ are finite, then $\t(x,\P)$ does not exceed the
greater one of them (to within a multiplicative constant).
The finiteness of the first test guarantees that $\min_{\Q\in\cC}\t(x,\Q)$ is
finite: this minimum is equal to $\t_{\cC}(x)$ to within a constant factor.
If this minimum was achieved on some measure $\Q\not=\P$, then both values
$\s(x,\Q)$ and $\s(x,\P)$ would be finite, contradicting to the definition of a
separator.
(Note that we proved a statement slightly stronger than promised: in place of
``greater of the two'', one can write ``the first of the two, if the second one
is finite''.)
 \end{proof}

The above theorem separates the randomness test into two parts (points at two
possible causes of non-randomness).
First, we must convince ourselves that $x$ is random with respect to the class
$\cC$.
For example in the case of a measure $B_{p}$, in the class $\cB$ of Bernoulli
measures, we must first be convinced that $\t_{\cB}(\omega)$ is finite.
This encompasses all the irregularity criteria.
If the independence of the sequence is taken for granted, we may assume that
the class test is satisfied.

After this, we know that our sequence is Bernoulli, and some kind of simple
test of the type of the law of large numbers is sufficient to find out, by just
which Bernoulli measure is it random: $B_{p}$ or some other one.
This second part, typicality testing, is analogous to parameter testing in statistics.

Separation is the only requirement of the separator test: its numerical value is
irrelevant.
For example in the Bernoulli test case,
no matter how crude the convergence criterion expressed by the separator test
$\s(x,\P)$, the maximum is always (essentially) the same universal test.

 \section{Are uniform tests too strong?}\label{sec:too-strong}

\subsection{Monotonicity and/or quasi-convexity}\label{subsec:monotonicity}


Uniform tests may seem too strong,
in case $\P$ is a non-computable measure.
In particular, randomness with respect to computable measures
(in the sense of Martin-L\"of or in the uniform sense, they are the same for
computable measures)
has certain intuitively desireable properties that uniform randomness
lacks.
One of these is monotonicity:
roughly, if $\Q$ is greater than $\P$ then if
$x$ is random with respect to $\P$, it should also be random with respect to
$\Q$.

  \begin{proposition}\label{propo:test-computable-mon}
 For computable measures $\P,\Q$, for all rational $\lambda>0$,
if $\lambda\P(A)\le \Q(A)$ for all $A$, then
  \begin{align}\label{eq:test-computable-mon}
 \m(\lambda)\cdot\lambda\t(x,\Q)\lem \t(x,\P).
  \end{align}
Here $\m(\lambda)$ is the discrete apriori probability of the rational $\lambda$.
To make the constant in $\lem$ independent of $\P,\Q$, one needs also to multiply
the left-hand side by $\eqm\m(\P,\Q)$.
  \end{proposition}
 \begin{proof}
 We have $1\ge \int\t(x,\Q)\,d\Q\ge \int\lambda\t(x, \Q)\,d\P$,
 hence $\lambda\t(x,\Q)$ is a $\P$-test.
 Using the trimming method of Theorem~\ref{thm:trim} in finding universal tests,
 one can show that the sum
  \begin{align*}
    \sum_{\lambda: \lambda\int\t(x,\Q)\,d\P<2} \m(\lambda)\cdot\lambda\t(x,\Q)
  \end{align*}
 is a $\P$-test, and hence $\lem \t(x,\P)$.
 Therefore this is true of each member of the sum, which is just what the theorem
 claims.
It is easy to see that the multiplicative constants depend here on $\P,Q$ only
via inserting a factor $\m(\P,\Q)$.
 \end{proof}

The intuitive motivation for monotonicity is this: if there are two devices with
internal randomness generators, outputting numbers with distributions $\P$ and
$\Q$, and if $\lambda\P\le\Q$, then it can be imagined that the second
device simulates the first one with probability $\lambda$, and does
its own thing otherwise.
Then every outcome intuitively plausible as the outcome of the first device,
must also be deemed a plausible outcome of the second one, since this could have
simulated the first one by chance.
(The numerical value of the randomness deficiency may be, of course, somewhat
larger, since we must believe in addition that the $\lambda$-probability
event occurred.)

Uniform randomness violates, alas, this property: if measure $\Q$ is larger, but
computationally more complex, then the randomness tests with respect to $\Q$ can
exploit this additional information, to make nonrandom some outcomes that were
random with respect to $\P$ (see the proof of
Theorem~\ref{thm:nonrandom-blind}).
This is just the reason of the difference between uniform and blind
(oracle-free) randomness, for which the analogous monotonicity property is obviously
satisfied.

Another situation for which we have some intuition on randomness is the mixture
(convex combination) of measures.
Imagine two devices with output measures $\P$ and $\Q$, and an outer box which
triggers one of them with some probabilities $\lambda$, $1-\lambda$.
As a whole, we obtain a system whose outcome is distributed by the measure
$\lambda\P + (1-\lambda)\Q$.
About which outcomes can we assert that are obtained randomly as a result of
this experiment?
Clearly both the outcomes random with respect to $\P$ and those random with
respect to $\Q$ must be accepted (with the understanding that if the coefficient
is small then some additional, but finite suspicion is added).
And there should not be any other outcomes.
A quantitative elaboration of this result (which in one direction follows from
monotonicity) is given below.

\begin{proposition}\label{propo:convexity}
Let $\P$ and $\Q$ be two computable measures.
\begin{alphenum}
\item\label{i:convexity.qconvex}
For a rational $0<\lambda<1$,
 \begin{align*}
 \m(\lambda)\cdot\t(x,\lambda\P+(1-\lambda)\Q)\lem\max(\t(x,\P),\t(x,\Q)).
 \end{align*}
\item\label{i:convexity.qconcave}
For arbitrary $0<\lambda<1$,
 \begin{align*}
\t(x,\lambda\P+(1-\lambda)\Q)\gem\min(\t(x,\P),\t(x,\Q)).
 \end{align*}
\end{alphenum}
The constants in $\lem$ depend on the length of the shortest programs
defining $\P$ and $\Q$ (their complexities), but not on $\lambda$ (or other aspects of
$\P,\Q$).
\end{proposition}

Statement~(\ref{i:convexity.qconvex}) could be called the \df{quasi-convexity} of
randomness tests (to within a multiplicative constant).
For a test with an exact quasi-convexity property (without any multiplicative constants)
there is a lower semicomputable semimeasure that
is neutral (after extending tests to semimeasures see~\cite{LevinUnif76,GacsExact80}).

Statement~(\ref{i:convexity.qconcave}) implies that no other random outcomes exist
for the mixture of $\P$ and $\Q$.
This could be called the \df{quasi-concavity} of
randomness tests (to within a multiplicative constant).
\begin{proof}

Part~(\ref{i:convexity.qconvex}) follows from Proposition~\ref{propo:test-computable-mon}.
Indeed, if $\lambda\ge 1/2$ then
Proposition~\ref{propo:test-computable-mon} implies
$ \m(\lambda)\cdot\t(x,\lambda\P+(1-\lambda)\Q)\lem\t(x,\P)$
(absorbing $1/2$ into the $\lem$).
If $\lambda<1/2$ then it implies
$ \m(1-\lambda)\cdot\t(x,\lambda\P+(1-\lambda)\Q)\lem\t(x,\Q)$
similarly, and we just recall $\m(\lambda)\eqm\m(1-\lambda)$.

Part~(\ref{i:convexity.qconcave}) follows from the fact that the right-hand
side is a test with respect to an arbitrary mixture of the measures $\P$ and
$\Q$, and trimming can convert it into uniform test.
\end{proof}

It is easy to see that all these statements exploit the computability of the
measures and the mixing coefficients in an essential way.
The corresponding counterexamples are easy to build once it is recognized that
the mixture of measures can be stronger from an oracle-computational point of
view than any of them, as well as in the other way.
For example, let us divide the segment $\clint{0}{1}$ into two halves and
consider the measures $\P$ and $\Q$ that are uniformly distributed over these
halves.
Their mixture with coefficients $\lambda$ and $1-\lambda$ will make the number
$\lambda$ obviously non-random (since it can be computed from this measure),
though with respect to one of the measures in can very well be random.
Taking instead of $\P$ and $\Q$ their mixtures, say, with coefficients
$1/3$ and $2/3$ and then reversed, one can make $\lambda$ random with respect to
both measures.

In this example the mixture contains more information than each of the original
measures.
It can also be the other way: bend the interval $\clint{0}{1}$ with the uniform
measure into a circle, and cut it into two half-circles by the points $p$ and
$p+1/2$.
Then the uniform measures on these half-circles make $p$ computable with respect
to them and thus non-random, while the average of these measures is the uniform
measure on the circle, with respect to which $p$ can very well be random.

Let us note that for blind (oracle-free) randomness, we can guarantee without
any restrictions that the set of points random with respect to the mixture of
$\P$ and $\Q$ is the union of points random with respect to $\P$ and $\Q$.
(In one direction this follows from monotonicity, which we already mentioned.
In the other one: if an outcome is not random with respect to $\P$ and not
random with respect to $\Q$, then there are two tests proving this, and their
minimum will a lower semicomputable test proving its non-randomness with respect
to the mixture.)

These are strong motives to modify the concept of randomness test in order to
reproduce these properties, while conserving other desireable properties (say
the existence of a universal test and with it the notion of a deficiency of
randomness).
Some such modifications can be seen
in~\cite{LevinUnif76,GacsExact80,LevinRandCons84}.

\subsection{Locality}

Imagine that some sequence $\omega$ is uniformly random with respect to
measure $\P$ and starts with $0$.
Change the values of the measure on sequences that start with $1$.
It is not guaranteed
that $\omega$ remains uniformly random since now the measure may become
stronger as an oracle (allowing to compute more).
But this looks
strange since the changes in measure are in the part of the universe that does
not touch $\omega$.

For blind (oracle-free) randomness, specifically this example
is impossible (one can force the test to zero on sequences beginning
with unity), but in principle the concept of test depends not only on the
measure along the sequence (not only on the probabilities of occurrences of nulls
and ones after its start).

For randomness with respect to computable measures, the situation is again better.

\begin{proposition}[Prequentiality]
Let $\P,\Q$ be two computable measures on the space $\Omega$ of binary
sequences, coinciding on all initial segments of some sequence $\omega$.
Then this sequence is simultaneously random or non-random with respect to
$\P$ and $\Q$.
\end{proposition}
\begin{proof}
This follows immediately from the randomness criterion in terms of the
complexity of the inital segments (Levin-Schnorr Theorem) in any of its variants
(Theorem~\ref{thm:rand-prefix}, Proposition~\ref{propo:rand-monot},
Corollary~\ref{coroll:ample-excess}).
\end{proof}

On the other hand, it is easy to modify one of the counterexamples
in~\ref{subsec:monotonicity} to violate prequentiality as well.

In case of an arbitrary constructive metric space an analogous statement holds,
though with a stronger requirement: we assume that two computable
measures are equal on all
sets contained in some neighborhood of the outcome $\omega$.
(In this case it is possible to multiply the test by a basic function without
changing it in $\omega$, and making it zero outside the neighborhood of
coincidence).

Here is yet another way to obtain a clearly prequential definition of
randomness, in which the randomness deficiency is a function of the sequence
itself and the measures of its initial segments.
For a
given sequence $\omega$ and a given sequence $\{q(i)\}$ of real numbers with
$1=q(0)\ge q(1)\ge q(2)\ge\dotsm\ge 0$, let
 \begin{align*}
   \t'(\omega,q)=\inf\t(\omega,\P),
 \end{align*}
where the minimum is taken over all measures $\P$ with $\P(\omega(1:n))=q(n)$.
The corresponding sets are effectively compact, so that this minimum will be a
lower semicomputable function of $\omega$ and the sequence $q$.
If for the sequence $\omega$ and the measures $q(i)$ of its initial segments,
the value $\t'(\omega,q)$ is finite, then the sequence $\omega$ can be called
\df{prequentially random}.

In other words, sequence $\omega$ is prequentially random with respect to
measure $\P$ if there is a (in general different) measure $\Q$ with respect to
which $\omega$ is random and which coincides with $\P$ on all initial segments
of $\omega$.

The requirement of prequentiality has been invoked in connection with a
theory that extends probability theory and statistics to models of forecasting:
see for example~\cite{ChernovShenVereshVovk08}
and~\cite{ShenVovkPrequential10}.
An example situation is the following.
Let $\omega(n)=1$ mean that there is rain on day $n$ and $0$ otherwise.
Suppose that a forecasting office makes daily forecasts $p(1),p(2),\dots$
of the probability of rain.
It is not necessarily proposing a coherent probability model of global
weather (a global probability distribution).
It just provides forecasts for the conditional probabilities along the
path corresponding to the weather that actually takes place.

Is it possible to estimate the quality of the forecast?
It seems that in some situations, yes: if say, all forecasts are close to zero
(say, less than $10\%$), and the majority of days (say more than $90\%$) is
rainy.
(It is said that the forecast is poorly \df{calibrated}.)
Naturally, there are other possible inconsistencies, not related to the
frequencies: the general question is whether the given sequence can be accepted
as randomly obtained with the predicted probabilities.
(Such a question arises also in the situation of estimating the quality
of a random number generator each of whose output
values is claimed to occur with whatever distribution the customer requires
at that time of the process, for that particular bit.)

An additional circumstance to consider at the estimation of the quality
of forecasts is that the forecaster can use a variety of information accessible
to her at the moment of prediction (say, the evening of the preceding day), and
not only the members of the sequence $\omega$.
The presence of such information must also be taken into account at the
estimation of the quality of the forecast.

The paper~\cite{ShenVovkPrequential10} proposes several different approaches
to this question, which turn out to be equivalent.
One involves a generalization of the notion of martingale (see
Definition~\ref{def:martingale}).
It would be interesting to establish a connection with uniform randomness tests
in the spirit of the above defined prequential deficiency.
(Admittedly, in place
of probabilities of initial segments, one must deal here with conditional
probabilities,
which is not quite the same, if these are not separated from zero.)

\section{Questions for future discussion}

We have already noted some questions that (in our view) would be interesting to
study.
In this section we collected a few more such questions.

\begin{enumerate}[1.]
\item
Consider the following method for generating a sequence $\omega\in\Omega$
using an arbitrary distribution $\P$ on $\Omega$ in which the probabilities of
all words are positive.
Take a random sequence $\rho$ of independent reals $\rho(1),\rho(2),\dots$
uniformly distributed over $\clint{0}{1}$.
At stage $n$, after outputting $\xi(1:n-1)$,  set $\xi(n)=1$ if
 \begin{align*}
   \rho(n)<\P(\xi(1:n-1)1)/\P(\xi(1:n-1)).
 \end{align*}
Considering this as a random process, the output distribution will be exactly
$\P$.
What sequences can be obtained on the ouput, from a Martin-L\"of-random
sequence of real numbers on the input?
(It can be verified that for computable measures $\P$ one gets exactly the
sequences that are Martin-L\"of-random with respect to $\P$.)

\item Recall the formula for the deficiency for computable measures:
  \begin{equation}\label{eq:sum-quotient-def}
\t(\omega,\P) \eqm \sum_{x\prefix\omega} \frac{\m(x)}{\P(x)}.
  \end{equation}
Both sides make sense for non-computable $\P$, but this formula is no more true.
Indeed, the right-hand side does not change significantly if a measure $\P$ is
replaced by some other one that is close to $\P$ but is much more powerful as an
oracle; and the left-hand side can become infinite while it was finite for $\P$.

Denote the righ-hand side by $t'(\omega,\P)$.
Does it make sense to take the finiteness of $t'(\omega,\P)$ as a definition
of randomness by a non-computable measure?
It will be at least monotonic (an increase of the measure will only increase
randomness).
With respect to mixtures of measures, we can say that it is
quasi-convex; moreover, it is proved in~\cite{GacsDiss78} that $1/t'(\omega,\P)$
is a concave function of $\P$.

Another possibility is to define the randomness deficiency for an infinite
sequence $\omega$ as $\log\sup_{x\prefix\omega}\M(x)/\P(x)$ (and
consider the corresponding definition of randomness).
For computable measures we obtain a definition equivalent to Martin-L\"of's
standard one.
Paper~\cite{GacsUnif05} shows that the uniform tests defined by this expression
(whether to use $\m(x)$ or $\M(x)$)
do not obey randomness conservation, while the universal uniform test does.
The work~\cite{GacsDiss78} shows that, on the other hand, an expression related
to the right-hand side of~\eqref{eq:sum-quotient-def}, whith the summation running
over all positive basic functions instead of
only the functions $1_{x\Omega}(\omega)$, obeys randomness conservation.

\item
Can we define a reasonable class of tests with the property in
Proposition~\ref{propo:test-computable-mon} holding for all measures $\P$
(or some stronger version of it) so that there exists an universal class?
For example, one may require
\begin{equation*}
\P\le c\cdot \Q \Rightarrow t(\omega,\P)\ge t(\omega,\Q)/c
\end{equation*}
(motivation: this is true for the right-hand side of formula~\eqref{eq:sum-quotient-def}.
Could one also require the quasi-convexity, as in
Proposition~\ref{propo:convexity}?
Papers~\cite{LevinUnif76} and~\cite{GacsExact80} provide some such examples, as
well as~\cite{LevinRandCons84}.

How about the quasi-concavity of~Proposition~\ref{propo:convexity}?
A uniform test with this property seems less likely, since our counterexample
seems more robust.

\item
Relativization in recursion theory means that we take some set $A$ and
artificially declare it ``decidable'' by adding some oracle that tells us
whether $x\in A$ for any given $x$.
Almost all the theorems of classical recursion theory can be relativized.
It is more delicate to declare some set $E$ ``enumerable''.
This means that we have some
enumeration-oracle that enumerates the set $E$.
The problem is, of course, that there are many enumerations.
Still we can give the definition of an $E$-enumerable set.
Let $W$ be a set of pairs of the form $\pair{x}{S}$
where $x$ is an integer and $S$ is a finite set of integers;
assume $W$ to be enumerable in the classical sense.
Then consider the set $S(E,W)$ of all $x$ such that $\pair{x}{S}\in W$
for some $S\subset E$.
The sets $S(E,W)$ (for fixed $E$ and all enumerable $W$) are
called \df{enumerable with respect to the enumeration-oracle $E$}.
(The relation
$\pair{x}{S}\in E$ means that we add $x$ to the $E$-enumeration as soon as
we see all elements of $S$ in $E$.)
A standard (decision) oracle for a set $A$ can
be considered a special case of an enumeration oracle (say, for the set
$\{2n: n\in A\}\cup \{2n+1: n\notin A\}$.

For some purposes, an enumeration oracle is as meaningful as a decision oracle:
for
example, we can speak about a lower semicomputable function with respect to
enumeration oracle $E$, since it can be defined in terms of enumerable sets.
But what can be proved for this kind of relativized notions?

For example, is there (for an arbitrary $E$) a maximal lower $E$-semicomputable
semimeasure?
Can one define prefix complexity with oracle $E$, and will it coincide with the
logarithm of the maximal semimeasure lower semicomputable relative to $E$ (if
the latter exists)?
What if we assume, in addition, that $E$ is the set of all basic balls in a
constructive metric space, containing a given point?

(For comparison: we could define an $E$-computable function as a function
whose graph is $E$-enumerable.
Then some familiar properties will hold; say, the
composition of $E$-computable functions is again $E$-computable.
On the other hand, we cannot guarantee that every non-empty
$E$-enumerable set is the range of a total $E$-computable function:
for some $E$ this is not so.)




\item
We may try do extend the definition of randomness in a different direction: to
lower semicomputable semimeasures (that is output distributions of probabilistic
machines that generate output sequence bit by bit).
Levin's motivation for his definition was his goal to define the independence of
the pair $\pair{x}{\eta}$ of infinite sequences as
randomness with respect to the semimeasure $\M\times \M$.
Correspondingly, the
the deficiency of randomness of
the pair $\pair{\xi}{\eta}$ with respect to $\M\times \M$
could be called the quantity of mutual information between the sequences $\xi$
and $\eta$.
This is motivated by the fact that the algorithmic mutual information
 \begin{align*}
\KP(x)+\KP(y)-\KP(x,y)=-\log(\m(x)\times\m(y))-\KP(x,y)
 \end{align*}
between finite objects $x,y$ indeed looks like
deficiency of randomness with respect to $\m\times\m$.

One possibility is to
require that $\M(z)/\Q(z)$ is bounded, where $\M$ is a priori probability on the
tree and $\Q$ is the semimeasure in question.
The other possibility is to use
random sequences for unbiased coin tossing and consider the output sequences in
all these cases.
It is not clear whether these two definitions coincide or if the
second notion is well-defined (that is for two different machines with the same
output distribution the image of the set of random sequences is the
same).
For \emph{computable measures} it is indeed the case.

\item (Steven Simpson) Can we use uniform tests (modified in a proper
way) for defining, say, $2$-randomness?
(The standard definition uses
non-semicontinuous tests, but maybe it can be reformulated.)
\end{enumerate}

\section*{Acknowledgements}

The authors are thankful to their colleagues with them they have discussed the
questions considered in the paper: first to L.~Levin with whom many of the
concepts of the paper originate, to A.~Bufetov and A.~Klimenko,
and also to  V.~V'yugin and other participants
of the Kolmogorov seminar (at the Mechanico-Mathematical school of Moscow State
University).
The paper was written with the financial support of the grants
NAFIT ANR-08- EMER-008-01, RFBR 0901-00709-a.

\bibliographystyle{plain}
\bibliography{ait,publ}

\end{document}


\title{Алгоритмические тесты и случайность относительно\\ классов мер}
\author{Л. Биенвеню\thanks{%
     Laurent Bienvenu,
     LIAFA, CNRS \& Universit\'e Paris Diderot, Paris 7, Case 7014, 75205
     Paris Cedex 13, France,
     e-mail: \texttt{Laurent dot Bienvenu at liafa dot jussieu dot fr}},
П. Гач\thanks{%
     Peter G\'acs,
     Department of Computer Science, Boston University,
     111 Cummington st., Room 138, Boston, MA 02215,
     e-mail: \texttt{gacs at bu dot edu}},
К. Рохас\thanks{%
     Cristobal Rojas,
     Department of Mathematics, University of Toronto,
     Bahen Centre, 40 St. George St., Toronto, Ontario, Canada, M5S 2E4,
     e-mail: \texttt{crojas at math dot utoronto dot ca}},
М. Хойруп\thanks{%
     Mathieu Hoyrup,
     LORIA -- B248, 615, rue du Jardin Botanique, BP 239, 54506
     Vand\oe uvre-l\`es-Nancy, France,
     e-mail: \texttt{Mathieu dot Hoyrup at loria dot fr}},
А. Шень\thanks{%
     Alexander Shen, LIF, Universit\`e Aix -- Marseille, CNRS,
     39, rue Joliot-Curie, 13453 Marseille cedex 13, France,
     on leave from ИППИ РАН, Б.~Каретный, 19, Москва.
     Supported by NAFIT ANR-08-EMER-008-01, RFBR 0901-00709-a
     grants. e-mail: \texttt{sasha dot shen at gmail dot com}.}
}
\date{}
\maketitle

\begin{abstract}
     В этой работе приводятся некоторые новые результаты об
     алгоритмической случайности по отношению к классам мер, а
     также подробно излагаются известные (но не опубликованные
     подробно) результаты об алгоритмических тестах случайности.

     Мы начинаем с переформулировки определения случайности по
     Мартин-Лјфу в терминах тестов случайности (функций,
     измеряющих степень ``неслучайности'' последовательностей).
     Приводится формула, выражающая значение универсального
     теста в терминах префиксной сложности. Рассматриваются
     также варианты определения дефекта случайности для конечных
     слов, связанные с универсальным тестом.

     Далее рассматривается (введјнное ещј Мартин-Лјфом) понятие
     бернуллиевой последовательности (как последовательности, не
     противоречащей гипотезе о том, что все испытания независимы
     и имеют одинаковую вероятность успеха). Показано, что
     определение с помощью универсального теста эквивалентно
     первоначальному определению Мартин-Лјфа и что
     последовательность является бернуллиевой тогда и только
     тогда, когда она случайна по Мартин-Лјфу относительно
     бернуллиевой меры $B_p$ при некотором $p$ (с оракулом для
     $p$).

     Затем этот же вопрос (о сравнении тестов относительно
     классов мер и тестов как функции двух аргументов~---
     последовательности и меры) применяется к произвольным
     эффективно замкнутым классам мер в канторовском
     пространстве. Изучаются свойства ортогональных классов мер
     и указываются предположения, в которых два понятия
     случайности (равномерная и безоракульная) совпадают.

     В заключение рассматриваются обобщения некоторых из
     указанных результатов на случай произвольных метрических
     про\-с\-т\-ранств.

\end{abstract}

\section{\texorpdfstring{Введение}{Introduction}}
     \label{sec:introduction}
Эта работа может рассматриваться как
продолжение~\cite{GacsUnif05} (которая в свою очередь является
развитием давних идей Л.~Левина)
и~\cite{HoyrupRojasRandomness09}.

Хорошо известны различные варианты определения понятия случайной
последовательности нулей и единиц, соответствующие равномерному
распределению (бросанию честной монеты). Большинство этих
результатов естественно переносится на случай произвольного
вычислимого распределения вероятностей на пространстве $\Omega$
бесконечных последовательностей нулей и единиц.

Наша цель~--- исследовать возможности определения случайности в
более общей ситуации, когда распределение на $\Omega$ не
является вычислимым (или когда рассматривается случайность в
других пространствах, не только в $\Omega$). Для этой цели мы
рассматриваем тест случайности $\mathbf{t}(\omega,P)$ как
функцию двух переменных, последовательности $\omega$ и меры $P$.
Большие значения такого теста, интуитивно говоря, соответствуют
ситуациям, когда гипотеза о том, что последовательность $\omega$
получилась в результате случайного выбора по мере $P$,
неправдоподобна.

Кроме того, следуя~\cite{LevinRand73}, мы будем рассматривать
тесты относительно классов мер, обладающих свойством типа
компактности. Такой тест, $t_\mathcal{C}(\omega)$, измеряет,
насколько неправдоподобным кажется появление последовательности
$\omega$ в результате случайного процесса, распределение
вероятностей которого (нам неизвестное) принадлежит классу
$\mathcal{C}$. Мы покажем, что для класса бернуллиевых мер
(независимые испытания с одинаковой вероятностью успеха)
возникающее понятие случайности относительно этого класса
совпадает с введјнным Мартин-Лјфом в~\cite{MLof66art}.

Для классов, меры в которых попарно ортогональны (в некотором
эффективном смысле), мы получаем разложение теста случайности по
данной мере на два: один проверяет случайность относительно
класса мер, а второй проверяет (достаточно грубо) соответствие
последовательности конкретной мере. Для случая бернуллиевых мер
в качестве второго теста можно взять просто закон больших чисел
и проверять, что предельная частота действительно равна
декларируемой вероятности успеха. Аналогичное разбиение возможно
и для других классов, соответствующих эргодическим стационарным
процессам.

Определение случайности с помощью равномерных тестов
$\mathbf{t}(\omega,P)$, вообще говоря, не обладает некоторыми
интуитивно желательными свойствами (скажем, не монотонно по $P$
в естественном смысле). Но для случая эффективно ортогональных
классов оно равносильно другому, ``слепому'' определению
случайности, в котором рассматриваются лишь тесты, вычисление
которых не использует меру $P$ как оракул.

Статья начинается с переформулировки определения случайности по
Мартин-Лјфу (относительно вычислимых мер) в терминах тестов.
Значения теста мы рассматриваем как количественную
характеристику ``неслучайности'' последовательности, и считаем
случайными последовательностями те, для которых тест конечен. Мы
рассматриваем два вида тестов (ограниченные в среднем и
ограниченные по вероятности) и показываем, что они близки друг к
другу.

Затем мы приводим формулу, которая выражает значение
(ограниченного в среднем) теста через префиксную сложность (и
даже два варианта такой формулы~--- с максимумом и с суммой).
Эта формула является количественным уточнением критерия
случайности Левина -- Шнорра (в форме с префиксной сложностью,
как в статье Чейтина~\cite{Chaitin75}). Далее мы обсуждаем некоторые варианты
движения в обратную сторону: как от дефекта случайности для
бесконечных последовательностей перейти к дефекту случайности
конечных.

Далее мы определяем понятие теста бернуллиевости (частный случай
случайности относительно класса мер, в данном случае~--- класса
бернуллиевых мер). Мы показываем, что множество бернуллиевых
последовательностей, для которых этот тест конечен, совпадает с
объединением по всем $p\in[0,1]$ множеств последовательностей,
случайных в смысле Мартин-Лјфа относительно бернуллиевой меры
$B_p$, при этом в определении случайности добавляется оракул для
$p$. Для этого мы вводим понятие равномерного бернуллиева теста
и устанавливаем количественный вариант указанного результата:
дефект бернуллиевости равен точной нижней грани (по $p$)
дефектов случайности относительно каждой из мер $B_p$.

После этого мы вводим понятие равномерного теста (уже не
ограничиваясь конкретным классом мер) и соответствующее ему
понятие равномерной случайности (которое для случая вычислимых
мер совпадает с определением Мартин-Лјфа).

Бернуллиевы меры обладают таким свойством: имея доступ к случайной по
одной из этих мер последовательности, мы можем восстановить
значение $p$ как предельную частоту единиц (закон больших
чисел). Мы показываем, что для подобных классов мер различные
определения случайности (равномерное и ``слепое'', когда тест не
использует меру) равносильны. Это утверждение обобщается и на
меры в произвольных конструктивных метрических пространствах.

Наконец, введјм некоторые полезные обозначения.

Мы будем писать $f(x)\lem g(x)$ для положительных функций $f$ и
$g$, если указанное неравенство выполняется с точностью до
мультипликативной константы, то есть $f(x)\le c g(x)$ для
некоторого $c$ и для всех~$x$.

Запись $f(x)\eqm g(x)$ означает $f(x)\lem g(x)$ и $g(x)\lem
f(x)$.

Обозначения $f(x)\lea g(x)$ и $f(x)\eqa g(x)$ имеют аналогичный
смысл (неравенство с точностью до аддитивной константы).

Через $\Lambda$ мы обозначаем пустое слово (строку нулевой
длины). Длина слова $x$ обозначается $|x|$. Запись $x\pref y$
или $y\extends x$ означает, что слово $x$ является началом
(префиксом) слова $y$. Элементы бесконечной последовательности
$x$ (а также буквы слова $x$) обозначаются $x(1),x(2),\ldots$;
еј начало длины $n$ обозначается $x(1:n)$.

Часто мы рассматриваем функции со значениям в множестве
$\overline{\mathbb{R}}_+=[0,+\infty]$ (неотрицательные функции,
возможно, бесконечные в некоторых точках). Множества натуральных
и действительных чисел обозначаются $\mathbb{N}$ и $\mathbb{Z}$;
через $\mathbb{B}$ иногда обозначается множество $\{0,1\}$.

Логарифмы, если не указано иное, берутся по основанию~$2$.

\section{\texorpdfstring{Случайность последовательностей\\
         относительно вычислимых мер}{Randomness with respect to computable
        measures}}
\label{sec:computable-measures}

Мы начнјм с определения случайности бесконечных двоичных
последовательностей относительно вычислимых мер.

\subsection{\texorpdfstring
{Перечислимые снизу функции на $\Omega$}
{Lower semicomputable functions on $\Omega$}}

\begin{definition}[Канторовское пространство]
        \label{def:cantor-space}
Множество $\{0,1\}^\mathbb{N}$ бесконечных двоичных
последовательностей мы называем \emph{двоичным канторовским
пространством} и обозначаем $\Omega$. Для каждого конечного
двоичного слова $x$ мы рассматриваем \emph{интервал} $x\Omega$,
состоящий из всех последовательностей, начинающихся на~$x$.
Интервалы являются \emph{базисными открытыми множествами}
стандартной топологии канторовского пространства;
\emph{открытыми} являются произвольные объединения интервалов.
\end{definition}

Понятие открытого множества, как и другие топологические
понятия, имеет эффективный аналог.

\begin{definition}
        \label{def:effectively-open}
\emph{Эффективно открытыми} множествами называют объединения
перечислимых семейств интервалов. \emph{Эффективно замкнутыми}
называются дополнения эффективно открытых множеств.

Далее можно определить \emph{эффективные $G_\delta$} множества
как счјтные пересечения $\cap_i U_i$ последовательности
эффективно открытых множеств $U_i$; при этом требуется, чтобы
$U_i$ было эффективно открыто равномерно по $i$ \textup(алгоритм
получает на вход $i$ и перечисляет интервалы, образующие
$U_i$\textup).

Будем называть функцию $t\colon\Omega\to[0,\infty]$
\emph{перечислимой снизу}, если

\textup{(а)}~для любого рационального $r$ множество
$U_r=\{\omega\mid r<t(\omega)\}$ открыто,

\textup{(б)}~и, более того, $U_r$ эффективно открыто равномерно
по $r$ \textup(существует алгоритм, который получает $r$ и
перечисляет интервалы, образующие $U_r$\textup).
\end{definition}

Условие (a) означает, что функция $t$ полунепрерывна снизу, так
что перечислимость снизу является эффективизацией понятия
полунепрерывности.

Понятие перечислимой снизу функции в дальнейшем играет важную
роль. Его можно определить различными (эквивалентными)
способами. Вот одна из таких переформулировок.

\begin{definition}
        \label{def:basic-function}
Функция с рациональными значениями, определјнная на $\Omega$,
называется \emph{базисной}, если еј значение на
последовательности $\omega$ определяется конечным началом
$\omega$.

Если это начало имеет длину $N$, то функция может принимать до
$2^N$ значений, и еј можно задать таблицей из $2^N$ строк, в
которой для каждого варианта начала \textup(двоичного слова
длины $N$\textup) указано рациональное значение функции. Поэтому
базисные функции можно считать конструктивными объектами.
\end{definition}

Следующее предложение легко следует из определений:

\begin{proposition}
        \label{prop:monotone-limits}
Перечислимые функции и только они являются поточечными пределами
вычислимых возрастающих последовательностей базисных функций.
\end{proposition}

Разность двух базисных функций тоже является базисной функцией,
поэтому вместо пределов возрастающих функций можно говорить о
суммах рядов, составленных из неотрицательных базисных функций.

\remm{Для бэровского пространства при естественном определении
базисных функций разность двух базисных функций уже может не
быть базисной. Тем не менее несложно убедиться, что предел
вычислимой возрастающей последовательности базисных функций
по-прежнему представим в виде суммы ряда.}

Вот ещј один вариант определения перечислимых снизу функций на
$\Omega$.

\begin{definition}[Порождающие функции]
        \label{def:generation}
Будем говорить, что определјнная на двоичных словах функция $T$
со значениями в $[0,+\infty]$ является \emph{перечислимой
снизу}, если множество пар $\langle x,r\rangle$, где $x$~---
двоичное слово, а $r$~--- рациональное число, меньшее $T(x)$,
перечислимо.

Для каждой такой функции $T$ определим функцию $t$ на
бесконечных последовательностях, положив
        $$
t(\omega)=\sup_{x\pref \omega} T(x);
        $$
будем говорить, что функция $T$ \emph{порождает} функцию $t$.
\end{definition}

\begin{proposition}
        \label{prop:lower-semicomputable-correspondence}
При этом порождаются все перечислимые снизу функции на $\Omega$
и только они.
\end{proposition}

Можно наложить дополнительные ограничения на порождающую функцию
$T$, сохранив возможность порождать любую перечислимую снизу
функцию на $\Omega$. Например, можно требовать, чтобы функция
$T$ была монотонной (это значит, что $T(x)\le T(y)$ при $x\pref
y$). В самом деле, от любой функции можно перейти к монотонной,
положив $T'(x)=\max_{z\pref x} T(z)$. Можно также потребовать,
чтобы функция $T$ принимала рациональные значения и была бы
вычислимой (а не только перечислимой снизу). В самом деле,
поскольку в определении участвует $\sup T(x)$ по всем $x$,
являющимся началом $\omega$, вместо увеличения $T(x)$ для
некоторого $x$ можно увеличить значения функции $T$ для всех его
продолжений достаточно большой длины, и эта задержка позволяет
сделать функцию $T$ вычислимой.

Следующее наблюдение использует компактность канторовского
пространства. Среди всех функций $T$, порождающих данную функцию
$t$, можно выбрать максимальную, положив
        $$
T(x)= \inf_{\omega\extends x} t(\omega).
        $$

\begin{proposition}
        \label{prop:max-generating}
Определјнная таким образом функция $T$ перечислима снизу и
порождает $t$. В еј определении можно заменить $\inf$ на $\min$.
\end{proposition}

\begin{proof}
Очевидно, что порождаемая функция не превосходит $t$. С другой
стороны, если $t(\omega)>r$, то в силу полунепрерывности снизу
это верно в некоторой окрестности $\omega$, и потому $T(x)\ge r$
для некоторого начала $x\pref \omega$. Таким образом,
порождаемая функция совпадает с $t$.

Остајтся убедиться, что $T$ перечислима снизу. В самом деле,
$r<\inf_{\omega\extends x} t(\omega)$ тогда и только тогда,
когда существует $r'>r$, для которого $r'<t(\omega)$ для всех
$\omega\extends x$. Последнее условие может быть
переформулировано так: открытое множество тех
последовательностей $\omega$, для которых $t(\omega)>r'$,
покрывает интервал $x\Omega$. Это открытое множество (по
определению перечислимости снизу) есть объединение перечислимого
семейства интервалов, и в силу компактности уже конечное число
интервалов образует подпокрытие. Поскольку это в какой-то момент
обнаруживается, указанное свойство перечислимо, и квантор
существования по $r'$ сохраняет перечислимость.

Полунепрерывность снизу также гарантирует, что минимум на
компактном множестве достигается, так что $\inf$ можно заменить
на $\min$.
\end{proof}

\subsection{\texorpdfstring
{Тесты случайности}
{Randomness tests}}
     \label{subsec:randomness-test}
Мы предполагаем, что читатель знаком с основными понятиями
теории меры (хотя бы для канторовского пространства $\Omega$).
Напомним, что мера (распределение вероятностей) $P$ на $\Omega$
задајтся своими значениями на цилиндрах $x\Omega$. Эти значения
задают неотрицательную действительную функцию на двоичных
словах, которую мы обозначаем той же буквой, что и саму меру:
        $$
P(x)=P(x\Omega).
        $$
При этом
        $$
P(\Lambda)=1, \quad P(x)=P(x0)+P(x1),
        $$
и любая неотрицательная функция на двоичных словах, обладающая
этими двумя свойствами, соответствует мере на $\Omega$.

Среди всех мер выделяются вычислимые.

\begin{definition}[Вычислимые меры]
        \label{def:computable-measures}
Действительное число $x$ называется \emph{вычислимым}, если
существует алгоритм, который по любому рациональному
$\varepsilon>0$ указывает рациональное приближение к $x$ с
абсолютной погрешностью не более $\varepsilon$.

Вычислимые действительные числа можно определять также как
пределы последовательностей $x_1,x_2,\ldots$, для которых
$|x_n-x_{n+k}|\le 2^{-n}$.

Функция, определјнная на словах \textup(или иных конструктивных
объектах\textup) и принимающая действительные значения,
называется \emph{вычислимой}, если еј значения вычислимы
равномерно по входу, то есть существует алгоритм, который по
входу и по $\varepsilon>0$ указывает $\varepsilon$-приближение к
значению функции на этом входе.

Мера на $\Omega$ называется \emph{вычислимой}, если вычислима
функция $P\colon\{0,1\}^*\to[0,1]$, соответствующая этой мере.
\end{definition}

Пусть фиксирована вычислимая (вероятностная) мера на $\Omega$.

\begin{definition}[Тест случайности по вычислимой мере]
        \label{computable-measure-test}
Перечислимая функция $t\colon\Omega\to[0,+\infty]$ называется
\textup(\emph{ограниченным в среднем}\textup) \emph{тестом}
относительно меры $P$ \textup($P$-тестом\textup), если еј
интеграл \textup(математическое ожидание\textup) по мере $P$ не
превосходит $1$:
        $$
\int t(\omega)\,dP(\omega)\le 1.
        $$

Последовательность $\omega$ \emph{проходит} тест $t$, если
$t(\omega)$ конечно \textup(напомним, что мы рассматриваем
перечислимые снизу функции, которые могут принимать бесконечные
значения\textup).

Последовательность $\omega$ называется \emph{случайной} по мере
$P$, если она проходит все $P$-тесты.
\end{definition}

Интуитивный смысл этого определения можно описать так:
$t(\omega)$ отражает ``количество закономерностей'' в $\omega$.
Строя тест, мы можем объявить ``закономерностью'' любое
(эффективно обнаруживаемое) свойство последовательности, надо
только следить, чтобы их было не слишком много, иначе интеграл
превысит границу.

Это определение эквивалентно (дајт тот же класс случайных
последовательностей) классическому определению Мартин-Лјфа (см.
ниже). Но сначала отметим, что среди тестов существует
универсальный (максимальный):

\begin{theorem}
        \label{th:universal-test-computable}
Для любой вычислимой меры $P$ существует универсальный
\textup(максимальный\textup) тест $u$: это означает, что для
любого $P$-теста $t$ найдјтся константа $c$, при которой
        $$
t(\omega)\le c\cdot u(\omega)
        $$
при всех $\omega\in\Omega$.
\end{theorem}

В частности, универсальность гарантирует, что всякая проходящая
тест $u$ последовательность проходит все другие тесты. Тем самым
множество последовательностей, проходящих тест $u$, совпадает с
множеством случайных последовательностей.

\begin{proof}
Будем перечислять все алгоритмы, задающие перечислимые снизу
функции. (Такой алгоритм порождает возрастающую
последовательность базисных функций.) Не все эти перечислимые
функции будут тестами (интеграл может превысить единицу), но мы
можем их фильтровать и не пропускать очередную функцию, пока не
будет установлено, что еј интеграл меньше (скажем)~2. (Напомним,
что мера $P$ вычислима, поэтому если интеграл меньше $2$, то мы
сможем в этом убедиться.) Такая фильтрация пропустит все тесты и
ещј некоторые функции, которые превосходят тесты не более чем
вдвое. Остајтся сложить все профильтрованные функции с
коэффициентами, сумма которых не превосходит $1/2$, скажем,
$1/2^{i+2}$.
\end{proof}

В такой форме тесты случайности фигурировали
в~\cite{GacsExact80}.

Убедимся, что это определение эквивалентно классическому
определению Мар\-тин-Лј\-фа:

\begin{definition}
        \label{def:martin-lof-version}
Пусть $P$~--- вычислимое распределение вероятностей на $\Omega$.
Последовательность открытых множеств $U_1,U_2,\ldots$ называется
\emph{тестом Мар\-тин-Лј\-фа}, если множество $U_i$ эффективно
открыто \textup(равномерно по $i$\textup) и его мера не
превосходит $2^{-i}$.

Последовательность $\omega$ \emph{проходит} этот тест, если она
не принадлежит пересечению $\cap_i U_i$. Такие пересечения
\textup(а также все их подмножества\textup) называют
\emph{эффективно нулевыми множествами}.
\end{definition}

Можно было бы рассматривать только эти пересечения (а не все
их подмножества): такие множества было бы логично называть
\emph{эффективными нулевыми множествами}. Они являются
эффективными $G_\delta$-множествами (пересечениями
последовательности равномерно эффективно открытых множеств).

Как мы уже говорили, это определение равносильно приведјнному
выше:

\begin{theorem}
        \label{th:original-equivalence}
Последовательность проходит все тесты Мар\-тин-Лј\-фа тогда и
только тогда, когда она проходит все ограниченные в среднем
тесты.
\end{theorem}

\begin{proof}
Если $t$~--- ограниченный в среднем тест, то множество $U_i$ тех
$\omega$, для которых $t(\omega)>2^i$, эффективно открыто и
имеет меру не более $2^{-i}$, так что получается тест
Мар\-тин-Лј\-фа. Поэтому если $\omega$ проходит все тесты
Мар\-тин-Лј\-фа, то $t(\omega)$ конечно.

С другой стороны, из любого теста Мар\-тип-Лј\-фа $\{U_i\}$
легко сделать ограниченный в среднем тест. Положим $t_i(\omega)$
равным $2^i$ внутри $U_i$ и нулю вне $U_i$; функция $t_i$
перечислима снизу и имеет среднее не больше $1$, то есть
представляет собой тест. Остајтся сложить, скажем, $t_{2i}$ с
весами $2^{-i}$: если $\omega$ лежит в $U_{2i}$, то такая сумма
будет не меньше $2^{i}$.
\end{proof}

В дальнейшем, говоря о случайности, мы будем иметь в виду
случайность в смысле (любого из) этих определений определений,
если не оговорено противное.

Ограниченные в среднем тесты не только отделяют случайные
последовательности от неслучайных, но и классифицируют случайные
последовательности: чем больше значение теста, тем ближе
последовательность к неслучайным. Удобно перейти при этом к
логарифмической шкале:

\begin{definition}
        \label{def:bounded-expectation-test}
Фиксируем некоторый универсальный \textup(ограниченный в
среднем\textup) $P$-тест $\at_P(\omega)$. Через
$\mathbf{d}_P(\omega)$ обозначим логарифм этого теста:
        $$
\at_P(\omega)=2^{\mathbf{d}(\omega)}.
        $$
\end{definition}

Можно сказать, что $\mathbf{d}_P(\omega)$ измеряет в битах
``дефект случайности'' последовательности $\omega$ (количество
закономерностей в $\omega$).

При нашем определении дефект бывает отрицательным (и даже может
быть равным $-\infty$): интеграл от теста не больше $1$, и
потому тест имеет значения, меньшие единицы. Можно изменить
универсальный тест (взять его полусумму с постоянным тестом,
везде равным единице) и добиться того, чтобы дефект был всегда
не меньше $-1$. Можно также сделать дефект целочисленным
(заменив каждое значение теста на максимальную степень двойки,
меньшую его). Чтобы избавиться от отрицательных дефектов, можно
разрешить тестам иметь любые конечные средние (не обязательно
меньше единицы):

\begin{proposition}
        \label{prop:log-deficiency}
Функция $\mathbf{d}_P$ является максимальной \textup(с точностью
до константы\textup) перечислимой снизу функцией на $\Omega$,
для которой $P$-среднее от $2^{\mathbf{d}_P(\cdot)}$ конечно.
\end{proposition}

\begin{remark}
     \label{rem:closing-tests}
\leavevmode\par
\textup{1}. Оригинальное определение Мартин-Лјфа также может
быть использовано для измерения дефекта случайности. Именно,
можно считать, что элементы множества $U_i$ имеют дефект $i$ или
больше. Этот способ измерения дефекта, использованный
в~\textup{\cite{ZvLe70}}, эквивалентен ограниченным по
вероятности тестам \textup(см. следующий
раздел~\textup{\ref{subsec:test-average-probability}}\textup).

\textup{2}. Мы определили функцию $\mathbf{d}_P(x)$ отдельно для
каждой меры $P$ \textup(с точностью до константы\textup). В
дальнейшем мы определим \textup(также с точностью до
константы\textup) функцию $\mathbf{d}(\omega,P)$ двух
аргументов, которая будет для каждой вычислимой меры $P$
совпадать с $\mathbf{d}_P$ \textup(с той же точностью\textup).
\end{remark}

\subsection{\texorpdfstring%
{Тесты, ограниченные по вероятности и в среднем}
{Probability- and expectation-bounded tests}}

     \label{subsec:test-average-probability}

Приведјнное выше определение ограниченного в среднем теста в
каком-то смысле аналогично определению префиксной колмогоровской
сложности; есть и другой вариант определения, который больше
похож на обычную сложность. (Определение префиксной и обычной
сложности можно найти, например, в
\cite{ShenUppsala,LiViBook97}; мы используем эти понятия лишь
как образец для аналогий, и потому не обсуждаем их подробно.)

Ослабим требования на тесты: вместо условия $\int
t(\omega)\,dP(\omega)\le 1$ будем требовать, чтобы для любого
$c>0$ множество последовательностей $\omega$, для которых
$t(\omega)>c$, имело бы меру не больше $1/c$. (Это условие
следует из прежнего согласно неравенству Чебышјва.) Такие тесты
будем называть \emph{ограниченными по вероятности}.

В логарифмической шкале это определение может быть
переформулировано так: $P$-мера множества последовательностей,
имеющих дефект больше $n$, не превосходит $2^{-n}$. Если
ограничиться целыми значениями $n$, то мы приходим к
классическому определению Мартин-Лјфа (см.
замечание~\ref{rem:closing-tests}).

Легко видеть, что и при этом определении среди всех тестов
существует максимальный (с точностью до умножения на константу).
Ему соответствует максимальная (с точностью до аддитивной
константы) функция дефекта. В самом деле, будем перечислять все
тесты (и почти-тесты, где условие на меру вдвое ослаблено) и
соответствующие им функции дефекта $d_i$. Затем возьмјм их
максимум (с аддитивными добавками, соответствующими весам):
        $$
\mathbf{d}(\omega)=\max_i [d_i(\omega)-i]-c.
        $$
Этот максимум может быть меньше $d_i$ только на $i+c$; с другой
стороны, множество тех $\omega$, для которых
$\mathbf{d}(\omega)>k$, представляет собой объединение множеств
$\{\omega\mid d_i(\omega)>k+i+c\}$. Меры этих множеств не
превосходят $O(2^{-k-i-c})$, и при подходящем $c$ их объединение
имеет меру не больше $2^{-k}$, как и требуется.

Возникает естественный вопрос: как связаны значения
универсального ограниченного в среднем теста $\at^\mathrm{aver}$
и\ универсального ограниченного по вероятности теста
$\at^\mathrm{prob}$? Как мы видели, они оба бесконечны на одних
и тех же последовательностях; более того, и конечные значения их
близки:

\begin{proposition}\label{prop:comparing-probability-average}
        $$
\mathbf{d}^\mathrm{aver}(\omega) \lea
\mathbf{d}^\mathrm{prob}(\omega) \lea
\mathbf{d}^\mathrm{aver}(\omega) +
2 \log \mathbf{d}^\mathrm{aver}(\omega)
        $$
\end{proposition}

\begin{proof}
Как мы видели, ограниченные в среднем тесты автоматически
ограничены по вероятности, откуда следует первое неравенство.
Чтобы доказать второе неравенство, рассмотрим произвольный
ограниченный по вероятности тест и его логарифм $d(\omega)$.
Покажем, что $d-2\log d$ ограничен в среднем (в логарифмической
шкале). В самом деле, событие ``$d(\omega)$ находится между
$i-1$ и $i$'' имеет вероятность не более $1/2^{i-1}$, интеграл
от $2^{d-2\log d}$ по этому множеству не превосходит
$2^{-i+1}2^{i-2\log i}=O(1/i^2)$, и потому интеграл по всему
пространству конечен.

Остајтся заметить, что неравенство $a\lea b+2\log b$ следует из
$b\gea a-2\log a$. В самом деле, из $b \ge a - 2\log a$ следует
$b\ge a/2$ (при достаточно больших $a$) и потому $\log a \le
\log b+1$, так что $a\lea b+2\log a \lea b+2\log b$.
\end{proof}

В общем случае вопрос о том, как связана ограниченность в
среднем и ограниченность по вероятности, разбирается в статье
\cite{ShaferShenVereshchaginVovk11}. Там показано (и это
несложно), что если $u\colon [1,+\infty]\to [0,+\infty]$~---
монотонная непрерывная функция, для которой $\int_1^\infty
u(t)/t^2\,dt\le 1$, то $u(t(\omega))$ является ограниченным в
среднем тестом для любого ограниченного по вероятности теста
$t$, и что это условие на $u$ нельзя улучшить. (Наша оценка
получается при $u(x)\sim x/\log^2 x$.)

\begin{remark}
Последнее предложение напоминает соотношение между простой
и префиксной сложностями \textup(как и с точки зрения соотношения между
определениями --- в одном ограничивается интеграл, в другом количество
объектов, --- так и по результатам\textup). Важно иметь в виду, что сейчас
разница между двумя величинами ограничена логарифмом \emph{дефекта},
и потому мала, если последовательность близка к случайной, в то время
как для разница между префиксной и обычной сложностями оценивается
через логарифм самих этих величин \textup(который велик для случайных объектов\textup).
\end{remark}

\begin{question}
Интересно было бы понять, отличаются ли два вида тестов лишь
некоторым сдвигом шкалы или более существенным образом.
Подтверждением такого более существенного различия могло бы
служить семейство последовательностей $\omega_i$ и $\omega_i'$,
для которых
        $$
\mathbf{d}^\mathrm{aver}(\omega_i)-
\mathbf{d}^\mathrm{aver}(\omega_i')\to +\infty
        $$
при $i\to\infty$, но
        $$
\mathbf{d}^\mathrm{prob}(\omega_i)-
\mathbf{d}^\mathrm{prob}(\omega_i')\to -\infty.
        $$
Авторы не знают, существует ли такое семейство.
\end{question}

\subsection{\texorpdfstring
{Формула для ограниченного в среднем дефекта}
{Expression for expectation-bounded tests}}

     \label{subsec:exact-expression}

Эта формула использует понятие априорной вероятности (или
префиксную сложность); напомним соответствующие определения
(подробнее см. в \cite{ShenUppsala,LiViBook97}).

\begin{definition}
        \label{def:prefix-complexity}
Множество двоичных слов называется \emph{беспрефиксным}, если ни
один из его элементов не является началом другого. Вычислимая
частичная функция $T$ из множества двоичных слов в себя
называется \emph{самоограниченным декомпрессором}, если еј
область определения является беспрефиксным множеством. Мы
определяем сложность $\KP_D(x)$ слова $x$ относительно
декомпрессора $D$ как минимальную длину слова $p$, для которого
$D(p)=x$. Среди всех самоограниченных декомпрессоров существует
\emph{оптимальный}, для которого функция $\KP_D$ минимальна с
точностью до аддитивной константы. Эта минимальная функция
\textup(для некоторого фиксированного оптимального
декомпрессора\textup) называется \emph{префиксной сложностью} и
обозначается $\KP(x)$.

Величина $\am(x)=2^{-\KP(x)}$ называется \emph{дискретной
априорной вероятностью} слова $x$.
\end{definition}

Название ``априорная вероятность'' связано с тем, что эта
функция является максимальной (с точностью до постоянного
множителя) в некотором классе вероятностных распределений. Мы
приведјм соответствующие определения и формулировки без
доказательства.

\begin{definition}
        \label{def:discrete-semimeasure}
Функция $f\colon\{0,1\}^*\to[0,\infty)$ называется
\emph{дискретной полумерой}, если $\sum_x f(x)\le 1$.
\end{definition}

Перечислимые снизу дискретные полумеры можно описать как
выходные распределения вероятностных алгоритмов, использующих
датчик случайных чисел и выдающих на выход некоторое слово (если
алгоритм останавливается; с некоторой вероятностью он может и не
остановиться).

\begin{proposition}
        \label{prop:apriori}
Функция $\am(x)$ является перечислимой снизу дискретной
полумерой, максимальной в этом классе с точностью до константы:
для любой перечислимой снизу дискретной полумеры $f$ найдјтся
такая константа $c$, что $c\cdot \am(x)\ge f(x)$ при всех $x$.
\end{proposition}

Теперь можно указать явную формулу для универсального
ограниченного в среднем теста случайности:

\begin{proposition}
        \label{prop:exact-expression-sum}
Для данной вычислимой меры $P$ универсальный ограниченный в
среднем тест $\at_P$ задајтся формулой:
        $$
\at_P (\omega)\eqm
\sum_{x\pref\omega}\frac{\am(x)}{P(x)}.
        $$
\textup(Если $P(x)=0$, соответствующая дробь считается
бесконечной.\textup)
\end{proposition}

\begin{proof}
Перечислимая снизу функция на бесконечных последовательностях
оп\-ре\-де\-ля\-ет\-ся как предел возрастающей последовательности базисных
функций. Можно представлять себе это возрастание так: в каждый
момент каждое слово $x$ имеет некоторый неотрицательный
рациональный ``вес'' $w(x)$, а значение функции на
последовательности равно сумме весов всех еј начал. Постепенно
веса (изначально равные нулю) увеличиваются; в каждый момент
лишь конечное число весов не равны нулю.

В терминах весов условие ограниченности в среднем записывается
как
        $$
\sum_x P(x)w(x)\le 1,
        $$
поэтому при умножении весов на $P(x)$ оно в точности
соответствует определению дискретной полумеры. Заметим, что
вычислимость меры $P$ гарантирует, что перечислимость снизу
сохраняется в обе стороны (при умножении и делении на $P$).

Более формально, функция
        $$
\sum_{x\pref\omega}\frac{\am(x)}{P(x)}
        $$
является перечислимым снизу ограниченным в среднем тестом; еј
интеграл в точности равен $\sum_x \am(x)$. С другой стороны,
любой перечислимый снизу тест может быть представлен в терминах
увеличения весов, и предельные значения этих весов, умноженные
на $P$, образуют перечислимую снизу полумеру.

(Заметим, что второе преобразование не однозначно: веса можно
перераспределять между двоичным словом и его продолжениями без
изменения функции на бесконечных последовательностях.)
\end{proof}

В этом рассуждении нам было важно, что $P$ (во второй части
рассуждения) и $1/P$ (в первой) перечислимы снизу.

Оказывается, что в этом предложении можно заменить сумму на
точную верхнюю грань:

\begin{theorem}
     \label{th:test-explicit-formula}
        \begin{equation}
\at_P(\omega)\eqm
\sup_{x\pref\omega}\frac{\am(x)}{P(x)}\eqm
\sum_{x\pref\omega}\frac{\am(x)}{P(x)} \label{eqn:exact}
        \end{equation}
\end{theorem}

Первое из равенств можно переписать в логарифмической шкале:
        $$
\mathbf{d}_P(\omega)\eqa
\sup_{x\pref\omega}[-\log P(x) - \KP(x)]
        $$

\begin{proof}
Поскольку верхняя грань не превосходит суммы, требует
доказательства только неравенство в одну сторону: надо
объяснить, почему верхняя грань не сильно меньше суммы.

Для данного теста $t$ рассмотрим функции $t_i$ (для всех
$i\in\mathbb{Z}$), определјнные так: $t_i(\omega)=2^i$, если
$t(\omega)>2^i$, и равно нулю в противном случае. Все они
перечислимы снизу, и их сумма отличается от $t$ не более чем
вдвое (в ту или другую сторону). Кроме того, для любой точки
$\omega$ величина $\sum_i t_i(\omega)$ превосходит $\sup_i
t_i(\omega)$ не более чем вдвое.

Преобразуем каждое $t_i$ в сумму весов, как описано выше. При
этом, поскольку $t_i$ имеет только два значения (нулевое и
ненулевое), можно считать, что вершины ненулевого веса образуют
беспрефиксное множество (на каждой ветви есть максимум одна
такая вершина).

Складывая веса вдоль каждой ветви $\omega$, мы получим $\sum_i
t_i(\omega)$, то есть универсальный тест $t(\omega)$ (с
точностью до константы). Если же вместо суммы весов брать
максимум, то мы получим нечто меньшее, но уменьшение будет не
более чем в два раза, поскольку мы складываем различные степени
двойки. (При этом важно, что для каждого $t_i$ в отдельности
переход от суммы к максимуму ничего не меняет, поскольку вдоль
каждой ветви только один член ненулевой.)
\end{proof}

\remm{Английский текст этого доказательства менее неформален.}

Это рассуждение остајтся в силе, если разрешить ненулевые веса
не для всех вершин, а только для слов определјнных длин. Пусть у
нас имеется некоторая вычислимая возрастающая последовательность
длин $n_1<n_2<n_3<\ldots$.

\begin{theorem}\label{th:additive-exact-expression}
        $$
\mathbf{d}_P(\omega)\eqa \sup_k [-\log P(\omega(1:n_k)) -
\KP(\omega(1:n_k))].
        $$
\end{theorem}

(Здесь $\omega(1:n)$ означает начало последовательности $\omega$
длины $n$.)

\begin{proof}
Переходя от перечислимых функций на последовательностях к суммам
весов, можно выбирать веса только разрешјнных длин.
\end{proof}

Эта теорема позволяет дать естественную характеризацию дефекта
случайности двумерных массивов (которые с точки зрения топологии
и меры ничем не отличаются от одномерных, и потому определение
ограниченного в среднем дефекта случайности на них очевидно
переносится). А именно, достаточно сравнивать вероятность и
сложность, скажем, для квадратов с центром в начале координат.
(В самом деле, можно расположить клетки плоскости в
последовательность таким образом, чтобы эти квадраты
соответствовали началам последовательности, и сослаться на
предыдущую теорему.)

\subsubsection*{Историческое отступление}

Формула для дефекта случайности является количественным
уточнением следующего критерия:

\begin{theorem}[Критерий случайности в терминах
                          префиксной сложности]

Последовательность $\omega$ случайна по вычислимой мере $P$
тогда и только тогда, когда разность $-\log P(x)-\KP(x)$
ограничена сверху для всех еј начал.
        \label{th:levin-schnorr-prefix}
\end{theorem}

Этот критерий был впервые сформулирован в~\cite{Chaitin75} со
ссылкой на Шнорра; доказательство (для произвольной меры) было
опубликовано впервые в~\cite{GacsExact80}. Ещј до этого Шнорр и
Левин (независимо в~\cite{Schnorr73} и~\cite{LevinRand73})
сформулировали близкий критерий случайности, использующий
несколько другой вид сложности (``монотонную сложность'').
Приведјм еј определение и соответствующую формулировку критерия.
(В цитированной работе Шнорра используется несколько другой вид
сложности, но позже Шнорр также использовал вариант сложности,
введјнный Левиным.)

\begin{definition}[Монотонная сложность]
        \label{def:monotone-complexity}
Будем называть два слова \emph{совместными}, если одно из них
является началом другого. Рассмотрим перечислимое множество $A$
пар слов $\langle x,y\rangle$, обладающее таким свойством: если
$\langle p,q\rangle\in A$, $\langle p',q'\rangle\in A$ и $p$
совместно с $p'$, то $q$ совместно с $q'$. Такое множество
\textup(``монотонный декомпрессор''\textup) задајт отображение
множества конечных и бесконечных последовательностей в себя,
определяемое такой формулой \textup(и обозначаемое той же буквой
$A$\textup):
        $$
A(p)=\sup \{x\mid (\exists p'\pref p) \langle p',x\rangle\in
            A\}.
        $$
Здесь $p$~--- конечная или бесконечная последовательность, $p'$
и $x$~--- двоичные слова, а $\sup$ понимается как наименьшее
общее продолжение, которое может быть конечным или бесконечным.
\textup(Условие на $A$ гарантирует, что общее продолжение
существует.\textup)

Далее мы определяем монотонную сложность $\KM_A(x)$ слова $x$
относительно $A$ как минимальную длину слова $p$, для которого
$A(p)\extends x$. Среди всех монотонных декомпрессоров
существует оптимальный, сложность относительно него минимальна
\textup(с точностью до константы\textup). Фиксируем оптимальный
декомпрессор $V$ и положим $\KM(x)=\KM_V(x)$.
\end{definition}

\begin{remark}
        \label{rem:oracle-computation}
Частным случаем монотонных декомпрессоров являются отображения,
задаваемые \emph{машинами Тьюринга с оракулом}. Представим себе
машину $M$ со односторонней входной лентой \textup(только для
чтения\textup), на которой написана конечная или бесконечная
последовательность $p$. Кроме того, машина имеет рабочую ленту, а
также одностороннюю выходную ленту \textup(только для
записи\textup). В процессе работы на этой ленте появляется
конечная или бесконечная последовательность $M(p)$
\textup(работа может закончиться, если машина придјт в
заключительное состояние или выйдет за границу входного слова,
или продолжаться бесконечно, если не будет ни того, ни
другого\textup). Легко убедиться, что отображение $p\mapsto
M(p)$ будет монотонным декомпрессором \textup(однако не все
монотонные декомпрессоры соответствуют таким машинам, так что
получается несколько более узкий класс отображений~--- что,
впрочем, само по себе не гарантирует, что получится существенно
другая функция сложности\textup).
\end{remark}

\remm{Класс отображений действительно более узкий: для примера
можно положить $A(\hat n 0)=0$ для всех чисел $n$ и $A(\hat
n)=0$ для всех $n$ из перечислимого неразрешимого множества,
множество пар строится легко, а машина не будет знать, читать ли
бит после окончания $\hat n$. Наверно, получается и другая
монотонная сложность, но почему? Возможно, Adam Day или Wolfgang
Merkle знают ответ.}

Монотонные декомпрессоры (или машины с оракулом описанного вида)
могут быть использованы для определения другого вида априорной
вероятности: \emph{априорной вероятности на дереве}
\textup(которую можно также назвать \emph{непрерывной априорной
вероятностью}\textup).

\begin{definition}
        \label{def:output-distribution}
Подадим на вход монотонного декомпрессора $A$ последовательность
независимых случайных битов и посмотрим на выходное
распределение на конечных и бесконечных последовательностях.
Обозначим через $M_A(x)$ вероятность того, что выходная
последовательность будет иметь начало $x$.

Легко видеть, что функция $M_A$ принимает неотрицательные
значения, перечислима снизу, $M_A(\Lambda)=1$ \textup(здесь
$\Lambda$~--- пустое слово\textup) и что $M_A(x)\ge
M_A(x0)+M_A(x1)$.

Функции, обладающие указанными свойствами, называются
\emph{перечислимыми снизу полумерами на дереве} \textup(или
\emph{непрерывными полумерами}\textup).
\end{definition}

Используя ту же конструкцию, что и для оптимального
декомпрессора, можно доказать такое утверждение~\cite{ZvLe70}:

\begin{proposition}
        \label{prop:apriori-tree}

\textup{(а)}~Всякая перечислимая снизу полумера на дереве
является выходным распределением $M_A$ для некоторого
монотонного декомпрессора~$A$.

\textup{(б)}~Среди всех перечислимых снизу полумер на дереве
существует максимальная \textup(с точностью до умножения на
константу\textup).
\end{proposition}

\begin{definition}
     \label{def:continuous-apriori}
Фиксируем некоторую максимальную перечислимую снизу полумеру на
дереве и назовјм еј \emph{априорной вероятностью на дереве}, или
\emph{непрерывной априорной вероятностью}. Обозначение:
$\ap(x)$.
\end{definition}

Соотношение между априорной вероятностью на дереве и монотонной
сложностью отчасти напоминает соотношение между дискретной
априорной вероятностью и префиксной сложностью. Однако в этом
случае $2^{-\KM(x)}$, хотя и является перечислимой снизу
полумерой на дереве, не является максимальной~\cite{GacsRel83}.
Другими словами, $\KA(x)=-\log \ap(x)$ не превосходит монотонной
сложности, но может быть меньше еј (и разница не ограничена).

Теперь можно сформулировать упомянутый критерий случайности; его
доказательство, технически не сложное, можно найти в
\cite{LiViBook97,GacsLnotesAIT88,ShenUppsala}.

\begin{proposition}
        \label{prop:levin-schnorr-monotone}
Для вычислимой меры $P$ на $\Omega$ и последовательности
$\omega\in\Omega$ следующие свойства равносильны:

\textup{(i)}~$\omega$ случайна по мере $P$;

\textup{(ii)}~$\limsup_{x\pref \omega} [-\log
P(x)-\KM(x)]<\infty$;

\textup{(iii)}~$\liminf_{x\pref \omega} [-\log
P(x)-\KM(x)]<\infty$;

\textup{(iv)}~$\limsup_{x\pref \omega} [-\log
P(x)-\KA(x)]<\infty$;

\textup{(v)}~$\liminf_{x\pref \omega} [-\log
P(x)-\KA(x)]<\infty;$
\end{proposition}

Критерий случайности с префиксной сложностью имеет два отличия:
в нјм разность (ограниченная сверху для случайных
последовательностей) не всегда ограничена снизу (в отличие от
последнего критерия); кроме того, в нјм $\limsup$ нельзя
заменить на $\liminf$.

В последнем можно убедиться на таком примере. Заметим, что к
всякому слову $x$ можно дописать некоторые биты, получив слово
$y$ c $\KP(y)\ge |y|$ (где $|y|$~--- длина слова $y$). В самом
деле, если бы это было не так, то для продолжений слова $x$ мы
имели бы $\am(y)\ge 2^{-|y|}$ и сумма $\sum_y \am(y)$ была бы
бесконечной. Построим последовательность, по очереди дописывая
длинные участки из нулей, чтобы сделать сложность существенно
меньше длины, а потом биты, которые вновь доводят сложность до
длины (как мы только что видели, это всегда возможно). Такая
последовательность не будет случайной по равномерной мере
(поскольку $\limsup$ разности бесконечен), но имеет бесконечно
много начал, у которых сложность не меньше длины, так что
$\liminf$ конечен.

Формула для (ограниченного в среднем) дефекта случайности имеет
любопытное следствие. Рассмотрим равномерную меру на
последовательностях (соответствующую независимым бросаниям
честной монеты). Эта мера инвариантна относительно перестановок,
и отсюда легко следует, что вычислимые перестановки членов
последовательности сохраняют случайность. Более того, они
сохраняют и дефект случайности (с точностью до константы).
Отсюда получаем такое следствие:

\begin{proposition}
        \label{prop:randomness-permutation}
Максимальная разность $|x|-\KP(x)$ для начал случайной
последовательности $\omega$ изменяется при вычислимой
перестановке членов последовательности не более чем на константу
\textup(зависящую от перестановки, но не от
последовательности\textup).
\end{proposition}

Некоторые более общие результаты такого типа можно найти в
\cite{LevinRand73,LevinUnif76,GacsUnif05}.

Другое следствие известно под названием ``леммы Миллера -- Ю''
(Miller--Yu ample access lemma):

\begin{corollary}
        \label{cor:ample-access}
Последовательность $\omega$ случайна относительно вычислимой
меры $P$ тогда и только тогда, когда
        $$
\sum_{x\pref\omega} 2^{-\log P(x) - \KP(x)} < \infty.
        $$
\end{corollary}

Отсюда, кстати, можно получить другое доказательство уже
упомянутого факта:

\begin{corollary}
        \label{cor:ample-access-corollary}
Для всякого слова $x$ найдјтся его продолжение $y$, у которого
               $\KP(y)>|y|$.
\end{corollary}

\begin{proof}
В самом деле, $x$ является началом некоторой случайной
последовательности, и у неј по лемме Миллера--Ю есть сколь
угодно длинные начала, сложность которых больше длины.
\end{proof}

\subsection{\texorpdfstring
{Игровая интерпретация}
{Game interpretation}}

Формула для дефекта случайности может быть интерпретирована в
игровых терминах. Рассмотрим игру Алисы и Боба с неполной
информацией. Алиса выбирает бесконечную последовательность нулей
и единиц. Боб выбирает (не видя последовательности Алисы) слово
$x$. Они встречаются и одновременно открывают свои ходы. После
этого, если $x$ является началом $\omega$, то Алиса платит Бобу
$2^{|x|}$ рублей. (Эта версия игры соответствует равномерной
мере, то есть независимым бросаниям симметричной монеты; в общем
случае Алиса платит Бобу $1/P(x)$.)

Как обычно для игр с неполной информацией, будем рассматривать
\emph{чистые} стратегии (возможности игроков, согласно правилам
игры), и \emph{смешанные стратегии} (распределения вероятностей
на чистых стратегиях). Легко видеть, что \emph{цена} этой игры
(в смысле смешанных стратегий, как это обычно понимается для игр
с неполной информацией) равна $1$. В самом деле, Боб может
указать пустое слово и получить $1$ в любом случае. С другой
стороны, если Алиса честно получает свою последовательность
бросанием монеты, то математическое ожидание еј проигрыша равно
$1$, как бы ни пошјл Боб.

Оказывается, что Боб может построить вероятностную стратегию,
которая гарантирует ему успех, если Алиса поленится бросать монеты
и принесјт неслучайную последовательность. Рассмотрим
вероятностный алгоритм $D$, который дајт на выходе двоичные
слова (а может и ничего не дать с положительной вероятностью).
Такой алгоритм является смешанной стратегией для Боба (если на
выходе не появляется никакого слова, то Боб пропускает игру и
ничего не получает).

Теперь можно заметить следующее:

(i)~Для любой вероятностной стратегии Боба математическое
ожидание еј выигрыша (как функция от последовательности Алисы)
является ограниченным в среднем тестом. (Отсюда уже следует, что
это математическое ожидание будет конечным, если
последовательность Алисы случайна в смысле Мартин-Лјфа.)

(ii)~Если $m(x)$~--- вероятность получить $x$ на выходе
алгоритма $D$, то математическое ожидание выигрыша Боба на
$\omega$ равно
        $$
\sum_{x\pref\omega} \frac{m(x)}{P(x)}.
        $$

(iii)~Поэтому, если взять алгоритм, порождающий на выходе
дискретную априорную вероятность $\am(x)$, то математическое
ожидание выигрыша Боба будет универсальным тестом (по доказанной
формуле для универсального теста).

Таким образом, использование априорной вероятности как смешанной
стратегии позволяет Бобу (в среднем) наказать Алису бесконечным
штрафом за любую неслучайность в еј последовательности.

Можно рассматривать немного более общую игру и разрешить Бобу
указывать (в качестве чистых стратегий) не одно слово $x$, а
некоторую базисную функцию $f$ с неотрицательными значениями.
При этом его выигрыш (для последовательности $\omega$,
принесјнной Алисой), равен $f(\omega)/\int
f(\omega)\,dP(\omega)$. (Знаменатель делает средний выигрыш
равным единице.) Ходу $x$ в старой игре при этом соответствует
базисная функция, равная $2^{|x|}$ на продолжениях $x$ и нулю в
остальных местах.

Такое обобщение не дајт по существу ничего нового: мы и так
разрешаем смешанные стратегии, а базисную функцию можно
представить смесью нескольких ходов. (Получив функцию $f$, Боб
может сделать ещј один вероятностный шаг и выбрать один из
интервалов, на которых $f$ постоянна, с соответствующей
вероятностью.) Таким образом мы приходим к другой формуле для
универсального теста:
        $$
\at_P(\omega)\eqm \sum_f \frac{\am(f)
f(\omega)}{\int f(\omega)\,dP(\omega)}.
        $$
Преимущество этой формулы в том, что она сохраняет смысл в более
общих случаях, чем канторовское пространство, когда никаких
выделенных интервалов нет и мы работаем прямо с каким-то классом
базисных функций.

Отметим в заключение, что игровая интерпретация теории
вероятностей, согласно которой случайность объекта есть не его
свойство, а, грубо говоря, тип гарантии, с которой этот объект
продајтся, развита в книге Шейфера и Вовка~\cite{ShaferVovk}.

\section{\texorpdfstring
{От тестов к сложностям}
{Tests and complexities}}

Формула~(\ref{eqn:exact}) выражает дефект случайности
бесконечной последовательности (значение универсального
ограниченного в среднем теста) через сложность еј конечных
начал. Возникает естественный вопрос: можем ли мы действовать в
обратном направлении и связать сложность конечного слова с
дефектом случайности его бесконечных продолжений?

Мы уже переходили от бесконечных последовательностей к конечным
в предложении~\ref{prop:max-generating}. Это можно сделать и для
универсального теста:

\begin{definition}
        \label{def:bar-test}
Фиксируем вычислимую меру $P$. Пусть $t$~--- ограниченный в
среднем тест на $\Omega$. Определим для любого конечного слова
$z$ значение $\bar t(z)$ как минимум значений $t$ на всех
продолжениях:
        $$
\bar t(z)= \inf_{\omega\extends x} t(\omega).
        $$
\end{definition}

Функция $\bar t$ однозначно определяется по $t$, и, напротив,
позволяет восстановить $t$, поэтому еј можно считать конечной
версией теста $t$. Интуитивно говоря, слово выглядит
неслучайным, если все его бесконечные продолжения имеют большой
дефект случайности с точки зрения $t$.

\begin{question}
Колмогоров~\cite{KolmPPI69} предлагал аналогичный подход к
конечным словам: для каждого слова $z$ можно рассмотреть
минимальный дефект случайности (относительно равномерного
распределения, понимаемый как разность между длиной и
сложностью) его \emph{конечных} продолжений. Есть ли тут
какая-то связь с функцией $\bar t$?
\end{question}

Покажем, каким образом можно определить функцию
$\overline{\at}_P$, соответствующую универсальному тесту, не
обращаясь к бесконечным последовательностям.

\begin{definition}[расширенный тест для вычислимой меры]
        \label{def:extended-test-computable}
Будем называть неотрицательную монотонную перечислимую снизу
функцию $T\colon \{0,1\}^*\to[0,+\infty]$ \emph{расширенным
тестом} для вычислимой меры $P$, если для любого $N$ среднее
значение $T$ на словах длины $N$ не превосходит~$1$:
        $$
\sum_{\{x\colon |x|=N\}} P(x)T(x) \le 1.
        $$
\end{definition}

Монотонность означает, что $T(x)\le T(y)$ при $x\pref y$. Она
гарантирует, что сумму по всем словам данной длины можно
заменить на сумму по произвольному конечному (или даже
бесконечному) беспрефиксному множеству $S$:
        $$
\sum_{x\in S} P(x)T(x) \le 1.
        $$
(В самом деле, продолжим слова из $S$ до какой-то большой общей
длины.)

\begin{proposition}
        \label{prop:extended-vs-average}
Всякий расширенный тест порождает \textup(в смысле
определения~\textup{\ref{def:generation}}\textup) некоторый
ограниченный в среднем тест на бесконечных последовательностях.
Обратно, всякий ограниченный в среднем тест на бесконечных
последовательностях порождается некоторым расширенным тестом.
\end{proposition}

\begin{proof}
Первая часть непосредственно следует из определения (и теоремы о
монотонной сходимости под знаком интеграла). В обратную сторону
можно положить $T=\bar t$ или сослаться на
предложение~\ref{prop:monotone-limits} (тогда
компактность не понадобится).
\end{proof}

Но можно рассматривать расширенные тесты и не упоминая
бесконечные последовательности. Обычным образом можно доказать,
что среди них существует максимальный:

\begin{proposition}
        \label{prop:universal-extended}
Пусть $P$~--- вычислимая мера. Среди всех расширенных тестов для
меры $P$ существует максимальный \textup(с точностью до
умножения на константу\textup).
\end{proposition}

\begin{definition}
        \label{def:universal-extended}
Будем называть этот максимальный тест \emph{универсальным
расширенным тестом} для меры $P$.
\end{definition}

\begin{proposition}
        \label{prop:universal-average-extended}
Универсальный расширенный тест для меры $P$ совпадает с
$\overline{\at}_P$ с точностью до ограниченного множителя.
\end{proposition}

\begin{proof}
Поскольку $\overline{\at}_P$ является расширенным тестом, то он
не превосходит универсального (с точностью до константы). С
другой стороны, универсальный расширенный тест задајт тест на
бесконечных последовательностях, и остајтся сравнить его с
максимальным.
\end{proof}

Это построение по существу использует компактность пространства
$\Omega$ (и потому, например, не проходит для
последовательностей натуральных чисел), но и без этого можно
построить максимальный расширенный тест, который будем
обозначать $\at_P(x)$; использование одного и того же
обозначения $\at_P$ не вызовет путаницы, так как в одном случае
аргументом являются бесконечные последовательности, а в
другом~--- конечные слова.

Определение расширенного теста позволяет изгнать бесконечные
последовательности, сохранив по существу то же понятие
универсального теста и даже немного обогатив его: отметим, что
не всякий расширенный тест, порождающий универсальный
ограниченный в среднем тест, является универсальным расширенным
тестом (его значение на каком-то слове может быть малым, что не
мешает универсальности на уровне бесконечных
последовательностей, поскольку значения на всех продолжениях
большие).

\medskip

Описанный способ перехода от тестов на бесконечных
последовательностях к тестам на словах не является единственно
возможным.

\begin{definition}
        \label{def:density-test}
Предположим, что вычислимая мера $P$ положительна на всех
интервалах: $P(x)>0$ для любого слова $x$. Обозначим через
$\hat{\at}_P(x)$ условное математическое ожидание
$\at_P(\omega)$ при условии, что $\omega$ начинается на $x$.
Другими словами, $\hat{\at}_P(x)$ есть среднее значение $t$ на
интервале $x\Omega$, то есть отношение интеграла
        $$
U(x)= \int\limits_{x\Omega} \at(\omega)\, dP(\omega)
        $$
к $P(x)$.
\end{definition}

Функция $U$ является перечислимой снизу полумерой. Более того,
она обладает свойствами меры, за исключением того, что мера
всего $\Omega$ не равна единице (отметим также, что $U(x)$ не
обязано быть вычислимым). Эта мера имеет плотность $\hat\at$
относительно $P$. Отсюда следует, что функция $\hat{\at}_P$
является мартингалом в смысле следующего определения:

\begin{definition}
        \label{def:martingale}
Функция $g\colon\{0,1\}^*\to\mathbb{R}$ называется
\emph{мартингалом} относительно распределения вероятностей $P$,
если
        $$
P(x)g(x)=P(x0)g(x0)+P(x1)g(x1)
        $$
для любого слова $x$. Если заменить знак ``$=$'' на ``$\ge$'',
получим определение \emph{супермартингала}.
\end{definition}

Будучи мартингалом, функция $\hat{\at}_P(x)$ не является
монотонной по~$x$.

Следующая теорема устанавливает соотношение между различными
мерами неслучайности двоичных слов:

\begin{theorem}
        \label{thm:long-inequality}
        $$
\frac{\am(x)}{P(x)} \lem
\at_P(x) \lem
\hat{\at}_P(x) \lem
\frac{\ap(x)}{P(x)},
        $$
где $\am(x)$~--- дискретная априорная вероятность слова $x$
\textup(см. предложение~\textup{\ref{prop:apriori}}\textup), а
$\ap(x)$~--- непрерывная априорная вероятность \textup(на
дереве, см. предложение~\textup{\ref{prop:apriori-tree}}\textup)
того же слова.
\end{theorem}

\begin{proof}
Первое неравенство можно даже усилить, заменив $\am(x)/P(x)$ на
сумму $\sum_{t\pref x} \am(t)/P(t)$: эта сумма является частью
выражения для $\at_P(\omega)$ для любого продолжения $\omega$
слова $x$.

Второе неравенство связывает среднее и наименьшее значения
случайной величины.

Последнее неравенство следует из сравнения перечислимой снизу
полумеры на дереве $U$ с максимальной.
\end{proof}

Отметим ещј, что $\hat{\at}_P(x)$ является мартингалом, а
$\ap(x)/P(x)$~--- лишь супермартингалом (максимальным среди
супермартингалов относительно $P$, с точностью до
мультипликативной константы).

\begin{remarks}\relax
\leavevmode\par
\textup{1}. Между первым и вторым членом неравенства последней
теоремы можно поместить ещј два:
        $$
{}  \lem \max_{t\pref x} \frac{\am(t)}{P(t)}
\lem \sum_{t\pref x} \frac{\am(t)}{P(t)}
\lem {}
        $$

\textup{2}. В логарифмической шкале имеем
        $$
-\log P(x)-\KP(x) \lea
\log \at_P(x) \lea
\log \hat{\at}_P(x) \lea
-\log P(x) -\KA(x).
        $$

\textup{3}. Мера $U$ зависит от $P$ \textup(напомним, что
$U$~--- это максимальная перечислимая снизу мера, имеющая
плотность относительно $P$\textup), и для различных мер $P$
\textup(например, с различными носителями\textup) меры $U$ могут
быть разными. Но зависимость эта не так велика: теорема
показывает, что возможные колебания ограничены разностью между
$\KP(x)$ и $\KA(x)$.

\textup{4}. Последнее неравенство в теореме
\textup($\hat{\at}_P(x)\lem \ap(x)/P(x)$\textup) нельзя заменить
на равенство. Пусть, например, мера $P$ равномерна, а в качестве
$x$ берутся начала возрастающей длины какой-то вычислимой
последовательности. Тогда $U(x)$ стремится к нулю
\textup(область интегрирования сходится к одноэлементному
множеству, имеющему меру нуль\textup), а $\ap(x)$ отделено от
нуля.

\textup{5}. Мы использовали компактность \textup(конечность
алфавита $\{0,1\}$\textup), доказывая
предложение~\textup{\ref{prop:max-generating}}. Вместо этого
можно было бы использовать
предложение~\textup{\ref{prop:lower-semicomputable-correspondence}}
и получить аналогичные результаты для бэровского пространства
последовательностей натуральных чисел.
\end{remarks}

Все перечисленные в теореме~\ref{thm:long-inequality} величины
могут быть использованы для характеризации случайности:
последовательность случайна тогда и только тогда, когда любая из
этих величин ограничена на еј начальных отрезках. В самом деле,
теорема Левина -- Шнорра гарантирует, что для случайной
последовательности последнее отношение ограничено, а первое нет.
Поскольку вторая величина монотонна, то для неслучайной
последовательности все величины, начиная со второй, стремятся к
бесконечности. Как мы уже упоминали, про первую величину этого
утверждать нельзя.

\begin{question}
Некоторые из величин, упомянутых в
теореме~\textup{\ref{thm:long-inequality}} \textup(вторая слева,
а также две промежуточные между первой и второй\textup),
монотонны. Первая величина \textup(см. обсуждение критерия
случайности\textup), а также величина $\hat{\at}_P(x)$
\textup(мартингал\textup), не монотонны. Что можно сказать про
последнюю?
\end{question}

Отметим, что все эти величины если и не монотонны, то близки к
монотонным.

\section{\texorpdfstring
{Бернуллиевы последовательности}
{Bernoulli sequences}}
     \label{sec:bernoulli}
Можно стараться определить случайность не относительно
конкретной меры, а относительно класса мер. (Интуитивно это
означает, что мы готовы поверить, что последовательность
получена в результате вероятностного процесса с распределением в
этом классе.) Впоследствии мы сделаем это для произвольного
\emph{эффективно замкнутого} класса мер, но для наглядности
начнјм с конкретного примера~--- класса \emph{бернуллиевых} мер.

\subsection{\texorpdfstring
{Тесты для бернуллиевых последовательностей}
{Bernoulli tests}}
       \label{subsec:bernoulli-test}
Бернуллиева мера $B_p$ соответствует последовательности
независимых бросаний не обязательно симметричной монеты;
вероятность появления единицы в каждом испытании равна
некоторому $p\in[0,1]$ (одному и тому же во всех испытаниях).
Отметим, что $p$ не обязано быть вычислимым.

\begin{definition}[ограниченный в среднем бернуллиев тест]
        \label{def:bernoulli-test}
Перечислимая снизу функция на бесконечных последовательностях
называется ограниченным в среднем бернуллиевым тестом, если еј
интеграл по любой мере $B_p$ \textup(при любом
$p\in[0,1]$\textup) не превосходит~$1$.
\end{definition}

\begin{proposition}[универсальный бернуллиев тест]
        \label{prop:universal-bernoulli}
Среди всех таких тестов существует максимальный \textup(с
точностью до мультипликативной константы\textup).
\end{proposition}

\begin{proof}
Перечислимая снизу функция есть предел возрастающей
последовательности базисных функций. Для каждой из этих базисных
функций еј интеграл по мере $B_p$ представляет собой многочлен
от $p$, и легко проверить, что он не больше $1$ при всех~$p$
(если это так). Соответственно можно фильтровать все негодные
функции и перечислять все бернуллиевы тесты. Складывая их с
коэффициентами, получаем универсальный.
\end{proof}

\begin{definition}
        \label{def:universal-bernoulli-test}
Фиксируем универсальный бернуллиев тест и обозначим его
$\at_\mathcal{B}(\omega)$. Его логарифм будем называть
\emph{дефектом бернуллиевости} и обозначать
$\mathbf{d}_\mathcal{B}(\omega)$. Последовательность называется
\emph{бернуллиевой}, если еј дефект конечен.
\end{definition}

Как и раньше, можно немного модифицировать определение, чтобы
считать дефект неотрицательным целым числом.

Как и для вычислимых мер, можно перейти к конечным
последовательностям:

\begin{definition}
       \label{extended-bernoulli}
Будем называть монотонную перечислимую снизу неотрицательную
функцию $T\colon\{0,1\}\to[0,+\infty]$ \emph{расширенным
бернуллиевым тестом}, если для любого натурального $N$ и для
любого $p\in[0,1]$ выполняется неравенство $\sum_{\{x\colon
|x|=N\}} B_p(x)T(x)\le 1$.
\end{definition}

Как и для вычислимых мер, тесты на конечных и бесконечных
последовательностях связаны:

\begin{proposition}
        \label{prop:extended-bernoulli-equivalence}
Всякий расширенный бернуллиев тест порождает бернуллиев тест на
$\Omega$. Напротив, всякий бернуллиев тест на $\Omega$
порождается некоторым расширенным бернуллиевым тестом.
\end{proposition}

Среди расширенных бернуллиевых тестов существует максимальный;
он порождает универсальный бернуллиев тест на $\Omega$. Как и
раньше, мы будем использовать одно и то же обозначение
$\at_\mathcal{B}$ для максимальных тестов на конечных и
бесконечных последовательностях.

\subsection{\texorpdfstring
{Другие варианты определения бернуллиевости}
{Versions of Bernoulli tests}}
     \label{subsec:bernoulli-variations}
Как и для случайности относительно вычислимых мер, есть разные
эквивалентные варианты определения. Можно рассматривать
ограниченные по вероятности тесты (вероятность события
$t(\omega)>N$ по любой из мер $B_p$ должна быть не больше
$1/N$). Можно, следуя определению Мартин-Лјфа для вычислимых
мер, назвать тестом вычислимую последовательность эффективно
открытых множеств $U_i$, для которых $B_p(U_i)\le 2^{-i}$ при
любом $i$ и при любом $p\in[0,1]$. Все эти варианты определения
эквивалентны (и доказывается это точно так же, как для
случайности по вычислимой мере).

Интересно, что первоначальное определение бернуллиевости, данное
Мартин-Лјфом в~\cite{MLof66art}, было немного другим. Сейчас мы
покажем, что оно также эквивалентно остальным, но это несколько
сложнее.

\begin{notation}
Через $\mathbb{B}(n,k)$ мы обозначаем множество всех слов длины
$n$, содержащих ровно $k$ единиц.
\end{notation}

Мартин-Лјф определяет тест бернуллиевости как семейство
перечислимых множеств слов $U_1\supset U_2\supset
U_3\supset\ldots$; каждое из множеств наследственно вверх, то
есть вместе с любым словом содержит все его продолжения.
Ограничение на эти множества такое: рассмотрим произвольные
целые $n\ge 0$ и $k$ от $0$ до $n$; через $\mathbb{B}(n,k)$
обозначим множество всех слов длины $n$, содержащих $k$ единиц
(и $n-k$ нулей); требуется, чтобы при всех $i$ доля слов в
$\mathbb{B}(n,k)$, принадлежащих $U_i$, была бы не больше
$2^{-i}$.

Для удобства сравнения заменим множества $U_i$ на перечислимую
снизу функцию $d$ с целыми значениями, для которой $U_i=\{x\mid
d(i)\ge i\}$. Наследственность множеств означает монотонность
этой функции; помимо этого, требуется, чтобы вероятность события
$d\ge i$ внутри любого множества $\mathbb{B}(n,k)$ была бы не
больше $2^{-i}$. Видно, что эти требования соответствуют
ограниченным по вероятности расширенным тестам (в
логарифмической шкале), но только вместо класса мер $B_p$ на
словах длины $n$ рассматривается другой класс мер, а именно
класс мер, сосредоточенных на словах данной длины с данным
числом единиц. Меры из класса $B_p$ принимают равные значения на
словах одинаковой длины с одинаковым числом единиц, поэтому
представимы в виде смеси равномерных мер на $\mathbb{B}(n,k)$ с
некоторыми коэффициентами, от замены $B_p$ на эти меры условие
становится более сильным.

Покажем, что тем не менее класс бернуллиевых последовательностей
не меняется от такой замены и, более того, универсальный тест
(как функция на бесконечных последовательностях) тоже не
меняется (с точностью до ограниченного множителя, как обычно).
Мы покажем это для ограниченных в среднем вариантов тестов
(соответственно изменив определение Мартин-Лјфа); на класс
бернуллиевых последовательностей это не влияет. Рассуждение для
ограниченных по вероятности тестов аналогично.

Дадим соответствующие определения.

\begin{definition}
        \label{def:combinatorial-bernoulli-test}
Функцию $f\colon \{0,1\}^*\to [0,+\infty]$ назовјм
\emph{комбинаторным бернуллиевым тестом}, если

\textup{(а)}~она перечислима снизу;

\textup{(б)}~она монотонна \textup(увеличивается при добавлении
битов в конец слова\textup);

\textup{(в)}~для любых целых $n,k$ с $0\le k\le n$ среднее
значение функции $f$ на множестве $\mathbb{B}(n,k)$ не
превосходит $1$.
\end{definition}

Можно сравнить эти требования со случаем равномерной меры: тогда
мы требовали, чтобы среднее по всему $\{0,1\}^n$ не превосходило
единицы; теперь требование сильное: среднее по по каждой из его
частей $\mathbb{B}(n,k)$ должно быть не больше~$1$.

Имея такой тест для слов ограниченной длины, можно продолжать
его по монотонности:

\begin{proposition}
        \label{prop:combinatorial-test-extension}
Пусть имеется функция $f$, определјнная на словах длины меньше
$n$ и удовлетворяющая требованиям \textup{(а)--(в)}. Тогда еј
продолжение по монотонности на слова больших длин также
удовлетворяет этим требованиям.
\end{proposition}

\begin{proof}
Будем продолжать еј на слова длины $n$, положив $f(x0)$ и
$f(x1)$ равным $f(x)$ для слов $x$ длины $n-1$. Множество
$\mathbb{B}(n,k)$ состоит из двух частей: слов, оканчивающихся
на нуль, и слов, оканчивающихся на единицу. Первые находятся во
взаимно однозначном соответствии с $\mathbb{B}(n-1,k)$,
вторые~--- с $\mathbb{B}(n-1,k-1)$. Функция сохраняет значения
при этом соответствии, поэтому среднее по каждой из частей не
больше $1$. Следовательно, и среднее по всему $\mathbb{B}(n,k)$
не больше~$1$.
\end{proof}

Как обычно, можно определить универсальный комбинаторный
бернуллиев тест:

\begin{proposition}[универсальный комбинаторный бернуллиев тест]

Среди комбинаторных бернуллиевых тестов существует максимальный
с точностью до мультипликативной константы.
     \label{prop:universal-combinatorial-test}%
\end{proposition}

\begin{definition}
        \label{def:universal-combinatorial-test}%
Фиксируем универсальный комбинаторный тест $\mathbf{b}(x)$ и
продолжим его на бесконечные последовательности, положив
        $$
\mathbf{b}(\omega)=\sup_{x\pref\omega} \mathbf{b}(x).
        $$
Полученную функцию будем называть универсальным комбинаторным
тестом на $\Omega$ и обозначать той же буквой $\mathbf{b}$.
\end{definition}

(В силу монотонности точную верхнюю грань в этом определении
можно заменить на предел.)

Покажем, что этот тест совпадает (с точностью до ограниченного
множителя) с введјнными ранее бернуллиевыми тестами в смысле
определения \ref{def:bernoulli-test}.

\begin{theorem}\label{th:bernoulli-test-equivalence}
        $$
\mathbf{b}(\omega)\eqm \at_\mathcal{B}(\omega).
        $$
\end{theorem}

\begin{proof}
Мы уже видели, что комбинаторный бернуллиев тест является
расширенным бернуллиевым тестом (из ограничений на среднее по
каждой части $\mathbb{B}(n,k)$ следует ограничение на
математическое ожидание по мере $B_p$, так как эта мера
постоянна на каждой части). Следовательно,
$\mathbf{b}(\omega)\lem \at_\mathcal{B}(\omega)$.

Обратное утверждение неверно: расширенный бернуллиев тест может
не быть комбинаторным тестом. Однако можно построить
комбинаторный тест, который принимает те же значения (с
точностью до константы) на бесконечных последовательностях, а
только это и утверждается в теореме.

Идея тут состоит в следующем. Рассмотрим расширенный бернуллиев
тест $t$ на словах длины $n$ и перенесјм его на слова
существенно большей длины $N$ (применяя старый тест к их началам
длины $n$). Получим некоторую функцию $t'$. Нам нужно показать,
что $t'$ близка к комбинаторному тесту (превышает его не более
чем в константу раз). Для этого надо усреднить $t'$ по множеству
$\mathbb{B}(N,K)$ для произвольного $K$ между $0$ и $N$. Другими
словами, нам нужно усреднить $t$ по распределению вероятностей
на $n$-битовых началах последовательностей длины $N$, содержащих
$K$ единиц. При $N\gg n$ это распределение будет близко к
бернуллиевому с вероятностью $p=K/N$.

В терминах теории вероятностей мы имеем урну с $N$ шарами, из
которых $K$ чјрных, и вынимаем из неј (без возвращения) $n$
шаров. Нам надо сравнить распределение вероятностей с
бернуллиевым, которое получилось бы при выборке с
возвращением. Покажем, что
        \begin{quote}
\emph{при $N=n^2$ распределение без возвращения не более чем в
$O(1)$ раз превосходит распределение с возвращением}.
        \end{quote}
(Кстати, обратное неравенство не верно: при $K=1$ без
возвращения мы не можем получить слово с двумя единицами, а с
возвращением можем. Но нам достаточно неравенства в эту
сторону.)

В самом деле, при выборке без возвращения вероятность вытащить
шар данного цвета равна отношению
        $$
\frac{\text{число оставшихся шаров этого цвета}}
     {\text{число всех оставшихся шаров}}.
        $$
Оставшихся шаров этого цвета не больше, чем в случае с
возвращением, а знаменатель не меньше $N-n$. Поэтому вероятность
любой комбинации при выборке без возвращения не больше
вероятности же комбинации с возвращением, умноженной на
$N/(N-n)$ в степени $n$. При $N=n^2$ возникает множитель
$(1+O(1/n))^n=O(1)$.

Таким образом, если взять расширенный бернуллиев $t$ и затем
определить $t'(x)$ на слове $x$ длины $N$ как $t$ на начале
слова $x$ длины $\lfloor\sqrt{N}\rfloor$, то полученная функция
$t'$ будет комбинаторным тестом с точностью до константы.
(Отметим, что еј монотонность следует из монотонности $t$.)
\end{proof}

\subsection{\texorpdfstring
{Критерий бернуллиевости}
{Criterion for Bernoulli sequences}}
        \label{subsec:bernoulli-criterion}
Естественно сравнивать понятие бернуллиевой последовательности
(для которой тест бернуллиевости конечен) с понятием случайной
по мере $B_p$ последовательности. Однако определение случайности
по Мартин-Лјфу предполагало вычислимость меры, и непосредственно
не применимо к мере $B_p$ при невычислимом $p$.

Можно, однако, релятивизировать определения Мартин-Лјфа,
разрешив обращаться к оракулу для $p$. С таким оракулом мера
$B_p$ становится вычислимой и определение случайности по
Мартин-Лјфу приобретает смысл.

Следующая теорема подтверждает интуитивный смысл бернуллиевых
последовательностей как последовательностей, случайных по мере
$B_p$ при некотором~$p$:

\begin{theorem}
        \label{th:bernoulli-and-uniform}
Последовательность $\omega$ является бернуллиевой тогда и только
когда, когда она случайна по мере $B_p$ с оракулом $p$ для
некоторого $p\in[0,1]$.
\end{theorem}

Говоря об оракуле $p$, мы имеем в виду возможность получать по
$i$ значение $i$-го бита в двоичном разложении $p$ (которое
единственно, за исключением тех случаев, когда $p$
двоично-рационально, а в этих случаях оба разложения вычислимы и
оракул тривиален).

Мы будем доказывать эту теорему (и притом в более сильной
количественной форме), введя понятие теста случайности по мерам
$B_p$ как функции двух аргументов (последовательности и $p$).
Требуемый результат получится как комбинация следующих
утверждений:

(а)~Среди таких ``равномерных'' тестов случайности существует
максимальный тест $\at(\omega,p)$.

(б)~Функция $\omega\mapsto\inf_p \at(\omega,p)$ совпадает (как
обычно, с точностью до ограниченного множителя) с универсальным
бернуллиевым тестом.

(в)~При фиксированном $p$ функция $\omega\mapsto\at(\omega,p)$
совпадает (с той же точностью) с релятивизированным относительно
$p$ максимальным тестом случайности относительно
($p$-)вычислимой меры $B_p$.

Из этих трјх утверждений легко следует
теорема~\ref{th:bernoulli-and-uniform}: последовательность
$\omega$ бернуллиева, если тест бернуллиевости на $\omega$
конечен; он равен точной нижней грани $\at(\omega,p)$, поэтому
его конечность означает, что $\at(\omega,p)<\infty$ при
некотором $p$, что равносильно $p$-релятивизованной случайности
по мере $B_p$.

Нам понадобится некоторая техническая подготовка. Тесты
случайности (как функции двух аргументов) тоже будут
перечислимыми снизу, но это понятие требует уточнения, поскольку
добавился второй аргумент, действительное число. (Впоследствии
мы рассмотрим и более общую ситуацию, когда вторым аргументом
является мера.) Дадим соответствующие определения.

\begin{definition}
        \label{def:pair-topology}
Назовјм \emph{базисными прямоугольниками} в пространстве
$\Omega\times [0,1]$ множества вида $x\Omega\times (u,v)$, где
$x$~--- двоичное слово, а $u,v$~--- рациональные числа, причјм
$u<v$. \textup(Техническая оговорка: числа $u,v$ могут лежать и
вне $[0,1]$, в этом случае по второй координате берјтся
пересечение $(u,v)$ с $[0,1]$.\textup)

Функция $f\colon \Omega\times [0,1]\to[-\infty,+\infty]$
называется \emph{перечислимой снизу}, если существует алгоритм,
который получает на вход рациональное $r$ и порождает
прямоугольники, в объединении дающие всј множество пар $\langle
\omega,p\rangle$ с $r<f(\omega,r)$.

Это определение, как и раньше, требует, чтобы прообраз
$(-\infty,r)$ был эффективно открытым множеством равномерно по
$r$, но только теперь мы рассматриваем эффективно открытые
множества в $\Omega\times[0,1]$, определјнные естественным
образом.

Аналогично определяется и \emph{перечислимость сверху},
равносильная перечислимости снизу функции $(-f)$.

Функцию с конечными действительными значения называют
\emph{вычислимой}, если она перечислима и снизу, и сверху.
\end{definition}

Поскольку пересечение эффективно открытых множеств эффективно
открыто, получаем такую формулировку:

\begin{proposition}
        \label{topological-computability}
Функция $f\colon \Omega\times [0,1]\to\mathbb{R}$ вычислима
тогда и только тогда, когда для каждого интервала $(u,v)$ с
рациональными концами его прообраз есть объединение
последовательности базисных прямоугольников, эффективно
порождаемой по $u$ и $v$.
\end{proposition}

Интуитивный смысл этого требования можно понять, если иметь в
виду, что задача ``указывать приближения к $\alpha$ с любой
заданной точностью'' равносильна задаче ``перечислять все
интервалы, содержащие $\alpha$''. Поэтому для вычислимой функции
$f$ мы можем находить приближения к $f(\omega,p)$, если нам дают
приближения к $\omega$ и $p$.

Определение (неотрицательной) перечислимой снизу функции можно
переформулировать, введя понятие базисной функции. Нам будет
важно, что базисные функции непрерывны, поэтому зависимость от
действительного аргумента будет кусочно-линейной, а не скачками.

\begin{definition}[базисные функции, бернуллиев случай]
        \label{def:basic-uniform-bernoulli}
Пусть $x$~--- двоичное слово, $(u,v)$~--- рациональный интервал,
а $k$~--- натуральное число, для которого $u+2^{-k}<v-2^{-k}$.
Определим функцию $g_{x,u,v,k}(\omega,p)$ так: если $\omega$ не
начинается на $x$, то она равна нулю; если $\omega$ начинается
на $x$, то зависимость от $p$ будет кусочно-линейной, причјм при
$p\notin (u,v)$ функция равна нулю, внутри $(u+2^{-k},v-2^{-k})$
функция равна $1$, а в промежутке линейно меняется.

Теперь рассмотрим наименьший класс функций, содержащий все
функции $g_{x,u,v,k}$ и замкнутый относительно линейных
комбинаций с рациональными коэффициентами, максимумов и
минимумов. Это счјтное множество функций, которые можно задавать
конструктивно, и эти функции будем называть \emph{базисными}.
\end{definition}

Теперь можно дать эквивалентное описание перечислимости снизу:

\begin{proposition}
        \label{prop:basic-limits}
Функция $f\colon\Omega\times[0,1]\to[0,+\infty]$ перечислима
снизу тогда и только тогда, когда она представима в виде
поточечного предела неубывающей последовательности базисных
функций.
\end{proposition}

\begin{proof}
Это было бы совсем ясно, если считать базисными функциями
характеристические функции базисных прямоугольников и максимумы
конечного числа таких функций. Но мы хотим, чтобы базисные
функции были непрерывны по $p$ (это будет важно в дальнейшем).
Поэтому надо заметить, что при $k\to\infty$ функция
$g_{x,u,v,k}$ стремится к характеристической функции
прямоугольника.
\end{proof}

Непрерывность базисных функций гарантирует такое важное
свойство:

\begin{proposition}
        \label{prop:computable-integral}
Пусть $f\colon\Omega\times[0,1]\to\mathbb{R}$~--- базисная
функция. Тогда значение интеграла $\int f(\omega,p)\,dB_p(\omega)$
является вычислимой функцией от $p$ \textup(и от базисной
функции $f$\textup).
\end{proposition}

(Вычислимость понимается в описанном выше смысле; отметим, что
всякая вычислимая функция непрерывна. Аналогичное утверждение
верно для любой вычислимой функции $f$, не только базисной, но
нам это не понадобится.)

Теперь мы готовы сформулировать и доказать важный технический
факт (доказанный в \cite{HoyrupRojasRandomness09}); он не раз
нам понадобится (в том числе и в более общей ситуации).

\begin{proposition}[усечение]
        \label{prop:trimming}
Пусть $\varphi\colon\Omega\times [0,1]\to[0,\infty]$~---
перечислимая снизу функция. Тогда можно построить другую
перечислимую снизу функцию $\varphi'(\omega,p)$, не
превосходящую $\varphi(\omega,p)$ в каждой точке, для которой
при любом $p$:

\textup{(а)}~$\int \varphi'(\omega,p)\,dB_p(\omega)\le 2$;

\textup{(б)}~если $\int \varphi(\omega,p)\,dB_p(\omega)\le1$, то
             $\varphi'(\omega,p)=\varphi(\omega,p)$ при всех $\omega$.
\end{proposition}

\begin{proof}
Согласно Предложению~\ref{prop:basic-limits}, можно представить
$\varphi$ в виде суммы ряда неотрицательных базисных функций:
$\varphi(\omega,p)=\sum_n h_n(\omega,p)$. Предложение
\ref{prop:computable-integral} гарантирует, что интеграл
      $$\int\sum_{i\le n} h_i(\omega,p)\,dB_p(\omega)$$
является вычислимой функцией от $p$ (равномерно по $n$), и
поэтому множество $S_n$ тех $p$, где этот интеграл меньше $2$,
эффективно открыто (равномерно по $n$).

Теперь положим $h'_n(\omega,p)=h_n(\omega,p)$, если $p\in S_n$,
и $h'_n(\omega,p)=0$ в противном случае. Функция $h'_n$ будет
перечислимой снизу, и интеграл $\int\sum_{i\le n}
h'_i(\omega,p)\,dB_p(\omega)$ будет меньше $2$ при всех $p$.
Положив $\varphi'=\sum h_n'$, мы получим перечислимую снизу
функцию, и по теореме о монотонной сходимости
$\int\varphi'(\omega,p)\,dB_p(\omega)$ не больше $2$ при всех
$p$.

Остајтся заметить, что если при некотором $p$ интеграл
$\int\varphi(\omega,p)\,dB_p(\omega)$ не превосходит $1$, то это
$p$ войдјт во все $S_n$ и переход от $h_n$ к $h'_n$, как и
переход от $\varphi$ к $\varphi'$, ничего не изменит.
\end{proof}

После этой подготовки мы можем определить равномерные тесты
бернуллиевости и доказать их свойства:

\begin{definition}
        \label{def:uniform-bernoulli}
Функцию $t$ от двух аргументов $\omega\in\Omega$ и $p\in[0,1]$
назовјм \emph{равномерным тестом бернуллиевости}, если

\textup{(а)}~она перечислима снизу \textup(в описанном выше
     смысле, как функция пары\textup{);}

\textup{(б)}~для любого $p\in[0,1]$ математическое ожидание
     $t(\omega,p)$ по мере $B_p$ \textup(то есть интеграл
     $\int
     t(\omega,p)\,dB_p(\omega)$\textup) не превосходит~$1$.
\end{definition}

Нам осталось доказать три обещанных утверждения:

\begin{lemma}
        \label{lemma:uniform-bernoulli-universal}
Существует \emph{универсальный} равномерный тест бернуллиевости
$\at(\omega,p)$, который является максимальным в этом классе
\textup(с точностью до константы\textup).
\end{lemma}

\begin{lemma}
        \label{lemma:uniform-bernoulli-and-bernoulli}
Для этого теста функция $\at'(\omega)=\inf_p \at(\omega,p)$
совпадает \textup(с точностью до ограниченного в обе стороны
множителя\textup) с универсальным тестом бернуллиевости
$\at_\mathcal{B}(\omega)$ в смысле определения
\textup{\ref{def:universal-bernoulli-test}}.
\end{lemma}

Из этих двух лемм вытекает, что последовательность $\omega$
является бернуллиевой тогда и только тогда, когда
$\at'(\omega)<\infty$, то есть $\at(\omega,p)<\infty$ при
некотором~$p$, и это позволяет завершить доказательство
теоремы~\ref{th:bernoulli-and-uniform} ссылкой на такое
утверждение:

\begin{lemma}
        \label{lemma:uniform-bernoulli-and-relativized}
Для фиксированного $p$ функция $\at_p(\omega)=\at(\omega,p)$
совпадает \textup(с точностью до ограниченного множителя\textup)
с релятивизованным относительно $p$ универсальным тестом
случайности относительно меры~$B_p$.
\end{lemma}

\begin{proof} (лемма~\ref{lemma:uniform-bernoulli-universal})
Будем перечислять все перечислимые снизу функции двух
аргументов. К каждой из них применим
предложение~\ref{prop:trimming}, и полученные суммы сложим с
коэффициентами, образующими вычислимо сходящийся ряд с суммой
меньше $1/2$.
\end{proof}

\begin{proof} (лемма~\ref{lemma:uniform-bernoulli-and-bernoulli})
Покажем, что функция $\at'$ является универсальным бернуллиевым
тестом. При любом $p$ математическое ожидание этой функции по
мере $B_p$ не больше $1$ (поскольку она не превосходит
$\at(\omega,p)$ для этого конкретного~$p$).

Кроме того, эта функция перечислима снизу. Это доказывается
аналогично предложению~\ref{prop:max-generating} с
использованием компактности. (Аналогичное утверждение в более
общей ситуации будет доказано в
предложении~\ref{prop:parametrized-minimum}.)

Таким образом, $\at'$ является бернуллиевым тестом.
Универсальность (максимальность) очевидно следует из того, что
любой бернуллиев тест можно рассматривать как функцию двух
переменных, которая будет равномерным бернуллиевым тестом.
\end{proof}

Аналогичное рассуждение показывает, что для естественным образом
определјнного универсального расширенного равномерного
бернуллиева теста $\at(x,p)$ (первым аргументом которого
является двоичное слово) величина $\inf_p \at(x,p)$ будет
универсальным расширенным бернуллиевым тестом.

\begin{proof} (лемма~\ref{lemma:uniform-bernoulli-and-relativized})
Предположим вначале, что $p$ вычислимо. Тогда мы можем
перечислять все интервалы, содержащие $p$, и функция
$\at_p\colon\omega\mapsto \at(\omega,p)$ перечислима снизу (чтобы
перечислять те интервалы, где $\at_p (\omega)>r$, мы перечисляем
прямоугольники, в которых $\at(\omega,p)>r$, и отбираем из них те,
где вторая проекция содержит $p$).

Аналогичное рассуждение можно провести для любого $p$ и
установить, что функция $\at_p$ перечислима снизу с
$p$-оракулом. Таким образом, $\at_p$ не превосходит
универсального релятивизованного теста для $B_p$.

Обратное рассуждение чуть сложнее. Пусть имеется некоторый тест
$t$ для $B_p$, перечислимый снизу с оракулом $p$. Мы должны
найти равномерный тест $t'(\omega,p)$, который мажорирует $t$
(при данном $p$). Другими словами, нужно продолжить функцию,
первоначально определјнную только для одного $p$, на все
значения $p$, и при этом ещј и гарантировать оценку для
интеграла.

Начнјм с простого случая, когда $p$ вычислимо. В этом случае
оракул не нужен и функция $t$ перечислима снизу. Добавив в неј
фиктивный второй аргумент $p$, мы получим перечислимую снизу
функцию двух аргументов. Но тестом эта функция, скорее всего, не
будет, так как про еј математическое ожидание по мере $B_q$ при
$q\ne p$ мы ничего не знаем. Тут нам помогает
предложение~\ref{prop:trimming}: с его помощью мы преобразуем
$t$ в перечислимую снизу функцию $t'(\omega,q)$ (которая теперь
уже реально зависит от~$q$), для которой $\int
t'(\omega,q)\,dB_q(\omega)\le 2$ при всех $q$, а
$t'(\cdot,p)=t(\cdot,p)$. Поделив $t'$ пополам, получим
равномерный тест.

Теперь рассмотрим случай невычислимого $p$. В этом случае $p$
иррационально, поэтому биты его двоичного разложения можно
получать, имея перечисление всех содержащих его интервалов
(дождавшись, пока появится интервал, однозначно определяющий
нужный нам бит). Поэтому машина, использующая оракул $p$, может
быть преобразована в машину, которая перечисляет снизу некоторую
функцию $\tilde t(\omega,q)$, совпадающую с $t(\omega)$ при
$q=p$. Эта функция вовсе не обязана быть равномерным тестом
бернуллиевости (поскольку условие на интеграл гарантировано
только при $q=p$, но еј опять же можно подвергнуть усечению с
помощью предложения~\ref{prop:trimming}.
\end{proof}

\section{\texorpdfstring
{Произвольные меры на $\Omega$}
{Arbitrary measures on $\Omega$}}
       \label{sec:measures-on-omega}
В этом разделе мы по-прежнему ограничиваемся двоичными
последовательностями, но меры на $\Omega$ могут быть любыми, а
не только бернуллиевыми.

\begin{notation}
Будем обозначать множество всех вероятностных распределений на
$\Omega$ через $\mathcal{M}(\Omega)$.
\end{notation}

(Напомним, что мера всего пространства $\Omega$ всегда
равна~$1$.)

\subsection{\texorpdfstring
{Равномерные тесты случайности}
{Uniform randomness tests}}

\begin{definition}
        \label{def:uniform-tests-cantor}
Назовјм \emph{равномерным тестом} перечислимую снизу
функцию $t(\omega,P)$ двух аргументов \textup(последовательности
$\omega$ и меры $P$ на $\Omega$\textup), для которой
        $$
\int t(\omega,P)\,dP(\omega) \le 1
        $$
для любой меры $P$.
\end{definition}

Здесь требует уточнения понятие перечислимой снизу функции.
Пространство всех мер $\mathcal{M}(\Omega)$ можно рассматривать
как замкнутое подмножество бесконечного произведения
        $$
\Xi=[0,1]\times[0,1]\times[0,1]\times\ldots
        $$
(мера задајтся своими значениями на интервалах, которых счјтное
число; эти значения должны удовлетворять соотношениям,
выделяющим замкнутое множество). Теперь определим базисные
множества, эффективно открытые множества и пр. для пространства
мер:

\begin{definition}
        \label{basic-sets-of-measures}
Базисное множество \textup(открытый \emph{интервал}\textup) в
пространстве мер задајтся конечным множеством условий вида
$u<P(y)<v$, где $y$~--- двоичное слово, а $u,v$~--- рациональные
числа. Оно состоит из всех мер $P$, удовлетворяющих этим
условиям. \emph{Базисное открытое множество} в пространстве
$\Omega\times\mathcal{M}(\Omega)$ имеет вид $x\Omega\times\beta$
\textup(произведение интервалов в $\Omega$ и в
$\mathcal{M}(\Omega)$\textup).

Перечислимость снизу, сверху и вычислимость теперь определяются
стандартным образом в терминах базисных открытых множеств, см.
определение~\textup{\ref{def:pair-topology}}.
\end{definition}

Мы будем использовать компактность пространств $\Omega$ (мы
рассматриваем последовательности над конечным алфавитом
$\{0,1\}$) и $\mathcal{M}(\Omega)$. Это свойство позволяет
выбирать из любого открытого покрытия конечное подпокрытие. Нам
понадобится эффективный вариант этого свойства:

\begin{definition}[эффективная компактность]
        \label{def:effective-compactness}
Компактное под\-мно\-же\-ство $C$ пространства $\mathcal{M}$
называется \emph{эффективно компактным}, если множество
        $$
\{S\mid \text{$S$~--- конечное семейство базисных множеств,
                  покрывающее $C$}\}
        $$
перечислимо.
\end{definition}

Само пространство $\mathcal{M}(\Omega)$, как легко видеть,
компактно и эффективно компактно. Компактно оно как замкнутое
множество в произведении компактных пространств, а эффективность
следует из того, что мы можем проверить, что данные базисные
множества покрывают всј пространство (речь идјт о линейных
равенствах и неравенствах с конечным числом переменных, а там
всј алгоритмически разрешимо). Отсюда легко следует такое
утверждение:

\begin{proposition}
        \label{prop:closed-compact}
Всякое эффективно замкнутое подмножество $\mathcal{M}(\Omega)$
эффективно компактно.
\end{proposition}

\begin{proof}
Пусть эффективно замкнутое множество $C$ имеет дополнение,
являющееся объединением базисных открытых множеств
$U_1,U_2,\ldots$; семейство $S$ является покрытием $C$ тогда и
только тогда, когда вместе с некоторым конечным набором из $U_i$
оно покрывает всј пространство. А это свойство перечислимо.
\end{proof}

Верно и обратное~--- что эффективно компактное подмножество
эффективно замкнуто. (Его дополнение есть объединение всех
интервалов, которые не пересекаются с некоторым конечным
покрытием множества, а такие ситуации можно перечислять.)

Теперь мы можем продолжить построение по аналогии с
бернуллиевыми мерами. Определим понятие базисной функции (по
аналогии с определением~\ref{def:basic-uniform-bernoulli};
конкретный вид базисных функций не имеет большого значения):

\begin{definition}[базисные функции для равномерных тестов по
произвольным мерам в $\Omega$]
        \label{def:basic-function-for-measures}
\emph{Базисные функции} на множестве
$\Omega\times\mathcal{M}(\Omega)$ определяются аналогично
определению~\textup{\ref{def:basic-uniform-bernoulli}}, начиная
с функций
        $$
g_{x,y,u,v,k}\colon\Omega\times\mathcal{M}(\Omega)\to [0,1].
        $$
Здесь $x,y$~--- двоичные слова, $u,v$~--- рациональные числа, а
$k$~--- натуральное число; значение $g_{x,y,u,v,k}(\omega,P)$
равно нулю, если $\omega$ не начинается на $x$; при $x\pref
\omega$ это значение кусочно-линейно зависит от $P(y)$, и равно
$0$ при $x\notin(u,v)$ и $1$ при $x\in(u+2^{-k},v-2^{-k})$.
\end{definition}

Как и в предложении~\ref{prop:basic-limits}, всякая перечислимая
снизу неотрицательная функция является пределом возрастающей
последовательности базисных функций (заметим, что в
$g_{x,y,u,v,k}(\omega,P)$ входит значение $P$ только на слове
$y$, но мы затем берјм минимумы и максимумы).

Далее, как в предложении~\ref{prop:computable-integral}, можно
заметить, что для базисной функции $f$ интеграл $\int
f(\omega,P)\,dP(\omega)$ является вычислимой функцией меры $P$
(и базисной функции~$f$).

Наконец, остајтся верным (с тем же доказательством) и аналог
предложения~\ref{prop:trimming}:

\begin{theorem}[усечение]
        \label{th:trimming-general}
Пусть $\varphi(\omega,P)$~--- перечислимая снизу функция. Тогда
существует перечислимая снизу функция $\varphi'(\omega,P)$, для
которой для любого $P$

\textup{(а)}~$\int \varphi'(\omega,P)\,dP(\omega) \le 2$;

\textup{(б)}~если $\int \varphi(\omega,P)\,dP(\omega) \le 1$, то
$\varphi'(\omega,P)=\varphi(\omega,P)$ при всех $\omega$.
\end{theorem}

Это позволяет построить универсальный тест как функцию
последовательности и произвольной меры на $\Omega$:

\begin{theorem}
     \label{th:universal-uniform-cantor}
Существует универсальный \textup(максимальный с точностью до
постоянного множителя\textup) равномерный тест.
\end{theorem}

\begin{proof}
Как и раньше, будем перечислять все перечислимые снизу функции,
подвергать каждую из них усечению (которое еј не портит, если
функция и так была тестом), и затем сложим получившиеся тесты
(или почти-тесты) с подходящими коэффициентами.
\end{proof}

\begin{definition}
     \label{def:uniform-randomness}
Фиксируем один из универсальных равномерных тестов и будем
обозначать его $\at(\omega,P)$. Будем говорить, что
последовательность является \emph{равномерно случайной} по мере
$P$ \textup(не обязательно вычислимой\textup), если
$\at(\omega,P)<\infty$.
\end{definition}

Покажем, что для вычислимых мер это определение согласуется с
прежним (определение~\ref{def:bounded-expectation-test}):

\begin{proposition}
     \label{comparison-uniform-classical}
Пусть $P$~--- вычислимая мера. а $\at_P(\omega)$~---
универсальный ограниченный в среднем тест для этой меры в смысле
определения~\textup{\ref{def:bounded-expectation-test}}. Тогда
$c_1\at_P(\omega)\le \at(\omega,P)\le c_2\at_P(\omega)$ для
некоторых $c_1,c_2>0$ и для всех~$\omega$.
\end{proposition}

Здесь константы $c_1$ и $c_2$ зависят от выбора меры $P$ и от
выбора теста $\at_P$ для этой меры (этот выбор был произвольно
сделан для каждой вычислимой меры).

\remm{Можно ли конструкцию максимального теста для данной меры
сделать равномерной по этой мере (несмотря на то, что программы
для мер образуют неразрешимое подмножество)? При этом получится
функция не от меры, а от еј номера, естественно. Если да, то как
связан равномерный тест с этой функцией? может быть, он равен
минимальному (по номерам данной меры) дефекту по мере с таким
номером?}

Это предложение показывает, в частности, что для вычислимых мер
равномерная случайность совпадает со случайностью в смысле
Мартин-Лјфа.

\begin{proof}
Для начала покажем, что $\at(\omega,P)\le c_2\at_P(\omega)$.
Заметим, что функция $\omega\mapsto\at(\omega,P)$ является
перечислимой снизу: поскольку мера $P$ вычислима, мы можем
эффективно перечислять все интервалы в пространстве мер, еј
содержащие. Поэтому она мажорируется максимальным перечислимым
снизу $P$-тестом $\at_P$.

Чтобы доказать обратное неравенство, добавим фиктивный аргумент
и рассмотрим функцию
     $$
t(\omega,Q)=\at_P(\omega).
     $$
Эта функция, естественно, не является равномерным тестом, так
как еј интеграл по мере $Q$ (при $Q\ne P$) может быть любым, но
после усечения она становится (почти) тестом, не меняясь на $P$.
\end{proof}

Можно пытаться определить тесты относительно некоторой одной (не
обязательно вычислимой) меры. Будем говорить, что функция
$f:\Omega\to[0,+\infty]$ \emph{перечислима снизу относительно
меры $P$}, если она получается из перечислимой снизу функции на
$\Omega\times\mathcal{M}(\Omega)$ фиксацией второго аргумента
равным~$P$. Теперь можно дать такое определение:

\begin{definition}
Пусть $P$~--- некоторая мера на $\Omega$. Будем называть
\emph{тестом случайности относительно $P$} перечислимую снизу
относительно $P$ функцию $f$, для которой $\int
f(\omega)\,dP(\omega)$ не превосходит единицы.
\end{definition}

 Теорема~\ref{th:trimming-general} показывает, однако, что
нового понятия случайности при этом не получается --- мы
приходим к тем же самым равномерным тестам:

\begin{theorem}
Пусть $P$~--- некоторая мера, а $t_P(\omega)$~--- тест
относительно этой меры. Тогда существует равномерный тест
$t'(\cdot,\cdot)$, для которого $t_P(\omega)\le 2t(\omega,P)$.
Напротив, сужение любого равномерного теста на меру $P$ является
$P$-тестом.
\end{theorem}

Для равномерных тестов можно
также ввести понятие расширенного теста:

\begin{definition}[расширенные равномерные тесты]
     \label{def:extended-uniform}
Функция $T\colon \{0,1\}^*\times
\mathcal{M}(\Omega)\to[0,1]$, перечислимая снизу и
монотонная по первому аргументу,
называется \emph{расширенным равномерным тестом}, если
     $$
\sum_{|x|=n} T(x,P)P(x)\le 1
     $$
для любого $n$ и любой меры $P$.
\end{definition}

Как и раньше, в силу монотонности можно суммировать не по словам
данной длины, а по любому беспрефиксному множеству.

\remm{Возникает, кстати, естественный вопрос --- а что
получится, если не требовать монотонности (и говорить про
суммирование по любому беспрефиксному множеству)? для тестов по
данной мере получалась просто априорная вероятность. А тут,
может быть, будет априорная вероятность с мерой как оракулом?
Может ли она быть определена на дереве (дискретная --- может) И
если такая априорная вероятность с мерой как оракулом --- или
вообще с точкой в конструктивном метрическом пространстве как
оракулом --- может быть определена, то верна ли теорема
Левина--Шнорра?}

Следующее предложение легко доказывается с использованием
аналога предложения~\ref{prop:basic-limits} (представления
неотрицательных перечислимых снизу функций как сумм рядов):

\begin{proposition}
     \label{prop:uniform-test-generation}
Любой равномерный тест $t(\omega,P)$ порождается некоторым
расширенным равномерным тестом в следующем смысле:
     $$
t(\omega,P)=\sup_{x\pref\omega} T(x,P).
     $$
Напротив, эта формула по любому расширенному равномерному тесту
$T$ дајт равномерный тест $t$.
\end{proposition}

Среди расширенных равномерных тестов тоже можно выбрать
максимальный (проводя аналогичное усечение и складывая
результаты). Фиксируем один из таких тестов; будем обозначать
его $\at(x,P)$ (где $x\in\{0,1\}*$, $P$~--- мера на $\Omega$).
Он порождает максимальный расширенный равномерный тест
$\at(\omega,P)$ (с точностью до ограниченного множителя).

\begin{remark}
Если мы захотим перенести развитую нами теорию на случай
некомпактного пространства последовательностей натуральных чисел
\textup(вместо компактного пространства последовательностей
нулей и единиц\textup), то можно и нужно определять расширенные
тесты непосредственно, а не через тесты на бесконечных
последовательностях. При этом можно доказать и существование
максимального среди них.
\end{remark}

Предложение~\ref{comparison-uniform-classical} позволяет
обобщить результат, известный нам для бернуллиевых мер:

\begin{theorem}
Пусть мера $P$ вычислима относительно оракула $A$ и, напротив,
оракул $A$ может быть эффективно восстановлен по известным
приближениям с произвольной точностью к значениям меры $P$. В
этом случае последовательность $\omega$ равномерно случайна
относительно меры $P$ тогда и только тогда, когда она случайна
по Мартин-Лјфу относительно меры $P$ с оракулом $A$.
\end{theorem}

(Поскольку добавление оракула $A$ делает меру $P$ вычислимой,
случайность по Мартин-Лјфу имеет смысл.)

\begin{proof}
Пусть $\at(\omega,P)$ бесконечно. Заметим, что функция
$\at(\cdot,P)$ является $A$-перечислимой снизу, так как мера $P$
вычислима с оракулом $A$. Поэтому $\omega$ не случайна по мере
$P$ с оракулом $A$.

Напротив, пусть $\omega$ не случайна по мере $P$ с оракулом $A$
и $t(\cdot)$~--- $A$-перечислимый снизу $P$-тест, для которого
$t(\omega)=\infty$. Поскольку оракул $A$ может быть восстановлен
по приближениям к $P$, то существует перечислимая снизу функция
$\bar t(\cdot,\cdot)$, для которой $\bar t(\cdot,P)=t(\cdot)$.
Остајтся преобразовать $\bar t$ в равномерный тест с помощью
теоремы~\ref{th:trimming-general}.
\end{proof}

Отметим, что не всякая мера $P$ удовлетворяет условию теоремы
(оно означает, что массовая проблема ``указывать приближения к
значениям меры $P$'' равносильна проблеме разрешения некоторого
множества; про степени таких массовых проблем см.
в~\cite{MillerDegreesCont04}). Впоследствии
(теорема~\ref{th:day-miller}) мы покажем, как можно
характеризовать равномерную случайность для произвольных мер (в
терминах случайности по Мартин-Лјфу с оракулом).

Другое применение техники усечения: покажем, что равномерные
тесты являются обобщением равномерных бернуллиевых тестов в
смысле определения~\ref{def:uniform-bernoulli}.

\begin{theorem}
     \label{th:uniform-tests-uniform-bernoulli}
Пусть $\at(\omega,P)$~--- универсальный равномерный тест
и $\at(\omega,p)$~--- универсальный равномерный тест
бернуллиевости из леммы~\textup{\ref{lemma:uniform-bernoulli-universal}}.
Тогда $\at(\omega,B_p)\eqm\at(\omega,p)$.
\end{theorem}

Здесь $B_p$~--- бернуллиева мера с параметром~$p$.

\begin{proof}
Для доказательства $\lem$-неравенства заметим, что функция
$\langle\omega,p\rangle\mapsto \at(\omega,B_p)$ является
равномерным бернуллиевым тестом, поскольку функция $p\mapsto
B_p$ вычислима (в естественном смысле).

Чтобы доказать обратное неравенство, заметим, что существует
вычислимое отображение из пространства мер в $[0,1]$, которое
переводит $B_p$ в $p$ (достаточно взять вероятность однобитового
слова). Комбинируя эту функцию с $\at(\omega,p)$, мы получаем
перечислимую снизу функцию $f(\omega,P)$, определјнную на всех
мерах и продолжающую наш тест на бернуллиевых мерах:
$f(\omega,B_p)=\at(\omega,p)$. Функция $f$ ещј не является
равномерным тестом, но к ней можно применить усечение.
\end{proof}

\subsection{\texorpdfstring
{Априорная вероятность с оракулом и\\
                    равномерные тесты}
{Relativized a priori probability and randomness tests}}
\label{subsec:uniform-exact}

Для вычислимой меры у нас было выражение для универсального
теста (предложение~\ref{prop:exact-expression-sum}) через
(дискретную) априорную вероятность. Аналогичное выражение
существует и для универсального равномерного теста:

\begin{theorem}
     \label{thm:uniform-exact}
     $$
\at(\omega,P)\eqm\sum_{x\pref\omega} \frac{\am(x|P)}{P(x)}.
     $$
\end{theorem}

Здесь, правда, нам ещј предстоит определить понятие априорной
вероятности относительно меры, то есть величину $\am(x|P)$. Мы
сейчас это сделаем, после чего вернјмся к доказательству.

\begin{definition}
     \label{def:uniform-semimeasure}
Будем называть неотрицательную функцию $t(x,P)$, аргументами
которой являются двоичное слово $x$ и мера $P$,
\emph{равномерной перечислимой снизу полумерой}, если она
перечислима снизу и $\sum_x t(x,P)\le 1$ для любой меры $P$ на
$\Omega$.
\end{definition}

\begin{proposition}
     \label{prop:apriori-universal-uniform}
Среди всех равномерных перечислимых снизу полумер существует
наибольшая с точностью до умножения на константу.
\end{proposition}

Это доказывается тем же способом, что и существование
универсального теста (и даже немного проще, поскольку здесь
ограничение на значения теста не зависит от меры).

\begin{definition}
Фиксируем одну из таких наибольших полумер и назовјм еј
\emph{априорной вероятностью относительно $P$}. Будем обозначать
еј $\am(x|P)$.
\end{definition}

(Мы используем чјрточку вместо запятой, чтобы подчеркнуть
родство с рассматриваемой обычно условной априорной
вероятностью.)

\begin{proof} (теоремы~\ref{thm:uniform-exact})
Нам нужно проверить две вещи. Во-первых, мы должны убедиться,
что правая часть формулы задајт равномерный тест. Каждый из
членов суммы можно рассматривать как функцию двух аргументов,
равную $0$ вне конуса продолжений $x$ и равную $\am(x|P)/P(x)$
внутри этого конуса. Для данного $x$ функции $\am(x|P)$ и
$1/P(x)$ перечислимы снизу (равномерно по $x$), и их
суммирование дајт перечислимую снизу функцию. Интеграл этой
функции по какой-либо мере $P$ равен сумме интегралов слагаемых,
то есть $\sum_x \am(x|P)$, и потому не превосходит~$1$.

Здесь есть особый случай, когда $P(x)=0$ для некоторого $x$. В
этом случае соответствующее слагаемое формулы становится
бесконечным для $\omega$, продолжающих $x$. Но поскольку мера
этого конуса равна нулю, то интеграл по нему равен нулю, и
потому слагаемое хоть и не равно $\am(x|P)$, но только меньше.
Таким образом, правая часть формулы есть равномерный тест и
потому не превосходит универсального равномерного теста: мы
доказали $\gem$-неравенство.

Вторая часть доказательства не так проста: наблюдая за
увеличением значений равномерного теста, мы должны распределить
это увеличение между различными членами суммы в правой части,
при этом сохранив перечислимость снизу. Трудность в том, что
если, скажем, сначала перечислимая снизу функция была равна $1$
на некотором эффективно открытом множестве $A$, а вне него
равнялась нулю, а потом это множество заменилось на большее
множество $B$, то разница (характеристическая функция
$B\setminus A$), вообще говоря, не будет перечислимой снизу, так
как в пространстве мер (как, например, и на отрезке) разность
двух интервалов не будет открытым множеством.

Решение этой проблемы состоит в том, что мы переходим к
непрерывным функциям. Пусть нам дан произвольный равномерный
тест $t(\omega,P)$. Поскольку он перечислим снизу, его можно
представить как предел неубывающей последовательности
неотрицательных базисных функций, или --- переходя к разностям
--- в виде суммы ряда из неотрицательных базисных функций:
$t(\omega,P)=\sum t_i(\omega,P)$.

Будучи базисной, функция $t_i(\omega)$ зависит лишь от некоторого
конечного начала последовательности $\omega$; обозначим длину
этого начала $n_i$. Для каждого слова $x$ длины $n_i$ мы
получаем некоторую перечислимую снизу функцию $t_{i,x}(P)$, при
этом $t_i(\omega,P)=t_{i,x}(P)$, если $\omega$ начинается на
$x$. Теперь положим $m_i(x,P)=t_{i,x}(P)\cdot P(x)$, если $x$
имеет длину $n_i$ (для остальных длин нуль). Функция $m_i$
перечислима снизу (как произведение двух перечислимых снизу
функций) равномерно по $i$, поэтому и сумма $m(x,P)=\sum
m_i(x,P)$ будет перечислимой снизу.

Покажем, что $m$ является полумерой, то есть что $\sum_x
m(x,P)\le 1$ при любом $P$. В самом деле, в сумме $\sum_x m_i(x,P)$
ненулевые члены соответствуют словам длины $n_i$, и эта сумма
равна $\sum_x t_{i,x}(P)P(x)$, то есть в точности интегралу
$\int t_i(\omega,P)\,dP(\omega)$, а сумма этих интегралов не
превосходит $1$ по условию.

Кроме того, если для всех начал $x$ последовательности $\omega$
мера $P(x)$ не равна нулю, то
     $$
\sum_{x\pref\omega} \frac{m_i(x,P)}{P(x)}=
           \frac{t_{i,x_i}(P)\cdot P(x)}{P(x)}=t_i(\omega,P)
     $$
(здесь $x_i$~--- начало $\omega$ длины $n_i$), поэтому после
суммирования по $i$
     $$
\sum_{x\pref\omega}\frac{m(x,P)}{P(x)}=t(\omega,P),
     $$
и остајтся воспользоваться максимальностью полумеры, чтобы
получить $\lem$-неравенство для случая, когда все начала
$\omega$ имеют ненулевую $P$-меру. Если же одно из них имеет
нулевую $P$-меру, то правая часть бесконечна, так что и тут
неравенство выполнено.
\end{proof}

\begin{question}
Для универсального теста случайности относительно равномерной
меры в этой формуле можно было заменить сумму на максимум. Можно
ли это сделать для равномерных тестов? (Применјнное тогда
рассуждение встречает трудности.) Можно ли разумно определить
априорную вероятность на дереве относительно меры, и доказать
равномерный вариант теоремы Левина--Шнорра?
\end{question}

Мы вернјмся к определению априорной вероятности с оракулом (и еј
связи с префиксной сложностью) в
разделе~\ref{subsec:apriori-revisited}

\subsection{\texorpdfstring
{Эффективно компактные классы мер}
{Effectively compact classes}}
      \label{subsec:effective-compact-classes}
Мы рассматривали бернуллиевы тесты, то есть полунепрерывные
снизу функции, интеграл от которых по любой бернуллиевой мере не
превосходит~$1$. В этом определении вместо бернуллиевых мер
можно рассматривать произвольный эффективно компактный класс:

\begin{definition}
     \label{def:cantor-class-test}
Пусть $\mathcal{C}$~--- эффективно компактный класс мер на
$\Omega$. Неотрицательная полунепрерывная снизу функцию
$t\colon\Omega\to[0,\infty]$ называется
$\mathcal{C}$-\emph{тестом}, если $\int t(\omega)\,dP(\omega)\le
1$ для любой меры $P\in\mathcal{C}$.
\end{definition}

\begin{theorem}
     \label{th:universal-class-test}
Пусть $\mathcal{C}$~--- эффективно компактный класс мер.

\textup{(а)}~Существует универсальный $\mathcal{C}$-тест
$\at_\mathcal{C}(\cdot)$.

\textup{(б)}~$\at_\mathcal{C}(\omega)=
\inf_{P\in\mathcal{C}} \at(\omega,P).$
\end{theorem}

\begin{proof}
Оба утверждения теоремы доказываются аналогично
леммам~\ref{lemma:uniform-bernoulli-universal}
и~\ref{lemma:uniform-bernoulli-and-bernoulli}.
\end{proof}

\begin{remark}
Поскольку класс $\mathcal{C}$ компактен, а
функция~$\at(\omega,P)$ полунепрерывна снизу, то $\inf$ в
утверждении~\textup{(б)} можно заменить на $\min$.
\end{remark}

\begin{question}
Можно ли найти какие-то критерии случайности по отношению к
естественным классам мер (в частности, в терминах сложности)?
Например, можно ли охарактеризовать бернуллиевы
последовательности в терминах сложности их начальных отрезков?
Можно показать, что главный член дефекта бернуллиевости можно
записать как
     $$
\log C_n^k - \KP(x|n,k)
     $$
для начального отрезка $x$ длины $n$, содержащего $k$ единиц.
Подобный критерий приведјн в~\cite{GacsLnotesAIT88}, но он
выглядит довольно искусственно.

Аналогичные вопросы естественно задать и для других классов
(марковские меры, инвариантные относительно сдвигов меры).
\end{question}

\remm{А что с мерами, соответствующими независимым испытаниям, с
ка\-ки\-ми-то ограничениями на вероятности (иначе любая
последовательность случайна --- можно считать, что мера на ней
сосредоточена). Это можно рассматривать как аналог вопроса о
дефекте независимости двух последовательностей.}

\subsection{\texorpdfstring
{Разреженные последовательности}
{Sparse sequences}}
     \label{subsec:sparse}
Бывают ситуации, в которых мы говорим о случайности, но это не
сводится к стандартной постановке (случайность данного
наблюдения $\omega$ относительно данной модели $P$). Сейчас мы
рассмотрим один из таких случаев~--- понятие разреженной
последовательности, введјнное
в~\cite{BienvenuRomashchenkoShen08}. Другой пример, который
можно назвать онлайн-случайностью, рассмотрен в
разделе~\ref{on-line}.

Будем называть $p$-разреженной последовательность, в которой
меньше единиц, чем в случайной по бернуллиевой мере $B_p$.
Другими словами, будем брать произвольные $B_p$-случайные
последовательности и заменять в них некоторые единицы на нули.
Всј, что получится таким образом, будет $p$-разреженным.

\begin{definition}[разреженные последовательности]
     \label{def:sparse}
Введјм покоординатный порядок на бесконечных последовательностях
нулей и единиц \textup(или конечных последовательностях одной
длины\textup): $\omega\le\omega'$, если $\omega(i)\le
\omega'(i)$ при всех $i$, то есть $\omega$ может быть получена
из $\omega'$ заменой некоторых единиц на нули.

Пусть $B_p$~--- бернуллиева мера для некоторого вычислимого $p$.
Будем говорить, что последовательность $\omega$ является
$p$-\emph{разреженной}, если $\omega\le\omega'$ для некоторой
$B_p$-случайной $\omega'$. \textup(В терминах множеств:
$p$-разреженные множества~--- это подмножества $p$-случайных
множеств.\textup)
\end{definition}

Покажем, что в определении разреженности можно избавиться от
квантора существования по $\omega'$ и дать критерий в терминах
монотонных тестов.

\begin{definition}
     \label{def:sparse-monoton}
Будем говорить, что функция $f\colon\Omega\to[0,\infty]$
монотонна, если $f(\omega')\ge f(\omega)$ при
$\omega'\ge\omega$.

Монотонная перечислимая снизу функция
$f\colon\Omega\to[0,\infty]$ называется \emph{тестом
$p$-раз\-ре\-жен\-но\-с\-ти}, если $\int t(\omega)\,dB_p(\omega)\le 1$.
Тест разреженности называем \emph{универсальным}, если он
максимален среди всех таких тестов \textup(с точностью до
умножения на константу, как обычно\textup).
\end{definition}

Монотонность тестов гарантирует, говоря неформально, что
закономерностью является наличие единиц на каких-то местах, а не
их отсутствие. (Отметим, что раньше мы говорили совсем о другой
монотонности, определяя расширенные тесты: там сравнивались
значения функции на конечном слове и его продолжении.)

\begin{proposition}
     \label{prop:universal-sparse-formula}
Рассмотрим универсальный тест $\at(\omega,P)$. Тогда величина
     $$
 r_p(\omega)= \min_{\omega'\ge\omega}\at(\omega,B_p)
     $$
задајт универсальный тест $p$-разреженности.
\end{proposition}

\begin{proof}
Тест $p$-разреженности по определению является тестом по мере
$B_p$. Используя его монотонность и сравнивая с универсальным,
получаем, что любой тест разреженности не превосходит $r_p$ (с
точностью до константы).

В обратную сторону нужно показать, что минимум в выражении для
$r_p$ достигается и что эта функция является тестом
$p$-разреженности. Перечислимость снизу доказывается с
использованием того, что свойство $\omega\le\omega'$ задајт
эффективно замкнутое подмножество эффективно компактного
пространства $\Omega\times\Omega$ (ср. ниже
предложение~\ref{prop:parametrized-minimum}). Монотонность и
неравенство для интеграла непосредственно следуют из
определения.
\end{proof}

Отсюда получаем критерий разреженности в терминах тестов:

\begin{theorem}
     \label{th:sparse-criterion}
Последовательность является $p$-разреженной \textup(получается
из $p$-слу\-чай\-ной заменой некоторых единиц на нули\textup)
тогда и только тогда, когда универсальный тест разреженности
$r_p(\omega)$ конечен.
\end{theorem}

Разреженность эквивалентна случайности по некоторому классу мер.
Чтобы описать этот класс, введјм понятие спаривания (coupling)
мер.

\begin{definition}
     \label{def:coupling}
Пусть $P,Q$~--- две меры на $\Omega$. Будем говорить, что мера
$P$ может быть спарена с $Q$ \textup(обозначение: $P\preceq
Q$\textup), если существует мера $R$ на $\Omega\times\Omega$,
для которой:

\textup{(а)}~первая проекция $R$ равна $P$, а вторая равна
$Q$\textup;

\textup{(б)}~мера $R$ целиком сосредоточена на парах $\langle
\omega,\omega'\rangle$, у которых $\omega\le\omega'$
\textup(вероятность этого события по мере $R$ равна~$1$\textup).
\end{definition}

Отметим, что отношение спаривания в этом определении
несимметрично (хотя из названия этого не видно); более наглядно
было бы говорить ``$P$ может быть помещена под $Q$'', если
$P\preceq Q$.

Следующий критерий спариваемости хорошо известен и восходит
к~\cite{Strassen65}; доказательство можно найти
в~\cite{BienvenuRomashchenkoShen08}.

\begin{proposition}
     \label{prop:coupling-criterion}
Свойство $P\preceq Q$ равносильно такому: для всякой монотонной
базисной функции $f$ выполнено неравенство
     $$
\int f(\omega)\,dP(\omega) \le \int f(\omega)\,dQ(\omega).
     $$
\end{proposition}

В этом критерии можно допустить произвольные монотонные функции,
или, наоборот, ограничится характеристическими функциями
множеств.

\begin{notation}
     \label{def:sparse-measure-class}
Пусть $\mathcal{S}_p$~--- класс всех мер $P$, для которых
$P\preceq B_p$.
\end{notation}

\begin{proposition}
     \label{prop:s-closed}
Класс мер $\mathcal{S}_p$ является эффективно замкнутым
\textup(и, следовательно, эффективно компактным\textup).
\end{proposition}

\begin{proof}
Каждое из неравенств предыдущего предложения (для каждой
базисной функции $f$) задајт эффективно замкнутое множество, и
их пересечение тоже будет эффективно замкнутым.
\end{proof}

\begin{theorem}
     \label{th:test-for-coupled-measures}
Универсальный тест $r_p$ является универсальным тестом для
класса мер $\mathcal{S}_p$.
\end{theorem}

Отсюда вытекает, что последовательность является $p$-разреженной
тогда и только тогда, когда она равномерно случайна относительно
некоторой меры из класса $\mathcal{S}_p$.

Доказательство этой теоремы основано на таком утверждении:

\begin{lemma}[монотонизация]
     \label{lemma:monotonization}
Пусть $t\colon\Omega\to\mathbb{R}$~--- базисная функция, и $\int
t(\omega)\,dQ(\omega)\le 1$ для любой меры $Q\in\mathcal{S}_p$.
Определим монотонную базисную функцию $\hat
t(\omega)=\max_{\omega'\le\omega} t(\omega)$\textup; это
определение корректно, так как $t(\omega)$ зависит только от
конечного числа позиций в $\omega$. Тогда $\int\hat
t(\omega)\,dB_p(\omega)\le 1$.
\end{lemma}

\begin{proof}
Пусть функция $t$ зависит только от первых $n$ координат. Для
каждого $x\in\{0,1\}^n$ выберем $x'\le x$, где $t(x')$ достигает
максимума (среди таких $x'$). Помимо распределения $B_p$
рассмотрим распределение $Q$, в котором бернуллиева мера $x$
перенесена на $x'$ (при этом меры из различных $x$ могут быть
отнесены к одному $x'$ и тогда складываются). Мы описали
поведение $Q$ на первых $n$ битах; следующие биты добавляются
независимо и вероятность единицы в каждой позиции равна $p$;
отметим также, что для математических ожиданий функций $t$ и
$\hat t$ важны только первые $n$ битов.

По построению $Q\preceq B_p$ (мы по существу описали меру на
парах), поэтому $\int t(\omega) \,dQ(\omega)\le 1$. Но этот
интеграл равен $\int \hat t(\omega)\,dB_p(\omega)$. Лемма
доказана.
\end{proof}

Вернјмся к теореме~\ref{th:test-for-coupled-measures}.

\begin{proof}
Всякий тест $p$-разреженности $t$ является тестом для класса
$\mathcal{S}_p$. В самом деле, интеграл $t$ по мере из класса
$\mathcal{S}_p$ не превосходит интеграла $t$ по мере $B_p$ в
силу монотонности теста и возможности спаривания.

В другую сторону: покажем, что для любого теста $t$ для класса
$\mathcal{S}_p$ существует не меньший его тест
$p$-разреженности. В самом деле, тест $t$ может быть представлен
в виде предела возрастающей последовательности базисных функций
$t_n$. Применив к ним лемму о монотонизации, получим
возрастающую последовательность базисных функций $\hat t_n$,
которые всюду не меньше $t_n$ и имеют интегралы не больше $1$ по
мере $B_p$. Их предел и будет требуемым тестом
$p$-разреженности.
\end{proof}

\subsection{\texorpdfstring
{Варианты определений случайности}
{Other definitions of randomness}}
      \label{subsec:randomness-variations}

Мы уже определили равномерную случайность последовательности
относительно произвольной (не обязательно вычислимой) меры.
Однако есть и другие варианты такого рода определений.

\subsubsection*{Оракулы}

Мы можем использовать определение случайности по Мартин-Лјфу,
добавляя оракул, который делает меру вычислимой. А именно,
назовјм последовательность $\omega$ случайной относительно меры
$P$, если существует оракул $A$, относительно которого $P$
вычислима, и при этом $\omega$ случайна в смысле Мартин-Лјфа с
оракулом $A$ относительно $P$.

(Мы говорим ``существует оракул $A$, делающий меру $P$
вычислимой'', а не ``для любого оракула $A$, делающего $P$
вычислимой'', поскольку среди таких оракулов есть и оракул,
делающий последовательность $\omega$ вычислимой. В этом случае
она не может быть случайной, если только не является атомом меры
$P$.)

Оказывается, что  этот вариант определения равносилен равномерной
случайности (как доказали Адам Дей и Джозеф
Миллер в~\cite{DayMiller10}). Доказательство этой эквивалентности требует
некоторых приготовлений.

Прежде всего зададим себе вопрос, почему нельзя взять в качестве
оракула саму меру (как это делалось для случая бернуллиевых мер,
когда мы в качестве оракула брали двоичное разложение числа
$p$). Дело в том, что выбор такого разложения не однозначен
($0.01111\ldots=0.10000\ldots$). Когда речь идјт об одном числе
$p$, то это не важно, поскольку неоднозначность возникает только
для рациональных $p$, и в этом случае оба представления
вычислимы. Однако для мер это уже не так: мера задајтся счјтным
количеством действительных чисел (скажем, вероятностями
отдельных слов, или условными вероятностями), и произвол в
выборе представления может не сводиться к конечному число
вариантов.

\begin{definition}
     \label{def:measure-representation}
Зафиксируем некоторый способ кодирования мер на $\Omega$
двоичными последовательностями, то есть вычислимое отображение
$\pi\mapsto R_\pi$ множества $\Omega$ в пространство мер.
Например, можно разбить последовательность $\pi$ на счјтное
число частей и каждую из них считать двоичной записью условной
вероятности единицы после некоторого начала \textup(начал тоже
счјтное число\textup). При этом возникает неоднозначность
\textup(если вероятность какого-то начала равна нулю, то
условные вероятности после него не играют роли\textup), но она и
так была.

Определим теперь \emph{r-тест} \textup(representation test, тест
случайности по данному представлению меры\textup) как
перечислимую снизу неотрицательную функцию $t(\omega,\pi)$, для
которой неравенство $\int t(\omega,\pi)\,dR_\pi(\omega)\le 1$
выполняется при всех $\pi$.
\end{definition}

Как мы уже обсуждали, одна и та же мера $P$ может иметь много
представлений (может быть много различных $\pi$, для которых
$R_\pi=P$), и значения теста для различных представлений одной и
той же меры могут быть разными.

Как обычно, легко доказать, что

(а)~всякая перечислимая снизу функция может быть эффективно
усечена (сделана не более чем вдвое превосходящей r-тест), при
этом если она уже была r-тестом, то она не изменится;

(б)~существует универсальный (максимальный с точностью до
константы) r-тест $\at(\omega,\pi)$.

При фиксированном $\pi$ функция $\at(\cdot,\pi)$ совпадает с
универсальным $\pi$-вычислимым ограниченным в среднем тестом
случайности относительно меры $R_\pi$. В самом деле, она
является таким тестом; с другой стороны, любой такой тест
перечисляется снизу машиной с оракулом, и эта машина может быть
применена к любому оракулу (но может не давать теста); мы
получаем перечислимую снизу функцию $t'(\omega,\pi)$, которая
совпадает с исходным тестом при данном $\pi$; остајтся применить
свойство~(а).

Как следствие этого простого рассуждения мы получаем, что
величина $\at(\omega,\pi)$ конечна тогда и только тогда, когда
последовательность $\omega$ случайна с оракулом $\pi$
относительно меры $R_\pi$.

\begin{theorem}[Дей--Миллер]
     \label{th:day-miller}
Последовательность $\omega$ равномерно случайна по мере $P$
тогда и только тогда, когда существует оракул $\pi$, делающий
меру $P$ вычислимой, а последовательность $\omega$~--- случайной
в смысле Мартин-Лјфа по мере $P$ с оракулом $\pi$.

Кроме того,
     $$
\at(\omega,P)\eqm \inf_{\{\pi\mid R_\pi=P\}} \at(\omega,\pi).
     $$
\end{theorem}

\begin{proof}
Докажем указанное в теореме равенство. Заметим, что если $t$~---
равномерный тест, то $t(\omega,R_\pi)$ как функция от $\omega$ и
$\pi$ представляет собой r-тест, и потому мажорируется
универсальным r-тестом.

Обратное утверждение несколько сложнее. Нам нужно доказать, что
функция в правой части перечислима снизу как функция
последовательности $\omega$ и меры $P$. (Условие на интеграл
после этого получается легко, поскольку мера $P$ имеет хотя бы
одно представление $\pi$.) Это можно доказать, используя
эффективную компактность множества пар $\langle P,\pi\rangle$,
для которых $P=R_\pi$. В общем виде (для произвольных
конструктивных метрических пространств) это утверждение составит
содержание леммы~\ref{lemma:global-minimum} на
с.~\pageref{lemma:global-minimum}, и мы не будем приводить
отдельного доказательства для рассматриваемого частного случая,
поскольку оно ничем не отличается от общего.

Осталось объяснить, как связаны доказанное равенство и
случайность с оракулом. Если $\at(\omega,P)$ конечно, то по
доказанному равенству существует $\pi$, при котором $R_\pi=P$ и
$\at(\omega,\pi)$ конечно. Как мы видели, это в свою очередь
означает, что $\omega$ случайна по мере $R_\pi$, то есть по мере
$P$, с оракулом $\pi$, который делает меру $P$ вычислимой.

Напротив, если $t(\omega,P)$ бесконечно, а оракул $A$ делает
меру $P$ вычислимой, то функция $t(\cdot,P)$ будет
$A$-перечислимой снизу, и еј интеграл не превосходит $1$ по мере
$P$, так что последовательность $\omega$ не будет случайной с
оракулом $A$ по мере $P$.
\end{proof}

\subsubsection*{Слепая (безоракульная) случайность}

Можно использовать определение эффективно нулевого множества
(или перечислимого снизу теста) и в ситуации невычислимой меры.
При этом может не существовать максимального эффективно нулевого
множества. Например, если мера $P$ сосредоточена на единственной
невычислимой последовательности $\pi$, то все интервалы, не
содержащие $\pi$, будут эффективно нулевыми множествами, а их
объединение (дополнение к синглетону $\{\pi\}$) таковым не
будет, иначе $\pi$ была бы вычислимой.

Тем не менее мы можем определить понятие случайной
последовательности как последовательности, не лежащей \emph{ни в
одном} эффективно нулевом множестве (что эквивалентно тому, что
все тесты на ней конечны). Кьјс-Хансен предложил называть такую
случайность ``гиппократовой'' (ссылаясь на легенду о враче
Гиппократе), но мы предпочитаем говорить о ``слепой'' (blind)
или ``безоракульной'' случайности.

\begin{definition}[слепые тесты]
     \label{def:blind-tests}
Перечислимые снизу \textup(без оракула\textup) функции
$t(\omega)$, для которых $\int t(\omega)\,dP(\omega)\le 1$,
будем называть \emph{слепыми}, или \emph{безоракульными},
тестами для меры $P$. Последовательность $\omega$ будем называть
\emph{безоракульно случайной} по мере $P$, если $t(\omega)$
конечно для любого такого теста~$t$.
\end{definition}

Как мы видели, может не существовать максимального слепого
теста.

Это понятие безоракульной случайности можно характеризовать во
введјнных ранее терминах:

\begin{theorem}
     \label{th:oblivious-randomness-criterion}
Последовательность $\omega$ является безоракульно случайной
относительно меры $P$ тогда и только тогда, когда она случайна
относительно любого эффективно компактного класса мер,
содержащего меру $P$.
\end{theorem}

\begin{proof}
Предположим, что $\omega$ не случайна относительно некоторого
эффективно компактного класса мер, содержащих меру $P$. Тогда
перечислимый снизу (безо всякого оракула) тест для этого класса
будет безоракульным тестом для $P$, так что $\omega$ не будет
безоракульно случайной.

С другой стороны, предположим, что существует некоторый
безоракульный тест $t$ для меры $P$, для которого $t(\omega)$
бесконечно. Тогда можно попросту рассмотреть класс всех мер $Q$,
для которых $t$ является тестом, то есть для которых $\int
t(\omega)\,dQ(\omega)\le1$. Этот класс эффективно замкнут, и
потому эффективно компактен. В самом деле, если для некоторой
меры $Q$ интеграл больше $1$, то уже для некоторого базисного
приближения $t_n$ к $t$ \textup(снизу\textup) этот интеграл
больше $1$, а последнее свойство задајт эффективно открытое
множество в пространстве мер. Перечисляя все $t_n$ и объединяя
все эти множества, получаем эффективно открытое множество.
\end{proof}

Из определения (или из последней теоремы) легко вывести, что из
равномерной случайности последовательности $\omega$ относительно
$P$ следует безоракульная случайность. Обратное утверждение
неверно:

\begin{theorem}
     \label{th:oblivious-not-uniform}
Существует последовательность $\omega$, безоракульно случайная
по некоторой мере $P$, но не являющаяся равномерно случайной по
этой мере.
\end{theorem}

\begin{proof}
Заметим, что безоракульная случайность не изменится, если мы
чуть-чуть изменим меру $P$ (так, чтобы мера любого множества
изменилась не более чем в $O(1)$ раз). С другой стороны, с
вычислительной точки зрения новая мера может быть гораздо
сильнее. Например, начнјм с равномерной бернуллиевой меры
$B_{1/2}$, в которой все испытания независимы и имеют
вероятность успеха $1/2$, и возьмјм некоторую случайную по ней
последовательность $\omega=\omega(1)\omega(2)\ldots$. Затем
рассмотрим чуть сдвинутую меру $B'$, в которой вероятности
успеха равна
$1/2+\omega(1)\varepsilon_1,1/2+\omega(2)\varepsilon_2,\ldots$;
здесь $\varepsilon_i$ настолько малы и так быстро сходятся к
нулю, что $B'$ отличается от $B$ на любом множестве не более чем
в константу раз. Тогда $B'$ содержит информацию об $\omega$, и
несложно построить равномерный тест $t$, для которого
$t(\omega,B')=\infty$.
\end{proof}

Однако бывают некоторые классы мер, для которых понятия
равномерной и безоракульной случайности совпадают. (В частности,
таковы бернуллиевы меры.) Чтобы сформулировать достаточные
условия такого совпадения, начнјм с определения.

\begin{definition}
     \label{def:effective-orthogonal}
Обозначим через $\Rands(P)$ множество последовательностей,
равномерно случайных относительно меры $P$. Назовјм класс мер
\emph{эффективно ортогональным}, если
        $$
\Rands(P)\cap\Rands(Q)=\varnothing
        $$
для любых двух различных  мер $P$ и $Q$ из этого класса.
\end{definition}

\begin{theorem}
     \label{th:comparing-uniform-oblivious}
Пусть $\mathcal{C}$~--- эффективно компактный и эффективно
ортогональный класс мер. Тогда для любой меры $P$ из этого
класса понятия равномерной и безоракульной случайности
относительно $P$ совпадают.
\end{theorem}

Утверждение этой теоремы выглядит парадоксально: мы утверждаем
нечто о случайности по одной мере $P$, а в условии стоит
возможность вложить $P$ в класс мер с некоторыми свойствами.
(Было бы естественно найти более явные достаточные условия на
$P$.)

Из этой теоремы следует, что построенная при доказательстве
теоремы~\ref{th:oblivious-not-uniform} мера не может быть
вложена в эффективно замкнутый эффективно ортогональный класс.

\begin{proof}
Мы уже отмечали, что в одну сторону утверждение прямо следует из
определения. Докажем обратное. Предположим, что
последовательность $\omega$ безоракульно случайна относительно
меры $P$. По теореме~\ref{th:oblivious-randomness-criterion} она
случайна относительно класса мер $\mathcal{C}$. Следовательно,
она равномерно случайна относительно некоторой меры $P'$ из
класса $\mathcal{C}$. Остајтся доказать, что $P=P'$.

Пусть это не так. Покажем, что существует компактный класс мер
$\mathcal{C}'$, который содержит $P$, но не содержит $P'$. В
самом деле, $P\ne P'$ означает, что некоторое слово имеет
различные меры относительно $P$ и $P'$, и можно найти замкнутое
условие на меру этого слова, которое отделит $P$ от $P'$. Теперь
рассмотрим эффективно компактный класс
$\mathcal{C}\cap\mathcal{C}'$. Он содержит $P$ и потому
последовательность $\omega$ будет случайной относительно этого
класса. Значит, и в нјм есть мера $P''$, относительно которой
последовательность $\omega$ равномерно случайна. Но $P'\ne P''$
(одна мера лежит в $\mathcal{C}'$, а другая~--- нет), так что
получаем противоречие с эффективной ортогональностью
класса~$\mathcal{C}$.
\end{proof}

\begin{remark}
Доказанная теорема применима, в частности, к классу бернуллиевых
мер. Может показаться, что еј можно доказать проще: если
последовательность $\omega$ случайна \textup(безоракульно или
тем более равномерно\textup) по мере $B_p$, то предел частоты
единиц в ней равен $p$ и тем самым он определяется самой
последовательностью \textup(и дополнительный оракул для $p$
ничего нового не дајт\textup). Это рассуждение, однако,
неправильно: хотя $p$ и определяется последовательностью, но не
является вычислимой \textup(или хотя бы непрерывной\textup)
функцией от неј. В самом деле, никакой начальный отрезок
последовательности не гарантирует, что еј предельная частота
будет в заданном интервале. Аналогичное рассуждение, однако,
можно применить к последовательностям, у которых дефект
случайности ограничен заранее известной константой. \textup(См.
подробнее в~\textup{\cite{HoyrupRojasCiE2009}}, где введено
относящееся к этому понятие послойной вычислимости.\textup) В
частности, можно показать, что если последовательность $\omega$
безоракульно случайна по мере $B_p$, то двоичное разложение $p$
вычислимо с оракулом $\omega$.
\end{remark}

\section{\texorpdfstring
{Нейтральная мера}
{Neutral measure}}
        \label{sec:neutral}
Следующая теорема, опубликованная в~\cite{LevinUnif76} и затем
в~\cite{GacsUnif05}, указывает на парадоксальное свойство
равномерной случайности, которая отличает еј от случайности с
оракулом.

\begin{definition}
     \label{def:neutral}
Назовјм меру на $\Omega$ \emph{нейтральной}, если всякая
последовательность является равномерно случайной относительно
этой меры.
\end{definition}

\begin{theorem}
     \label{th:neutral}
Существует нейтральная мера; более того, существует мера $N$,
для которой $\at(\omega,N)\le 1$ при всех $\omega\in\Omega$.
\end{theorem}

Прежде чем доказывать эту теорему, отметим, что все вычислимые
последовательности обязаны быть атомами нейтральной меры (иметь
положительную вероятность). В самом деле, можно построить тест,
который ищет длинные отрезки вычислимой последовательности,
имеющие малую меру (этого можно дождаться, имея
последовательность и меру в качестве аргументов), и присваивает
им большое значение дефекта.

Отсюда следует, что нейтральная мера не может быть вычислимой. В
самом деле, для вычислимой меры легко построить вычислимую
последовательность, не являющуюся атомом (надо из двух
продолжений слова выбирать то, которое имеет меньшую~--- или
хотя бы не сильно б\'ольшую~--- меру). Аналогичное рассуждение
показывает, нейтральная мера не эквивалентна никакому оракулу
(нет оракула, который делал бы еј вычислимой и, напротив, мог бы
быть восстановлен по приближениям к ней). В самом деле, если
$A$~--- такой оракул, то (как мы видели) равномерная случайность
равносильна случайности с оракулом $A$, и можно повторить то же
рассуждение.

Нейтральная мера не может быть также перечислимой сверху или
снизу, но при нашем определении (когда мера всего пространства
должна равняться единице) это не дајт ничего нового. Некоторые
более осмысленные (и менее тривиальные) варианты этого
утверждения доказаны в~\cite{LevinRandCons84}.

\begin{proof}
Рассмотрим универсальный тест $\at(\omega,P)$. Мы утверждаем,
что существует мера $N$, для которой $\at(\omega,N)\le 1$ для
любой последовательности $N$. Другими словами, для каждой
последовательности $\omega$ мы имеем условие на $N$, состоящее в
том, что $\at(\omega,N)\le 1$, и надо доказать, что все эти
условия совместны (пересечение их непусто). Каждое условие
задајт замкнутое множество в компактном пространстве всех мер
(вспомним, что $\at$ перечислимо снизу, и тем более
полунепрерывно снизу), поэтому достаточно доказать совместность
любого конечного числа условий	.

Итак, возьмјм $k$ произвольных последовательностей
$\omega_1,\ldots,\omega_k$. Нам надо доказать, что существует
мера $N$, для которой $\at(\omega_i,N)\le 1$ при всех
$i=1,\ldots,k$. Эту меру мы будем искать в виде выпуклой
комбинации мер, сосредоточенных в $\omega_1,\ldots,\omega_k$.
Таким образом, нам надо показать, что $k$ замкнутых подмножеств
$k$-мерного симплекса (соответствующих $k$ условиям) имеют общую
точку. Это делается с помощью известной топологической леммы
(используемой в стандартном доказательстве теоремы Брауэра о
неподвижной точке):

\begin{lemma}
     \label{lemma:topology-sperner}
Пусть симплекс с вершинами $1,\ldots,n$ покрыт $k$ замкнутыми
множествами $A_1,\ldots,A_k$. Пусть при этом вершина $i$ всегда
принадлежит множеству $A_i$, ребро $i$--$j$ целиком лежит в
объединении $A_i\cup A_j$, и так далее \textup(грань
$(i_1,\ldots,i_s)$ целиком лежит в $A_{i_1}\cup\ldots\cup
A_{i_s}$\textup). Тогда пересечение $A_1\ldots, A_k$ непусто.
\end{lemma}

Приведјм для полноты стандартное доказательство этой леммы.
Рассмотрим симплициальное разбиение данного симплекса на мелкие
симплексы и пометим каждую их вершину каким-то числом $i$ от $1$
до $k$, с тем условием, чтобы эта вершина лежала в $A_i$. Более
того, можно предполагать что вершина $i$ помечена числом $i$,
вершины на отрезке $i$--$j$ помечены либо числом $i$, либо
числом $j$, и так далее. Комбинаторная лемма Шпернера говорит,
что есть симплекс разбиения, у которого все вершины помечены
по-разному. Устремляя размер максимального симплекса разбиения к
нулю, и выбирая предельную точку последовательности получившихся
разноцветных симплексов, находим искомую точку в пересечении
всех $A_i$. Лемма доказана.

Применим теперь доказанную лемму. Согласно этой лемме, нам
достаточно доказать, что любая точка на грани симплекса
покрывается объединением соответствующих множеств. Пусть,
скажем, есть точка (мера) $X$, являющаяся смесью вершин
$\omega_1$, $\omega_5$ и $\omega_7$; это значит, что мера $X$
сосредоточена в множестве $\{\omega_1,\omega_5,\omega_7\}$. Нам
нужно показать, что точка $X$ покрыта объединением множеств
$A_1$, $A_5$ и $A_7$, в нашем случае это означает, что одно из
чисел $\at(\omega_1,X)$, $\at(\omega_3,X)$ и $\at(\omega_7,X)$
не превосходит единицы. Но мы знаем, что $\int \at(\omega,X)
\,dX(\omega)\le 1$ согласно определению теста, а этот интеграл
является взвешенным средних этих трјх величин, так что хотя бы
одна из них не превосходит~$1$.
\end{proof}

\section{\texorpdfstring
{Случайные точки метрического пространства}
{Random points in metric spaces}}
     \label{sec:randomness-metric-spaces}
Б\'ольшая часть сформулированных нами результатов о случайных
последовательностях битов переносится на случай
последовательностей натуральных чисел, а часто и на более общий
случай метрических пространств. Мы сейчас изложим такие
обобщения (а также и некоторые новые~--- даже и для пространства
$\Omega$~--- результаты).

\subsection{\texorpdfstring
{Конструктивные метрическое пространства}
{Constructive metric spaces}}
     \label{subsec:constructive-metric}

Определяя конструктивные метрические пространства и перечислимые
снизу функции на них, мы
следуем~\cite{GacsUnif05,HoyrupRojasRandomness09} (см.
также\cite{GacsLnotesAIT88}).

\begin{definition}
     \label{def:constructive-metric}
Конструктивное метрическое пространство
$\mathbf{X}=(X,d,D,\alpha)$ состоит из полного сепарабельного
метрического пространства $(X,d)$, в котором выделено счјтное
плотное множество $D$ и его нумерация $\alpha\colon\mathbb{N}\to
D$ \textup(определјнная на всјм $\mathbb{N}$\textup). Требуется,
чтобы расстояние $d(\alpha(m),\alpha(n))$ было бы эффективно
вычислимо \textup(с любой требуемой точностью\textup) по $m$ и
$n$.

Открытые шары с центрами в точках из $D$ и рациональными
радиусами будем называть \emph{базисными шарами}; множество всех
таких шаров образует \emph{канонический базис} топологии $X$ как
метрического пространства.

Будем называть последовательность $s_1,s_2,\ldots$ точек из $D$
\emph{сильно фундаментальной}, если $d(s_m, s_n)\le 2^{-m}$ при
$n>m$. По предположению пространство $X$ полно, поэтому всякая
такая последовательность имеет единственный предел, и точки
пространства можно отождествить с классами эквивалентности
сильно фундаментальных последовательностей.
\end{definition}

Часто мы будем обозначать конструктивное метрическое
пространство $\mathbf{X}$ просто $X$ (если $d$, $D$ и $\alpha$
ясны из контекста).

\begin{examples}\end{examples}

1. В \emph{дискретном} метрическом пространстве
$\Sigma=\{s_1,s_2,\ldots\}$ расстояние между любыми двумя
различными точками равно $1$; множество $D$ содержит все точки и
$\alpha(i)=s_i$.

2. Можно добавить к натуральным числам бесконечный элемент,
положив $\overline{\mathbb{N}}=\mathbb{N}\cup\{\infty\}$.
Расстояние при этом можно определить как обычное расстояние
между обратными величинами (при этом, естественно,
$1/\infty=0$). Это метрическое пространство можно назвать
\emph{одноточечной компактификацией} пространства натуральных
чисел c дискретной метрикой предыдущего примера.

3. Вещественная прямая $\mathbb{R}$ с обычным расстоянием
$d(x,y)=|x-y|$ также является конструктивным метрическим
пространством (с каким-либо естественным выбором множества $D$ и
его нумерации). То же самое можно сказать о еј положительной
части $\mathbb{R}_+=[0,\infty)$; можно добавить и бесконечный
элемент $+\infty$, но тогда надо изменить метрику (например,
перенеся еј с отрезка).

4. Имея два конструктивных метрических пространства $\mathbf{X}$
и $\mathbf{Y}$, можно рассмотреть их произведение
$\mathbf{Z}=\mathbf{X}\times \mathbf{Y}$ с одной из естественных
метрик (например, сумма расстояний по обеим координатам);
множество $D_\mathbf{Z}$ также является произведением
$D_\mathbf{X}\times D_\mathbf{Y}$.

5. Пусть $X$~--- конечный или счјтный алфавит (с фиксированной
нумерацией). Тогда множество $X^\mathbb{N}$, состоящее из
бесконечных последовательностей $x=(x(1),x(2),\ldots)$ c
$x(i)\in X$, превращается в метрическое пространство, если
положить $d(x,y)=2^{-n}$, где $n$~--- минимальный индекс $i$,
где $x(i)$ и $y(i)$ различаются. Это пространство является
обобщением рассмотренного нами двоичного канторовского
пространства; шары в нјм являются цилиндрами: для данной
конечной последовательности точек $z\in X^*$ мы берјм все
продолжения $z$.

\begin{remark}
Каждая точка $x$ конструктивного метрического пространства
$\mathbf{X}$ может рассматриваться как ``массовая проблема'' в
смысле Медведева~\textup{\cite{MedvedevMass55}}: по данному
рациональному $\varepsilon>0$ указать номер точки из $D$,
приближающей $x$ с погрешностью не более $\varepsilon$. Легко
проверить, что эта проблема эквивалентна \textup(в смысле
Медведева\textup) проблеме перечисления всех базисных шаров,
содержащих точку $x$.
\end{remark}

\begin{remark}
Конструктивное метрическое пространство является частным случаем
более общего \textup(и часто полезного\textup) понятия
конструктивного топологического пространства.

\emph{Конструктивное топологическое пространство}
$\mathbf{X}=(X,\tau,\nu)$ состоит из топологического
пространства $X$, базиса открытых множеств $\tau$ и его
нумерации $\nu$: это значит, что $\tau=\{\nu(1),\nu(2),\ldots\}$
и что открытыми множествами в $X$ являются объединения множеств
из $\tau$.

Если $Z$ является непустым подмножеством эффективного
топологического пространства, то оно само становится эффективным
топологическим пространством: мы пересекаем все базисные
множества с $Z$, не меняя их нумерацию. В этом состоит важное
преимущество понятия конструктивного топологического
пространства \textup(по сравнению с конструктивными метрическими
пространствами\textup): там мы можем перенести метрику на
подмножество, но никакого естественного способа выделить в нјм
счјтное плотное множество с нумерацией не видно.

В этой статье мы, однако, не будем стараться обобщить наши
результаты на конструктивные топологические пространства
\textup(для чего нужно было бы сформулировать точно, какой класс
конструктивных топологических пространств мы хотим
рассматривать\textup), а в качестве компромисса ограничимся
подмножествами метрических пространств.
\end{remark}

\remm{Кажется, для большинства результатов достаточно таких
требований:

1. Пересечение двух базисных открытых множеств представляется в
виде объединения равномерно перечислимого семейства базисных
открытых множеств.

2. Существует некоторое перечислимое отношение ``быть
отделјнными'' (to be manifestly disjoint) для базисных открытых
множеств, при этом выполнено такое свойство: отделјнные
множества заведомо не пересекаются, и для любых двух точек
найдјтся два отделјнных базисных открытых множества, их
содержащие.

Это используется, например, при доказательстве того, что
эффективно компактное подмножество конструктивного
топологического пространства эффективно замкнуто.}

Определив структуру конструктивного топологического
пространства, мы можем теперь определить некоторые эффективные
варианты топологических понятий (для конструктивных метрических
пространств):

\begin{definition}
     \label{def:effectively-open-metric}
Открытое подмножество $U$ конструктивного метрического
пространства $X$ называется \emph{эффективно открытым}, если оно
является объединением перечислимого семейства базисных шаров.
Дополнения эффективно открытых множеств мы называем
\emph{эффективно замкнутыми}.

Пусть $A$~--- некоторое подмножество $X$. Будем говорить, что
множество $U \subseteq X$ \emph{открыто на $A$}, если найдјтся
эффективно открытое в $X$ множество $V$, для которого $U\cap
A=V\cap A$.
\end{definition}

Отметим, что в последнем определении $U$ не обязано быть частью
$A$, но реально играет роль лишь пересечение $U$ с $A$.

Теперь можно определить вычислимость в терминах эффективно
открытых множеств.

\begin{definition}[вычислимые функции]
     \label{def:computable-mappings}
Пусть $f\colon X\to Y$~--- отображение метрических пространств,
определјнное на всјм $X$. Функция $f$ \emph{непрерывна} тогда и
только тогда, когда прообраз любого базисного открытого шара
$B\subset Y$ является открытым подмножеством $X$\textup; назовјм
функцию $f$ \emph{вычислимой}, если этот прообраз является
эффективно открытым множеством \textup(равномерно по
$B$\textup).

Частичная функция $f$ из $X$ в $Y$, область определения которой
содержит некоторое подмножество $A\subset X$, \emph{вычислима на
$A$}, если прообраз любого базисного открытого шара $B$ является
эффективно открытым на $A$ \textup(равномерно по $B$\textup). Из
этого определения видно, что поведение функции вне $A$ на
вычислимость \textup(на $A$\textup) не влияет.

Частным случаем вычислимости функций можно считать вычислимость
точек: точку $x\in X$ будем называть вычислимой, если функция на
одноэлементном пространстве $f\colon\{0\}\to X$ с $f(0)=x$
вычислима.

Если функция $f$, определјнная в единственной точке $x_0\in X$ и
принимающая значение $y_0\in Y$, вычислима на своей области
определения, то мы говорим, что $y_0=f(x_0)$ является
$x_0$-\emph{вычислимым}. Если функция $f\colon Y\times Z\to Y$,
определјнная на $X\times\{z_0\}$, вычислима на своей области
определения, мы называем функцию $g\colon X\to Y$, отображающую
$x$ в $g(x)=f(x,z_0)$, $z_0$-\emph{вычислимой}, или
\emph{вычислимой относительно $z_0$.}
\end{definition}

Несложно проверить, что наше определение вычислимой точки
эквивалентно более привычным:

\begin{proposition}
     \label{prop:computable-points}
Следующие свойства точки $x$ конструктивного метрического
пространства $\mathbf{X}=(X,d,D,\alpha)$ равносильны:

\textup{(i)}~$x$ вычислима;

\textup{(ii)}~множество базисных шаров, содержащих $x$,
перечислимо;

\textup{(iii)}~существует вычислимая последовательность
$z_1,z_2,\ldots$ элементов $D$ \textup(заданных своими
$\alpha$-номерами\textup), для которой $d(x,z_n)\le 2^{-n}$ при
всех $n$.
\end{proposition}

Следующее предложение дајт более привычную переформулировку
определения вычислимости:

\begin{proposition}
     \label{prop:sequence-computability}
Пусть $f\colon X\to Y$~--- отображение метрических пространств.
Функция $f$ вычислима тогда и только тогда, когда существует
вычислимое преобразование \textup(задаваемое машиной с
оракулом\textup), которое переводит любую сильно фундаментальную
последовательность точек из $D_\mathbf{X}$ с пределом $x$ в
сильно фундаментальную последовательность точек из
$D_\mathbf{Y}$ с пределом $f(x)$.

Если $f$~--- частичная функция с областью определения $Z\subset
X$ и значениями в $Y$, то еј вычислимость на $Z$ равносильна
существованию вычислимого преобразования, которое применимо к
любой сильно фундаментальной последовательности с пределом $x\in
Z$ и преобразует еј в сильно фундаментальную последовательность
с пределом $f(x)$.
\end{proposition}

\begin{remark}
В некотором смысле $x_0$-вычислимость аналогична вычислимости с
оракулом \textup(и в канторовском пространстве совпадает с
ней\textup), но в общем случае надо иметь в виду, что
определение имеет более сложную природу: в этом определении
машина имеет доступ к некоторой последовательности
$s_1,s_2,\ldots$, сходящейся к $x_0$, но эта последовательность
не фиксирована.
\end{remark}

\begin{definition}[перечислимость снизу]
     \label{def:lower-semicomputable-metric}
Пусть дано конструктивное метрическое пространство
$\mathbf{X}=(X,d,D,\alpha)$. Функция $f\colon X\to
[-\infty,\infty]$ \emph{полунепрерывна снизу}, если множества
$\{x\mid f(x)>r\}$ открыты при любом рациональном $r$
\textup(отсюда следует, что они открыты при любом $r$, не
обязательно рациональном\textup). Она называется
\emph{перечислимой снизу}, если эти множества эффективно открыты
при любом рациональном $r$ \textup(равномерно по $r$\textup).
Частичная функция $f$ из $X$ в $Y$, определјнная \textup(по
крайней мере\textup) на некотором подмножестве $A$ пространства
$X$, называется \emph{перечислимой снизу на $A$}, если множества
$\{x\mid f(x)>r\}$ эффективно открыты на $A$ равномерно по $r$.

Аналогично определяется и перечислимость сверху; она равносильна
перечислимости снизу функции $-f$.
\end{definition}

Легко проверить, что функция $f\colon X\to\mathbb{R}$ вычислима
тогда и только тогда, когда она перечислима сверху и снизу.

Как и раньше, можно дать эквивалентное определение
перечислимости снизу с помощью базисных функций.

\begin{definition}[базисные функции в конструктивном метрическом пространстве]
     \label{basic-functions-metric}
Определим счјтное множество \emph{базисных} функций
$\mathcal{E}=\{e_1,e_2,\ldots\}$ в конструктивном метрическом
пространстве $\mathbf{X}=(X,d,D,\alpha)$ следующим образом. Для
каждой точки $d\in D$ и для любых положительных рациональных
чисел $r$ и $\varepsilon$ определим функцию
$g_{d,r,\varepsilon}$: еј значение в точке $x$ определяется
расстоянием от $x$ до $d$ и равно $1$, если это расстояние не
больше $r$, равно нулю, если расстояние не меньше
$r+\varepsilon$, и линейно меняется, когда расстояние пробегает
$[r,r+\varepsilon]$. См. рис.~\textup{\ref{fig:basic-function}}.

Затем мы рассматриваем все функции, получаемые из семейства
$g_{d,r,\varepsilon}$ замыканием относительно рациональных
линейных комбинаций и операций максимума и минимума. Множество
таких функций и обозначается $\mathcal{E}$; они имеют
естественную нумерацию.
\end{definition}

\begin{figure}[ht]
\begin{center}
   \includegraphics[scale=0.5]{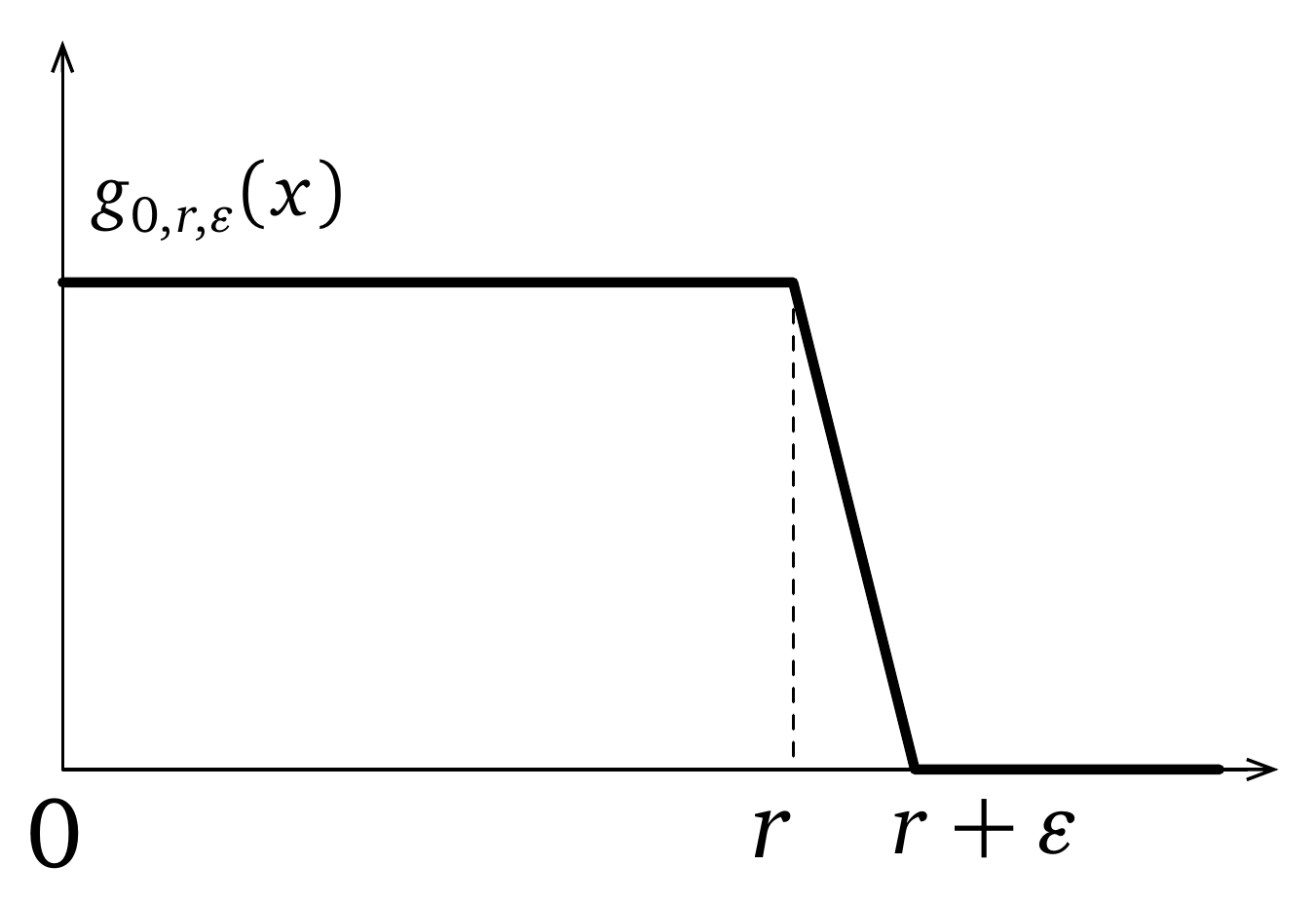}
\end{center}
\caption{Функции, из которых строятся базисные.}
\label{fig:basic-function}
\end{figure}

\begin{proposition}
     \label{prop:lower-semicomputable-metric}
Функция $f\colon X \to [0,+\infty]$, определјнная на
конструктивном метрическом пространстве, перечислима снизу тогда
и только тогда, когда она является поточечным пределом
неубывающей вычислимой последовательности базисных функций.
\end{proposition}

Если в этом предложении отказаться от требования вычислимости,
то вместо перечислимости снизу получится полунепрерывность
снизу.

\begin{definition}
     \label{def:lower-semicomputability-relative-to-a-point}
Аналогичным образом определяется перечислимость снизу
\emph{относительно точки $z_0$} некоторого метрического
пространства \textup(как перечислимость снизу на произведении
области определения и $\{z_0\}$\textup).
\end{definition}

Выбор метрики в конструктивном метрическом пространстве часто не
влияет на понятие вычислимости. В частности, мы можем не
различать эквивалентные метрики в смысле следующего определения
(конструктивный вариант равномерной непрерывности тождественного
отображения пространства в себя с другой метрикой):

\begin{definition}[равномерная непрерывность, эквивалентность]
Пусть $X$ и $Y$~--- конструктивные метрические пространства, а
$f\colon X\to Y$~--- вычислимая функция. Мы говорим, что $f$
\emph{равномерно непрерывна}, если для любого рационального
$\varepsilon>0$ существует рациональное $\delta>0$, для которого
$d_\mathbf{X}(x_1,x_2)\le\delta$ гарантирует
$d_\mathbf{Y}(f(x_1),f(x_2))\le\varepsilon$. Мы говорим, что $f$
\emph{эффективно равномерно непрерывна}, если по $\varepsilon$
можно эффективно найти соответствующее $\delta$.

Две метрики на одном и том же пространстве называются
\emph{\textup(эффективно\textup) эквивалентными}, если
тождественная функция, рассматриваемая как отображение
пространства с одной метрикой в пространство с другой метрикой,
\textup(эффективно\textup) равномерно непрерывна \textup(в обе
стороны\textup).
\end{definition}

Например, в $\mathbb{R}^2$ евклидова метрика и $L_1$-метрика
эффективно эквивалентны.

Понятие эффективной компактности
(определение~\ref{def:effective-compactness}) очевидно
переносится на произвольные конструктивные метрические
пространства. Есть эффективный аналог и у более слабого понятия
локальной компактности.

\begin{definition}[эффективная компактность и локальная
 компактность в метрических пространствах]
     \label{effective-compactness-metric}
Компактное подмножество $C$ эффективного метрического
пространства~$X$ называется \emph{эффективно компактным}, если
можно перечислять все конечные его покрытия базисными
множествами пространства $X$.

Подмножество $C$ метрического пространства называется
\emph{локально компактным}, если его можно покрыть шарами $B$, у
которых пересечения $\overline B\cap C$ компактны. \textup(Здесь
$\overline B$~--- замыкание шара $B$.\textup) Подмножество $C$
будем называть \emph{эффективно локально компактным}, если
существует вычислимая последовательность базовых шаров $B_k$,
покрывающих $C$, для которых $\overline B_k\cap C$ эффективно
компактны равномерно по~$k$.
\end{definition}

\begin{examples}\end{examples}

1. Конечное дискретное пространство компактно, а бесконечное~---
локально компактно.

2. Отрезок $[0,1]$ является эффективно компактным. Прямая
$\mathbb{R}$ является эффективно локально компактным
пространством.

3. Для конечного алфавита $X$ пространство бесконечных
последовательностей $X^\mathbb{N}$ является эффективно
компактным. (Для бесконечного алфавита это пространство не будет
локально компактным.)

4. Пусть $\alpha\in[0,1]$~--- перечислимое снизу действительное
число, не являющееся вычислимым. (Такие числа существуют;
например, число $\sum 2^{-\KP(x)}$ является таковым.) Рассмотрим
отрезок $[0,\alpha]$ как конструктивное метрическое
пространство: перечислимость $\alpha$ снизу позволяет
перенумеровать рациональные числа, меньшие $\alpha$. Это
конструктивное метрическое пространство будет компактным, но не
эффективно компактным.

Следующее предложение дајт полезный критерий эффективной
компактности в конструктивных метрических пространствах.

\begin{proposition}
     \label{epsilon-net}
\textup{(а)}~Компактное подмножество $C$ конструктивного
метрического пространства $X$ эффективно компактно тогда и
только тогда, когда по любому рациональному $\varepsilon>0$
можно эффективно указать конечное покрытие множества~$C$
базисными шарами радиуса $\varepsilon$.

\textup{(б)}~Пусть $C$~--- эффективно компактное подмножество
конструктивного метрического пространства. Тогда из любого
перечислимого семейства базисных шаров, покрывающих $C$, можно
эффективно выбрать конечное покрытие.
\end{proposition}

\begin{proof}
Пусть для каждого $\varepsilon$ мы умеем указывать конечное
покрытие $S_\varepsilon$ множества $C$ шарами радиуса
$\varepsilon$. Вместе с таким покрытием будем перечислять и все
покрытия \emph{гарантированно} большими шарами (это значит, что
для каждого шара $B(x,\varepsilon)$ из покрытия $S_\varepsilon$
в этом новом покрытии найдјтся шар $B(y,\sigma)$ c
$\sigma>\varepsilon+d(x,y)$). Компактность $C$ гарантирует, что
при этом (если сделать это для всех $\varepsilon$) будут
перечислены все покрытия $C$. (В самом деле, если есть какое-то
покрытие $S'$, которое в перечисление не попадјт, то для каждого
$\varepsilon$ есть один из шаров покрытия $S_\varepsilon$,
который не попадает гарантированно внутрь одного из шаров
покрытия $S'$. Применив компактность и взяв предельную точку
центров этих непопадающих шаров, получим противоречие.)

Остальные утверждения доказываются совсем просто.
\end{proof}

Следующее предложение обобщает
предложение~\ref{prop:closed-compact} и доказывается аналогичным
рассуждением.

\begin{proposition}
     \label{prop:compact-subsets-metric}
Всякое эффективно замкнутое подмножество эффективно компактного
множества эффективно компактно.
\end{proposition}

(Вспоминая определения, это утверждение можно переформулировать
так: пересечение эффективно компактного множества в пространстве
$X$ с эффективно замкнутым подмножеством пространства $X$
является эффективно компактным.)

Как и раньше, верно и обратное: всякое эффективно компактное
подмножество конструктивного метрического пространства
эффективно замкнуто. В самом деле, мы можем рассмотреть
всевозможные покрытия этого множества базисными шарами, а также
шары, заведомо (по соотношению расстояния и радиусов) не
пересекающиеся с множествами покрытия. Все такие шары вместе в
объединении дадут дополнение нашего эффективно компактного
множества.

Образ компактного пространства при непрерывном отображении
компактен. Это утверждение имеет эффективный аналог (с
аналогичным доказательством):

\begin{proposition}
     \label{prop:effectively-compact-image}
Пусть $C$~--- эффективно компактное подмножество конструктивного
метрического пространства $X$, а $f$~--- вычислимое на $C$
отображение $C$ в другое конструктивное метрическое
пространство. Тогда $f(C)$ эффективно компактно.
\end{proposition}

Утверждение о том, что полунепрерывная снизу функция на
компактном множестве достигает минимума, тоже имеет вычислимый
аналог (мы приведјм сразу параметрический вариант):

\begin{proposition}[минимум с параметром]
     \label{prop:parametrized-minimum}
Пусть $Y,Z$~--- конструктивные метрические пространства,
$f\colon Y\times Z\to[0,\infty]$~--- перечислимая снизу функция,
а $C$~--- эффективно компактное подмножество $Y\times Z$. Тогда
функция
     $$
g(y)=\inf_{\{z\mid (y,z)\in C\}} f(y,z)
     $$
\textup(где $\inf$ можно заменить на $\min$ в силу компактности,
только надо считать, что минимум пустого множества равен
$+\infty$\textup) перечислима снизу.

Вместо эффективной компактности $C$ достаточно предполагать, что
проекция $C_Y=\{y\mid \exists z\,(y,z)\in C\}$ эффективно
замкнута и покрыта вычислимой последовательностью шаров $B_k$,
для которых $(\overline{B}_k\times Z)\cap C$ эффективно
компактно равномерно по $k$.
\end{proposition}

Последнее условие заведомо выполняется, если $Y$ эффективно
локально компактно, а $Z$ эффективно компактно.

\begin{proof}
Для начала воспроизведјм классическое доказательство
полунепрерывности снизу. Нам надо проверить, что для любого $r$
множество $\{y\mid r<g(y)\}$ открыто. Это множество можно
представить в виде объединения, заметив, что условие $r<g(y)$
равносильно условию
      $$
(\exists r'>r)\, \forall z \left[(y,z)\in C \Rightarrow
                                    f(y,z)> r'\right],
      $$
и достаточно проверить, что множество
      $$
U = \left\{y\mid \forall z \left[(y,z)\in C \Rightarrow
                                    f(y,z)> r'\right]\right\}
     $$
открыто. Множество $U$ можно представить в виде
    $$
U=(Y\setminus C_Y)\cup\bigcup_k (B_k\cap U).
    $$
Множество $Y\setminus C_Y$ открыто по условию, поэтому
достаточно показать, что каждое из множеств $B_k\cap U$ открыто.
Положим $F_k=\overline B_k \times Z$; по предположению, $F_k\cap
C$ компактно. Условие $f(y,z)>r'$ в силу полунепрерывности
задајт некоторое открытое множество пар $V$, следовательно,
$F_k\cap C\setminus V$~--- замкнутое подмножество компактного
множества и потому компактно. Поэтому его проекция
      $$
\{ y \in\overline B_k \mid \exists z\, (y,z)\in F_k\cap C\setminus V\}
      $$
является непрерывным образом компактного множества и потому
компактна (тем самым замкнута), а еј дополнение в $B_k$ (то есть
$B_k\cap U$) открыто.

Теперь надо перевести это рассуждение на эффективный язык.
Прежде всего заметим, что можно ограничиться рациональными $r$ и
$r'$. Затем надо заметить, что множество $V$ эффективно открыто
(равномерно по $r'$), а множества $F_k\cap C\setminus V$
равномерно эффективно замкнуты (и будучи подмножествами
эффективно компактного множества, эффективно компактны). Поэтому
их проекции (как образы) эффективно компактны и эффективно
замкнуты, а дополнения эффективно открыты.
\end{proof}

Следствием этого предложения является такая лемма:

\begin{lemma}
     \label{lemma:global-minimum}
Пусть $X,Z,Z'$~--- метрические пространства, причјм $X$ локально
компактно, а $Z$ компактно. Пусть $f\colon Z\to Z'$~---
непрерывная функция, множество значений которой совпадает с
$Z'$, а $t\colon X\times Z\to[0,+\infty]$~--- полунепрерывная
снизу функция. Тогда функция $t_f\colon X\times Z'\to
[0,+\infty]$, определјнная формулой
     $$
t_f(x,z')=\inf_{\{z\mid f(z)=z'\}} t(x,z),
     $$
является полунепрерывной снизу.

Если $X$, $Z$, $Z'$~--- конструктивные метрические пространства,
$X$ эффективно локально компактно, $Z$ эффективно компактно,
функция $f$ вычислима, а функция $t$ перечислима снизу, то
функция $t_f$ перечислима снизу.
\end{lemma}

\begin{proof}
Приведјм рассуждение сразу для эффективного варианта. Применим
предыдущее предложение с $Y=X\times Z'$ и $$C=X\times\{ (f(z),z)
\mid z\in Z\}.$$ Тогда $$t_f(x,z')=\inf_{(x,z',z)\in C}
t(x,z).$$ Множество $Y$ эффективно локально компактно как
произведение эффективно локально компактного и эффективно
компактного множеств. Проекция $C$ на $Y$ совпадает со всем $Y$
и потому замкнута. Поэтому можно применить предыдущее
предложение.
\end{proof}

\subsection{\texorpdfstring
{Меры на конструктивном метрическом пространстве}
{Measures on constructive metric spaces}}
     \label{subsec:measures-on-metric-space}
В метрическом пространстве выделяются борелевские множества
(минимальная $\sigma$-алгебра, содержащая открытые множества), и
можно говорить о мерах на борелевских множествах. Они обладают
следующим свойством \emph{регулярности}:

\begin{proposition}[регулярность]
     \label{prop:regularity}
Пусть $P$~--- мера на полном сепарабельном метрическом
пространстве. Тогда любое измеримое множество $A$ можно
приблизить большими открытыми множествами:
     $$
P(A)=\inf_{G\supseteq A} P(G),
     $$
где $G$ открыто.
\end{proposition}

На мерах можно ввести расстояние:

\begin{definition}[расстояние Прохорова]
     \label{def:prokhorov}
Определим расстояние от точки $x$ до множества $A$ в метрическом
пространстве как $d(x,A)=\inf_{y\in A} d(x,y)$. Определим
$\varepsilon$-окрестность множества $A$ как
$A^\varepsilon=\{x\mid d(x,A)<\varepsilon\}$.

\emph{Расстояние Прохорова} $\rho(P,Q)$ между двумя мерами $P$ и
$Q$ определяется как точняя нижняя грань тех $\varepsilon>0$,
для которых $P(A)\le Q(A^\varepsilon)+\varepsilon$, а также
$Q(A)\le P(A^\varepsilon)+\varepsilon$, при всех
борелевских~$A$.
\end{definition}

Известно, что определјнная таким образом функция расстояния
действительно является метрикой; тем самым на множестве всех
вероятностных мер в метрическом пространстве возникает структура
метрического пространства. Есть и другие способы ввести
расстояние на пространстве мер, дающие эквивалентные метрики (в
том смысле, что тождественное отображение со сменой метрики
равномерно непрерывно в обе стороны).

\begin{definition}[пространство мер]
     \label{def:measure-space}
Пусть $\mathbf{X}$~--- конструктивное метрическое пространство.
Введјм в пространстве мер на нјм структуру конструктивного
метрического пространства $\mathbf{M}=\mathcal{M}(\mathbf{X})$.
В качестве счјтного всюду плотного множества $D_\mathbf{M}$
возьмјм множество мер, сосредоточенных на некотором конечном
подмножестве множества $D_\mathbf{X}$, и принимающих
рациональные значения. Такие меры имеют естественное описание
как конструктивные объекты, нужно указать номера элементов этого
конечного подмножества и их меры. Таким образом мы получаем
нумерацию $\alpha_\mathbf{M}$ точек в $D_\mathbf{M}$.

Вычислимые точки этого конструктивного метрического пространства
называют \emph{вычислимыми} мерами на $\mathbf{X}$.
\end{definition}

Определјнное таким образом понятие вычислимости является
обобщением введјнного нами ранее понятия вычислимой меры на
канторовском пространстве (а также естественного понятия
вычислимой меры на бэровском пространстве последовательностей
натуральных чисел).

Имеет место аналог предложения~\ref{prop:computable-integral}:
интеграл $\int f(\omega,P)\,dP(\omega)$ базисной функции $f$ по
мере $P$ является вычислимой функцией от $P$ и $f$.

Вот ещј один родственный результат:

\begin{proposition}
     \label{prop:integral-as-a-function}
Пусть $f$~--- ограниченная эффективно равномерно непрерывная
функция. Тогда еј интеграл по мере $P$, рассматриваемый как
функция от меры $P$, является эффективно равномерно непрерывной
функцией.
\end{proposition}

\begin{proof}
Без ограничения общности можно считать, что $f$ неотрицательна
(добавим константу). Пусть меры $P$ и $P'$ близки. Тогда
$P'(A)\le P(A_\varepsilon)+\varepsilon$, где $A_\varepsilon$
обозначает $\varepsilon$-окрестность множества $A$. Тогда
     $$
 \int f \,dP' \le \int f_\varepsilon \,dP + \varepsilon,
     $$
где $f_\varepsilon(x)$ есть точная верхняя грань $f$ на
$\varepsilon$-окрестности точки $x$. (Интеграл неотрицательной
функции $g$ определяется мерами множеств $G_t=\{x\mid g(x)\ge
t\}$, согласно теореме Фубини об изменении порядка
интегрирования эту меру как функцию от $t$ надо проинтегрировать
по $t$. При этом если $f(x)\ge t$, то $f_\varepsilon(x)\ge t$ в
$\varepsilon$-окрестности точки $x$.) Остајтся воспользоваться
эффективной равномерной непрерывностью функции $f$, чтобы
узнать, с какой точностью надо задавать меру, чтобы получить
данную точность в интеграле.
\end{proof}

С другой стороны, мера $P(B)$ базисного шара не обязана быть
вычислимой, но она перечислима снизу (равномерно по $B$). Как
показано в~\cite{HoyrupRojasRandomness09}, это свойство
(равномерная перечислимость снизу) также является критерием
вычислимости меры~$P$.

Известно, что для компактного сепарабельного метрического
пространства $X$ пространство мер на $X$ с описанной метрикой
также компактно. Это утверждение имеет конструктивный вариант,
который доказывается стандартным образом:

\begin{proposition}
     \label{prop:measure-space-compactness}
Если конструктивное метрическое пространство $\mathbf{X}$
эффективно компактно, то и пространство $\mathcal{M}(X)$
вероятностных мер на $X$ эффективно компактно.
\end{proposition}

Для пространства $\Omega$ это упоминалось в
предложении~\ref{prop:closed-compact}.

\begin{examples}\end{examples}

Бесконечное дискретное метрическое пространство ($\mathbb{N}$)
не компактно, и множество мер на нјм (которое можно отождествить
с множеством функций $P\colon\mathbb{N}\to[0,1]$, для которых
$\sum_i P(i)=1$) тоже не компактно. С другой стороны, если
перейти к полумерам, заменив условие на $\sum_i P(i)\le 1$, то в
естественной метрике получится компактное пространство. В самом
деле, в отличие от равенства неравенства достаточно проверять
для конечных сумм, и каждое неравенство задајт замкнутое
множество в компактном произведении $[0,1]^\mathbb{N}$, так что
пересечение таких множеств тоже будет компактным. Легко понять,
что оно будет и эффективно компактным.

Переход к полумере соответствует компактификации пространства
$\mathbb{N}$: недостающая часть суммы ряда переносится на
бесконечную точку.

\subsection{\texorpdfstring
{Случайность в метрическом пространстве}
{Randomness in metric spaces}}
      \label{subsec:uniform-test-metric}

В пространстве $\Omega$ мы определяли

\begin{itemize}
   \item
случайность относительно вычислимых мер (в смысле Мартин-Лјфа;
см. определение~\ref{computable-measure-test} на языке тестов);
  \item
равномерную случайность относительно произвольных мер (тест
является функцией последовательности и меры,
оп\-ре\-де\-ле\-ние~\ref{def:uniform-tests-cantor});
   \item
случайность относительно эффективно компактного класса мер,
определение~\ref{def:cantor-class-test};
   \item
слепую (безоракульную) случайность,
определение~\ref{def:blind-tests}.
\end{itemize}

Все эти понятия с небольшими изменениями переносятся на
произвольное конструктивное метрическое пространство. В этом
разделе мы обсудим эти обобщения и их свойства, а затем более
подробно рассмотрим случайность относительно ортогональных
классов мер.

Для вычислимых мер тест определяется как перечислимая снизу
функция на метрическом пространстве, интеграл от которой не
превосходит $1$. Среди таких тестов существует максимальный с
точностью до константы. Как и раньше, это доказывается с помощью
усечения: мы перечисляем все перечислимые снизу функции,
принудительно превращая их в тесты или почти тесты, и затем
складываем их с коэффициентами, образующими сходящийся ряд.

Это делается как и раньше, при этом мы рассматриваем
перечислимые снизу функции как возрастающие пределы базисных.
Важно, что интеграл от базисной функции по вычислимой мере
вычислим. Более того, интеграл от базисной функции по
произвольной мере вычислимо зависит от этой функции и от этой
меры. Для базисных точек в пространстве мер это ясно, далее надо
воспользоваться предложением~\ref{prop:integral-as-a-function}.

\medskip

Легко обобщить на случай конструктивных метрических пространств
и понятие равномерного теста (см.
определение~\ref{def:uniform-tests-cantor} для случая
канторовского пространства). Такой тест представляет собой
перечислимую снизу функцию двух аргументов $t(x,P)$, где $x$~---
точка нашего метрического пространства, а $P$~--- мера на этом
пространстве. Условие на интеграл, как и раньше, имеет вид $\int
t(x,P)\,dP(x)\le 1$.

Как и раньше, существует универсальный тест, и это можно
доказать с помощью техники усечения:

\begin{theorem}[усечение в метрических пространствах]
     \label{th:metric-trimming}
Пусть $u(x,P)$~--- перечислимая снизу функция, первый аргумент
которой~--- точка конструктивного метрического пространства, а
второй~--- мера на этом пространстве. Тогда существует
равномерный тест $t(x,P)$, для которого $u(x,Q)\le 2t(x,Q)$ при
всех $Q$, для которых функция $u_Q\colon x\mapsto u(x,Q)$
является тестом по мере $Q$, то есть $\int u(x,Q)\,dQ(x)\le 1$.
\end{theorem}

Доказательство повторяет рассуждение из
теоремы~\ref{th:trimming-general}, при этом используется тот
факт, что для базисной функции $b(x,P)$ на произведении
пространств интеграл $\int b(x,P)\,dP(x)$ является вычислимой
(непрерывной) функцией от $P$ (что доказывается аналогично
приведјнному нами рассуждению про вычислимость интеграла).

Универсальный тест мы будем обозначать $\at(x,P)$. Вообще-то для
каждого конструктивного метрического пространства он свой, но
обычно понятно, какое пространство имеется в виду, так что в
обозначение оно не входит.

Помимо существования универсального теста, из возможности
усечения следует возможность ``униформизации'' в следующем
смысле. Определим тест относительно меры $P$ (не обязательно
вычислимой) на конструктивном метрическом пространстве $X$:

\begin{definition}
     \label{def:metric-p-test}
Пусть $\mathbf{X}=(X,d,D,\alpha)$~--- конструктивное метрическое
пространство, а $P$~--- мера на нјм. Назовјм $P$-\emph{тестом
случайности} функцию $f\colon X\to [0,+\infty]$, если она
перечислима снизу относительно $P$ и если $\int f(x)\,dP(x)\le
1$.
\end{definition}

В этом определении фигурирует только мера $P$. Ясно, что из
равномерного теста можно получить тест относительно $P$.
Оказывается, что (с точностью до константы) так получаются все
тесты относительно $P$:

\begin{theorem}[униформизация]
     \label{th:metric-tests-uniformization}
Пусть $P$~--- некоторая мера на конструктивном метрическом
пространстве $X$, и дан некоторый $P$-тест $t_{P}(x)$. Тогда
существует равномерный тест $t'(\cdot,\cdot)$, для которого
$t_{P}(x)\le 2t'(x,P)$.
\end{theorem}

\begin{proof}
Из определения перечислимой снизу относительно $P$ функции сразу
же следует, что она является сужением некоторой перечислимой
снизу функции двух аргументов. Остајтся применить предыдущую
теорему к этому продолжению.
\end{proof}

Многие результаты (например, теорема~\ref{th:day-miller})
обобщаются на произвольные метрические пространства. Вот ещј
один пример такого обобщения (равномерный вариант так называемой
``случайности по Курцу'', Kurtz randomness):

\begin{proposition}
     \label{prop:kurtz-tests-uniform}
Пусть $X$~--- конструктивное метрическое пространство, а $S$~---
эффективно открытое подмножество пространства $X\times
\mathcal{M}(X)$. Если множество $S_P=\{x\mid (x,P)\in S\}$ имеет
$P$-меру $1$ для некоторой меры $P\in\mathcal{M}(X)$, то $S_P$
содержит все равномерно $P$-случайные точки.
\end{proposition}

\begin{proof}
Характеристическая функция $1_S(x,P)$ множества $S$, равная
единице внутри множества и нулю снаружи, перечислима снизу и
потому является пределом вычислимой возрастающей последовательности
базовых функций $g_n(x,P)$ с $0\le g_n(x,P)\le 1$.
Последовательность функций $G_n\colon P\mapsto \int
g_n(x,P)\,dP(x)$ представляет собой неубывающую
последовательность непрерывных вычислимых (равномерно по $n$)
функций. По теореме о монотонной сходимости значения $G_n(P)$
стремятся к единице для тех мер $P$, для которых $P(S_P)=1$.
Определим для каждой меры $P$ числа $n_k(P)$ как минимальные
значения $n$, для которых $G_n(P)>1-2^{-k}$. Эти числа
перечислимы сверху как функции от $P$ (в естественном смысле;
заметим, что для мер $P$, при которых $P(S_P)<1$, некоторые из
$n_k(P)$ бесконечны). Соответственно функции $1-g_{n_k(P)}(x,P)$
как функции от $x$ и $P$ (такую функцию мы считаем нулјм при
бесконечном $n_k(P)$, независимо от $x$) перечислимы снизу,
равномерно по $k$. Положим теперь
$t(x,P)=\sum_{k>0}(1-g_{n_k(P)}(x,P))$. Эта функция является
равномерным тестом, поскольку при данном $P$ еј $k$-е слагаемое
равно нулю, если $n_k(P)$ бесконечно, и имеет интеграл по мере
$P$ не больше $2^{-k}$ при конечном $n_k(P)$.

В условии теоремы говорится о мере $P$, для которой $P(S_P)=1$.
Тогда все $n_k(P)$ конечны, а $g_{n_k(P)}(x,P)=0$ для любого $x$
вне $S_P$. Поэтому все слагаемые в сумме, образующей тест, равны
единице, и $x$ не является равномерно $P$-случайной точкой.
Следовательно, $S_P$ включает в себя все равномерно
$P$-случайные точки.
\end{proof}

\subsection{\texorpdfstring
{Априорная вероятность с оракулом}
{A priori probability with an oracle}}
       \label{subsec:apriori-revisited}

В разделе~\ref{subsec:uniform-exact} мы определили априорную
вероятность с условием, роль которого играла мера на $\Omega$.
Теперь, введя понятие конструктивного метрического пространства,
мы можем заметить, что это определение естественно обобщается на
любое пространство $\mathbf{X}$: мы рассматриваем
неотрицательные перечислимые снизу функции
$m\colon\mathbb{N}\times X\to[0,+\infty]$, для которых $\sum_i
m(i,x)\le 1$ при любом $x\in X$.

Среди таких функций существует максимальная с точностью до
константы. Это доказывается методом усечения: мы не будем
повторять рассуждение подробно, отметим лишь, что перечислимую
снизу функцию $m(i,x)$ можно получить как сумму ряда из базисных
функций, каждая из которых отлична от нуля только для одного
$i$.

Максимальную из таких функций будем называть \emph{априорной
вероятностью с условием $x$} и обозначать $\am(i|x)$.

Мы считали первый аргумент натуральным числом, но это не
существенно: можно рассматривать слова (или любые другие
дискретные конструктивные объекты). Частным случаем этого
определения является определение априорной вероятности
относительно меры (раздел~\ref{subsec:uniform-exact}), а также
стандартные понятия априорной вероятности с оракулом (что
соответствует $\mathbf{X}=\Omega$) и условной априорной
вероятности (что соответствует $\mathbf{X}=\mathbb{N}$).

По аналогии с теоремой Дея--Миллера, можно выразить априорную
вероятность с условием в произвольном эффективно компактном
метрическом пространстве $\mathbf{X}$ через априорную
вероятность с оракулом.

\begin{proposition}
     \label{prop:day-miller-apriori}
Пусть $F\colon\Omega\to X$~--- вычислимое отображение, образом
которого является всј пространство~$X$. Тогда
     $$
\am(i|x)\eqm \min_{\{\pi\mid F(\pi)=x\}} \am(i|\pi).
     $$
\end{proposition}

\begin{proof}
Рассуждаем как в доказательстве теоремы~\ref{th:day-miller}.
Функция $(i,\pi)\mapsto \am(i|F(\pi))$ является перечислимой
снизу на $\mathbb{N}\times\Omega$, откуда получается
$\lem$-неравенство.

Чтобы получить обратное неравенство, мы пользуемся
леммой~\ref{lemma:global-minimum} и замечаем, что функция в правой
части корректно определена (минимум достигается) и перечислима
снизу.
\end{proof}

Заметим, что априорная вероятность $\am(i|\pi)$ в правой части
предложения~\ref{prop:day-miller-apriori} может быть выражена
через префиксную сложность (с оракулом). Для случая условий в
метрических пространствах не ясно, как определять префиксную
сложность с таким условием (можно говорить о функциях с
перечислимым относительно точки $x$ графиком, но неясно, как
искать среди них универсальную). Можно формально определить $\KP(i|x)$
как $\max_{\{\pi\mid F(\pi)=x\}} \KP(i|\pi)$, тогда
$\KP(i|x)\eqa-\log \am(i|x)$, но вряд ли это можно считать
удовлетворительным определением префиксной сложности (скажем,
обычные рассуждения, где используется самоограниченность
программы, при таком определении уже не применимы, хотя многие
результаты остаются верными; например, формулу
$\KP(i,j|x)\lea\KP(i|x)+\KP(j|x)$ можно доказать, не приписывая
друг к другу самоограниченные программы, а рассуждая с
вероятностями) --- честнее просто говорить о логарифме априорной
вероятности.

\begin{remark}
Аналогичным образом можно добавлять точки конструктивных
метрических пространств в качестве условий и в другие наши
определения. Например, можно рассматривать равномерные тесты на
$\Omega$ с условиями в произвольном конструктивном метрическом
пространстве $X$: это будут перечислимые снизу функции
$t(\omega,P,x)$, для которых $\int t(\omega,P,x)\,dP(\omega)\le
1$ при всех $x$. Можно также фиксировать вычислимую меру $P$,
например, равномерную, и определить тесты относительно этой меры
с условиями в $X$.
\end{remark}

\section{\texorpdfstring
{Классы ортогональных мер}
{Classes of orthogonal measures}}

Как и в случае пространства $\Omega$, для мер в произвольном
метрическом пространстве можно определить понятие эффективно
компактного класса и универсального теста случайности
относительно этого класса. Сохраняется (с тем же
доказательством) и формула для универсального теста относительно
класса (теорема~\ref{th:universal-class-test}).

Класс бернуллиевых мер (как и многие другие часто используемые
классы) обладает важным свойством: случайная по одной из мер
этого класса последовательность однозначно определяет меру, по
которой она случайна. Именно это обстоятельство по существу было
использовано в теореме~\ref{th:comparing-uniform-oblivious}. В
этом разделе мы рассмотрим тот же вопрос в более общей ситуации,
когда речь идјт о мерах на конструктивном метрическом
пространстве $\mathbf{X}=(X,d,D,\alpha)$.

Это обобщение включает в себя естественные примеры: конечные и
бесконечные марковские цепи, стационарные эргодические процессы
(см. ниже).

Сначала приведјм определение ортогональности мер, которое можно
считать классическим аналогом эффективной ортогональности
(определение~\ref{def:effective-orthogonal}).

\begin{definition}
     \label{def:orthogonal}
Пусть $P,Q$~--- две меры на $(X,\mathcal{A})$, где $X$~---
некоторое пространство, а $\mathcal{A}$~--- некоторая
$\sigma$-алгебра подмножеств $X$. Говорят, что меры $P$ и $Q$
\emph{ортогональны}, если пространство можно разбить на два
непересекающихся множества $U$ и $V$ из $\mathcal{A}$, для
которых $P(V)=Q(U)=0$.

Говорят, что класс $\mathcal{C}$ является \emph{ортогональным},
если существует измеримая функция $\varphi\colon
X\to\mathcal{C}$, для которой $P(\varphi^{-1}(P))=1$ для любой
меры $P\in\mathcal{C}$.
\end{definition}

Когда мы говорим об измеримости функции $\varphi$, имеется в
виду, что в метрическом пространстве $\mathcal{M}(X)$ определены
борелевские множества.

\begin{examples}\label{ex:orthogonal}\end{examples}

1. В ортогональном классе мер любые две (различные) меры $P$ и
$Q$ ортогональны. В самом деле, множества $\{P\}$ и $\{Q\}$
борелевские (замкнутые), и потому их прообразы измеримы (и,
очевидно, не пересекаются). Обратное утверждение неверно: класс
$\mathcal{C}$ попарно ортогональных мер не обязан быть
ортогональным, даже если он эффективно компактен. Пусть
$\lambda$~--- равномерная мера на отрезке $[0,1]$. Для каждой
точки $x\in[0,1]$ рассмотрим меру $\delta_x$, сосредоточенную в
точке~$x$. Тогда класс $\{\lambda\}\cup \{\delta_x\mid
x\in[0,1]\}$ эффективно компактен, и его элементы попарно
ортогональны. Однако он не является ортогональным классом:
условие ортогональности требует, чтобы мера $\delta_x$ была
значением $\varphi$ на $x$, а тогда $\varphi^{-1}(\lambda)$
будет пустым.

2. Пусть $P$ и $Q$~--- две меры. Если $\Rands(P)$ и $\Rands(Q)$
не пересекаются, то эти меры ортогональны (в качестве $U$ и $V$
можно взять, скажем, случайные и неслучайные по мере $P$
последовательности). Обратное, вообще говоря, неверно: меры
$\lambda$ и $\delta_x$ ортогональны, но множества случайных
последовательностей пересекаются, если $x$ взять случайным по
мере $\lambda$.

В качестве примера рассмотрим стационарные эргодические
процессы.

\begin{definition}
     \label{def:stationary-ergodic}
Рассмотрим на пространстве $\Omega$ бесконечных двоичных
последовательностей преобразование левого сдвига:
     $$
T\colon \omega(1)\omega(2)\ldots\mapsto \omega(2)\omega(3)\ldots.
     $$
Распределение вероятностей $P$ на $\Omega$ назовјм
\emph{стационарным}, если
    $$P(T^{-1}(A))=P(A)$$
для любого борелевского множества $A$. Легко проверить, что это
эквивалентно требованию
      $$P(x)=P(0x)+P(1x)$$
для всех слов $x$.

Борелевское множество $A\subset \Omega$ назовјм
\emph{инвариантным} относительно сдвига, если $T(A)\subset A$.
Например, множество последовательностей, в которых частота
единиц стремится к $1/2$, является инвариантным. Стационарное
распределение называется \emph{эргодическим}, если любое
инвариантное борелевское множество имеет меру $0$ или $1$.
\end{definition}

Вот пример стационарного процесса.

\begin{example}
     \label{ex:stationary-process}
Пусть $Z_1,Z_2,\ldots$~--- последовательность независимых
одинаково распределјнных случайных величин, принимающих значения
$0$ и $1$ с вероятностями соответственно $0.9$ и $0.1$.
Определим $X_0,X_1,X_2,\ldots$ так: $X_0$ принимает значения
$0,1,2$ с равными вероятностями и независима от всех $Z_i$,
$X_n=X_0+\sum_{i=1^n}Z_n \bmod 3$. Наконец, пусть $Y_n=0$ при
$X_n=0$ и $Y_n=1$ при $X_n\ne 0$. Легко видеть, что процесс
$Y_0,Y_1,\ldots$ является стационарными; можно доказать, что он
эргодический. Поскольку он является функцией марковской цепи
$X_n$, его называют \emph{скрытой марковской цепью}
\textup(hidden Markov chain\textup).
\end{example}

Следующее утверждение является следствием эргодической теоремы
Биркгофа о поточечной сходимости. Через $g_x$ будем обозначать
индикатор события $x\pref\omega$, то есть $g_x(\omega)$ равно
$1$ при $x\pref\omega$ и $0$ в противном случае.

\begin{proposition}
     \label{prop:ergodic}
Пусть $P$~--- стационарное распределение вероятностей на
пространстве $\Omega$.

\textup{(а)}~Для почти всех по мере $P$ последовательностей
$\omega$ последовательность
     $$
 A_{x,n}(\omega)=\tfrac{1}{n}
       (g_x(\omega)+g_x(T\omega)+\ldots+g_x(T^{n-1}\omega))
     $$
сходится.

\textup{(б)}~Для эргодического процесса этот предел равен
$P(x)$.
\end{proposition}

(Для неэргодических процессов предел может зависеть от
$\omega$.)

Общая теорема Биркгофа касается произвольных пространств и
сохраняющих меру преобразований, а в качестве $g_x$ можно взять
произвольную интегрируемую функцию, и гарантировать поточечную
сходимость (в эргодическом случае~--- к математическому
ожиданию).

Утверждение (б) показывает, что класс $\mathcal{C}$ эргодических
мер является ортогональным. В самом деле, будем называть
последовательность $\omega$ ``стабильной'', если для неј
существуют пределы из п.~(а) при любом $x$. Легко понять, что в
этом случае эти пределы задают некоторую меру $Q_\omega$.
Определим функцию $\varphi\colon\Omega\to\mathcal{C}$, положив
$\varphi(\omega)=Q_\omega$, если мера $Q_\omega$ является
эргодической, и выбрав в качестве значения произвольную
эргодическую меру, если $Q_\omega$ не является эргодической или
$\omega$ не является стабильной. Пункт (б) гарантирует, что
$P(\varphi^{-1}(P))=1$ для любой эргодической меры~$P$.

Тут используется, что множество стабильных последовательностей
борелевское. Заметим, что класс эргодических мер незамкнут, но это в
определении не предполагается.

\medskip

Мы видели (пример~\ref{ex:orthogonal}.2), что две меры могут
быть ортогональны, но иметь общие случайные последовательности.
Однако для вычислимых мер, как мы сейчас докажем, это
невозможно.

Будем говорить, что две меры эффективно ортогональны, если
классы равномерно случайных относительно них последовательностей
не пересекаются. (Это позволяет переформулировать
определение~\ref{def:effective-orthogonal} так: класс мер
эффективно ортогонален, если любые две меры в этом классе
эффективно ортогональны.)

\begin{theorem}
Две вычислимые меры на конструктивном метрическом пространстве
ортогональны тогда и только тогда, когда они эффективно
ортогональны.
\end{theorem}

\begin{proof}
Как мы уже говорили, в одну сторону это верно для любых мер.
Докажем обратное утверждение. Пусть даны две вычислимые
ортогональные меры $P,Q$; по определению ортогональности
существует измеримое множество $A$, для которого $P(A)=1$,
$Q(A)=0$. В силу регулярности
(предложение~\ref{prop:regularity}) найдјтся последовательность
открытых множеств $G_n$, содержащих $A$, для которых
$Q(G_n)<2^{-n}$. Поскольку $G_n$ содержит $A$, то $P(G_n)=1$.
Множество $G_n$ открыто, поэтому найдјтся конечное объединение
$H_n\subset G_n$ базисных шаров, для которого $P(H_n)>1-2^{-n}$;
для $H_n$ мера $Q$ тоже меньше $2^{-n}$. Перебором можно найти
вычислимую последовательность множеств $H_n$ с такими свойствами
($P$-мера большая, $Q$-мера маленькая).

Рассмотрим теперь $\limsup H_n$, то есть множество $\bigcap_m
U_m$, где $U_m=\bigcup_{n>m} H_n$. Согласно
предложению~\ref{prop:kurtz-tests-uniform}, каждое из множеств
$U_m$, а значит, и их пересечение, содержит все $P$-случайные
точки. С другой стороны, множества $U_n$ образуют тест в смысле
Мартин-Лјфа относительно меры $Q$, поэтому это пересечение не
содержит ни одной $Q$-случайной точки.
\end{proof}

Мы видели, что эргодические меры образуют ортогональный класс.
Более детальный анализ показывает, что они эффективно
ортогональны.

\begin{theorem}
     \label{th:effective-ergodic}
Эргодические меры на канторовском пространстве образуют
эффективно ортогональный класс.
\end{theorem}

\begin{proof}
В статье~\cite{VyuginErgodic98} приведено доказательство
эффективной эргодической теоремы, из которого следует, что

(а)~Равномерно случайные последовательности по стационарной мере
стабильны (в том смысле, что для них существует указанный выше
предел частот);

(б)~Для равномерно случайных по эргодической мере
последовательностей этот предел совпадает с мерой $P(x)$.

Опишем коротко схему доказательства. Для любых рациональных
чисел $0<\alpha<\beta$ рассмотрим перечислимую снизу функцию
$\omega\mapsto\sigma(\omega,\alpha,\beta)$, которая считает,
сколько раз величина $A_{x,n}(\omega)$ с ростом $n$ пересекла
промежуток $(\alpha,\beta)$ слева направо (была меньше $\alpha$
и стала больше $\beta$). Затем можно доказать, что
$(1+\alpha^{-1})(\beta-\alpha)\int
\sigma(\omega,\alpha,\beta)\,dP(\omega)\le1$, то есть функция
$(1+\alpha^{-1})(\beta-\alpha)\sigma(\alpha,\beta)$ является
ограниченным в среднем тестом. Следовательно, для равномерно
случайных (и даже для безоракульно случайных)
последовательностей число таких пересечений интервала конечно.

Чтобы доказать (б), если (а) уже гарантировано, достаточно для
каждого $x$ установить, что
     $$
\liminf_n A_{x,n}(\omega) \le P(x) \le \limsup_n A_{x,n}(\omega)
     $$
для всех случайных $\omega$. Рассмотрим, например, первое
неравенство (второе аналогично). Достаточно показать для любых
$k$ и $m$, что
     $$\inf_{n\ge m}A_{x,n}(\omega)\le P(x)+2^{-k}$$
для случайной $\omega$. Множество
     $$
S_{x,k,m}=\{(\omega,P)\mid (\exists n\ge m) A_{x,i}(\omega) <
           P(x)+2^{-k}\}
     $$
является эффективно открытым, и по теореме Биркгофа множество
$S_{x,k,m}(P)=\{x\mid (x,P)\in S_{x,k,m}\}$ имеет $P$-меру $1$,
если мера $P$ эргодическая.
Предложение~\ref{prop:kurtz-tests-uniform} гарантирует, что
множество $S_{x,k,m}(P)$ содержит все равномерно $P$-случайные
точки.

Другой подход к доказательству состоит в том, чтобы установить
сходимость частот к $P(x)$ для случайных по Мартин-Лјфу
последовательностей относительно вычислимых эргодических мер
(при этом доказательство должно выдерживать релятивизацию), не
используя явного теста, как это сделано
в~\cite{BienvDayHoyrMezhShen10}. После этого можно сослаться на
теорему~\ref{th:day-miller} и получить сходимость для равномерно
случайных последовательностей.
\end{proof}

Для работы с ортогональными классами мер полезно понятие
сепаратора. В этом определении, говоря об измеримости функций,
мы имеем в виду их измеримость по Борелю (прообраз борелевского
множества является борелевским); меры мы тоже считаем
определјнными на борелевских множествах.

\begin{definition}[функция-сепаратор]
     \label{def:separator-function}
Пусть $\mathcal{C}$~--- класс мер на конструктивном метрическом
пространстве $\mathbf{X}$. Измеримую функцию
      $$s\colon X\times \mathcal{M}(\mathbf{X})\to[0,+\infty]$$
назовјм \emph{сепаратором} для класса $\mathcal{C}$, если
$\int s(x,P)\,dP(x)\le 1$ для любой меры $P$,
а для любых двух различных мер $P,Q\in \mathcal{C}$ и для любой
точки $x$ хотя бы одно из значений $s(x,P)$ и $s(x,Q)$
бесконечно.

Сепаратор называется \emph{тестом-сепаратором}, если он
перечислим снизу как функция $x$ и~$P$.
\end{definition}

В определении сепаратора мы требуем $\int s(x,P)\,dP(x)\le 1$
для всех мер (а не только для мер из класса $P$), но это не
очень существенно, поскольку к тест-сепаратору можно применить
усечение.

Следующая теорема связывает понятие ортогонального класса мер с
сепараторами, а также устанавливает, что для случая эффективно
ортогонального класса каждая мера может быть восстановлена по
случайной (по этой мере) последовательности.

\begin{theorem}
Пусть $\mathcal{C}$~--- класс мер на конструктивном метрическом
пространстве.

\textup{(а)}~Если борелевский класс мер $\mathcal{C}$
ортогонален, то для него существует сепаратор.

\textup{(б)}~Класс $\mathcal{C}$ является эффективно
ортогональным тогда и только тогда, когда существует
тест-сепаратор.
\end{theorem}

(Что касается обратного к (а) утверждения, то авторы не знают,
верно ли оно.)

\begin{proof}
Докажем сначала (а). Пусть $\varphi(x)$~--- функция, которая для
каждого $x\in X$ указывает меру в соответствии с определением
ортогональности. По предположению эта функция борелевская,
поэтому (см.~\cite{KuratowskiTop}) еј график является
борелевским множеством. Положим $s(x,P)=1$ при $P\notin
\mathcal{C}$, а также при $P\in\mathcal{C}$ и $\varphi(x)=P$, и
положим $s(x,P)=\infty$ при $P\in\mathcal{C}$ и $\varphi(x)\ne
P$.

Докажем теперь утверждение (б). Если класс $\mathcal{C}$
эффективно ортогонален, то универсальный равномерный тест и
будет тест-сепаратором для класса $\mathcal{C}$. С другой
стороны, пусть имеется тест-сепаратор для класса $\mathcal{C}$.
Пусть $P$ и $Q$~--- две различные меры для класса $\mathcal{C}$,
и точка $x$ равномерно случайна по обеим мерам. Поскольку $s$
является равномерным тестом случайности, то $s(x,P)$ и $s(x,Q)$
конечны, что противоречит определению теста-сепаратора.
\end{proof}

\medskip

Следующий результат менее ожидаем; он показывает, что для случая
эффективно компактного класса мер из существования
полунепрерывного снизу сепаратора следует существование и
перечислимого снизу сепаратора (то есть теста-сепаратора).

\begin{theorem}
     \label{th:semicontinuous-test-separator}
Пусть для эффективно компактного класса мер существует
полунепрерывный снизу сепаратор $s(x,P)$. Тогда этот класс
эффективно ортогонален.
\end{theorem}

\begin{proof}
Пусть $\mathcal{C}$~--- эффективно компактный класс мер на
конструктивном метрическом пространстве. Мы должны показать, что
в предположениях теоремы для любых двух различных мер
$P_1,P_2\in\mathcal{C}$ множества случайных последовательностей
не пересекаются:
     $$
\Rands(P_1) \cap \Rands(P_2) = \varnothing.
     $$
Возьмјм в пространстве мер два непересекающихся базисных
замкнутых шара $B_1$ и $B_2$, содержащих меры $P_1$ и $P_2$, и
рассмотрим классы мер $\mathcal{C}_1=\mathcal{C}\cap B_1$ и
$\mathcal{C}_2=\mathcal{C}\cap B_2$. Это непересекающиеся
эффективно компактные классы мер, содержащие $P_1$ и $P_2$.
Рассмотрим теперь функции
     $$
t_i(x)=\inf_{P\in\mathcal{C}_i} s(x,P).
     $$
Для любого $x$ хотя бы одно из значений $t_1(x)$ и $t_2(x)$
бесконечно. Функции $t_1$ и $t_2$ полунепрерывны снизу (первая
часть доказательства
предложения~\ref{prop:parametrized-minimum}) и являются
$\mathcal{C}_1$- и $\mathcal{C}_2$-тестами.

Теперь мы можем применить рассуждение, аналогичное
доказательству предложения~\ref{prop:kurtz-tests-uniform}. Пусть
$k>1$~--- целое число. Рассмотрим открытое множество
$S_k=\{x\mid t_1(x)>2^k\}$. Поскольку $t_1$ является
$\mathcal{C}_1$-тестом, то $P(S_k)< 2^{-k}$ для всех
$P\in\mathcal{C}_1$. С другой стороны, поскольку для каждого $x$
одно из значений $t_1(x)$ и $t_2(x)$ бесконечно, то $P(S_k)=1$
для всех $P\in\mathcal{C}_2$. Характеристическая функция
$1_{S_k}$ множества $S_k$ полунепрерывна снизу, и потому может
быть представлена как поточечный предел неубывающей
последовательности (не обязательно вычислимой!) базисных функций
$g_{k,n}$. Рассуждая как в
предложении~\ref{prop:kurtz-tests-uniform}, мы заключаем, что
найдјтся $n=n_k(P)$, при котором $\int g_{k,n}\,dP>1-2^{-k}$ для
любого $P\in\mathcal{C}_2$. Эффективная компактность класса
$\mathcal{C}_2$ позволяет выбрать $n$ общим для всех
$P\in\mathcal{C}_2$.

Зафиксируем достигнутое: для всякого $k$ найдјтся базисная
функция $h_k$, для которой
     \begin{align*}
\int h_k \, dP &< 2^{-k} \text{ при всех $P\in\mathcal{C}_1$};\\
\int h_k \, dP &> 1 - 2^{-k} \text{ при всех $P\in\mathcal{C}_2$}.
     \end{align*}
Такую функцию можно найти эффективно по $k$ перебором.

Теперь можно построить перечислимую снизу функцию
     $$t_1'(x)=\sum_k h_k(x).$$
Она является тестом для класса $\mathcal{C}_1$. Функция
$t_2'(x)=\sum_k (1-h_k(x))$ по аналогичным причинам будет тестом
для класса $\mathcal{C}_2$. Эти тесты должны быть конечны для
случайных по мерам $P_1$ и $P_2$ последовательностей, а
одновременно для обоих тестов это быть не может.
\end{proof}

Смысл введјнного нами понятия теста-сепаратора можно пояснить
следующим образом. Универсальный тест $\at(\omega,P)$ в силу
эффективной ортогональности позволяет разделить
последовательности, случайные по разным мерам из класса
$\mathcal{C}$: глядя на последовательность $\omega$, равномерно
случайную по одной из мер этого класса (=случайную относительно
класса $\mathcal{C}$), мы ищем $P\in\mathcal{C}$, для которого
$\at(\omega,P)$ конечно. Такая мера $P$ в классе $\mathcal{C}$
единственна (согласно определению эффективной ортогональности).

Последнее, однако, может выполняться и для неуниверсального
теста, и такие тесты мы назвали тестами-сепараторами.
Неуниверсальный тест менее требователен к идее случайности, и
описывает еј, так сказать, в первом приближении: может
оказаться, что та последовательность, которую он считает
случайной (на которой значение $t(\omega,P)$ конечно), более
серьјзный тест уже отбракует. (Обратное невозможно, так как
универсальный тест максимален.) Важно только, чтобы уже эта
предварительная грубая отбраковка позволяла разделить меры из
класса $\mathcal{C}$, то есть чтобы ни одна последовательность
не казалась ``в первом приближении случайной'' сразу по двум
мерам.

\begin{definition}
Для данного тест-сепаратора $s(x,P)$ будем называть элемент $x$
\emph{случайным в первом приближении относительно $P$}, если
значение этого тест-сепаратора конечно: $s(x,P)<\infty$.
\end{definition}

В качестве примера рассмотрим класс бернуллиевых мер. В качестве
такого ``теста в первом приближении'' можно вспомнить слова фон
Мизеса, который самым первым свойством случайной
последовательности (\emph{коллектива}, как он говорил) называл
устойчивость частот. Свойство устойчивости частот (усиленный
закон больших чисел в современной терминологии) состоит в том,
что $S_n(\omega)/n\to p$. Здесь $S_n(\omega)$~--- количество
единиц в начальном отрезке последовательности $\omega$ длины
$n$, а $p$~--- параметр бернуллиевой меры $B_p$.

Можно пытаться использовать это свойство для построения
сепаратора разными способами, вот несколько возможных
требований:

(1)~$S_n(\omega)/n\to p$ с некоторой фиксированной скоростью
сходимости.

(2)~$S_n(\omega)/n\to p$ без указания конкретной скорости
сходимости.

(3)~Можно вспомнить доказательство
теоремы~\ref{th:effective-ergodic} для класса всех эргодических
стационарных мер на $\Omega$, и получить тест, гарантирующий
сходимость всех частот $A_{x,n}(\omega)$ к соответствующим
вероятностям $P(x)$.

Наиболее простое и естественное (с математической точки зрения)
требование (2) не записывается в виде перечислимого снизу теста,
но чтобы поправить дело, можно перейти к (1) и фиксировать
скорость сходимости. Вот один из возможных вариантов. (Для
простоты мы будем использовать только неравенство Чебышјва, и
получим сходимость частот не на всех отрезках, а только по
степеням двойки. Более аккуратная оценка позволила бы получить
сходимость частот по всем начальным отрезкам.)

Неравенство Чебышјва гарантирует, что
    $$
B_p(\{x\in\mathbb{B}^n\colon |S_n(x)-np|>
        \lambda n^{1/2}(p(1-p))^{1/2}\})\le\lambda^{-2}.
    $$
Здесь $S_n(x)$~--- частота единиц в слове длины~$n$. Поскольку
$p(1-p)\le 1/4$, отсюда следует, что
    $$
B_p(\{x\in\mathbb{B}^n\colon |S_n(x)-np|>
        \lambda n^{1/2}/2\})\le\lambda^{-2}.
    $$
Положив, скажем, $\lambda=n^{0.1}$ (и опуская множитель $1/2$,
что лишь ослабляет утверждение), получаем
    $$
B_p(\{x\in\mathbb{B}^n\colon |S_n(x)-np|> n^{0.6}\})\le n^{-0.2}.
    $$
Чтобы ряд сходился, ограничимся лишь членами вида $n=2^{k}$:
    $$
B_p(\{x\in\mathbb{B}^{2^k}\colon |S_{2^k}(x)-2^kp|>
       2^{0.6k}\})\le 2^{-0.2k}.
    $$
Теперь для бесконечной последовательности $\omega$ и для $p\in
[0,1]$ положим
     $$
g(\omega,B_p)=\sup \{ k\colon | S_{2^k}(\omega)-2^kp|>2^{0.6k}\}.
     $$
При этом
     $$
\int g(\omega,B_p)\,dB_p(\omega) \le \sum_k k\cdot 2^{-0.2k}
         = c < \infty
     $$
и поделив на $c$, получаем тест. Это тест-сепаратор, так
как $g(\omega,B_p)<\infty$ влечјт сходимость последовательности
$2^{-k} S_{2^k}(\omega)$ к $p$ и для двух разных $p$ такого
случиться не может.

Теорема~\ref{th:comparing-uniform-oblivious} обобщается на
случай произвольного метрического пространства (доказательство
остајтся практически тем же, надо использовать базисные шары
вместо начальных отрезков): для всякого эффективно компактного
эффективно ортогонального класса мер и всякой меры в этом классе
слепая (безоракульная) случайность равносильна равномерной
случайности. В связи с этим естественно спросить, нельзя ли
произвольную эргодическую меру поместить в некоторый эффективно
компактный класс. Оказывается, что нет.

\begin{theorem}
     \label{th:strange-ergodic}
Рассмотрим стационарные \textup(инвариантные относительно
сдвига\textup) меры на пространстве~$\Omega$. Среди них
существует эргодическая мера, не содержащаяся ни в каком
эффективно компактном классе стационарных эргодических мер.
\end{theorem}

Прежде чем доказывать эту теорему, приведјм некоторые
вспомогательные утверждения.

\begin{proposition}
     \label{prop:ergodic-nonergodic}
Как эргодические, так и неэргодические меры плотны в классе
стационарных мер.
\end{proposition}

\begin{proof}
Для начала покажем, что всякую стационарную меру можно
приблизить эргодической. Без ограничения общности можно считать,
что вероятность $P(x)$ появления любого слова $x$ по этой мере
строго положительна. (Если нет, можно подмешать немного
равномерной меры.) Фиксируем какое-либо $n$ и рассмотрим
значения меры $P(x)$ на строках длины не более $n$. Существует
марковский процесс с таким же распределением вероятностей, в
котором вероятность следующего бита определяется $n-1$
предыдущими битами: для любого $x\in\mathbb{B}^{n-2}$ и любых
битов $b,b'$ вероятность перехода от $bx$ к $xb'$ равна
$P(bxb')/P(bx)$. В этом процессе все вероятности перехода
положительны, и поэтому он является эргодическим. С ростом $n$
он стремится к исходной стационарной мере.

С другой стороны, всякую стационарную меру $P$ можно приблизить
и неэргодической мерой. Очевидно, достаточно рассмотреть случай,
когда сама мера $P$ эргодическая. Тогда, согласно эргодической
теореме, можно найти последовательность, в которой предельные
частоты всех подслов соответствуют мере. (Почти все --- в смысле
этой меры --- последовательности таковы.) Взяв длинный кусок
этой последовательности и зациклив его, можно найти
периодическую последовательность, в которой частоты слов длины
не больше $n$ отличаются от меры $P$ не более чем на
$\varepsilon$ (для любых данных $n$ и $\varepsilon>0$). (При
зацикливании образуются новые слова на месте склейки, но при
большой длине это не играет роли.) Теперь можно рассмотреть
меру, соответствующую случайным сдвигам этой последовательности
(она сосредоточена на конечном множестве последовательностей ---
их столько, каков минимальный период). Эта мера не эргодична, но
близка к $P$.
\end{proof}

Нам понадобится ещј одно утверждение:

\begin{proposition}
     \label{prop:ergodic-g-delta}
Множество эргодических мер образует $G_\delta$-подмножество в
метрическом пространстве всех мер на $\Omega$.
\end{proposition}

\begin{proof}
Мы можем ограничиться (замкнутым) множеством стационарных мер.
Рассмотрим функцию $A_{x,n}$ на $\Omega$, положив значение
$A_{x,n}(\omega)$ равным доле вхождений слова $x$ среди первых
$n$ возможных позиций (мы прикладываем $x$ к $\omega$, начиная с
первой, второй,\ldots, $n$-ой позиции и смотрим долю
совпадений). Эргодическая теорема гарантирует, что при каждом
$x$ последовательность функций $A_{x,1},A_{x,2},\ldots$ сходится
в смысле $L_1$ (на самом деле имеет место даже и сходимость
почти всюду). При этом стационарная мера $P$ будет эргодической
тогда и только тогда, когда пределом этой последовательности
будет константа $P(x)$.

В силу существования предела для стационарных мер достаточно
проверять, что константа $P(x)$ является предельной точкой. Для
любых $x$, $N$ и $\varepsilon$ множество $S_{x,N,\varepsilon}$
тех $P$, для которых существует $n\ge N$ с
     $$
\int |A_{x,n}(\omega)-P(x)|\,dP(\omega)<\varepsilon,
     $$
является открытым, а пересечение этих множеств по всем
$x,N,\varepsilon$ и дајт указанный выше критерий эргодичности
стационарной меры.
\end{proof}

Теперь мы можем доказать теорему~\ref{th:strange-ergodic}.

\begin{proof}
Объединение всех эффективно компактных классов эргодических мер
является $F_\sigma$-множеством. Предположим, что оно включает в
себя все эргодические меры. Тогда множество неэргодических мер
является $G_\delta$-множеством, которое плотно в силу
предложения~\ref{prop:ergodic-nonergodic}. С другой стороны, как
мы видели в предложениях~\ref{prop:ergodic-g-delta}
и~\ref{prop:ergodic-nonergodic}, множество неэргодических мер
также является плотным $G_\delta$-множеством. Но пересечение
этих двух множеств пусто, что противоречит теореме Бэра о
категории. Теорема~\ref{th:strange-ergodic} доказана.
\end{proof}

Тем не менее вопрос, с которого мы начали, остајтся открытым:

\begin{question}
Существует ли эргодическая мера, для которой понятия равномерной
и слепой \textup(безоракульной\textup) случайности не совпадают?
\end{question}

\bigskip

Возвращаясь к произвольным эффективно компактным эффективно
ортогональным классам, мы можем связать универсальные тесты с
тестами для класса (см. теорему ~\ref{th:universal-class-test})
и тестами-сепараторами.

\begin{theorem}
     \label{th:universal-class-separator-tests}
Пусть $\mathcal{C}$~--- эффективно компактный класс эффективно
ортогональных мер. Пусть $\at_\mathcal{C}(x)$~--- универсальный
тест случайности для этого класса, а $s(x,P)$~--- некоторый
тест-сепаратор для $\mathcal{C}$. Тогда универсальный
равномерный тест $\at(x,P)$ для мер этого класса можно выразить
так:
     $$
\at(x,P)\eqm  \max (\at_\mathcal{C}(x),s(x,P))
     $$
для всех $P\in\mathcal{C}$ и для всех $x$.
\end{theorem}

\begin{proof}
Заметим прежде всего, что $\at_\mathcal{C}(x)$ и $s(x,P)$ не
превосходят универсального равномерного теста $\at(x,P)$ (что
следует из его универсальности).

С другой стороны, покажем, что если $\at_\mathcal{C}(x)$ и
$s(x,P)$ конечны, то $\at(x,P)$ не превосходит наибольшего из
них (с точностью до константы). Конечность первого теста
гарантирует, что $\min_{Q\in\mathcal{C}}\at(x,Q)$ конечен: этот
минимум равен $\at_\mathcal{C}(x)$ с точностью до константы.
Если этот минимум достигался бы на какой-то мере $Q\ne P$, то
оба значения $s(x,Q)$ и $s(x,P)$ были бы конечны, что
противоречит определению сепаратора. (Заметим, что мы доказали
чуть более сильное утверждение, чем обещали: вместо
``наибольшего из них'' можно написать ``первого из них, если
второй конечен''.)
\end{proof}

Утверждение этой теоремы позволяет разбить проверку случайности
по некоторой мере $P$ из класса $\mathcal{C}$ на две части
(указывает две возможные причины неслучайности). Во-первых, мы
должны убедиться, что $x$ случайно относительно класса
$\mathcal{C}$. Например, в случае меры $B_p$ из класса
$\mathcal{B}$ бернуллиевых мер мы вначале должны убедиться, что
$\at_\mathcal{B}(\omega)$ конечно. После этого мы знаем, что
наша последовательность бернуллиева и достаточно какой-то
простой проверки типа закона больших чисел, чтобы выяснить, по
какой именно бернуллиевой мере она случайна: $B_p$ или какой-то
другой. Вторую часть можно рассматривать как аналогичную
параметрическому тестированию в статистике.

С количественной точки зрения (если нас интересует не просто
случайность и неслучайность, но и значение теста) вторая часть
тестирования не важна: про сепаратор нам надо знать лишь,
конечно или бесконечно его значение.

\section{\texorpdfstring
{Равномерные тесты: слишком сильные\\ требования?}
{Are uniform tests too strong?}}
     \label{sec:uniform-too-strong}

\subsection{\texorpdfstring
{Монотонность и квази-выпуклость}
{Monotonicity and quasi-convexity}}
     \label{subsec:monotone-quasi-convex}

С интуитивной точки зрения равномерные тесты случайности (в
наиболее общей форме см. определение~\ref{def:metric-p-test})
могут казаться слишком сильными, если мера $P$ не вычислима. И
действительно, они не обладают некоторыми интуитивно
желательными свойствами, которыми обладает понятие случайности
относительно вычислимых мер (в смысле Мартин-Лјфа или
равномерное~--- для вычислимых мер это одно и то же). Одним из
таких свойств является монотонность: большая (с точностью до
константы) мера имеет больше случайных объектов.

\begin{proposition}
     \label{prop:computable-monotonicity}
Пусть $P$ и $Q$~--- две вычислимые меры, а $\lambda>0$~---
рациональное число, причјм $\lambda P(A)\le Q(A)$ для всех $A$.
Тогда
     $$\am(\lambda)\cdot \lambda\cdot \at(x,Q)\lem \at(x,P).$$
Здесь $\am(\cdot)$~--- дискретная априорная вероятность
рационального числа $\lambda$; константа в $\lem$ определяется
сложностью пары программ, задающих $P$ и $Q$.
\end{proposition}

\begin{proof}
Функция $\lambda \at(\cdot,Q)$ является $P$-тестом, так как
    $$
\int \lambda \at(x,Q)\,dP(x)\le\int\at(x,Q)\,dQ(x)\le 1.
    $$
Используя усечение, мы заключаем, что сумма
     $$
\sum_{\{\lambda\mid \lambda\cdot\int\at(x,Q)\,dP(x)<2\}}
             \am(\lambda)\cdot\lambda\cdot \at(x,Q)
     $$
с точностью до константы является $P$-тестом, и потому не
превосходит $\at(x,P)$. Тем более это верно и для всех членов
этой суммы. (Упомянутая константа обратно пропорциональна
$\am(P,Q)$.)
\end{proof}

Интуитивная мотивировка свойства монотонности такова: если есть
два устройства с внутренними датчиками случайности, генерирующие
объекты с выходным распределением $P$ и $Q$, и $\lambda P \le
Q$, то можно представить себе, что с вероятностью $\lambda$
второе устройство моделирует первое, а в остальных случаях
делает что-то свој. Тогда всякий объект, который с интуитивной
точки зрения правдоподобен на выходе первого устройства, должен
считаться правдоподобным и на выходе второго: вдруг оно-таки
промоделировало первое? (Численное значение дефекта, конечно,
может быть немного больше, так как мы дополнительно должны
поверить, что произошло событие с вероятностью $\lambda$.)

Для равномерной случайности, увы, это свойство не выполнено:
если мера $Q$ больше, но вычислительно сложнее, то тесты
случайности относительно $Q$ могут использовать эту
дополнительную информацию, чтобы сделать неслучайными некоторые
объекты, которые относительно $P$ были случайными (см.
доказательство теоремы~\ref{th:oblivious-not-uniform}). Именно в
этом причина отличия равномерной случайности от слепой
(безоракульной), для которой аналогичное свойство выполнено по
очевидным причинам.

Другая ситуация, в которой у нас есть некоторая интуиция
случайности~--- это смесь (выпуклая комбинация) мер. Представим
себе два устройства с выходными мерами $P$ и $Q$, и внешнюю
оболочку, которая с какими-то вероятностями $\lambda$ и
$1-\lambda$ запускает одно из них. В целом мы получаем систему,
выход которой распределјн по мере $\lambda P + (1-\lambda) Q$.
Про какие объекты мы готовы поверить, что они случайно получены
в результате такого эксперимента? ясно, что это должны быть
случайные по мере $P$ объекты, а также случайные по мере $Q$
объекты (при этом если коэффициент мал, то должно добавляться
дополнительное удивление, но конечное). И других объектов быть
не должно. Количественное уточнение этого результата (который в
одну сторону следует из монотонности) дајтся в следующем
предложении.

\begin{proposition}
     \label{prop:convexity}
Пусть $P$ и $Q$~--- две вычислимые меры.

\textup{(а)}~$\am(\lambda)\cdot\at(x,\lambda P+(1-\lambda)Q)\lem
       \max(\at(x,P),\at(x,Q))$\textup;

\textup{(б)}~$\at(x,\lambda P+(1-\lambda)Q)\gem
       \min(\at(x,P),\at(x,Q))$.

В первом утверждении $\lambda$~--- рациональное число в $(0,1)$,
а $\am(\lambda)$~--- его дискретная априорная вероятность. Во
втором утверждении $\lambda$ может быть любым. Константы в
$\lem$ не зависит от $\lambda$ \textup(определяются сложностью
пары мер $P$ и $Q$\textup).
\end{proposition}

Первое утверждение можно назвать \emph{квази-выпуклостью} тестов
случайности (с точностью до константы). Для тестов, обладающих
свойством квази-выпуклости в уточнјнном варианте, без умножения
на константу, можно построить нейтральную перечислимую снизу
полумеру (в некотором точном смысле этого слова, при надлежащем
обобщении понятия теста на полумеры,
см.\cite{LevinUnif76,GacsExact80}).

Второе утверждение можно назвать \emph{квази-вогнутостью}; оно
показывает, что никаких новых случайных объектов относительно
смеси $P$ и $Q$ не появляется.

\begin{proof}
Первое утверждение является ослаблением
предложения~\ref{prop:computable-monotonicity}. Если $\lambda\ge
1/2$, то из этого предложения следует, что
    $\am(\lambda)\cdot\at(x,\lambda P+(1-\lambda)Q)\lem \at(x,P)$
(множитель $1/2$ можно включить в $\lem$). При $\lambda\le1/2$
верно аналогичное неравенство $\am(1-\lambda)\cdot\at(x,\lambda
P+(1-\lambda)Q)\lem \at(x,Q)$; при этом $\am(1-\lambda)\eqm
\am(\lambda)$.

Второе утверждение следует из того, что правая часть (как
функция от $x$) является тестом относительно любой смеси мер $P$
и $Q$, и можно воспользоваться усечением, чтобы сделать еј
равномерным тестом.
\end{proof}

Легко понять, что все эти утверждения существенно используют
вычислимость мер и коэффициентов в смеси. Соответствующие
контрпримеры легко построить, если осознать, что смесь мер может
быть как более сильным с вычислительной точки зрения оракулом,
чем каждая из них (если пропорции смешивания невычислимы), так и
наоборот. Например, разделим отрезок $[0,1]$ на две половины и
рассмотрим две меры $P$ и $Q$, равномерно распределјнные на этих
половинах. Их смесь с коэффициентами $\lambda$ и $1-\lambda$
делает число $\lambda$ заведомо неслучайным (поскольку оно может
быть вычислено относительно этой меры), хотя по одной из мер оно
вполне может быть случайным. (Взяв вместо $P$ и $Q$ их смеси,
скажем, с коэффициентами $1/3$ против $2/3$ и наоборот, можно
сделать $\lambda$ случайным по обеим мерам.)

В этом примере смесь содержит больше информации, чем каждая из
мер. Может быть и наоборот: свернјм отрезок $[0,1]$ с
равномерной мерой в окружность и разобьјм его на две
полуокружности точками $p$ и $p+1/2$. Тогда равномерные меры на
этих полуокружностях делают $p$ вычислимым относительно них и
потому неслучайным, а среднее этих мер есть равномерная мера на
окружности, относительно которой $p$ вполне может быть
случайным.

Заметим, что для слепой (безоракульной) случайности мы можем
безо всяких ограничений гарантировать, что множество случайных
относительно смеси $P$ и $Q$ точек будет объединением множеств
точек, случайных относительно $P$ и относительно $Q$. (В одну
сторону это следует из монотонности, которую мы уже отмечали. В
другую: если точка не случайна относительно $P$ и не случайна
относительно $Q$, то есть два теста, это доказывающих, и их
минимум будет перечислимым снизу тестом, доказывающим еј
неслучайность относительно смеси.) Было бы интересно
модифицировать понятие теста случайности, чтобы восстановить эти
свойства, сохранив другие желательные свойства (скажем,
существование универсального теста и тем самым понятие дефекта
случайности). Некоторые предложения такого рода имеются в
\cite{LevinUnif76,GacsExact80,LevinRandCons84}.

\subsection{\texorpdfstring
{Локальность}
{Locality}}
     \label{subsec:locality}\label{on-line}

Представим себе, что последовательность $\omega$ случайна по
равномерной мере и начинается с нуля. Теперь изменим эту меру на
последовательностях, начинающихся с единицы. Может оказаться,
что последовательность перестанет быть случайной, так как
значения меры теперь могут быть использованы как оракул
(например, последовательность может стать вычислимой
относительно новой меры). Но это выглядит странным, так как
изменение меры происходит не в той части, где лежит наша
последовательность.

Для слепой (безоракульной) случайности конкретно этот пример,
как легко видеть, невозможен (тест можно принудительно обнулить
на последовательностях, начинающихся с единицы), но в принципе
понятие теста зависит от меры не только вдоль последовательности
(не только от вероятности появления нуля и единицы после еј
начал).

\remm{А можно тут привести пример или доказать, что их не
бывает?}

Опять же для вычислимых мер ситуация лучше.

\begin{proposition}[преквенциальное свойство]
     \label{prop:prequential}
Пусть $P$ и $Q$~--- две вычислимые меры на пространстве $\Omega$
бесконечных последовательностей нулей и единиц, совпадающие на
всех началах некоторой последовательности $\omega$. Тогда эта
последовательность одновременно случайна или не случайна по
мерам $P$ и $Q$.
\end{proposition}

\begin{proof}
Это немедленно следует из критерия случайности в терминах
сложности начальных отрезков (теорема Левина--Шнорра) в любом из
его вариантов (теорема~\ref{th:levin-schnorr-prefix},
предложение~\ref{prop:levin-schnorr-monotone} и
следствие~\ref{cor:ample-access}).
\end{proof}

Для невычислимых мер это (как показывают примеры, аналогичные
рассмотренным в предыдущем разделе) неверно.

В случае вычислимых мер в произвольном пространстве имеет место
аналогичное утверждение, правда, с более сильным требованием: мы
предполагаем, что две меры совпадают на любых множествах,
содержащихся в некоторой окрестности последовательности
$\omega$. (В этом случае можно умножить тест на базисную
функцию, не изменив его в $\omega$ и сделав нулевым вне
окрестности совпадения.)

Вот ещј один способ, позволяющий получить заведомо
преквенциальное определения случайности, в котором дефект
случайности является функцией от самой последовательности и от
мер еј начальных отрезков. Для данной последовательности
$\omega$ и для данной последовательности $\{q(i)\}$
действительных чисел, для которых $1=q(0)\ge q(1)\ge
q(2)\ge\ldots\ge 0$, положим
     $$
\at'(\omega,q)=\inf \at(\omega,P),
     $$
где минимум берјтся по всем мерам $P$, для которых
$P(\omega(1:n))=q(n)$. Соответствующие множества (для случай
двоичных последовательностей) эффективно компактны, так что этот
минимум будет перечислимой снизу функцией от $\omega$ и
последовательности $q$. Если для последовательности $\omega$ и
мер $q(i)$ еј начальных отрезков значение $\at'(\omega,q)$
конечно, то последовательность $\omega$ можно назвать
\emph{преквенциально случайной}.

Другими словами, последовательность $\omega$ преквенциально
случайна по мере $P$, если существует (вообще говоря, другая)
мера $Q$, относительно которой $\omega$ случайна и которая
совпадает с $P$ на всех начальных отрезках $\omega$.

Требование преквенциальности связано с попытками перенести
понятия теории вероятностей и статистики в ситуацию
последовательных предсказаний членов последовательности,
ср.~\cite{ChernovShenVereshVovk08,ShenVovkPrequential10}.
Рассмотрим, например, прогноз погоды, в котором $\omega(n)$
означает, что в день $n$ идјт дождь. Метеобюро перед каждым днјм
указывает число $p(n)$, которое оно называет вероятностью дождя
в день $n$. (При этом на следующие дни никакого распределения
вероятностей не указывается. Другими словами, вместо глобального
распределения вероятностей бюро прогнозов указывает лишь
условные вероятности вдоль пути, соответствующего фактической
погоде.)

Можно ли оценить качество прогноза? Кажется, что в некоторых
ситуациях да: если, скажем, все предсказания близки к нулю
(скажем, меньше $10\%$), а большинство дней (скажем, более
$90\%$) были дождливые. (Говорят, что такой прогноз плохо
\emph{калиброван}.) Но, естественно, возможны и какие-то другие
виды несоответствий, не только частотные: общий вопрос состоит в
том, можно ли воспринимать данную последовательность как
полученную случайно с предсказанными вероятностями. (Другая
ситуация, где возникает такой вопрос, это оценка качества
датчика случайных битов, который выдајт бит с заказанным
распределением, на каждом шаге своим.)

Дополнительным обстоятельством при оценке качества предсказания
является то, что предсказатель может использовать разнообразную
информацию, доступную на момент предсказания (скажем, вечер
предыдущего дня), а не только предыдущие члены
последовательности $\omega$. Наличие такой информации должно
учитываться и при оценке качества предсказания.

В статье~\cite{ShenVovkPrequential10} обсуждаются подобные
вопросы и предлагаются различные варианты определений, в
частности, связанные с понятием мартингала, и доказывается
эквивалентность некоторых из них. Интересно было бы установить
связь и с равномерными тестами случайности в духе приведјнного
выше преквенциального определения дефекта (правда, вместо
вероятностей начальных отрезков тут возникают условные
вероятности, что не совсем то же самое, если они не отделены от
нуля).

\section{\texorpdfstring
{Вопросы}
{Questions}}

Мы уже отмечали некоторые вопросы, которые (на наш взгляд) было
бы интересно изучить. В этом разделе мы собрали ещј несколько
таких вопросов.

1. Рассмотрим следующий метод порождения последовательности
$\xi\in\Omega$, распределјнной в соответствии с данным
распределением $P$ на $\Omega$, при котором вероятности всех
слов ненулевые. Возьмјм случайную последовательность $\rho$
независимых равномерно распределјнных на $[0,1]$ случайных
чисел. После того как $\xi(1:n-1)$ уже построена, мы полагаем
$\xi(n)=1$ тогда и только тогда, когда
     $$
\rho(n) < P(\xi(1:n-1)1)/P(\xi(1:n-1)).
     $$
Если рассматривать это как вероятностный процесс, то выходное
распределение будет в точности $P$. Спрашивается, какие
последовательности можно получить на выходе, если начинать со
случайных по Мартин-Лјфу последовательностей действительных
чисел. (Можно проверить, что для вычислимых мер $P$ получаются в
точности случайные по Мартин-Лјфу относительно $P$
последовательности.)

2. Вспомним формулу для дефекта случайности по вычислимой мере:
     $$
\at(\omega,P)\eqm\sum_{x\pref \omega}\frac{\am(x)}{P(x)}.
     $$
Обе части имеют смысл при произвольном $P$, но они могут быть
различны. Причину этого мы уже обсуждали: небольшое изменение
меры почти не влияет на правую часть, но может изменить еј
вычислительную силу как оракула и существенно изменить левую.

Обозначим правую часть этого равенства через $t'(\omega,P)$.
Может быть, имеет смысл считать конечность $t'$ определением
случайности по невычислимым мерам? По крайней мере она будет
монотонной (от увеличения меры случайность будет только расти).
Относительно смеси мер она будет квази-выпуклой, более того, в
\cite{GacsDiss78} доказано, что $1/t'(\omega,P)$ является
вогнутой функцией от $P$.

Другое возможное определение дефекта случайности для бесконечной
последовательности $\omega$ по мере $P$ таково:
$\log\sup_{x\prec\omega} [\ap(x)/P(x)]$. Для вычислимых мер мы
вновь получаем определение, равносильное стандартному
определению Мартин-Лјфа. В работе \cite{GacsUnif05} показано,
что определјнные таким образом тесты случайности (а также
аналогичные, использующие монотонную сложность вместо априорной
вероятности) не обладают некоторым естественным свойством
(\emph{сохранения случайности})~--- в отличие от равномерных
тестов. В \cite{GacsDiss78} показано, что, с другой стороны, что
если в определении $t'$ в правой части рассмотреть сумму по всем
вычислимым функциям с конечным числом рациональных значений, а
не только характеристические функции множеств $x\Omega$, то
свойство сохранения информации выполняется.

3. Можно ли разумно определить дефект случайности
последовательности относительно произвольных мер, добившись его
монотонности (в каком-нибудь естественном смысле)? Например,
можно было бы потребовать
     $$
 P\le c\cdot Q \Rightarrow t(\omega,P)\ge t(\omega,Q)/c.
     $$
(отметим в качестве мотивировки, что правая часть приведјнной
выше формулы для дефекта обладает этим свойством). Можно ли
рассчитывать при этом на свойство квази-выпуклости? Некоторые
попытки такого рода предприняты в
\cite{LevinUnif76,GacsExact80}, а также в \cite{LevinRandCons84}

Свойство квази-вогнутости, видимо, обеспечить труднее (в этой
ситуации приведјнные контрпримеры выглядят более устойчивыми).

4. Стандартной процедурой в теории вычислимости является
релятивизация: некоторое множество $A$ объявляется разрешимым по
определению, и алгоритмам разрешается использовать ``оракул''
для этого множества (отвечающий на вопросы о принадлежности
ему). Это увеличивает класс вычислимых функций, но большинство
результатов теории вычислимости остаются верными. Более сложная
ситуация возникает, когда мы объявляем некоторое множество $E$
перечислимым по определению. Проблема тут в том, что его можно
перечислять в разном порядке, и разные перечисляющие его
``оракулы'' могут дать разные результаты. Тем не менее
существует естественное понятие \emph{перечислимого относительно
$E$} множества (как ещј говорят, множества, \emph{сводящегося по
перечислимости} к $E$). Приведјм соответствующее определение.
Пусть $W$ есть некоторое множество пар вида $\langle
x,S\rangle$, где $x$~--- натуральное число, а $S$~--- конечное
множество натуральных чисел. Будем считать, что множество $W$
перечислимо. Тогда для любого множества $E$ можно рассмотреть
множество $S(E,W)$, состоящее из всех $x$, при которых $\langle
x,S\rangle\in W$ при некотором $S\subset E$. (Неформально
говоря, пара $\langle x,S\rangle$ понимается как инструкция:
выдавать на выход $x$, обнаружив в перечислении $E$ все элементы
из списка $S$.) Добавление стандартного разрешающего оракула для
множества $A$ можно рассматривать как частный случай такой
сводимости, положив $E=\{2n\mid n\in A\} \cup \{2n+1\mid n\notin
A\}$.

В некоторых ситуациях можно пытаться обойтись такого рода
оракулом. Скажем, вполне можно говорить о перечислимой снизу
относительно $E$ функции, поскольку еј можно определить в
терминах перечислимых множеств. Но не очевидно, что определјнные
таким образом понятия обладают привычными для нас свойствами.

Можно ли утверждать (для произвольного $E$), что существует
максимальная перечислимая снизу относительно $E$ полумера? Можно
ли определить префиксную сложность с оракулом $E$ и будет ли она
совпадать с логарифмом максимальной полумеры (если таковая
существует)? Что будет, если дополнительно предположить, что $E$
есть множество всех базисных шаров в конструктивном метрическом
пространстве, содержащих некоторую точку?

(Для сравнения напомним, что можно определить $E$-вычислимую
функцию как функцию, график которой $E$-перечислим. При этом
выполняются некоторые знакомые свойства, например, композиция двух
$E$-вычислимых функций является $E$-вычислимой. Но, скажем,
утверждение о том, что всякое непустое $E$-перечислимое
множество является областью значений всюду определјнной
$E$-вычислимой функции, уже не гарантировано: при некоторых $E$
это не так.)

5. Можно пытаться обобщать понятие случайности в другом
направлении, рассматривая не вычислимые меры, а перечислимые
полумеры, то есть выходные распределения вероятностных машин,
выдающих бит за битом (и, возможно, выдающих конечную
последовательность с положительной вероятностью). Это предложил
Левин, имея в виду определить независимость двух
последовательностей $\alpha$ и $\beta$ как случайность пары
$(\alpha,\beta)$ относительно полумеры $\ap\times\ap$. (Такой
подход предполагает, что все последовательности случайны
относительно полумеры $\ap$.) Соответственно дефект случайности
пары $(\alpha,\beta)$ относительно $\ap\times\ap$ можно было бы
назвать количеством общей информации в последовательностях
$\alpha$ и $\beta$, по аналогии с конечными словами, где
взаимная информация слов $x$ и $y$, определяемая как
        $$
\KP(x)+\KP(y)-\KP(x,y)=- \log (\am(x)\times\am(y))-\KP(x,y),
        $$
выглядит как дефект случайности относительно $\am\times\am$.

Одна из возможностей для определения случайности
последовательности относительно полумеры $Q$ такая: потребовать
ограниченности отношения $\ap(x)/Q(x)$ для начальных отрезков
$x$ последовательности $\omega$. Другой вариант: рассматривать
случайные по Мартин-Лјфу относительно равномерной меры
последовательности случайных битов, использовать их как исходы
датчика случайных битов в вероятностной машине и смотреть, какие
последовательности могут получиться на выходе. Неизвестно,
совпадают ли это определения. Неизвестно также, корректно ли
второе из них (в том смысле, что двум вероятностным машинам с
одним и тем же выходным распределением соответствует одно и то
же множество образов случайных последовательностей). Для
вычислимых мер (машин, выдающих бесконечные последовательности с
вероятностью $1$) это действительно так.

6. (Этот вопрос задал С.~Симпсон) Можно ли предложить
естественное понятие тестов случайности для, скажем, 2-случайных
последовательностей? (Обычное определение на языке тестов
соответствует тестам, не являющимся полунепрерывными.)

\remm{А чем плохо определение просто с оракулом для
$\mathbf{0}'$?}

\section*{Благодарности}

Авторы благодарны своим коллегам, с котороми они обсуждали
рассматриваемые в статье вопросы, в первую очередь Л.~Левину, к
которому восходят многие из понятий этой статьи, А.~Буфетову и
А.~Клименко, а также В.~Вьюгину и другим участникам
колмогоровского семинара (на мехмате МГУ в Москве). Статья
написана при финансовой поддержке грантов NAFIT
ANR-08-EMER-008-01, RFBR 0901-00709-а.